\documentclass[a4paper,10pt]{article}
\usepackage{dsfont}
\usepackage{amssymb}
\usepackage{latexsym}
\usepackage{amsmath}
\usepackage{xcolor}
\usepackage{comment}
\usepackage{amsthm}% for \proof,\newtheorem
\usepackage{layout} % for layout 
\usepackage{ulem}
\usepackage{enumerate}%\begin{enumerate}[{(}i{)}]
\usepackage{bm}
\usepackage{bbm} %\mathbbm{ABCabc\nabla}
\usepackage[bbgreekl]{mathbbol} %\mathbb{abcABC\nabla\bbalpha\bbbeta\Delta  }
\usepackage{authblk}
\usepackage{setspace} % for gyoukai chousei
\newif\ifrs
\rstrue
\ifrs \usepackage{mathrsfs} \fi  % Use \mathscr{*}
\newif\ifcol
\coltrue %%% Activate "notice"
\newtheorem{theorem}{Theorem}[section]
\newtheorem{ass}[theorem]{Assumption}

\newtheorem{lemma}[theorem]{Lemma}
\newtheorem{definition}[theorem]{Definition}
\newtheorem{proposition}[theorem]{Proposition}

\newtheorem{remark}[theorem]{Remark}

\numberwithin{equation}{section}
\newtheorem{theorem*}{Theorem}
\newtheorem{ass*}[theorem*]{Assumption}
\newtheorem{note*}[theorem*]{Note}
\newtheorem{lemma*}[theorem*]{Lemma}
\newtheorem{definition*}[theorem*]{Definition}
\newtheorem{proposition*}[theorem*]{Proposition}
\newtheorem{corollary*}[theorem*]{Corollary}
\newtheorem{remark*}[theorem*]{Remark}
\newtheorem{example*}[theorem*]{Example}
\numberwithin{equation}{section}
\newif\ifcol
\coltrue
\ifcol
\newcommand{\colorr}{\color[rgb]{0.8,0,0}}

\newcommand{\colorn}{\color[rgb]{1,1,1}}

\else
\newcommand{\colorr}{\color{black}}% {{\color[rgb]{0.8,0,0}}
\newcommand{\colorn}{\color{black}}% {\color[rgb]{1,1,1}}
\fi
\excludecomment{en-text}
\includecomment{forrecord}
\excludecomment{forrecord}
\includecomment{scrap}
\excludecomment{scrap}
\includecomment{redundant}
\excludecomment{redundant}
\includecomment{futurework}
\excludecomment{futurework}

\includecomment{todelete}
\excludecomment{todelete}

\includecomment{expooneold}
\excludecomment{expooneold}
\includecomment{expoonenew}
\excludecomment{expoonenew}

\includecomment{makeitbetter}
\excludecomment{makeitbetter}
\def\mH{{\mathfrak H}}

\def\bd{\begin{description}}
\def\ed{\end{description}}

\def\D2{\bbD_{2,\infty-}}

\def\tj{{t_j}}
\def\tjm{{t_{j-1}}}

\def\B{{\bf B}}
\def\C{{\bf C}}
\def\D{{\bf D}}
\def\E{{\bf E}}

\def\O{{\bf O}}
\def\P{{\bf P}}

\def\T{{\bf T}}
\def\U{{\bf U}}

\def\cala{{\cal A}}

\def\cald{{\cal D}}
\def\cale{{\cal E}}
\def\calf{{\cal F}}

\def\calh{{\cal H}}
\def\cali{{\cal I}}

\def\calk{{\cal K}}
\def\call{{\cal L}}

\def\simleq{\ \raisebox{-.7ex}{$\stackrel{{\textstyle <}}{\sim}$}\ }

\def\ep{\epsilon}
\def\half{\frac{1}{2}}

\def\nn{\nonumber}
\def\be{\begin{equation}}
\def\ee{\end{equation}}
\def\bea{\begin{eqnarray}}
\def\eea{\end{eqnarray}}
\def\beas{\begin{eqnarray*}}
\def\eeas{\end{eqnarray*}}
\def\bi{\begin{itemize}}
\def\ei{\end{itemize}}

\def\bd{\begin{description}}
\def\ed{\end{description}}
\newcommand{\bbD}{{\mathbb D}}

\newcommand{\bbN}{{\mathbb N}}

\newcommand{\bbR}{{\mathbb R}}

\newcommand{\bbV}{{\mathbb V}}

\newcommand{\bbZ}{{\mathbb Z}}

\usepackage{mathrsfs}

\makeatletter
\newsavebox{\@brx}
\newcommand{\llangle}[1][]{\savebox{\@brx}{\(\m@th{#1\langle}\)}%
  \mathopen{\copy\@brx\kern-0.5\wd\@brx\usebox{\@brx}}}
\newcommand{\rrangle}[1][]{\savebox{\@brx}{\(\m@th{#1\rangle}\)}%
  \mathclose{\copy\@brx\kern-0.5\wd\@brx\usebox{\@brx}}}
\makeatother

\newcommand\vsp{\vspace{1cm}}
\newcommand\hsp{\hspace{1cm}}
\newcommand\vspsm{\vspace{5mm}}

\newcommand\vspssm{\vspace{2.5mm}}

\newcommand\sfz{{\sf z}}

\newcommand\sfi{{\tt i}}
\newcommand\tti{{\tt i}}

\newcommand{\stout}[1]{\ifmmode\text{\sout{\ensuremath{#1}}}\else\sout{#1}\fi}

\newcommand{\comm}[1]{{#1}}

\newcommand{\delc}[1]{{}}
\newcommand{\delb}[1]{{[]}}

\newcommand\rze{{r_{0}}}
\newcommand\ron{{r_{1}}}

\newcommand\son{{s_{1}}}
\newcommand\stw{{s_{2}}}

\newcommand\tpr{{t'}}

\newcommand\vpr{{v'}}

\newcommand\jon{{j_{1}}}
\newcommand\jtw{{j_{2}}}
\newcommand\jth{{j_3}}

\newcommand{\bnorm}[1]{{\bigl\| #1 \bigr\|}}
\newcommand{\Bnorm}[1]{{\Bigl\| #1 \Bigr\|}}
\newcommand{\bbnorm}[1]{{\biggl\| #1 \biggr\|}}
\newcommand{\snorm}[1]{{\left\| #1 \right\|}}
\newcommand{\norm}[1]{{\left\| #1 \right\|}}

\newcommand{\abs}[1]{{\left| #1 \right|}}

\newcommand{\babs}[1]{{\bigl| #1 \bigr|}}
\newcommand{\Babs}[1]{{\Bigl| #1 \Bigr|}}
\newcommand{\bbabs}[1]{{\biggl| #1 \biggr| }}

\newcommand{\cbr}[1]{{\left\{ #1 \right\} }}
\newcommand{\ncbr}[1]{{\{ #1 \} }}
\newcommand{\bcbr}[1]{{\bigl\{ #1 \bigr\} }}
\newcommand{\Bcbr}[1]{{\Bigl\{ #1 \Bigr\} }}
\newcommand{\bbcbr}[1]{{\biggl\{ #1 \biggr\} }}

\newcommand{\rbr}[1]{{\left( #1 \right) }}
\newcommand{\nrbr}[1]{{( #1 ) }}
\newcommand{\brbr}[1]{{\bigl( #1 \bigr) }}
\newcommand{\Brbr}[1]{{\Bigl( #1 \Bigr) }}
\newcommand{\bbrbr}[1]{{\biggl( #1 \biggr) }}

\newcommand{\sbr}[1]{{\left[ #1 \right] }}
\newcommand{\bsbr}[1]{{\bigl[ #1 \bigr] }}
\newcommand{\Bsbr}[1]{{\Bigl[ #1 \Bigr] }}
\newcommand{\bbsbr}[1]{{\biggl[ #1 \biggr] }}

\newcommand{\abr}[1]{{\left\langle #1 \right\rangle }}
\newcommand{\babr}[1]{{\bigl\langle #1 \bigr\rangle }}
\newcommand{\Babr}[1]{{\Bigl\langle #1 \Bigr\rangle }}

\newcommand{\dabr}[1]{{\llangle[] #1 \rrangle[] }}
\newcommand{\bdabr}[1]{{\llangle[\big] #1 \rrangle[\big] }}

\newcommand{\subotimes}[1]{{\underset{ #1 }{\otimes}}}

\newcommand{\clop}[1]{{\left[ #1 \right) }}

\newcommand\qtor{{\sf qTor}}
\newcommand\mqtor{{\sf mqTor}}
\newcommand\qtan{{\sf qTan}}

\newcommand\tr{\tj}
\newcommand\tl{\tjm}

\newcommand{\mR}{\mathfrak R}
\newcommand{\mS}{\mathfrak S}
\newcommand{\mT}{\mathfrak T}

\newcommand{\frakg}{{\mathfrak g}}

\newcommand{\scrC}{{\mathscr C}}

\newcommand{\bbf}{{\mathbb f}}

\newcommand{\bbone}{{\mathbb 1}}

\newcommand{\tfor}{\quad\text{for}\ }
\newcommand{\tforsm}{\ \text{for}\ }
\newcommand{\tfornsp}{\text{for}\ }
\newcommand{\tif}{\quad\text{if}\ }
\newcommand{\tifsm}{\ \text{if}\ }

\newcommand{\tand}{\quad\text{and}\quad}
\newcommand{\wtand}{\qquad\text{and}\qquad}
\newcommand{\tandsm}{\ \text{and}\ }

\newcommand{\torsm}{\ \text{or}\ }
\newcommand{\twith}{\quad\text{with}\quad}

\newcommand{\tas}{\quad\text{as}\quad}
\newcommand{\tassm}{\ \text{as}\ }

\newcommand{\mpl}{{m+1}}
\newcommand{\ilamv}{{\lambda_v}}
\newcommand{\blamc}{{\bar\lambda_C}}
\newcommand{\kerfvn}[2]{{f_{#1,n}(#2)}}
\newcommand{\charf}[1]{{1_{#1}}}
\newcommand{\charfn}[1]{{1_{n,#1}}}

\newcommand{\vc}[2]{{v^{#1}_{#2}}}

\newcommand{\edgeWt}{{\boldsymbol\theta}}

\newcommand{\vertWt}{{\boldsymbol{q}}}
\newcommand{\vertwtlow}{{\boldsymbol{q}}}

\newcommand{\barq}{{{\bar  q}}}
\newcommand{\bartheta}{{ {\bar  \theta}}}

\newcommand{\DLamCali}{\cali_n^{(\lambda)}}

\newcommand{\betacycle}[2]{\bar\beta_{#1}(#2)}

\newcommand{\Comp}{\scrC}
\newcommand{\cpr}{C_0}
\newcommand{\cppr}{C_1}
\newcommand{\hatc}{\widehat C}

\newcommand{\funcSec}{\cali_n^\T}

\newcommand{\numCBplus}{{\big| \CBplus\big|}}
\newcommand{\cplusrest}{{\Comp_+^{\vee}}}
\newcommand{\CBplus}{\Comp_\B^{\vee}}
\newcommand{\CUBplus}{\Comp_{\U\B}^{\vee}}
\newcommand{\CB}{\Comp_\B}
\newcommand{\CBo}{\Comp_\B^1}
\newcommand{\CBt}{\Comp_\B^{2+}}

\newcommand{\Ctp}{\Comp^{2+}}
\newcommand{\Ctpz}{\Comp^{2+}_{0}}
\newcommand{\Ctpo}{\Comp^{2+}_{\O}}
\newcommand{\Ctppe}{\Comp^{2+}_{\P\E}}
\newcommand{\cpe}{\Ctppe}

\newcommand{\Cpearg}[1]{\Comp^{2+}_{\P\E(#1)}}
\newcommand{\Ctpa}{\Comp^{2+}_{\bf a}}
\newcommand{\Ctpb}{\Comp^{2+}_{\bf b}}

\newcommand{\Cqzero}{\Comp_{0}}
\newcommand{\Cqom}{\Comp_{1-}}

\newcommand{\Cszero}{\Comp^0}
\newcommand{\Csone}{\Comp^1}
\newcommand{\Ctz}{\Comp^2_{0}}
\newcommand{\Ctt}{\Comp^2_{2}}
\newcommand{\Ctto}{\Comp^2_{2,1}}
\newcommand{\Cttt}{\Comp^2_{2,2+}}

\newcommand{\tensorc}{{\Comp}_{\T}}
\newcommand{\nottenc}{{\Comp}_{\U\T}}
\newcommand{\twochaos}{{\Comp}_{\bf\romanTwo}}

\newcommand{\tensorg}[1]{{\Comp}_{\T,#1}}

\newcommand{\cblam}[1]{{\Comp^{\bar\lambda}_{#1}}}

\newcommand{\shiftg}[2]{(#1)_{+ #2}}

\newcommand{\olog}{\bar O}

\newcommand{\romanTwo}{I\hspace{-0.8pt}I}

\newcommand{\expoT}{e_{\T}}

\newcommand{\cmpns}{M_{0,n}}
\newcommand{\pertur}[1]{N_{#1,n}}

\newcommand{\qtorker}[2]{\calk_{#1,#2}}
\newcommand{\dotd}{\dot{d}}
\newcommand{\ddotd}{\ddot{d}}

\newcommand{\kerone}[2]{z_{#1}(#2)}
\newcommand{\zndelta}{Z_N^\Delta }

\newcommand{\maxs}[1]{\bar s(#1)}

\newcommand{\singlegraph}[2]{
  (\cbr{#1},0,#2)
}

\newcommand{\graphcycle}{
  \begin{tikzpicture}
    \node[shape=circle,draw=black,inner sep=1pt] (1) at (120:1){0};
    \node[shape=circle,draw=black,inner sep=1pt] (2) at (180:1) {0};
    \node[shape=circle,draw=black,inner sep=1pt] (3) at (0:1) {0};
    \node[shape=circle,draw=black,inner sep=1pt] (4) at (60:1) {0};
    \draw[thick] (1) to [out=210,in=90,looseness=0.5] (2);
    \draw[dashed,thick] (2) to [out=270,in=270,looseness=1.6] (3);
    \draw[thick] (3) to [out=90,in=330,looseness=0.5] (4);
    \draw[thick] (4) to [out=150,in=30,looseness=0.5] (1);
  
    \node  at (150:0.8) {1};
    \node  at (30:0.8) {1};
    \node  at (90:0.8) {1};
    \node  at (125:1.45) {$v^G_{1}$};
    \node  at (170:1.4) {$v^G_{2}$};
    \node  at (10:1.7) {$v^G_{I(G)-1}$};
    \node  at (55:1.5) {$v^G_{I(G)}$};
   \end{tikzpicture}
}

\newcommand{\graphcycletwo}{
  \begin{tikzpicture}
    \node[shape=circle,draw=black,inner sep=1pt] (1) at (-0.5,0) [label=below:$v^G_{1}$]{0};
    \node[shape=circle,draw=black,inner sep=1pt] (2) at ( 0.5,0) [label=below:$v^G_{2}$]{0};
    \node  at (60:1.5) {};
    \node  at (240:1.5) {};
    \draw[thick] (1) -- node[midway, above] {2} (2);
   \end{tikzpicture}
}

\newcommand{\graphpath}{
  \begin{tikzpicture}
    \node[shape=circle,draw=black,inner sep=1pt] (1) at (120:1){1};
    \node[shape=circle,draw=black,inner sep=1pt] (2) at (180:1) {0};
    \node[shape=circle,draw=black,inner sep=1pt] (3) at (0:1) {0};
    \node[shape=circle,draw=black,inner sep=1pt] (4) at (60:1) {1};
    \draw[thick] (1) to [out=210,in=90,looseness=0.5] (2);
    \draw[dashed,thick] (2) to [out=270,in=270,looseness=1.6] (3);
    \draw[thick] (3) to [out=90,in=330,looseness=0.5] (4);
    \node  at (150:0.8) {1};
    \node  at (30:0.8) {1};
    \node  at (125:1.45) {$v^G_{1}$};
    \node  at (170:1.4) {$v^G_{2}$};
    \node  at (10:1.7) {$v^G_{I(G)-1}$};
    \node  at (55:1.5) {$v^G_{I(G)}$};
   \end{tikzpicture}
}

\newcommand{\graphpathtwo}{
  \begin{tikzpicture}
    \node[shape=circle,draw=black,inner sep=1pt] (1) at (-.5,0) [label=below:$v^G_{1}$]{1};
    \node[shape=circle,draw=black,inner sep=1pt] (2) at ( .5,0) [label=below:$v^G_{2}$]{1};
    \node  at (60:1.5) {};
    \node  at (240:1.5) {};
    \draw[thick] (1) -- node[midway, above] {1} (2);
   \end{tikzpicture}
}

\newcommand{\graphMtwon}[2]{
  \begin{tikzpicture}
    \node[shape=circle,draw=black,inner sep=1pt] (1) at (0,0) [label=below:$#1$]{1};
    \node[shape=circle,draw=black,inner sep=1pt] (2) at (1,0) [label=below:$#2$]{1};
    \draw[thick] (1) -- node[midway, above] {1} (2);
   \end{tikzpicture}
}

\newcommand{\graphMtwonon}[2]{
  \begin{tikzpicture}
    \node[shape=circle,draw=black,inner sep=1pt] (1) at (0,0) [label=below:$#1$]{0};
    \node[shape=circle,draw=black,inner sep=1pt] (2) at (1,0) [label=below:$#2$]{0};
    \draw[thick] (1) -- node[midway, above] {2} (2);
   \end{tikzpicture}
}

\newcommand{\graphMtwtw}[2]{
  \begin{tikzpicture}
    \node[shape=circle,draw=black,inner sep=1pt] (1) at (0,0) [label=below:$#1$]{2};
    \node[shape=circle,draw=black,inner sep=1pt] (2) at (2/3,0) [label=below:$#2$]{1};
   \end{tikzpicture}
}

\newcommand{\graphMtwth}[2]{
  \begin{tikzpicture}
    \node[shape=circle,draw=black,inner sep=1pt] (1) at (0,0) [label=below:$#1$]{1};
    \node[shape=circle,draw=black,inner sep=1pt] (2) at (2/3,0) [label=below:$#2$]{1};
   \end{tikzpicture}
}

\newcommand{\graphMtwfo}[2]{
  \begin{tikzpicture}
    \node[shape=circle,draw=black,inner sep=1pt] (1) at (0,0) [label=below:$#1$]{0};
    \node[shape=circle,draw=black,inner sep=1pt] (2) at (1,0) [label=below:$#2$]{1};
    \draw[thick] (1) -- node[midway, above] {1} (2);
   \end{tikzpicture}
}

\newcommand{\graphMtwtwth}[2]{
  \begin{tikzpicture}
    \node[shape=circle,draw=black,inner sep=1pt] (1) at (0,0) [label=below:$#1$]{1};
    \node[shape=circle,draw=black,inner sep=1pt] (2) at (1,0) [label=below:$#2$]{0};
    \draw[thick] (1) -- node[midway, above] {1} (2);
   \end{tikzpicture}
}

\newcommand{\graphMthon}[3]{
  \begin{tikzpicture}
    \node[shape=circle,draw=black,inner sep=1pt] (1) at (0,0) [label=below:$#1$]{1};
    \node[shape=circle,draw=black,inner sep=1pt] (2) at (1,0) [label=below:$#2$]{1};
    \node[shape=circle,draw=black,inner sep=1pt] (3) at (2,0) [label=below:$#3$]{1};
    \draw[thick] (1) -- node[midway, above] {1} (2);
   \end{tikzpicture}
}

\newcommand{\graphMthtw}[3]{
  \begin{tikzpicture}
    \node[shape=circle,draw=black,inner sep=1pt] (1) at (0,0) [label=below:$#1$]{1};
    \node[shape=circle,draw=black,inner sep=1pt] (2) at (1,0) [label=below:$#2$]{0};
    \node[shape=circle,draw=black,inner sep=1pt] (3) at (2,0) [label=below:$#3$]{1};
    \draw[thick] (1) -- node[midway, above] {1} (2);
    \draw[thick] (2) -- node[midway, above] {1} (3);
   \end{tikzpicture}
}

\newcommand{\graphMthth}[3]{
  \begin{tikzpicture}
    \node[shape=circle,draw=black,inner sep=1pt] (1) at (0,0) [label=below:$#1$]{2};
    \node[shape=circle,draw=black,inner sep=1pt] (2) at (2/3,0) [label=below:$#2$]{1};
    \node[shape=circle,draw=black,inner sep=1pt] (3) at (4/3,0) [label=below:$#3$]{1};
   \end{tikzpicture}
}

\newcommand{\graphMthfi}[3]{
  \begin{tikzpicture}
    \node[shape=circle,draw=black,inner sep=1pt] (1) at (0,0) [label=below:$#1$]{2};
    \node[shape=circle,draw=black,inner sep=1pt] (2) at (3/3,0) [label=below:$#2$]{0};
    \node[shape=circle,draw=black,inner sep=1pt] (3) at (6/3,0) [label=below:$#3$]{1};
    \draw[thick] (2) -- node[midway, above] {1} (3);
   \end{tikzpicture}
}

\newcommand{\graphonept}[2]{
  \begin{tikzpicture}
    \node[shape=circle,draw=black,inner sep=1pt] (1) at (0,0) [label=below:$#1$]{#2};
   \end{tikzpicture}
}

\newcommand{\graphExample}[9]{
  \begin{tikzpicture}
    \node[shape=circle,draw=black,inner sep=1pt] (1) at (0,0) [label=below:$#7$]{#1};
    \node[shape=circle,draw=black,inner sep=1pt] (2) at (1.5,0) [label=below:$#8$]{#2};
    \node[shape=circle,draw=black,inner sep=1pt] (3) at (3.0,0) [label=below:$#9$]{#3};
    \draw[thick] (1) -- node[midway, above] {#4} (2);
    \draw[thick] (2) -- node[midway, above] {#5} (3);
   \end{tikzpicture}
}

\usepackage{cases}
\usepackage{tikz} 
\usepackage{subcaption}
\usepackage{here}

\DeclareMathOperator\supp{supp}
\DeclareMathOperator*{\argmax}{arg\,max}

\DeclareMathOperator*{\Dom}{Dom}

\begin{document}
%\setstretch{1.5} % gyoukan settei for all pages
%\pagecolor{black}\color{white} %white on black

%%%%%%%%%%%%%%%%%%%%%%%%%%%%%%%%%%%%%%%%%
%%%%%%%%%%%%%%%%%%%%%%%%%%%%%%%%%%%%%%%%%
\title{Order estimate of functionals related to fractional Brownian motion 
and %its application to 
asymptotic expansion of the quadratic variation 
%of stochastic differential equation% driven by fBm
of fractional stochastic differential equation
%arXiv
%Expansion of the Asymptotically Conditionally Normal Law %ISM Research Memo
%Submit
\footnote{
This work was in part supported by 
Japan Science and Technology Agency CREST JPMJCR14D7, JPMJCR2115; 
Japan Society for the Promotion of Science Grants-in-Aid for Scientific Research 
No. 17H01702 (Scientific Research); 
%No. 24340015 (Scientific Research), 
%Nos. 24650148 and 
%No. 26540011 (Challenging Exploratory Research); 
%the Global COE program ``The Research and Training Center for New Development in Mathematics'' of the Graduate School of Mathematical Sciences, University of Tokyo; 
%NS Solutions Corporation; 
and by a Cooperative Research Program of the Institute of Statistical Mathematics. %%
%The main parts in this paper were presented at 
%International conference ``Statistique Asymptotique des Processus Stochastiques VII'', Universit\'e du Maine, Le Mans, March 16-19, 2009, 
%MSJ Spring Meeting 2010, March 24-27, 2010, Keio University, Mathematical Society of Japan, 
%and 
%International conference ``DYNSTOCH Meeting 2010'', Angers, June 16-19, 2010. 
%The author thanks to the organizers of the meetings for opportunities of the talks. 
}
}
%%%% Author %%%%%%%%%%%%%%%%%%%%%%%
%%%%%%%%%%%%%%%%%%%%%%%%%%%%%%%
\author{Hayate Yamagishi and Nakahiro Yoshida}
\affil{Graduate School of Mathematical Sciences, University of Tokyo
\footnote{Graduate School of Mathematical Sciences, 
University of Tokyo: 3-8-1 Komaba, Meguro-ku, Tokyo 153-8914, Japan. 
e-mail: 
yhayate@ms.u-tokyo.ac.jp,
nakahiro@ms.u-tokyo.ac.jp}}
\affil{CREST, Japan Science and Technology Agency%\footnote{}
}
%%%% Authors %%%%%%%%%%%%%%%%%%%%%%%
%%%%%%%%%%%%%%%%%%%%%%%%%%%%%%%
%\author[1,3]{Masayuki Uchida}
%\author[2,3]{Nakahiro Yoshida}
%\affil[1]{Graduate School of Engineering Science, Osaka University
%\footnote{Graduate School of Engineering Science, Osaka University: Toyonaka, Osaka 560-8531, Japan}
%        }
%\affil[2]{Graduate School of Mathematical Sciences, University of Tokyo
%\footnote{Graduate School of Mathematical Sciences, University of Tokyo: 3-8-1 Komaba, Meguro-ku, Tokyo 153-8914, Japan. e-mail: nakahiro@ms.u-tokyo.ac.jp}
%        }
%\affil[3]{CREST, Japan Science and Technology Agency
%%\footnote{}
%        }
%%%% Date %%%%%%%%%%%%%%%%%%%%%%%%
%%%%%%%%%%%%%%%%%%%%%%%%%%%%%%%
%\date{September 14, 2010, \\
%%Revised February 25, 2012
%}
%%%%%%%%%%%%%%%%%%%%%%%%%%%%%%%
\maketitle
%%%%%%%%%%%%%%%%%%%%%%%%%%%%%%%
%%%%%%%%%%%%%%%%%%%%%%%%%%%%%%%
\ \\
{\it Summary}\;
We derive an asymptotic expansion for the quadratic variation of a stochastic process satisfying a stochastic differential equation driven by a fractional Brownian motion,  based on the theory of asymptotic expansion of Skorohod integrals converging to a mixed normal limit. 
In order to apply the general theory, it is necessary to estimate functionals that are a randomly weighted sum of products of multiple integrals of the fractional Brownian motion, in expanding the quadratic variation and identifying the limit random symbols. 
To overcome the difficulty, we introduce two types of exponents by means of the {\it weighted graphs} capturing the structure of the sum in the functional, and investigate how the exponents change by the action of the Malliavin derivative and its projection. 

%%%%%%%%%%%%%%%%%%%%%%%%%%%%%%%%%%%%%%%%%%%%%%%%%%%%%%%%%%%%
%%%%%%%%%%%%%%%%%%%%%%%%%%%%%%%%%%%%%%%%%%%%%%%%%%%%%%%%%%%%
%%%%%%%%%%%%%%%%%%%%%%%%%%%%%%%%%%%%%%%%%%%%%%%%%%%%%%%%%%%%
%%%%%%%%%%%%%%%%%%%%%%%%%%%%%%%%%%%%%%%%%%%%%%%%%%%%%%%%%%%%
%%%%%%%%%%%%%%%%%%%%%%%%%%%%%%%%%%%%%%%%%%%%%%%%%%%%%%%%%%%%

\section{Introduction}
{Asymptotic expansion has been recognized as a basic technique in statistics and other fields. 
It was initiated for independent variables and has been generalized to various dependent structures. 
As will explained in the last part of Introduction, historically the first attempts to dependent structures were directed 
to mixing Markov processes since it is possible to develop a theory parallel to that in the independent case 
by taking advantage of the Markovian property and the mixing property.

Developments in 70s-80s in martingale central limit theorems (cf. Jacod and Shiryaev \cite{JacodShiryaev2003}) 
and their applications to statistics for stochastic processes 
motivated studies of asymptotic expansion of martingales. 
Mykland \cite{Mykland1992,Mykland1993} and 
Yoshida \cite{Yoshida1997, yoshida2001malliavin} studied asymptotic expansion for martingales 
about moments and distributions, respectively, as a refinement of the martingale central limit theorem. 
Though the first application was to ergodic diffusion processes, i.e., Markovian processes, 
the martingale expansion does not need the Markovian property and as a matter of fact, and 
it was also applied to the volatility parameter estimation from the temporally-discrete samples in a finite-time horizon, 
in the case where the limit distribution of the estimator is Gaussian. 
Differently from the mixing approach, a non-Gaussian distribution appears in the higher-order term, 
which suggests a possibility of extension of the theory to the so-called non-ergodic statistics. 

Quantitative finance treats high frequency data, i.e., temporally discretely sampled data over a fixed time interval, 
and the realized volatility is one of the very basic statistic there. 
The error of the realized volatility is asymptotically mixed normal as the number of observations tends to infinity. 
Together with its applications, the martingale central limit theorem of mixture type was investigated 
by Jacod \cite{Jacod1997} and others. 
The martingale expansion was generalized to martingales with a mixed normal limit by Yoshida \cite{yoshida2013martingale} 
(updated by arXiv:1210.3680v3). 
The random symbols called {\it tangent} and {\it torsion} play an important role. 
The tangent corresponds to the correction term in the classical martingale expansion, but the torsion to the martingale newly appears. 
The asymptotic expansion formula is given in terms of the random symbols involving the Malliavin derivatives of functionals 
though the mixed normal central limit theorem can be handled by It\^o calculus. 
The torsion disappears when the asymptotic variance is deterministic, as it is the classical case of the martingale central limit theorem, 
and the expansion formula returned to the classical one. 
This martingale expansion was applied to 
the realized volatility of a solution to a stochastic differential equation (SDE) by Yoshida \cite{yoshida2012asymptotic}, 
the power variation by Podolskij and Yoshida \cite{podolskij2016edgeworth}, 
the pre-averaging estimator by Podolskij et al. \cite{podolskij2017edgeworth}, 
and to the Euler-Maruyama scheme by Podolskij et al. \cite{podolskij2018edgeworthToappear}. 

The martingale central limit theorem of mixture type 
can be reproduced by the stable limit theorem for Skorohod integrals presented by Nourdin, Nualart and Pecati \cite{nourdin2016quantitative}. 
By extending their interpolation formula to the second-order, 
Nualart and Yoshida \cite{nualart2019asymptotic} gave an asymptotic expansion for Skorohod integrals, 
and applied it to the quadratic variation of a fractional Brownian motion. 
The expansion formula is specified with newly introduced random symbols 
{\it quasi-tangent} and {\it quasi-torsion}. 
The expansion of the realized volatility can be reproduced by this method since the Skorohod integral generalizes the It\^o integral. 
However, the correspondence between \{tangent, torsion\} and \{quasi-tangent, quasi-torsion\} is not obvious. 
In fact, in this application, the quasi-tangent vanishes though the tangent does not. 
}%
\begin{en-text}
To apply the Skorohod integral approach, the essential steps are identification of the limit of the random symbols and 
estimation of the Sobolev norms of variables obtained by repeated Malliavin derivatives and their projections. 
Generally, both steps were not easy so far because of lack of a methodology for systematic assessment of 
the magnitude of random polynomials of multiple Wiener integrals with respect to the fractional Brownian motion. 
In this paper, we will approach this problem. 
\end{en-text}

In this paper,
we consider the following one-dimensional SDE driven by a fractional Brownian motion (fBm):
\begin{align}\label{210430.1615}
X_t =X_0+ 
\int_0^t V^{[2]}(X_s)ds+\int_0^t V^{[1]}(X_s) dB_s
\end{align}
for $t\in[0,T]$ with $X_0\in\bbR$ and a fixed terminal value $T>0$.
The Hurst parameter of the driving fBm 
$B=(B_t)_{t\in[0,T]}$ is restricted to $H\in(1/2,3/4)$.
The stochastic integral in (\ref{210430.1615}) is a pathwise Young integral, and 
it is known that 
there exists a unique solution under some regularity conditions on $V^{[i]}$ ($i=1,2$), 
as Nualart and Rascanu \cite{rascanu2002differential} detailed.
The quadratic variation (realised volatility) of the process $X_t$
is defined by
\begin{align}\label{220420.1130}
  \bbV_n =n^{2H-1}\sum_{j=1}^n (\Delta_jX)^2,
\end{align}
where
$\Delta_jX=X_\tj-X_\tjm$ with 
$\tj=jT/n$ for $n\in\bbN$ and $j\in\cbr{1,...,n}$.
This estimator converges in $L^p$ to
\begin{align*}
  \bbV_\infty = T^{2H-1} \int_0^T V^{[1]}(X_t)^2dt, 
\end{align*}
and 
Le{\'o}n and Lude{\~n}a \cite{leon2007limits} proved that
the rescaled error of the convergence 
$Z_n=n^{1/2}\brbr{\bbV_n-\bbV_\infty}$
%\begin{align}
%  Z_n=n^{1/2}\brbr{\bbV_n-\bbV_\infty}
%  \label{220404.1601}
%\end{align}
weakly converges to  %as $n$ goes to $\infty$,
a centered mixed normal distribution.
This convergence is a stable convergence in fact. 
% $MN(0,G_\infty)$ with conditional variance

{%\colorr 
In the case where the Hurst index $H$ is equal to $1/2$, the asymptotic expansion for $Z_n$ 
was approached by \cite{yoshida2012asymptotic} supported by \cite{yoshida2013martingale} (updated by arXiv:1210.3680v3), 
where the author developed a method of asymptotic expansion for martingales with a mixed normal limit. 
Although the functional corresponding to $\bbV_n$ in that paper also converges to a mixed normal distribution,
the technique used there %to obtain the asymptotic expansion %, namely the martingale embedding method, 
cannot apply $Z_n$ for the fBm,
since there is no obvious martingale structure in $\bbV_n$ of (\ref{220420.1130}).
Our strategy of expansion will rely on 
the asymptotic expansion for Skorohod integrals presented by \cite{nualart2019asymptotic}, 
that used Malliavin calculus as their central tools
and a second-order interpolation formula in the frequency domain to expand the characteristic function of the functional in question.
}

\begin{en-text}
{weighted quad variation. 
The main term of this error term behaves like the weighted quadratic variation of fBm.
}
\end{en-text}

{%\colorr 
To apply the asymptotic expansion for Skorohod integrals, the essential steps are identification of the limit of the random symbols and 
estimation of the Sobolev norms of variables obtained by repeated Malliavin derivatives and their projections. 
Generally, both steps were not easy so far because of lack of a methodology for systematic assessment of 
the magnitude of random polynomials of multiple Wiener integrals with respect to the fractional Brownian motion.}
%In this paper, we will approach this problem.
%%
\begin{en-text}
Although this paper is an application of the general theory developed in \cite{nualart2019asymptotic},
verifying the assumption given in the paper in this context is rather a heavy task.
In particular, functionals in a certain form appear in the calculation repeatedly, 
and we need to find out the order of each functional as $n$ goes to $\infty$. 
\end{en-text}
To deal with this problem, we introduce two exponents which give an upper bound of the order of functionals 
with some easy calculation.
These exponents are applicable to a wider class of functionals related to fBm than treated in this specific problem of 
the asymptotic expansion of the quadratic variations. 
{%\colorr 
Our exponents in a sense generalize the exponent defined by \cite{yoshida2020asymptotic} 
for functionals of a Brownian motion 
(i.e. fBm of $H=1/2$) having a structure similar to that we encounter in this paper.
The stability and contraction properties under the Malliavin derivatives and the projections will be clarified 
with the help of the exponents of functionals.

The organization of the paper is as follows.
Section \ref{220512.0943} gives the statement of the main result of the asymptotic expansion and 
outlines the strategies of this paper, 
with a overview of the general theory of asymptotic expansion for Skorohod integrals given by \cite{nualart2019asymptotic}.
Section \ref{210429.1655} is devoted to the theory of exponents for functionals.
The authors hope that this section is of interest even independently from other parts. 
In Section \ref{220317.2030}, 
we first give a stochastic expansion of $\bbV_n$ and identify the random symbols appearing in the expansion, and then 
we prove the main result by verifying the assumptions of the general theory in the present situation. 
To carry out this plot, we extensively use the exponents introduced in Section \ref{210429.1655}. 
Section \ref{220405.1158} is collecting some technical lemmas used in the preceding sections.
To clarify the background of the theory we will work with and to understand what we will do, the following paragraphs give some historical remarks about the initiation of the asymptotic expansion theory for independent models and the generalizations to stochastic processes.}

{%\colorr 
%Asymptotic expansion has been recognized as a basic technique in statistics and other fields. It was initiated for independent variables and has been generalized to various dependent structures. 
%
The basic references for independent cases include
Cram\'er \cite{cramer1928composition, cramer2004random}, 
Gnedenko and Kolmogorov \cite{gnedenko1954limit}, 
Bhattacharya \cite{Bhattacharya1971}, 
Petrov \cite{Petrov1975}, 
Bhattacharya and Ranga Rao \cite{BhattacharyaRanga1976}, 
Bhattacharya and Ghosh \cite{BhattacharyaGhosh1978}, and Bhattacharya et al. \cite{bhattacharya2016course}. 
Applications of the asymptotic expansion have spread to major areas of statistics as their theoretical basis.
Asymptotic theory concerning higher-order optimality of statistical inference 
cannot exist without asymptotic expansion: 
Akahira and Takeuchi \cite{AkahiraTakeuchi1981}, Pfanzagl \cite{Pfanzagl1985} and Ghosh \cite{Ghosh1994}, 
if picked up from the literature in early days. 
%if we list some of literature in the early days. 
%
One of the fields where asymptotic expansion techniques have been well developed is multivariate analysis; see 
Okamoto \cite{okamoto1963asymptotic}, Anderson \cite{anderson1962introduction}, 
Fujikoshi, Ulyanov and Shimizu \cite{fujikoshi2011multivariate}. 
%Informative textbooks for asymptotic expansion in the multivariate analysis are 
%
The reader is referred to Efron \cite{efron1979bootstrap} and Hall \cite{Hall1992} 
for the theory of Bootstrap methods and asymptotic expansion. 
A variant of the ordinary asymptotic expansion is 
the saddle-point approximation; see 
Barndorff-Nielsen \cite{barndorff2012parametric}, 
Jensen \cite{jensen1995saddlepoint} and 
Pace and Salvan \cite{PaceSalvan1997}. 
Information geometry, that is a differential geometry on a parameter space as a Riemannian manifold 
with the $\alpha$-affine connection, 
stems from a geometric interpretation of the expansion formula for statistical estimators (Amari \cite{Amari1985}). 
Asymptotic expansion serves as a basic tool of theory of 
information criteria for 
model selection as well as prediction problems; cf.  
Konishi and Kitagawa \cite{KonishiKitagawa1996}, 
Uchida and Yoshida \cite{UchidaYoshida2001,UchidaYoshida2004}, Komaki \cite{Komaki1996}.

After successful applications of asymptotic expansion to independent models, 
it was a natural attempt to generalize this method to dependent cases. 
The mixing property is a useful structure as an extension of independency since it determines the higher-order structure of the asymptotic distribution 
of an additive functional of the mixing process. 
G\"otze and Hipp \cite{GotzeHipp1983, GotzeHipp1994} completed 
the theory for discrete-time approximately Markovian processes having a mixing property. 
Kusuoka and Yoshida \cite{KusuokaYoshida2000} and Yoshida \cite{Yoshida2004} 
presented asymptotic expansion for 
for continuous-time mixing ($\ep$-) Markov processes, 
and it was applied to statistics for stochastic processes by 
Sakamoto and Yoshida 
\cite{
%SakamotoYoshida1996, 
%SakamotoYoshida1998a, 
SakamotoYoshida2003, 
SakamotoYoshida2004, 
%SakamotoYoshida2008, 
SakamotoYoshida2009, 
SakamotoYoshida2010} 
and 
Uchida and Yoshida \cite{UchidaYoshida2001,UchidaYoshida2004}. 
Regularity of the distribution is essential for asymptotic expansion. 
In other words, a fast decay of the characteristic function of an additive functional of the process is necessary to validate the Fourier inversion. 
The Markovian property enables us to reduce the estimate of the characteristic functional of the additive functional into the estimate of 
the characteristic function of summands. 
The Malliavin calculus takes the place of 
the classical Cram\'er condition on the regularity of the distribution 
to obtain a decay of each temporally local characteristic function. 
We refer the reader to Yoshida \cite{yoshida2016asymptotic} for a short exposition of  
asymptotic expansion and additional references therein. 
At this point, it would be worth mentioning that recently Tudor and Yoshida provided 
asymptotic expansion for general Wiener functionals in \cite{tudor2019asymptotic}, and 
arbitrary order of asymptotic expansion for Wiener functionals in 
\cite{tudor2019high} in the central limit case. 
Their scheme was applied to a stochastic wave equation in \cite{tudor2019high} 
and to the variation of a mixed fractional Brownian motion in \cite{tudor2020asymptotic}.
They characterized the expansion formula by the gamma factors of the functional. 
It is a natural way because the gamma factors correspond to the cumulants, that determined the formula in the classical cases under mixing. 
Logically, we cannot exploit the Markovian property, however, they showed the non-degeneracy of the asymptotic variance is sufficient 
to validate the expansion, without any additional non-degeneracy condition like Cram\'er's condition. 

The methodology we will rely on is stemming from the martingale expansion.
Differently from the mixing approach, this method has more freedom in that the cumulants do not admit classical regulations, nor determine the expansion. 
For example, in our expansion formula, a term corresponding to the fifth-order appears but it was impossible in the classical expansion since 
the fifth-order cumulant is always negligible in the classical setting. 
The classical frame may seem to have nothing to do with the martingale-Skorohod integral approach in this sense. 
However, as seen in the main body of this article, a kind of temporally local chaotic expansion plays an essential role in the analysis. 
According to the definition of the exponents, we can intuitively say the degree of the chaos is reflecting the speed of mixing, and 
hence, the method of exponents is quite important to bridge the two approaches, in other words, to incorporate the classical computational method into the new frame of asymptotic expansion. 
}

\section{Main results and strategies}
\label{220512.0943} 
In this section, firstly we state the statement of the main result and
overview the strategies to prove it.
\subsection{Asymptotic expansion of quadratic variations of SDE driven by fBm}

Let $(X_t)_{t\in[0,T]}$ the solution %\redb{Under what conditions?? Nualart Sauss..}
to SDE (\ref{210430.1615}).
We define the functional $G_\infty$, 
which plays the role of the asymptotic variance of the realized volatility
$\bbV_n$ defined at (\ref{220420.1130}),
by
\begin{align}
  G_\infty =2 c_H^2 T ^{4H-1}\int^T_0 (V^{[1]}(X_{t}))^4 dt,
  \label{220301.1030}
\end{align}
with the constant $c_H^2$ defined at (\ref{210417.1805}).
We impose the following conditions on SDE (\ref{210430.1615}).
\begin{ass}\label{220404.1535}
  (i) $V^{[i]}\in C^\infty_b(\bbR)$ for $i=1,2$, 
  where $C^\infty_b(\bbR)$ is the set of the smooth functions whose derivatives of any order are bounded together with itself.
  
  \item 
  (ii) The functional $G_\infty$ satisfies $G_\infty^{-1}\in L^{\infty-}$.
\end{ass}
\noindent 
The second condition is assumed to secure the nondegeneracy of the distribution of $Z_n$.

We introduce some notations to state the main result.
We write $\phi(z;\mu,v)$ for the density function of the normal distibution 
with mean $\mu\in\bbR$ and variance $v>0$.
For a (polynomial) random symbol 
$\varsigma(\xi) = \sum_\alpha c_\alpha\,\xi^\alpha$
with a random variable $c_\alpha$ and $\alpha\in\bbZ_{\geq0}$,
where $\xi$ is a dummy variable, 
the action of the adjoint $\varsigma(\partial_z)^*$ to $\phi(z;0,G_\infty)$ under the expectation
is  defined by 
\begin{align}
  E\sbr{\varsigma(\partial_z)^* \phi(z;0,G_\infty)}
  =\sum_\alpha (-\partial_z)^\alpha E\sbr{c_\alpha\phi(z;0,G_\infty)}.
  \label{220401.1952}
\end{align}
We denote by $\hat\cale(M,\gamma)$ 
the set of measurable functions $f:\bbR\to\bbR$
such that 
$\abs{f(z)}\leq M(1+\abs{z})^\gamma$ for all $z\in\bbR$.

\begin{theorem} \label{220404.1630}
  Suppose that $(X_t)_{t\in[0,T]}$ is the solution to SDE (\ref{210430.1615}) satisfying Assumption \ref{220404.1535}.
  Define $Z_n$ by
  \begin{align}
    Z_n=n^{1/2}\brbr{\bbV_n-\bbV_\infty}
    \label{220404.1601}
  \end{align}
  % (\ref{220404.1601})
  and let $r_n=n^{2H-3/2}$.
  Then, for any $M,\gamma>0$ the following estimate holds:
\begin{align}\label{220421.1700}
  \sup_{f\in\hat\cale(M,\gamma)}
  %\abs{\hat\Delta_n(f)}
  \abs{E\sbr{f(Z_n)}- \int_{\bbR}f(z)\hat p_n(z)dz}
  =o(r_n)
  \quad\text{as }n\to\infty,
\end{align}
where 
\begin{align*}
  %\hat p_n(z)&=
  %E\sbr{\exp\rbr{-\frac12 G_\infty^{-1} z^2}}
  %+r_n E\sbr{\mS(\partial_z)^* \exp\rbr{-\frac12 G_\infty^{-1} z^2}}
  %\\
  \hat p_n(z)&=
  E\sbr{\phi(z;0,G_\infty)}
  +r_n E\sbr{\mS(\partial_z)^* \phi(z;0,G_\infty)}
  %\\
  %\hat\Delta_n(f)&=
  %\abs{E\sbr{f(Z_n)}- \int_{z\in\bbR}f(z)\hat p_n(z)dz}
\end{align*}
with 
\begin{align}
  \mS=\mS^{(3,0)}+\mS^{(1,0)},
  \label{220404.1631}
\end{align}
where random symbols $\mS^{(3,0)}$ and $\mS^{(1,0)}$ 
are specified at (\ref{220308.2520}) and (\ref{220308.2521}).
\end{theorem}
The proof of this result is given in Section \ref{220317.2030}.

\begin{remark}
    The condition on the coefficient of the SDE (Assumption \ref{220404.1535} (i))
    is rather restrictive and this excludes basic examples like a fractional OU (fOU) process.
    The assumption that the coefficients are bounded is used to obtain 
    the Malliavin derivatives and their estimates.
    Since the Malliavin derivatives of fOU process has a simple expression,
    we have an asymptotic expansion for the quadratic variation of an fOU process as well. 
    See Section \ref{220518.2114}.
\end{remark}

\subsection{Overview on asymptotic expansion of Skorohod integrals {and strategies}}
\label{220421.2330}
We review the theory of asymptotic expansion of Skorohod integrals 
described in \cite{nualart2019asymptotic} in the one-dimensional case without reference variables.
Let $(\Omega, \calf, P)$ a complete probability space equipped with an isonormal Gaussian process 
$W=\cbr{W(h)_{h\in\calh}}$ 
on a real Hilbert space $\calh$.
We denote the Malliavin derivative operator by $D$ and 
its adjoint operator, namely the divergence operator or the Skorohod integral, by $\delta$.
For a real separable Hilbert space $V$, 
we write $\bbD^{k,p}(V)$ for the Sobolev space of $V$-valued random variables 
that has the Malliavin derivatives up to $k$-th order which have finite moments of order $p$.
We write $\bbD^{k,p}=\bbD^{k,p}(\bbR)$, 
$\bbD^{k,\infty}(V)=\cap_{p>1}\bbD^{k,p}(V)$ and
$\bbD^{\infty}(V)=\cap_{p>1,k\geq1}\bbD^{k,p}(V)$.
We refer to the monograph \cite{nualart2006malliavin} for a detailed account on this subject.

Consider a sequence of random varibales $Z_n$ defined on the probability space 
$(\Omega, \calf, P)$ written as 
\begin{align*}
  Z_n=M_n+r_nN_n,
\end{align*}
where
$M_n=\delta(u_n)$ is the Skorohod integral of $\calh$-valued random variable
$u_n\in\Dom(\delta)$,
$N_n$ is a random variable and 
$(r_n)_{n\in\bbN}$ is a sequence of positive numbers such that 
$\lim_{n\to\infty}r_n=0$.
We consider a stable convergence of $Z_n$ to 
a mixed normal distribution $G_\infty^{1/2}\zeta$,
where $G_\infty$ is a positive random variable 
and $\zeta$ is a standard normal distribution independent of $\calf$.

Nualart and Yoshida \cite{nualart2019asymptotic} introduced the following random symbols.
To write random symbols, we use a different notation from the paper 
due to the one-dimensional setting.
%We write $\mS(\xi)=\sum_{k}\chi_k\,\xi^k$ (a finite sum) for a polynomial random symbol,
% where $\chi_k$ are coefficient random variables
% %\redb{$\chi_k\,\xi^k$ is the random symbol of degree $k$ with its coefficient random variable $\chi_k$}
%  and $\xi$ stands for a dummy variable.
We may write
$D_{u_n}F=\abr{DF,u_n}_\calh$ for a random variable $F$ regular enough.
The quasi-tangent is defined by 
\begin{align*}
  \qtan_n=
  r_n^{-1}\brbr{\babr{DM_n, u_n}_\calh-G_\infty}
  =r_n^{-1}\brbr{D_{u_n}M_n-G_\infty}.
\end{align*}
The quasi-torsion and modified quasi-torsion are defined by 
\begin{align*}
  \qtor_n&=
  r_n^{-1}\Babr{D\babr{DM_n, u_n}_\calh,u_n}_\calh
  =r_n^{-1}(D_{u_n})^2M_n
  \\
  \mqtor_n&=
  r_n^{-1}\babr{DG_\infty,u_n}_\calh
  =r_n^{-1}D_{u_n}G_\infty.
\end{align*}

We write
\begin{align*}
  G^{(2)}_n
  %G^{(2)}_n(\sfz) 
  &= r_n\qtan_n
  = D_{u_n}M_n - G_\infty
  \\
  G^{(3)}_n
  %G^{(3)}_n(\sfz) 
  &= r_n\mqtor_n = D_{u_n}G_\infty
\end{align*}
and define the following random symbols
\begin{align*}
  \mS^{(3,0)}_n(\tti\sfz)&=
  \frac13r_n^{-1}(D_{u_n[\tti\sfz]})^2 M_n[\tti\sfz]=
  \frac13\qtor_n[\tti\sfz]^3\\
  \mS^{(2,0)}_{0,n}(\tti\sfz)&=
  \frac12 r_n^{-1}G^{(2)}_n(\sfz)=
  \frac12 r_n^{-1}\brbr{D_{u_n[\tti\sfz]} M_n[\tti\sfz] - G_\infty[\tti\sfz]^2}=
  \frac12 \qtan_n[\tti\sfz]^2\\
  \mS^{(1,0)}_n(\tti\sfz)&=N_n[\tti\sfz]\\
  \mS^{(2,0)}_{1,n}(\tti\sfz)&=
  D_{u_n[\tti\sfz]} N_n[\tti\sfz]
\end{align*}
for $\tti\sfz\in\tti\bbR$.
We consider the random symbols 
$\mS^{(3,0)}$,
$\mS^{(2,0)}_{0}$,
$\mS^{(1,0)}$ and
$\mS^{(2,0)}_{1}$.
These work as the limits of 
the above random symbols
$\mS^{(3,0)}_n$,
$\mS^{(2,0)}_{0,n}$,
$\mS^{(1,0)}_n$ and
$\mS^{(2,0)}_{1,n}$, respectively.
Let
\begin{align*}
  \Psi(\sfz)=
  \exp\rbr{2^{-1} G_\infty[\tti\sfz]^2}.
\end{align*}
for $\sfz\in\bbR$.

Since the distribution $\call(Z_n)$ on $\bbR$ 
for $Z_n$ treated in this paper is regular enough,
the asymptotic expansion (\ref{220421.1700}) holds for some non-differentiable functions $f$.
The following set of conditions {\bf [D]} from Nualart and Yoshida \cite{nualart2019asymptotic}
work as a sufficient condition to validate the expansion.
This parameter $l$ about differentiability below is $l=9$ in this case.
For a one-dimensional functional $F$ , we write 
$\Delta_F =\sigma_F$ for the Malliavin covariance matrix of $F$.

\begin{itemize}
\item [{\bf [D]}]
\begin{itemize}
  \item [(i)]
  $u_n\in\bbD^{l+1,\infty}(\mH)$,
  %$u_n\in\bbD^{l+1,\infty}(\mH\otimes\bbR^\sfd)$,
  $G_\infty \in \bbD^{l+1,\infty}(\bbR_+)$,
  %$G_\infty \in \bbD^{l+1,\infty}(\bbR^\sfd \otimes_+ \bbR^\sfd)$,
  %$W_n, 
  $N_n\in\bbD^{l,\infty}$.
  %$N_n\in\bbD^{l,\infty}(\bbR^\sfd)$,
  %$W_\infty\in\bbD^{l\vee\sfd_2,\infty}(\bbR^\sfd)$,
  %$X_n\in\bbD^{l,\infty}(\bbR^{\sfd_1})$,
  %$X_\infty\in\bbD^{l\vee(\sfd_2+1),\infty}(\bbR^{\sfd_1})$,

  \item [(ii)]
  There exists a positive constant $\kappa$ such that the following estimates hold for every $p>1$:
  \begin{align}
    &\norm{u_n}_{l,p}=O(1)
    \label{220215.1241}\\
    &\bnorm{G^{(2)}_n}_{l-2,p}=O(r_n)
    \label{220215.1242}\\
    &\bnorm{G^{(3)}_n}_{l-2,p}=O(r_n)
    \label{220215.1243}\\
    &\bnorm{D_{u_n}G^{(3)}_n}_{l-1,p}=O(r_n^{1+\kappa})
    \label{220215.1244}\\
    &\bnorm{D^2_{u_n}G^{(2)}_n}_{l-3,p}=O(r_n^{1+\kappa})
    \label{220215.1245}\\
    &\norm{N_n}_{l-1,p}=O(1)
    \label{220215.1246}\\
    &\bnorm{D^2_{u_n} N_n}_{l-2,p}=O(r_n^{\kappa})
    \label{220215.1247}
  \end{align}

  \item [(iii)]
  For each pair 
  %$(\mT_n,\mT)= (\mS_n^{(3,0)},\mS^{(3,0)})$,
  %$(\mS_{0,n}^{(2,0)},\mS_0^{(2,0)})$,
  %$(\mS_n^{(1,0)},\mS^{(1,0)})$ and 
  %$(\mS_{1,n}^{(2,0)},\mS_1^{(2,0)})$,
  $(\mT_n,\mT)= (\mS_n^{(3,0)},\mS^{(3,0)})$,
  $(\mS_{0,n}^{(2,0)},\mS_0^{(2,0)})$,
  $(\mS_n^{(1,0)},\mS^{(1,0)})$ and 
  $(\mS_{1,n}^{(2,0)},\mS_1^{(2,0)})$,
  the following conditions are satisfied.
  \begin{itemize}
    \item [(a)]
    $\mT$ is a polynomial random symbol the coefficients of which are in
    $L^{1+}=\cup_{p>1} L^p$.
    %$\bbD^{\check\sfd + \beta_x +1,1+} = \cup_{p>1} \bbD^{\check\sfd + \beta_x +1,p} $
    \item [(b)]
    For some $p>1$, there exists a polynomial random symbol $\bar\mT_n$ that has 
    $L^p$ coefficients and the same degree as $\mT$, 
    \begin{align*}
    E \sbr{\Psi (\sfz)\mT_n(\tti\sfz)} = E[\Psi (\sfz)\bar\mT_n(\tti\sfz)]
    %\hspsm{\tforsm \sfz\in\bbR \tandsm n\in\bbN}
    %E[\Psi (\sfz,\sfx)\mT_n(\tti\sfz,\tti\sfx)] = E[\Psi (\sfz,\sfx)\bar\mT_n(\tti\sfz,\tti\sfx)]
    \end{align*}
    and $\bar\mT_n\to\mT$ in $L^p$.
  \end{itemize}

  \item [(iv)]
  \begin{itemize}
    \item [(a)] $G^{-1}_{\infty}\in L^{\infty-}$
    \item [(b)]
    There exist %$c\in (-1,0)\cup(0,1)$ and 
    $\kappa>0$ such that
    \begin{align*}
    P[\Delta_{M_n}<s_n]=O(r_n^{1+\kappa})
    %P[\Delta_{(cM_n+W_\infty,X_\infty)}<s_n]=O(r_n^{1+\kappa})
    \end{align*}
    for some positive random variables $s_n\in\bbD^{l-2,\infty}$ satisfying 
    $\sup _{n\in\bbN}(\norm{s_n^{-1}}_p + \norm{s_n}_{l-2,p})<\infty$
    for every $p>1$.
  \end{itemize}
\end{itemize}
\end{itemize}
Note that $\bar\mT_n\to\mT$ in $L^p$ means that
the coefficient random variable of each degree of $\bar\mT_n$ converges in $L^p$ to 
the counterpart of $\mT$.
The functional $s_n$ in (iv) (b) is used to make
a truncation functional to gain local asymptotic non-degeneracy.
See Section 7 of \cite{nualart2019asymptotic} for a construction of a truncation function.

Define the random symbol 
$\mS_n = 1 +r_n\mS$ with
%\begin{align*}
%  \mS_n = 1 +r_n\mS,
%\end{align*}
%where the random symbol $\mS$ is defined by
\begin{align*}
  \mS(\tti\sfz) = 
  \mS^{(3,0)}(\tti\sfz) + \mS^{(2,0)}_{0}(\tti\sfz) +
  \mS^{(1,0)}(\tti\sfz) + \mS^{(2,0)}_{1}(\tti\sfz)
\end{align*}
and $\hat p_n(z)$ by
\begin{align*}
  \hat p_n(z) = 
  E\sbr{\mS_n(\partial_z)^* \phi(z;0,G_\infty)},
\end{align*}
where the action of the adjoint of a random symbol is defined at (\ref{220401.1952}).
We denote by $\hat\cale(M,\gamma)$ 
the set of measurable functions $f:\bbR\to\bbR$
such that 
$\abs{f(z)}\leq M(1+\abs{z})^\gamma$ for all $z\in\bbR$.
The following theorem rephrases Theorem 7.7 of \cite{nualart2019asymptotic}.
\begin{theorem}
  Suppose that Condition {\bf [D]} is satisfied.
  Then, for each $M,\gamma\in\bbR_{>0}$.
  \begin{align*}
    \sup_{f\in\hat\cale(M,\gamma)}
    \abs{E\sbr{f(Z_n)}- \int_{\bbR}f(z)\hat p_n(z)dz}
    =o(r_n)
    \quad\text{as }n\to\infty.
  \end{align*}
\end{theorem}

\vspsm
%\subsection{About estimate of functionals of a certain form}
In order to justify the asymptotic expansion given in Theorem \ref{220404.1630},
we will check Condition {\bf [D]} with 
$u_n$, $G_\infty$, $N_n$ and $\mS$ in the context of 
the quadratic variaiton of $X_t$,
whose definitions are given at (\ref{220404.1632}), (\ref{220301.1030}), 
(\ref{220404.1633}) and (\ref{220404.1631}), respectively.
While verifying the conditions for those functionals,
we repeatedly need to estimate the order of functionals of a certain form.
For example, 
\begin{align*}
  \cali^{}_{n}=
  %\cali^{M(3,0)}_{5,n}=
  n^{6H-3/2} \sum_{j_1,j_2,j_3=1}^n
  \rbr{D_{1_\jtw}a_{t_{\jon-1}}} a_{t_{\jtw-1}} a_{t_{\jth-1}}
  I_2(1_\jon^{\otimes 2})I(1_\jth)\beta_n\rbr{\jtw,\jth},
\end{align*} 
where 
$1_j$ is the indicator function of $\clop{\tjm,\tj}$,
$I_2(1_\jon^{\otimes 2})$ and $I(1_\jth)$ are (multiple) stochastic integrals and 
$\beta_n(j_2,j_3)$ is the inner product between $1_{j_2}$ and $1_{j_3}$.
When we asses the order (e.g. in $L^p$ norm) of functionals of this form, 
we use the IBP formula of Malliavin calculus 
until the multiple stochastic integrals in functionals disappear.
When we consider $E\bsbr{(\cali^{M(3,0)}_{5,n})^4}$, 
we encounter with a linear combination of functionals of the following form:
\begin{align*}
  E\bbsbr{n^{\alpha} \sum_{j_1,...,j_{12}=1}^n
  %D_{1_\jon}\rbr{D_{1_\jtw}a_{t_{\jon-1}}} a_{t_{\jtw-1}} a_{t_{\jth-1}}
  A_n(j_1,...,j_{12})
  \prod_{v_1\neq v_2=1,...,12}
  \beta_n\rbr{j_{v_1},j_{v_2}}^{\theta_{v_1,v_2}}},
  %\beta_n\rbr{j_1,j_3}^{\theta_{1,3}}
  %...
  %\beta_n\rbr{j_{11},j_{12}}^{\theta_{11,12}}
\end{align*} 
where $A_n(j_1,...,j_{12})$ are random variables and $\theta_{v_1,v_2}$ are non-negative integers. 
The problem is that 
in the linear combination there appear %{functionals with}
a huge number of patterns of various $A_n$ and $\theta_{v_1,v_2}$
caused by  the Leibniz rule of the actions of $D_{1_j}$.
Even after desolving the multiple stochastic integrals in functionals by the duality of $D$ and $\delta$,
it is not a simple matter to estimate the sum with a product of $\beta_n(j_{v_1},j_{v_2})$, 
which mainly controls the order of the functional.
Furthermore, when checking Condition {\bf [D]},
we need to consider the functionals 
obtained by the action of $D_{u_n}$ and $D^i$ on these functionals as well.
%\redb{functional no sum-structure, beta no tameni order estimate was complicated dearu koto wo remark}
To handle this issue, we devote Section \ref{210429.1655} 
to develope two exponents that can be calculated with ease 
based on some graphical representation of the structure of the sum.
These exponents can be used in other contexts related to fBm and 
we hope they have interest of their own.
{
  The same obstacle has appeared in the case of functionals of classical Brownian motion,
  and in \cite{yoshida2020asymptotic} the author defined an exponent. 
  The exponents defined in this paper can be thought of a generalization of it.
  (See Remark \ref{220502.2011} and \ref{220502.2012}.)
}
In Section \ref{220317.2030}, by extensively applying the exponents, 
we derive the stochastic expansion of the functional in question,
calculate the random symbol in the asymptotic expansion 
and verify the assumption of the general theory from \cite{nualart2019asymptotic}.

\vspssm
We review the definitions and notations about fractional Brownian motion and related Malliavin calculus.
Let $B=\cbr{B_t\mid t\in[0,T]}$ a fractional Brownian motion 
with Hurst parameter $H\in\rbr{1/2,3/4}$
defined on some complete probability space $(\Omega,\calf, P)$.
An inner product on the set $\cale$ of step functions on $[0,T]$ is defined by 
\begin{align*}
  \abr{1_{[0,t]},1_{[0,s]}}_\calh = E\sbr{B_s\,B_t}
  = \frac12\rbr{\abs{t}^{2H}+\abs{s}^{2H}-\abs{t-s}^{2H}}.
\end{align*}
and $\calh$ is the closure of $\cale$ with respect to
$\snorm{\cdot}_\calh=\abr{\cdot,\cdot}_\calh^{1/2}$.
The map $\cale\ni1_{[0,t]}\mapsto B_t\in L^2(\Omega, \calf, P)$
can be extended linear-isometrically to $\calh$.
We denote this map by $\phi\mapsto B(\phi)$ and 
the process $\cbr{B(\phi), \phi\in \calh}$ is an isonormal Gaussian process. 
In the following sections, tools from Malliavin calculus are based on 
this isonormal Gaussian process.
The Hilbert space $\calh$ contains the linear space $\abs\calh$
of measurable functions $\phi:[0,T]\to\bbR$ such that 
\begin{align*}
  \int_0^T\int_0^T \abs{\phi(s)}\abs{\phi(t)}\abs{t-s}^{2H-2}dsdt<\infty,
\end{align*}
and for $\phi$ and $\psi$ in $\abs\calh$, the inner product of $\calh$ is written as 
\begin{align*}
  \abr{\phi,\psi}_\calh=
  H(2H-1)\int^T_0\int^T_0\phi(s)\psi(t)\abs{t-s}^{2H-2}dsdt.
\end{align*}
For a measurable function $\phi$ on $[0,T]^l$, we define
\begin{align*}
  \norm\phi_{\abs\calh^{\otimes l}}^2
  =\alpha_H^l \int_{[0,T]^{l}} \int_{[0,T]^{l}} \abs{\phi(u)} \abs{\phi(v)}
  \abs{u_1-v_1}^{2H-2}...\abs{u_{l}-v_{l}}^{2H-2}dudv,
\end{align*}
and the space
$\abs\calh^{\otimes l}=\bcbr{\phi:[0,T]^l\to\bbR\mid \norm\phi_{\abs\calh^{\otimes l}}<\infty}$ 
of measurable functions
forms a subspace of the $l$-fold tensor product space $\calh^{\otimes l}$ of $\calh$.
We refer to \cite{nualart2006malliavin} or 
\cite{nourdin2012selected} for a detailed account on fBm. 
We collects some basic estimates related to SDE driven by fBm with Hurst index $H>1/2$
in Section \ref{220423.1450}.

\vspsm
Recall we write 
$\tj=jT/n$ for $n\in\bbN$ and $j\in\cbr{1,...,n}$.
Let $I_j=\clop{\tjm,\tj}\subset[0,T]$,
and for brevity 
we write $1_j$ (or $1_{n,j}$) for the indicator function $1_{I_j}$ of this interval $I_j$.
Note that $\tj$, $I_j$ and $1_j$ are dependent on $n$.
We write $[n]=\cbr{1,..,n}$.
Define $\rho_H(k)$, $c_H$ and 
$\beta_n\rbr{j_1,j_2}$ by %=\beta_{j_1,j_2}
 %:= \abr{1_{\jon}, 1_\jtw}_\calh$  for $j_1,j_2\in[n]$ by
\begin{align}
  \rho_H(k) &:= \frac12 \rbr{\abs{k+1}^{2H} +\abs{k-1}^{2H} -2\abs{k}^{2H}}
  =\abr{1_{[0,1]}, 1_{[k,k+1]}}_\calh
  &&\tfornsp k\in\bbZ
  \nn\\
  c_H^2 &%= c_{H,(\ref{210417.1805})}^2
  =\sum_{k\in\bbZ} \rho_H(k)^2 
  \label{210417.1805}
  \\
  \beta_n\rbr{j_1,j_2}=\beta_{j_1,j_2}%=\beta_{j_1,j_2;T}
  &:= \abr{1_{\jon}, 1_\jtw}_\calh
  = T^{2H} n^{-2H} \rho_H(\jon-\jtw)
  &&\tfornsp \jon,\jtw\in[n].
  \nn%\label{220422.1911}
\end{align}

We use the following notation to express asymptotic orders.
Let  $\alpha\in\bbR$ and $(F_n)_{n\in\bbN}$ a sequence of random variables.
We write 
\begin{alignat*}{5}
  &F_n=O_{L^{p}}(n^\alpha) 
  \;&&[\text{resp. } F_n=o_{L^{p}}(n^\alpha)]
  &\quad&\tif\;
  \norm{F_n}_{p}=O(n^\alpha) 
  \;&&[\text{resp. }  \norm{F_n}_{p}=o(n^\alpha)]
  \qquad(p\geq1)
  \\
  &F_n=O_{L^{\infty-}}(n^\alpha) 
  \;&&[\text{resp. } F_n=o_{L^{\infty-}}(n^\alpha)]
  &\quad&\tif\;
  F_n=O_{L^{p}}(n^\alpha) 
  %\norm{F_n}_{p}=O(n^\alpha) 
  \;&&[\text{resp. }  F_n=o_{L^{p}}(n^\alpha)]
  %\;&&[\text{resp. }  \norm{F_n}_{p}=o(n^\alpha)]
  \quad\text{for any } p>1.
\end{alignat*}
For $(F_n)_{n\in\bbN}$ with $F_n\in\bbD^\infty$,
\begin{alignat*}{5}
  &F_n=O_M(n^\alpha) 
  \;&&[\text{resp. } F_n=o_M(n^\alpha)]
  &\quad&\tif\;
  \norm{F_n}_{i,p}=O(n^\alpha) 
  \;&&[\text{resp. } \norm{F_n}_{i,p}=o(n^\alpha)]
  \quad\text{for any } i\in\bbN \tandsm p>1.
\end{alignat*}
In the case of small ``$o$'', the asymptotic order of $\norm{F_n}_{i,p}$ may depends on $i$ and $p$.

%\newpage
\section{Estimates of functionals involving Wiener chaoses}\label{210429.1655}
We introduce two exponents which enable us to tell
the order of functionals of certain forms which appear in Section \ref{220317.2030}.
The first exponent is somewhat coarser with a wider class of functionals 
to apply than the second one.  
Below we introduce some graphical representations to express functionals
to make the arguments precise.
We use the notation 
$[n]=\{1,...,n\}$ for $n\in\bbZ_{\geq1}$.
For finite sets $V'\subset V$ and a map $j$ from $V$ to $[n]$
(i.e. $j\in[n]^V$),
we write $j_{V'}=j|_{V'}\in[n]^{V'}$ for the restriction of $j$ to $V'$.
Especially when $V'=\cbr{v}$, we write $j_v=j_{\cbr{v}}$ and 
identify $j_v=j(v)$.

%%\newpage
\subsection{Edge-vertex-weighted graph
and functionals}
For a set $V$, we define 
$p(V):=\cbr{(v,v')\in V^2|v\neq v'}/_\equiv$,
where $\equiv$ means the equivalence relation induced by 
$(v,v')\equiv(v',v)$ for $v\neq v'\in V$.
We write $[v,v']$ for the equivalence class $(v,v')$ belongs to. 
Given a nonempty finite set $V$,
consider functions 
$\edgeWt:p(V)\to\bbZ_{\geq0}$ and
$\vertWt:V\to\bbZ_{\geq0}$,
and we say the triplet $G=(V,\edgeWt,\vertWt)$ is
{\it a edge-vertex-weighted graph}
({\it weighted graph} for short). 
We call $V$, $\edgeWt$ and $\vertWt$
the set of vertices, the weight function on edges and the weight function on vertices, respectively.
For brevity, we write
$\edgeWt_{v,v'}=\edgeWt_{v',v}=\edgeWt([v,v'])$
and $\vertwtlow_v=\vertWt(v)$.
We denote 
by $V(G)$ or $V_G$
the set $V$ of vertices of a weighted graph $G=(V,\edgeWt, \vertWt)$.

For a weighted graph $G=(V,\edgeWt,\vertWt)$,
we say that
the graph defined by 
$(V,\cbr{[v,v']\in p(V): \edgeWt([v,v'])>0})$
is {\it the non-weighted graph of $G$}, and denote it by $\frakg(G)$.
A weighted graph $G=(V,\edgeWt,\vertWt)$ is {\it connected}
if the non-weighted graph $\frakg(G)$ is connected.
%($V$, $\cbr{[v,v']\in p(V): \edgeWt([v,v'])>0}$) 
%
For a subset $V'$ of $V$, we call the weighted graph 
$G|_{V'}=(V',\edgeWt|_{p(V')},\vertWt|_{V'})$
{\it a subgraph of $G$} or the subgraph of $G$ obtained by restricting to $V'$.
Consider the equivalence relation $\sim_\edgeWt$ on $V$ induced by 
$v\sim_\edgeWt v'$ for $v\neq v'\in V$ such that $\edgeWt([v,v'])>0$.
When $V'(\subset V)$ is a equivalence class with respect to $\sim_\edgeWt$,
we say the subgraph $G|_{V'}$ of $G$ is a {\it component} of $G$.
We denote the set of components of $G$ by $\Comp(G)$.

We denote by $\singlegraph{v}{k}$ %$G_{(\cbr{v},k)}$
a weighted graph $(V,\edgeWt,\vertWt)$ 
such that $V$ is written as $V=\cbr{v}$ with some $v$ (i.e. the vertex set is a singleton), and
the weight $\vertWt(v)$ on the only vertex $v$ is $k$.
Notice that $p(V)$ is empty and 
we write $0$ for the weight function $\edgeWt$ on the empty edge set.
We sometimes use graphical notations to express weighted graphs. 
For example, we write the following graphs 
for $\singlegraph{v}{k}$ and 
a weighted graph $(V,\edgeWt,\vertWt)$ with
$V=\cbr{1,2,3}$ and $\edgeWt_{1,3}=0$, respectively.
%$(\edgeWt([1,2]),\edgeWt([1,3]),\edgeWt([2,3]))=(\edgeWt_{1,2},\edgeWt_{1,3},\edgeWt_{2,3})$ and 
%$(\edgeWt([1,2]),\edgeWt([1,3]),\edgeWt([2,3]))=(3,2,4)$ and 
%$(\vertWt(1), \vertWt(2), \vertWt(3))=(\vertwtlow_1,\vertwtlow_2,\vertwtlow_3)$
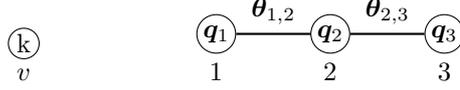
\begin{figure}[H]
  \centering
  \begin{subfigure}[t]{0.23\textwidth}
    \centering
    \graphonept{v}{k}
    %\caption{$G^{M(3,0)}_{1}$}
  \end{subfigure}
  \begin{subfigure}[t]{0.23\textwidth}
    \centering
    \graphExample{$\vertwtlow_1$}{$\vertwtlow_2$}{$\vertwtlow_3$}{$\edgeWt_{1,2}$}{$\edgeWt_{2,3}$}{}{1}{2}{3}
    %\caption{$G^{M(3,0)}_{1}$}
  \end{subfigure}
\end{figure}
\noindent
We may often drop information about vertices ($v$ and $1,2,3$ in the above examples) 
when it is not important.

\vspssm
For a finite family $(G^{(\gamma)})_{\gamma\in\Gamma}$
of weighted graphs
$G^{(\gamma)}=(V^{(\gamma)}, \edgeWt^{(\gamma)}, \vertWt^{(\gamma)})$,
we say $(G^{(\gamma)})_{\gamma\in\Gamma}$ is {\it disjoint}
if $(V^{(\gamma)})_{\gamma\in\Gamma}$ is mutually disjoint,
that is $V^{(\gamma)}\cap V^{(\gamma')}=\emptyset$ 
for $\gamma\neq\gamma'\in\Gamma$.
For a disjoint family $(G^{(\gamma)})_{\gamma\in\Gamma}$ of weighted graphs, 
we define a weighted graph 
$\vee_{\gamma\in\Gamma} G^{(\gamma)}=(V,\edgeWt, \vertWt)$ by
\begin{align*}
  V&=\bigsqcup_{\gamma\in\Gamma} V^{(\gamma)}
  \\
  \edgeWt([v,v'])&=
  \begin{cases}
    \edgeWt^{(\gamma)}([v,v'])
   \quad&\text{ if $v,v'\in V^{(\gamma)}$ for some $\gamma\in\Gamma$}   \\
    0 &\text{ if $v\in V^{(\gamma)}$ and $v'\in V^{(\gamma')}$ for $\gamma\neq\gamma'$}
  \end{cases}
  \\
  \vertWt(v) &= \vertWt^{(\gamma)}(v)
  \qquad\qquad\quad
  \text{for $v\in V^{(\gamma)}$ with some $\gamma\in\Gamma$}.
\end{align*}
%$\vee (G^{(\gamma)})_{\gamma\in\Gamma}$.
When $\Gamma=\cbr{1,...,k}$ for some $k\geq1$,
we may simply write
\begin{align*}
  G^{(1)}\vee ...\vee G^{(k)}=\vee_{\gamma\in\Gamma} G^{(\gamma)}.
\end{align*}
\begin{makeitbetter}
  Add what this operation stands for.
\end{makeitbetter}

\vspssm%\noindent
For a finite subset $V$ of $\bbZ$, a weighted graph $G=(V,\edgeWt, \vertWt)$
and an integer $k\in\bbZ$,
we define a weighted graph $\shiftg{G}{k}$ by
\begin{align*}
  \shiftg{G}{k}
=(V+k, 
(\edgeWt([v-k,v'-k]))_{[v,v']\in p(V+k)},  (\vertWt(v-k))_{v\in V+k})
%(\edgeWt_{v-k,v'-k})_{[v,v']\in p(V+k)},  (\vertwtlow_{v-k})_{v\in V+k})
\end{align*}
with the notation $V+k=\cbr{v+k: v\in V}$.

\vspssm
\begin{definition}
Let  $G=(V,\edgeWt, \vertWt)$ a weighted graph. For
$\sigma:V\to\bbZ_{\leq0}$ satisfying
$%\begin{align*}
  \vertWt(v) + \sigma(v)\geq0,
  %\vertwtlow_v + \sigma(v)\geq0,
$ %\end{align*}
we define $\abr\sigma\vertWt: V\to\bbZ_{\geq0}$ by
\begin{align*}
  (\abr\sigma\vertWt)(v) = \vertWt(v) + \sigma(v)
  %(\abr\sigma\vertWt)(v) = \vertwtlow_v+\sigma(v)
  \quad\tfor v\in V,
\end{align*}
and 
we denote the weighted graph $(V, \edgeWt, \abr{\sigma}\vertWt)$ by 
$\abr{\sigma}(G)$.

For
$\tau:p(V)\to\bbZ_{\geq0}$ 
satisfying
\begin{align*}
  \tau_v:=\sum_{\substack{v'\in V,  v'\neq v }}\tau([v,v'])\leq \vertWt(v)
  %\tau_v:=\sum_{\substack{v'\in V,  v'\neq v }}\tau([v,v'])\leq \vertwtlow_v
  \quad\tfor v\in V,
\end{align*}
we define
$\dabr\tau\edgeWt: p(V)\to\bbZ_{\geq0}$ and
$\dabr\tau\vertWt: V\to\bbZ_{\geq0}$ by
\begin{alignat*}{2}
  (\dabr\tau\edgeWt)([v,v'])&=
  \edgeWt([v,v']) + \tau([v,v'])
  &&\quad\tfor [v,v']\in p(V),
  \\
  (\dabr\tau\vertWt)(v) &= \vertWt(v)-\tau_v
  %(\abr\tau\vertWt)(v) &= \vertwtlow_v-\tau_v
  &&\quad\tfor v\in V,
\end{alignat*}
and 
we denote the weighted graph $(V, \dabr\tau\edgeWt, \dabr\tau\vertWt)$ by 
$\dabr\tau(G)$.
If $V$ is a singleton, we read $\dabr{\tau}(G)=G$ by convention.
\end{definition}
Notice that 
if $\dabr\tau(\abr\sigma G)$ is well-defined, 
then $\abr\sigma(\dabr\tau G)$ is also well-defined and
$\dabr\tau(\abr\sigma G) = \abr\sigma(\dabr\tau G)$.
%\delc{Maybe we can add the commutation of two $\sigma$ or $\tau$.}
If $\sigma(v_0)<0$ for some $v_0\in V$ and $\sigma(v)=0$ for $v\neq v_0$,
we write
$\abr{v_0}_{\sigma(v_0)}(G) =\abr\sigma(G)$.
Similarly, if $\tau([v_0,v_0'])>0$ for some $[v_0,v_0']\in p(V)$ and 
$\tau([v,v'])=0$ for $[v,v']\neq [v_0,v_0']$,
we write
$\dabr{v_0,v_0'}_{\tau([v_0,v_0'])}(G) =\dabr\tau(G)$.

\begin{definition}
For a weighted graph $G=(V,\edgeWt,\vertWt)$,
we define $I(G)$, $\bartheta(G)$ and $\barq(G)$ by %the following indices.
\begin{alignat*}{4}
  I(G) =& \babs{V},&\quad 
  \bartheta(G) &= \sum_{[v,v']\in p(V)} \edgeWt([v,\vpr])&\tand
  \barq(G) =&  \sum_{v\in V} \vertWt(v),
\end{alignat*}
respectively.
For a connected weighted graph $G=(V,\edgeWt,\vertWt)$,
we define $s(G)$ by
\begin{align*}
  s(G) &= 2\brbr{\bartheta(G)-(I(G)-1)} + \barq(G).
\end{align*}
\end{definition}
We sometimes abbreviate 
$I(G)$ %, $\bartheta(G)$, $\barq(G)$ 
and $s(G)$ as 
$I_G$ %, \redc{$\bartheta_G$, $\bar q_G$} 
and $s_G$, respectively.
Notice that 
\begin{align*}
  \bartheta(G)\geq I(G)-1
\end{align*}
for a connected weighted graph $G$, 
and hence we have $s(G)\geq0$.

\begin{proposition}\label{220110.2555}
  Let $(G^{(\gamma)})_{\gamma\in\Gamma}$ 
  a disjoint family of weighted graphs 
  $G^{(\gamma)}=(V^{(\gamma)},\edgeWt^{(\gamma)},\vertWt^{(\gamma)})$ 
  labeled by a nonempty finite set $\Gamma$. % with $k=\abs{\Gamma}$.
  %Set $V=\sqcup_{\gamma\in\Gamma} V^{(\gamma)}$.
  Set $G=(V,\edgeWt,\vertWt):=\vee_{\gamma\in\Gamma} G^{(\gamma)}$.
  Consider $\sigma:V\to\bbZ_{\leq0}$  and $\tau:p(V)\to\bbZ_{\geq0}$  
  such that 
  $G^{\sigma,\tau}=(V,\edgeWt^{\sigma,\tau},\vertWt^{\sigma,\tau})
  :=\abr\sigma\dabr\tau(G)$ 
  %:=\abr\sigma\dabr\tau(\vee_{\gamma\in\Gamma} G^{(\gamma)})$ 
  is well-defined.
  Then the following relations hold:
  \begin{align}
    \bartheta(G^{\sigma,\tau})
    &=\sum_{\gamma\in\Gamma}\bartheta(G^{(\gamma)})+\sum_{[v,v']\in p(V)}\tau([v,v'])
    \nn\\
    \barq(G^{\sigma,\tau})
    &=
    \sum_{\gamma\in\Gamma} \barq(G^{(\gamma)})
    +\sum_{v\in V} \sigma(v)
    -2\sum_{[v,v']\in p(V)}\tau([v,v']).
    \label{220116.2106}
  \end{align}
  If $G^{(\gamma)}$ is connected for all $\gamma\in\Gamma$ and
  the weighted graph $G^{\sigma,\tau}$ is connected, then
  $s(G^{\sigma,\tau})$ satisfies
  \begin{align}
    s(G^{\sigma,\tau})
    &=\sum_{\gamma\in\Gamma}s(G^{(\gamma)}) 
    + \sum_{v\in V}\sigma(v) - 2(\abs{\Gamma}-1).
    \label{211225.2120}
  \end{align}
\end{proposition}
\begin{proof}
  For $v\in V=\sqcup_{\gamma\in\Gamma} V^{(\gamma)}$,
  we write $\gamma(v)$ for $\gamma$ such that $v\in V^{(\gamma)}$.
  \begin{align*}
    \bartheta(G^{\sigma,\tau})
    &=\sum_{[v,v']\in p(V)} \edgeWt^{\sigma,\tau}([v,\vpr])
    =\sum_{[v,v']\in p(V)} \edgeWt([v,v'])
    + \sum_{[v,v']\in p(V)} \tau([v,v'])
    %=\sum_{\substack{[v,v']\in p(V)\\\gamma(v)=\gamma(v')}} \edgeWt^{(\gamma(v))}([v,v'])
    %+ \sum_{[v,v']\in p(V)} \tau([v,v'])
    \\
    &=\sum_{\gamma\in\Gamma} \sum_{[v,v']\in p(V^{(\gamma)})} \edgeWt^{(\gamma)}([v,v'])
    +\sum_{[v,v']\in p(V)}\tau([v,v'])
    =\sum_{\gamma\in\Gamma}\bartheta(G^{(\gamma)})
    +\sum_{[v,v']\in p(V)}\tau([v,v'])
  \\
    \barq(G^{\sigma,\tau})
    &=\sum_{v\in V} (\vertWt^{(\gamma(v))}(v) +\sigma(v) -\tau_v)\\
    &=
    \sum_{\gamma\in\Gamma} \sum_{v\in V^{(\gamma)}} \vertWt^{(\gamma)}(v) 
    +\sum_{v\in V} \sigma(v)
    -\sum_{v\in V} \sum_{\substack{v'\in V\\v'\neq v }}\tau([v,v'])\\
    &= 
    \sum_{\gamma\in\Gamma} \barq(G^{(\gamma)})
    +\sum_{v\in V} \sigma(v)
    -2\sum_{[v,v']\in p(V)}\tau([v,v'])
  \end{align*}

  We can prove the relation (\ref{211225.2120}) as follows.
  \begin{align*}
    s(G^{\sigma,\tau})
    &=2\brbr{\bartheta(G^{\sigma,\tau})-(I(G^{\sigma,\tau})-1)} + \barq(G^{\sigma,\tau})
    \\&=
    2\sum_{\gamma\in\Gamma}\bartheta(G^{(\gamma)})
    +2\sum_{[v,v']\in p(V)}\tau([v,v'])
    -2\sum_{\gamma\in\Gamma}I(G^{(\gamma)})+2
    \\&\qquad
    +\sum_{\gamma\in\Gamma} \barq(G^{(\gamma)})
    +\sum_{v\in V} \sigma(v)
    -2\sum_{[v,v']\in p(V)}\tau([v,v'])
    %\\&=
    %2\sum_{\gamma\in\Gamma}\rbr{\bartheta(G^{(\gamma)}) - I(G^{(\gamma)}) + 1}
    %-2\abs{\Gamma} + 2)
    %+\sum_{\gamma\in\Gamma} \barq(G^{(\gamma)})
    %+\sum_{v\in V} \sigma(v) 
    %%\\&=
    %%2\sum_{\gamma\in\Gamma}\rbr{\bartheta(G^{(\gamma)}) - I(G^{(\gamma)}) + 1}
    %%+\sum_{\gamma\in\Gamma} \sum_{v\in V^{(\gamma)}} q^{(\gamma)}_v 
    %%+\sum_{v\in V} \sigma(v) 
    %%-2(\abs{\Gamma} - 1)
    \\&=
    \sum_{\gamma\in\Gamma}
    \rbr{2\brbr{\bartheta(G^{(\gamma)}) -I(G^{(\gamma)}) + 1} +\barq(G^{(\gamma)})}
    +\sum_{v\in V} \sigma(v) 
    -2(\abs{\Gamma} - 1)
    \\&
    =\sum_{\gamma\in\Gamma}s(G^{(\gamma)}) + \sum_{v\in V}\sigma(v) - 2(\abs\Gamma-1)
  \end{align*}
\end{proof}

We proceed to define functionals using the weighted graphs explained above.
First we introduce some classes for notational convenience.
For a nonempty finite set $V$, 
we denote by $\cala(V)$
the set of families $A=(A_n(j))_{j\in[n]^V,n\in\bbN}$ 
with $A_n(j)\in\bbD^\infty$
satisfying the following conditions:
\begin{itemize}
  \item For 
  %$n\in\mathbb{N}$, $j\in[n]^V$ and 
  $i\in\bbN$, 
  $D^i A_n(j)\in \abs\calh^{\otimes i}$ a.s. 
  
  \item $\sup_{n\in\mathbb{N}} \sup_{j\in[n]^V} \norm{A_n(j)}_{L^p(P)} <\infty$ 
  for $p\geq1$.
  
  \item $D^i A_n(j)$ has a version 
  $\rbr{D^i_{s_1,...,s_i} A_n(j)}_{s_1,...,s_i\in[0,T]}$ which satisfies
  \begin{align}\label{220119.1721}
    %C(i,p)=
    \sup_{n\in\mathbb{N}} \sup_{j\in[n]^V} \sup_{s_1,...,s_i\in[0,T]} 
    \norm{D^i_{s_1,...,s_i} A_n(j)}_{L^p(P)} 
    <\infty
  \end{align}
  %for some constant $C(i,p)$ 
  for every $i\in\bbZ_{\geq0}$ and $p\geq1$.
\end{itemize}
Similarly, for a nonempty finite set $V$, 
we denote by $\calf(V)$
the set of families 
$\bbf=\rbr{\bbf_v}_{v\in V}
=\brbr{\kerfvn{v}{j}}_{j\in[n],n\in\bbN, v\in V}$ 
with $\kerfvn{v}{j}\in\abs\calh$ 
which has a version 
$\rbr{\kerfvn{v}{j}(s)}_{s\in[0,T]}$ 
such that
with some $C_\bbf>0$
\begin{align}
  \abs{\kerfvn{v}{j}(s)}\leq C_\bbf \charfn{j}(s)
  \label{210518.0001}
\end{align}
for $v\in V, n\in\bbN$, $j\in[n]$ and $s\in[0,T]$.
\begin{scrap}
  {\begin{alignat}{2}
  &\supp(\kerfvn{v}{j})\subset \overline{I_{n,j}}=[Tn^{-1}(j-1),Tn^{-1}j]
  &&\qquad\tforsm
  v\in V, n\in\bbN \tandsm j\in[n]
  \nn\\
  %C_{(\ref{210518.0001})}=
  &\max_{v\in V} \sup_{n\in\bbN} \sup_{j\in[n]} \sup_{s\in[0,T]} \abs{\kerfvn{v}{j}(s)}
  <\infty.&
  %\label{210518.0001}
\end{alignat}}
\end{scrap}
For $V'(\subset V)$, we write
$\bbf|_{V'}=(\bbf_v)_{v\in{V'}}(\in\calf(V'))$.
We denote 
$\bbone=\brbr{\charfn{j}}_{j\in[n],n\in\bbN}$,
and abuse this notation $\bbone$ for $(\bbone)_{v\in V}$ with any $V$.

\vspsm
Recall that we have defined
\begin{align*}
  \beta_n(j_1,j_2)=\abr{1_{n,j_1}, 1_{n,j_2}}_{\calh}
\end{align*}
for $n\in\bbN$ and $j_1,j_2\in[n]$
%$\beta_{j_1,j_2} = \abr{1_{j_1}, 1_{j_2}}$
with 
$1_{n,j} = 1_{(Tn^{-1}(j-1),Tn^{-1}j]}$
for $j\in[n]$.

\begin{definition}
Let $n\in\bbN$ and $V$ a nonempty finite set.
For $j\in[n]^V$ and a map $\edgeWt:p(V)\to\bbZ_{\geq0}$,
we define 
\begin{align*}
  \beta_{n,V}(j,\edgeWt) =
  \prod_{[v,v']\in p(V)} \beta_{n}(j_v,j_{v'})^{\edgeWt([v,v'])}.
  %\prod_{[v,v']\in p(V)} \beta_{n,V}(j,[v,v'])^{\edgeWt([v,v'])},
  %\prod_{[v,v']\in p(V)} \beta_{j_{v},j_{v'}}^{\edgeWt_{v,v'}}.
\end{align*}
(Notice that
$\beta_n(j_v,j_{v'}) =\abr{1_{n,j_v}, 1_{n,j_{v'}}}_\calh$ is well-defined
%\begin{align*}
%  \beta_{n,V}(j,[v,v'])=\beta_n(j_v,j_{v'}) =\abr{1_{n,j_v}, 1_{n,j_{v'}}}_\calh
%\end{align*}
for each $[v,v']\in p(V)$.)
When $V$ is a singleton, we read
$\beta_{n,V}(j,\edgeWt)=1$ 
%$\betaprod{V,\edgeWt}{j}=\prod_{[v,v']\in p(V)} \beta_{j_{v},j_{v'}}^{\edgeWt_{v,v'}}=1$
by convention, since $p(V)=\emptyset$.
For notational convenience, we may write
\begin{align*}
  \beta_n^{G}(j)=\beta_{n,V}(j,\edgeWt)
\end{align*}
for a weighted graph $G=(V,\edgeWt,\vertWt)$.

For
$n\in\bbN$, 
a weighted graph $G=(V,\edgeWt,\vertWt)$,
$j\in[n]^V$ and 
$\bbf\in\calf(V)$,
we define a functional $B_n^G(j,\bbf)$ by
\begin{align*}
  B_n^G(j,\bbf) &= 
  \beta_n^{G}(j)\; %\beta_{n,V}(j,\edgeWt)\;
  %\prod_{[v,v']\in P(V)} \beta_{j_{v},j_{v'}}^{\edgeWt_{v,v'}}.
  \prod_{v\in V}I_{\vertWt(v)}(\kerfvn{v}{j_v}^{\otimes\vertWt(v)}).
\end{align*}
If $G=(V,\edgeWt,\vertWt)$ is a connected weighted graph,
we define another functional $\check{B}_n^G(j,\bbf)$
\begin{align*}
  \check{B}_n^G(j,\bbf) &= 
  \beta_n^{G}(j)\; %\beta_{n,V}(j,\edgeWt)\; 
  %\prod_{[v,v']\in P(V)} \beta_{j_v,j_{v'}}^{\edgeWt([v,v'])}
  \delta^{\barq(G)}\Brbr{\subotimes{v\in V} (\kerfvn{v}{j_v}^{\otimes\vertWt(v)})}
  %=
  %\delta^{\barq(G)}  \Brbr{\subotimes{v\in V} (\kerfvn{v}{j_v}^{\otimes \vertwtlow_v})}
  %\prod_{[v,v']\in P(V)} \beta_{j_v,j_{v'}}^{\edgeWt_{v,v'}}
\end{align*}
We read
$I_{q}(f^{\otimes q})= 1$ if $q=0$ and
$\delta^{\barq(G)}\Brbr{\subotimes{v\in V} (\kerfvn{v}{j_v}^{\otimes\vertWt(v)})}=1$
if $\barq(G)=0$ by convention.
\end{definition}
%Notice that $\betaprod{G}{j}$ is independent of $\vertWt$.
We write $B_n^G(j)$ and $\check{B}_n^G(j)$ for
$B_n^G(j,\bbone)$ and $\check{B}_n^G(j,\bbone)$, respectively.
%Also we sometimes abbreviate $B_n(G,j,\bbf)$ as $B^G_n(j)$.
%
We often use the next lemma.
\begin{lemma}
  For a weighted graph $G=(V,\edgeWt,\vertWt)$, $n\in\bbN$, $j\in[n]^V$ and $\bbf\in\calf(V)$,
\begin{alignat*}{2}
  \beta_n^{G}(j) 
  &=\prod_{C\in \Comp(G)} \beta_{n}^{C}(j_{V_C}),&\qquad
  B_n^G(j,\bbf) &= \prod_{C\in \Comp(G)} B_n^C(j_{V_C},\bbf|_{V_C}),
\end{alignat*}
where 
%$C(G)$ is the set of components of $G$ and 
$V_C$ is the set of vertices of $C$,
$j_{V_C}=j|_{V_C}$ and
$\bbf|_{V_C}=(\bbf_v)_{v\in{V_C}}(\in\calf(V_C))$.
\end{lemma}

\vspssm
\begin{definition}
  For $n\in\bbN$, $\alpha\in\bbR$,
  a weighted graph $G=(V,\edgeWt,\vertWt)$,
  $A\in\cala(V)$ and $\bbf\in\calf(V)$,
  %\redb{$A$ and $\bbf$ satisfies Assumption \ref{220119.2343}.}
we define a functional $\cali_n(\alpha, G, A, \bbf)$ by
\begin{align*}
  \cali_n(\alpha, G, A, \bbf)&=
  n^\alpha \sum_{j\in[n]^V} A_n(j)\; B_n^G(j,\bbf).
  %=
  %n^\alpha \sum_{j\in[n]^V} A_n(j)
  %  \prod_{C\in \Comp(G)} B_n^C (j_{V_C}, \bbf|_{V_C})
  %  \label{220118.2623}
  %%\\&=
    %%\sum_{j\in[n]^V} A_n(j)
    %%\prod_{c\in C_+} B_n(c;j_{V_c};\bbf|_{V_c})
    %%\prod_{c\in C_0} B_n(c;j_{V_c})
  %\\&=
  %n^\alpha\sum_{j\in[n]^V} A_n(j)
  %\prod_{c\in C_+(G)}\bbcbr{ 
  %  %\betaprod{c}{j_{V_c}}
  %  \bbrbr{\prod_{[v,v']\in P(V_c)} \beta_{j_{v},j_{v'}}^{\edgeWt_{v,v'}}}
  %  \bbrbr{\prod_{v\in V_c} I^{\vertwtlow_v}(\kerfvn{v}{j_v}^{\otimes \vertwtlow_v})}}
  %  \prod_{c\in C_0(G)}%\betaprod{c}{j_{V_c}}
  %  \bbcbr{\prod_{[v,v']\in P(V_c)} \beta_{j_{v},j_{v'}}^{\edgeWt_{v,v'}}},
  %  \nn
\end{align*}
\end{definition}
The functional $\cali_n(\alpha, G, A, \bbf)$ can be transformed as:
\begin{align}
  \cali_n(\alpha, G, A, \bbf)&=
  n^\alpha \sum_{j\in[n]^V} A_n(j)
    \prod_{C\in \Comp(G)} B_n^C (j_{V_C}, \bbf|_{V_C})
    \label{220118.2623}
  %\\&=
    %\sum_{j\in[n]^V} A_n(j)
    %\prod_{c\in C_+} B_n(c;j_{V_c};\bbf|_{V_c})
    %\prod_{c\in C_0} B_n(c;j_{V_c})
  \\&=
  n^\alpha\sum_{j\in[n]^V} A_n(j)
  \prod_{C\in \Comp_+(G)}\bbcbr{ 
    %\betaprod{c}{j_{V_c}}
    \bbrbr{\prod_{[v,v']\in p(V(C))} \beta_n\rbr{j_{v},j_{v'}}^{\edgeWt_{v,v'}}}
    \bbrbr{\prod_{v\in V(C)} I_{\vertwtlow_v}(\kerfvn{v}{j_v}^{\otimes \vertwtlow_v})}}
   \nn\\&\hspace{90pt}\times
    \prod_{C\in \Comp_0(G)}%\betaprod{c}{j_{V_c}}
    \bbcbr{\prod_{[v,v']\in p(V(C))} \beta_n\rbr{j_{v},j_{v'}}^{\edgeWt_{v,v'}}},
    \nn
\end{align}
where 
$\Comp_+(G)=\cbr{C\in \Comp(G)\mid \barq(C)>0}$ and 
$\Comp_0(G)=\cbr{C\in \Comp(G)\mid \barq(C)=0}$.

%%\newpage
\subsection{First exponent}\label{211006.0300}
In order to obtain estimates of orders of the functionals
having the form (\ref{220118.2623}),
% belonging to $\cals$,
we define the exponent based on a weighted graph.
First we define the following exponent $e(C)$
for a connected weighted graph $C=(V,\edgeWt,\vertWt)$ as follows.
For $C$ with $s(C)\in\cbr{0,1}$, 
\begin{align*}
  e(C)=2-I(C)-s(C);
\end{align*}
if $s(C)\geq2$,
\begin{subnumcases}{e(C)=}
  (2-I(C)) -1 + (1/2-2H) - H(s(C)-2)
  &\text{if $\barq(C)$ is odd, 
  or if $\barq(C)$ is even 
  and  $\displaystyle\max_{v\in V} \vertwtlow(v) > \frac{1}{2}\barq(C)$}
  \label{210820.0123}
  \\
  (2-I(C)) + 2(1/2-2H) - H(s(C)-2)
  &\text{if $\barq(C)$ is even %$\barq(G)\in2\bbZ_{\geq0}$
    and $\displaystyle\max_{v\in V} \vertwtlow(v) \leq \frac{1}{2}\barq(C)$}.
\label{210820.0124}
\end{subnumcases}
Notice that the case (\ref{210820.0124}) includes the case $\barq(C)=0$.

Since $-H \geq 1/2 - 2H\geq - 1$ for $H\in(1/2,3/4)$,
we have
\begin{align}\label{211029.1300}
  e(C) \geq 2 - I(C)- s(C),
\end{align}
for any connected weighted graph $C$, and
\begin{align}\label{211121.1302}
  (2-I(C)) -1 + (1/2-2H) - H(s(C)-2)\leq
  e(G) \leq
  (2-I(C)) + 2(1/2-2H) - H(s(C)-2).
\end{align}
for $C$ with $s(C)\geq2$.

For $\alpha\in\bbR$ and a weighted graph $G$,
we define the exponent $e(\alpha,G)$ by
\begin{align*}
  e(\alpha,G) = \alpha  + \sum_{C\in \Comp(G)} e(C).
  %\label{211028.2347}
\end{align*}
Here we abuse the notation $e$ for connected weighted graphs and $(\alpha,G)$.

\begin{remark}\label{220502.2011}
  An exponent was defined in \cite{yoshida2020asymptotic} 
  for functionals having the following form:
  \begin{align*}
    \cali_n = 
    n^\alpha \sum_{j\in[n]^m} A_n(j) I_{q_1}(1_{j_1}^{\otimes q_1}) \cdots 
    I_{q_m}(1_{j_m}^{\otimes q_m})
    \qquad\twith q_1,...,q_m\geq0,
  \end{align*}
  where the multiple stochastic integrals are defined with respect to a classical Brownian motion.
  With a weighted graph $G = \vee_{v\in[m]} (\cbr{v}, 0 , q_v)$,
  the above functional $\cali_n$ is written as 
  \begin{align*}
    \cali_n=\cali_n(\alpha, G, A, \bbone),
  \end{align*}
  where we abused the notation $\cali_n$ for $H=1/2$.
  The corresponding exponent $e(\alpha,G)$ is calculated as 
  \begin{align}
    e(\alpha,G) 
    = \alpha + \sum_{v\in[m]} e\rbr{(\cbr{v}, 0, q_v)}
    = \alpha 
    + \sum_{v\in[m]: q_v=0} 1 
    + \sum_{v\in[m]: q_v=1} 0
    + \sum_{v\in[m]: q_v\geq2} (\frac12 - H q_{v}).
    \label{220502.1914}
  \end{align}
  
  Since the exponent in \cite{yoshida2020asymptotic} is written as 
  \begin{align}
    e(\cali_n) &=
      \alpha -\frac12 \sum_{v\in[m]}q_v +m 
    %-\frac12 \sum_{v\in[m]: q_v>0} 1
    -\frac12 \Babs{\cbr{v\in[m]\mid q_v>0}}
    =%\nn\\&=
    \alpha 
    + \sum_{v\in[m]: q_v>0}\rbr{-\frac12 q_v + \frac12}
    %+ \sum_{v\in[m]: q_v=0}\rbr{-\frac12 q_v + 1}
    + \sum_{v\in[m]: q_v=0}1,
    \label{220502.1915}
  \end{align}
    the two exponents (\ref{220502.1914}) and (\ref{220502.1915}) coincide with $H=1/2$.
\end{remark}

\subsubsection{Estimate of the order of expectation of the functional 
$\cali_n(\alpha,G,A,\bbf)$}
For a weighted graph $G$,
we write
$\Comp_+(G) =\cbr{C\in \Comp(G)\mid \barq(C)>0}$ and
$\Comp_0(G) =\cbr{C\in\Comp(G)\mid \barq(C)=0}$.

\begin{proposition}\label{210108.1642}
  Let $\alpha\in\bbR$,
  $G=(V,\edgeWt,\vertWt)$ a weighted graph, 
  $A\in\cala(V)$ and $\bbf\in\calf(V)$.
  Then,
  \begin{align}\label{220118.1424}
    E[\cali_n] = O(n^{e(\alpha,G)})
  \end{align}
  as $n\to\infty$ for
  $\cali_n = \cali_n(\alpha,G,A,\bbf)$
  given by (\ref{220118.2623})
\end{proposition}

\begin{proof}
Without loss of generality, we can assume $\alpha=0$.
We will prove the statement only in the case
$\bbf=\bbone$
since we can extend the argument below using Lemma \ref{201229.1452}(1)(a) and (2)
to general cases $\bbf\in\calf(V)$.

We prove the estimate (\ref{220118.1424}) 
by induction with respect to $\abs{\Comp_+(G)}$.
For a weighted graph $G=(V,\edgeWt, \vertWt)$ with $\abs{\Comp_+(G)}=0$ 
(i.e. $\vertWt=0$)
and $A\in\cala(V)$, 
we have
%by Lemma \ref{211028.2340} and Assumption \ref{201229.1450},
\begin{align*}
  \Babs{ E[\cali_n] }
  &=
  \abs{ E\bbsbr{ \sum_{j\in[n]^V} A_n(j) \prod_{C\in \Comp_0(G)}
    %\Bcbr{\prod_{\vpr<\vppr\in V_c} \beta_{j_\vpr,j_\vppr}^{\edgeWt_{\vpr,\vppr}}}
    \beta_{n,V_C}(j_{V_C}, \edgeWt|_{p(V_C)})}}
    %\beta_n^c(j_{V_c})}}
  %\leq%\\&\leq
  %\sum_{j\in[n]^V}\babs{E\sbr{A_n(j)}} \prod_{c\in C_0(G)} 
  %\beta_{n,V_c}(j_{V_c},\edgeWt|_{p(V_c)})
  %%\beta_n^c(j_{V_c})
  \\
  &\leq
  \sup_{n\in\bbN, j\in[n]^V} \norm{A_n(j)}_{L^1(P)}\;
  \sum_{j\in[n]^V}\prod_{C\in \Comp_0(G)} 
  \beta_{n,V_C}(j_{V_C},\edgeWt|_{p(V_C)})
  %\beta_{n,V_c}(j_{V_c}, \edgeWt|_{p(V_c)})
  \\
  &=C_{(\ref{220119.1721})}
  \prod_{C\in \Comp_0(G)}  \Brbr{\sum_{j\in [n]^{V_C}}
  \beta_{n,V_C}(j, \edgeWt|_{p(V_C)})
  %\beta_n^c(j)
  %\prod_{[v,v']\in p(V_c)} \beta_{j_{v},j_{v'}}^{\edgeWt_{v,v'}}
  }
  %%=C_{(\ref{220119.1721})}
  %%\prod_{c\in C_0(G)}  \bar\beta_{n,V_c}(\edgeWt|_{p(V_c)})
  %&&\text{(by Assumption \ref{201229.1450})}
  \\&=
  O(n^{\sum_{C\in \Comp_0(G)}e(C)})
  =O(n^{e(0,G)}),
  &&\text{(by Lemma \ref{211028.2340})}
  %\\&=O(n^{e(0,G)}),
  %&&\text{(by the definition (\ref{211028.2347}))}
\end{align*}
where $C_{(\ref{220119.1721})}$ is a constant obtained from
the bound (\ref{220119.1721}) in the assumption on $A$.
We applied Lemma \ref{211028.2340} to
the connected weighted graphs 
$C=(V(C), \edgeWt|_{p(V(C))},0)\in \Comp_0(G)$ with no weights on vertices.

\vspssm
Let $d\in\mathbb{Z}_{\geq0}$, and
assume that the estimate (\ref{220118.1424}) 
holds true for $G$ with $\abs{\Comp_+(G)}\leq d$. 
We are going to prove the estimate %(\ref{220118.1424})
for $G$ satisfying $\abs{\Comp_+(G)}=d+1$.
Fix such $G=(V,\edgeWt,\vertWt)$ and pick $C_1\in \Comp_+(G)$ arbitrarily and 
we denote $\cplusrest = \Comp_+(G)\setminus \cbr{C_1}$.
We write $V_1=V(C_1)$ (the vertex set of $C_1$).
For brevity, we may omit $(G)$ of $\Comp(G)$, $\Comp_+(G)$ etc.
For $\pi:p(V_1)\to\bbZ_{\geq0}$,
we write
$\pi_v = \sum_{v'\in V_1,v'\neq v}\pi([v,v'])$ for $v\in V_1$ and 
$\bar\pi = \sum_{[v,v']\in p(V_1)}\pi([v,v'])$.
Notice that 
\begin{align}\label{220110.2619}
  2\bar\pi=\sum_{v\in V_1}\pi_v.
\end{align}
We set 
\begin{align}
  \Pi(\vertWt|_{V_1})=%\Pi_{V_{C_1}}(\vertWt|_{V_{C_1}})=
  %\Pi(V_{c_1},\vertWt|_{V_{c_1}})=
  \cbr{\pi:p(V_1)\to\bbZ_{\geq0}\mid
  \vertwtlow(v)\geq\pi_v \tforsm v\in V_1}.
  \label{211225.2001}
\end{align}
If $\abs{V_1}=1$, then $p(V_1)=\emptyset$ and 
we consider $\Pi(\vertWt|_{V_{1}})$ is a singleton 
with the function $\pi:\emptyset\to\bbZ_{\geq0}$, 
for which we read $\pi_v=0$ for $\cbr{v}=V_1$ and $\bar\pi=0$.
By the product formula (Lemma \ref{211031.2620}),
\begin{align}
  \cali_n =&
    \sum_{j\in[n]^V} A_n(j)\;
      \bbrbr{\prod_{v\in V_1}I_{\vertwtlow_v}(\charf{j_v}^{\otimes \vertwtlow_v})}\,
      \beta_{n}^{C_1}(j_{V_1})
      %\prod_{[v,v']\in p(V_{c_1})} \beta_{j_{v},j_{v'}}^{\edgeWt_{v,v'}}
      \prod_{C\in \cplusrest\sqcup \Comp_0} B_n^C(j_{V_C})
      %\prod_{c\in \cplusrest\sqcup C_0} B_n(c, j_{V_c}, \bbone|_{V_c})
      %\prod_{c\in C_0} B_n(c;j_{V_c})}
  %\nn\\&
  %  \prod_{v\in V_{c_1}} I^{\vertwtlow_v}(1_{j_v}^{\otimes \vertwtlow_v}) 
  %    \prod_{\vpr<\vppr\in V_{c_1}}\beta_{j_{\vpr},j_{\vppr}}^{\edgeWt_{\vpr,\vppr}} \times  
  %  \prod_{c\in \cplusrest}\bigg\{
  %    \prod_{v\in V_c} I^{\vertwtlow_v}(1_{j_v}^{\otimes \vertwtlow_v})
  %    \prod_{\vpr<\vppr\in V_c}\beta_{j_{\vpr},j_{\vppr}}^{\edgeWt_{\vpr,\vppr}} 
  %     \bigg\}
  %  \prod_{c\in C_0}\bigg\{  
  %    \prod_{\vpr<\vppr\in V_c}\beta_{j_\vpr,j_\vppr}^{\edgeWt_{\vpr,\vppr}}  
  %    \bigg\}
  \nn\\=&
  \sum_{j\in[n]^V} A_n(j)
  \sum_{\pi\in \Pi(\vertWt|_{V_1})} C(\vertWt|_{V_1}, \pi)~
      \delta^{(\barq(C_1) - 2\bar\pi)}
      \Brbr{\subotimes{v\in V_{1}} (
        1_{j_v}^{\otimes(\vertwtlow_v - \pi_v)})}
      \prod_{[v,\vpr]\in p(V_{1})} \beta_{n}\rbr{j_{v},j_{\vpr}}^{\pi_{v,v'}}
      %\cdot\prod_{v'<v''\in V_{c_1}} \langle f_{v'}, f_{v''}\rangle^{\edgeWt_{v',v''}}
      \nn\\&\hspace{120pt}\times
      \beta_{n}^{C_1}(j_{V_{1}})
      %\prod_{[v,\vpr]\in p(V_{c_1})}\beta_{j_{v},j_{\vpr}}^{\edgeWt_{v,\vpr}} 
      \prod_{C\in \cplusrest\sqcup \Cqzero} B_n^C(j_{V_C})
    %\prod_{c\in \cplusrest\sqcup C_0} B_n(c, j_{V_c}, \bbone|_{V_c})
  %\nn\\&\times  
  %  \prod_{c\in \cplusrest }\bigg\{
  %    \prod_{v\in V_c} I^{\vertwtlow_v}(1_{j_v}^{\otimes \vertwtlow_v})
  %    \prod_{\vpr<\vppr\in V_c}\beta_{j_{\vpr},j_{\vppr}}^{\edgeWt_{\vpr,\vppr}} 
  %    \bigg\}
  %  \prod_{c\in C_0} \cbr{ 
  %    \prod_{\vpr<\vppr\in V_c}\beta_{j_\vpr,j_\vppr}^{\edgeWt_{\vpr,\vppr}}}
  \nn\\=:&
  \sum_{\pi\in \Pi(\vertWt|_{V_{1}})}
  C(\vertWt|_{V_{1}}, \pi)~ \cali_n^{(\pi)}
  \label{211101.1201}
\end{align}
where we write $\pi_{v,v'}=\pi([v,v'])$ for $[v,v']\in p(V_{1})$ for short,
and $C(\vertWt|_{V_{1}}, \pi)$ is a integer-valued constant 
dependent only on $\vertWt|_{V_{1}}$ and $\pi$.
The functional $\cali_n^{(\pi)}$ can be written as 
\begin{align*}
  \cali_n^{(\pi)}=
  &\sum_{j\in[n]^V} A_n(j)~
      \delta^{(\barq(C_1) - 2\bar\pi)}
      \Brbr{\subotimes{v\in V_{1}}(1_{j_v}^{\otimes(\vertwtlow_v - \pi_v)})}
      %\beta_{n,V_{c_1}}(j_{V_{c_1}}, \edgeWt+\pi)
      %\beta_{n,V_{c_1}}(j_{V_{c_1}}, \edgeWt|_{V_{c_1}}+\pi)
      \prod_{[v,\vpr]\in p(V_{1})} 
      \beta_n\rbr{j_{v},j_{\vpr}}^{\pi_{v,v'}+\edgeWt_{v,\vpr}}
      \prod_{C\in \cplusrest\sqcup \Cqzero} B_n^C(j_{V_C})
      %\prod_{c\in \cplusrest\sqcup C_0} B_n(c, j_{V_c}, \bbone|_{V_c}).
\end{align*}

\vspssm
Set $\dot V =(\sqcup_{C\in \cplusrest} V(C)) \sqcup \{0\}$.
For $\pi\in \Pi(\vertWt|_{V_{1}})$, let
\begin{align*}
  I(\pi)
  %I_{G,c_1}(\pi)
  %I(\vertWt|_{V_{c_1}},\pi;(\vertWt|_{V_c})_{c\in \cplusrest})
  =\bcbr{i:V_{1}\times\dot V\to\bbZ_{\geq0}\mid
  \text{satisfying }(\ref{220110.2620})\tandsm (\ref{211101.1825})}
\end{align*}
  with 
\begin{alignat}{2}
  \sum_{v'\in\dot V} i(v,v') &=\vertwtlow(v) - \pi_v
  &&\quad\text{ for $v\in V_{1}$}\label{220110.2620}\\
  \sum_{v\in V_{1}} i(v,v') &\leq \vertwtlow(v')
  &&\quad\text{ for $v'\in \bigsqcup_{C\in \cplusrest} V(C)$}.\label{211101.1825}
\end{alignat}
We write $i_{v'} = \sum_{v\in V_1} i(v,v')$ for $v'\in\dot V$ and
$\bar i_C = \sum_{v'\in V(C)} i_\vpr$ for $C\in \cplusrest$.
%\redb
We have $\barq(C)\geq \bar i_C(\geq0)$ by (\ref{211101.1825}).
We abbreviate $i(v,v')$ as $i_{v,v'}$.
%$I_{G,c_1}(\pi)$ as $I(\pi)$.
%$I(\vertWt|_{V_{c_1}},\pi;(\vertWt|_{V_c})_{c\in \cplusrest})$ as $I(\vertWt,\pi)$.
By using the duality of $\delta$ and $D$, 
we have
\begin{align}
  %\Expe{ \cali_n^{(\pi)} }
  E \sbr{ \cali_n^{(\pi)} }
  =&
  %E\bigg[\sum_{j\in[n]^V} A_n(j)
  %\cdot\delta^{(\barq(c_1) - 2\bar\pi)}
  %\Brbr{\subotimes{v\in V_{c_1}} (1_{j_v}^{\otimes(\vertwtlow_v - \pi_v)})}
  %\cdot\prod_{[v,\vpr]\in p(V_{c_1})} \beta_{j_{v},j_{\vpr}}^{\pi_{v,v'} + \edgeWt_{v,\vpr}}
  %%%%\cdot\prod_{v'<v''\in V_{c_1}} \langle f_{v'}, f_{v''}\rangle^{\edgeWt_{v',v''}}
  %\prod_{c\in \cplusrest\sqcup C_0} B_n(c;j_{V_c};\bbone|_{V_c})
  %\bigg]
  %\nn\\=&
  \sum_{j\in[n]^V} E\sbr{ A_n(j)\;
  \delta^{(\barq(C_1) - 2\bar\pi)}
  \Brbr{\subotimes{v\in V_1} (1_{j_v}^{\otimes(\vertwtlow_v - \pi_v)})}\;
  \prod_{[v,v']\in p(V_1)} 
  \beta_n\rbr{j_{v},j_{\vpr}}^{\pi_{v,v'} +\edgeWt_{v,v'}}
  %\cdot\prod_{v'<v''\in V_{c_1}} \langle f_{v'}, f_{v''}\rangle^{\edgeWt_{v',v''}}
  \prod_{C\in \cplusrest\sqcup\Cqzero} B_n^C(j_{V_C})}
  \nn\\=&
  \sum_{j\in[n]^V} E\bigg[ \sum_{i\in I(\pi)} C(i)\;
  D^{i_0}_{\otimes_{v\in V_1} (1_{j_{v}}^{\otimes i_{v,0}})} A_n(j)
  %\cdot\delta^{(\barq({c_1}) - 2\bar\pi)}
  %\big(\otimes_{v\in V_{c_1}} (
  %  1_{j_v}^{\otimes(\vertwtlow_v - \pi_v)})\big)
  \prod_{[v,\vpr]\in p(V_1)} 
  \beta_n\rbr{j_{v},j_{v'}}^{\pi_{v,v'} +\edgeWt_{v,v'}}
  %\cdot\prod_{v'<v''\in V_{c_1}} \langle f_{v'}, f_{v''}\rangle^{\edgeWt_{v',v''}}
  \nn\\&\hspace{70pt} \times  
  \prod_{C\in\cplusrest}\rbr{
    \beta_{n}^{C}(j_{V_{C}})
    \prod_{v'\in V_C} D^{i_{v'}}_{\otimes_{v\in V_1} (1_{j_{v}}^{\otimes i_{v,v'}})} 
    I_{\vertwtlow_{v'}}(1_{j_{v'}}^{\otimes \vertwtlow_{v'}})}\,
  %\prod_{[v,\vpr]\in p(V_{c})}\beta_{j_{v},j_{\vpr}}^{\edgeWt_{v,\vpr}}}
  \prod_{C\in\Cqzero} B_n^C(j_{V_C})
  %\cbr{\prod_{\vpr<\vppr\in V_c}\beta_{j_\vpr,j_\vppr}^{\edgeWt_{\vpr,\vppr}}}
  \bigg]
  \nn\\=&
  \sum_{i\in I(\pi)} E\bigg[ C(i) C(\vertWt,i) \sum_{j\in[n]^V} \bigg\{
    D^{i_0}_{\otimes_{v\in V_1} (1_{j_{v}}^{\otimes i_{v,0}})} A_n(j)
    \prod_{[v,\vpr]\in p(V_1)} 
    \beta_n\rbr{j_{v},j_{v'}}^{\pi_{v,v'} +\edgeWt_{v,v'}}
  %\cdot\prod_{v'<v''\in V_{c_1}} \langle f_{v'}, f_{v''}\rangle^{\edgeWt_{v',v''}}
  \nn\\&\hsp \times  
  \prod_{C\in \cplusrest} \bbrbr{
    \beta_{n}^{C}(j_{V_{C}})
    \Brbr{\prod_{v\in V_1,v'\in V_C} \beta_n\rbr{j_{v}, j_{v'}}^{i_{v,v'}}}
    \Brbr{\prod_{v'\in V_C} I_{(\vertwtlow_{v'}-i_{v'})}
    (1_{j_{v'}}^{\otimes (\vertwtlow_{v'}-i_{v'})})}
    %\prod_{[v,\vpr]\in p(V_{c})}\beta_{j_{v},j_{\vpr}}^{\edgeWt_{v,\vpr}} 
  }
  \prod_{C\in\Cqzero} B_n^C(j_{V_C})
  %\prod_{c\in C_0}B_n(c;j_{V_c};\bbone|_{V_c})
    \bigg\}\bigg]
  \nn\\=:&
  \sum_{i\in I(\pi)} C(i) C(\vertWt,i) E\sbr{\cali_n^{(\pi,i)}},
  \label{211101.2100}
\end{align}
where $C(i)$ and $C(\vertWt,i)$ are the integers which come from 
the repetition by the Leibniz rule and the action of $D$ to multiple stochastic integrals respectively.
%
%For $c\in \cplusrest$, set $\bar i_c = \sum_{v'\in V_c} i_\vpr$ and we have
%$\barq(c)\geq \bar i_c(\geq0)$ by (\ref{211101.1825}).
%$\delthetab = \sum_{c\in \cplusrest} \bar i_c = \sum_{c\in C^\calb_+} \bar i_c$.
For each $\pi\in\Pi(\vertWt|_{V_{c_1}})$ and 
$i\in I(\pi)$,
let
$\CBplus = \cbr{C\in \cplusrest\mid \bar i_C>0}$,
$\CUBplus = \cbr{C\in \cplusrest\mid \bar i_C=0}$ and 
$\CB = \cbr{C_1}\sqcup\CBplus$.
%$\CBplus(\pi,i) = \{ c\in \cplusrest: \bar i_c>0\}$,
%$\CUBplus(\pi,i) = \{ c\in \cplusrest: \bar i_c=0\}$ and 
%$\CB(\pi,i) = \cbr{c_1}\sqcup\CBplus(\pi,i)$.
%We omit $(\pi,i)$ and simply write $\CBplus$, $\CUBplus$ and $\CB$, respectively.
The functional $\cali_n^{(\pi,i)}$ can be written as
\begin{align*}
  \cali_n^{(\pi,i)}
  =&
  n^{-i_0} \sum_{j\in[n]^V} \bigg\{
  n^{i_0} D^{i_0}_{\otimes_{v\in V_{1}} (1_{j_{v}}^{\otimes i_{v,0}})} A_n(j)\,
  \prod_{[v,\vpr]\in p(V_1)} \beta_n\rbr{j_{v},j_{v'}}^{\pi_{v,v'} + \edgeWt_{v,v'}}
  \bbrbr{\prod_{C\in \CBplus} \beta_{n}^{C}(j_{V_{C}})}
  \nn\\&\hspace{50pt} \times  
  \bbrbr{\prod_{\substack{v\in V_{1}\\ v'\in \sqcup_{C\in\CBplus} V(C)}} 
  \beta_n\rbr{j_{v}, j_{v'}}^{i_{v,v'}}}
  \Brbr{\prod_{v'\in \sqcup_{C\in\CBplus} V(C)} I_{(\vertwtlow_{v'}-i_{v'})}
  (1_{j_{v'}}^{\otimes (\vertwtlow_{v'}-i_{v'})})}
  %\cdot\prod_{v'<v''\in V_{c_1}} \langle f_{v'}, f_{v''}\rangle^{\edgeWt_{v',v''}}
  %\nn\\&\hspace{50pt} \times  
  \prod_{C\in \Cqzero\sqcup\CUBplus}B_n^C (j_{V_C})
  \bigg\}
\end{align*}
By (\ref{220110.2619}) and (\ref{220110.2620}), 
we have the following relation
\begin{align}\label{210102.1051}
  \barq(C_1)
  = \sum_{v\in V(C_1)} \vertWt(v)
  = i_0 + \sum_{C\in \cplusrest} \bar i_C + 2\bar\pi
  = i_0 + \sum_{C\in \CBplus} \bar i_C + 2\bar\pi
\end{align}

\vspssm
%Fix $\pi\in\Pi(V_{c_1};\vertWt|_{V_{c_1}})$ and $i\in I(\vertWt,\pi)$.
We set the maps 
$\sigma^{(i)}:V_{1}\to\bbZ_{\leq0}$ and 
$\tau^{(i)}:p(\sqcup_{C\in \CB} V(C))\to\bbZ_{\geq0}$ as 
\begin{alignat*}{2}
  \sigma^{(i)}(v) &=-i(v,0)
  \hspace{16pt}\tfor v\in V_{1}
\\
  \tau^{(i)}([v,v'])&=
  \begin{cases}
    i(v,v')
    & \tfor v\in V_{1}
    \tandsm v'\in \sqcup_{C\in \CBplus} V(C)
    \\ 
    0
    & \tfor [v,v']\in 
    p(\sqcup_{c\in \CB} V(C))\setminus
    \cbr{[v_1,v_2]\mid v_1\in V_1, v_2 \in \sqcup_{C\in \CBplus} V(C)},
  \end{cases}
\end{alignat*}
and let
\begin{align*}
  C^{(\pi,i)}=
  \bdabr{\tau^{(i)}}\Brbr{\babr{\sigma^{(i)}}\brbr{\dabr{\pi} C_1}
  \vee \Brbr{\mathop\vee_{C\in \CBplus} C}}.
  %(V_{c^{(\pi,i)}},\edgeWt_{c^{(\pi,i)}}, \vertWt_{c^{(\pi,i)}})
\end{align*}
We define $G^{(\pi,i)}=(V,\edgeWt^{(\pi,i)},\vertWt^{(\pi,i)})$ 
and $A^{(\pi,i)}$ by
\begin{align*}
  G^{(\pi,i)}&=
  C^{(\pi,i)} \vee \Brbr{\mathop\vee_{C\in \Cqzero\sqcup\CUBplus} C}
  \\
  A^{(\pi,i)}&=(A^{(\pi,i)}_n(j))_{j\in[n]^V, n\in\bbZ} = 
\big(n^{i_0} D^{i_0}_{\otimes_{v\in V_{1}} (1_{j_{v}}^{\otimes i_{v,0}})}
 A_n(j)\big)_{j\in[n]^V, n\in\bbN}.
\end{align*}
The weighted graph $C^{(\pi,i)}$ is connected and we have
\begin{align*}
  V(C^{(\pi,i)})=\sqcup_{C\in\CB}V(C) \tand
  \Comp(G^{(\pi,i)})= \cbr{C^{(\pi,i)}}\sqcup\CUBplus\sqcup\Cqzero.
\end{align*}
We can see that 
%\redb
{$A^{(\pi,i)}\in\cala(V)$ by Lemma \ref{201229.1452}(1)(a)} and
$\cali_n^{(\pi,i)}$ is written as
\begin{align*}%\label{211101.2012}
  \cali_n^{(\pi,i)}
  =
  %n^{-i_0} \sum_{j\in[n]^V} \bigg\{
  %n^{i_0} D^{i_0}_{\otimes_{v\in V_{c_1}} (1_{j_{v}}^{\otimes i_{v,0}})} A_n(j)
  %\cdot\prod_{[v,\vpr]\in p(V_{c_1})} \beta_{j_{v},j_{\vpr}}^{\pi_{v,v'} + \edgeWt_{v,\vpr}}
  %%%%%\cdot\prod_{v'<v''\in V_{c_1}} \langle f_{v'}, f_{v''}\rangle^{\edgeWt_{v',v''}}
  %\nn\\&\hsp \times  
  %\prod_{c\in \cplusrest} \bbrbr{
  %  \prod_{v'\in V_c} I^{(\vertwtlow_{v'}-i_{v'})}(1_{j_{v'}}^{\otimes (\vertwtlow_{v'}-i_{v'})})
  %  \prod_{v\in V_{c_1}, v'\in V_c} \beta_{j_{v}, j_{v'}}^{i_{v,v'}}
  %  \prod_{[v,\vpr]\in p(V_{c})}\beta_{j_{v},j_{\vpr}}^{\edgeWt_{v,\vpr}}}
  %\prod_{c\in C_0}B_n(c;j_{V_c};\bbone|_{V_c})
  %\bigg\}
  %\nn\\=&
  \cali_n(-i_0, G^{(\pi,i)}, A^{(\pi,i)}, \bbone).
\end{align*}
Indeed, we have 
\begin{alignat*}{2}
  \vertWt^{(\pi,i)}(v) 
  &= \vertWt(v) - \pi_v - i(v,0) - 
  \sum_{\substack{v'\in\sqcup_{C\in\CB}V(C)\\ v'\neq v}} \tau^{(i)}([v,v'])
  = \vertWt(v) - \pi_v - \sum_{v'\in \dot V}i(v,v')=0
  %= \vertWt(v) - \pi_v - i_{v,0} - \sum_{c\in\CBplus}\sum_{v'\in V_c}i_{v,v'}=0
  &&\tfor v\in V(C_1)
  \\
  \vertWt^{(\pi,i)}(v')
  &= \vertWt(v') -
  \sum_{\substack{v\in\sqcup_{C\in\CB}V(C)\\ v\neq v'}} \tau^{(i)}([v,v'])
  = \vertwtlow(v') - \sum_{v\in V_{1}}i(v,v')=\vertwtlow(v')-i_{v'}
  &&\tfor v'\in \sqcup_{C\in \CBplus} V(C),
\end{alignat*}
where we used (\ref{220110.2620}) at the first equality,
and hence for each $j\in[n]^V$
\begin{align*}
  B_n^{C^{(\pi,i)}}(j_{V\rbr{C^{(\pi,i)}}})
  &=
  \prod_{[v,\vpr]\in p(V_{1})}
    \beta_n\rbr{j_{v},j_{\vpr}}^{\edgeWt_{v,\vpr}+\pi_{v,v'}}
  \prod_{C\in \CBplus} \beta_{n}^{C}(j_{V_{C}})
  \Brbr{\prod_{\substack{v\in V_{1}\\v'\in \sqcup_{C\in\CBplus} V_C}} 
  \beta_n\rbr{j_{v}, j_{v'}}^{i_{v,v'}}}
  \\&\hsp\times
  \Brbr{\prod_{v'\in \sqcup_{C\in\CBplus} V_C} I_{(\vertwtlow_{v'}-i_{v'})}
  (1_{j_{v'}}^{\otimes (\vertwtlow_{v'}-i_{v'})})}.
  \end{align*}

Notice that
\begin{align*}%\label{220118.1720}
  \babs{\Comp_+(G^{(\pi,i)})}\leq d.
\end{align*}
Indeed, if $\CBplus$ is empty, then
we have
$C^{(\pi,i)} = \babr{\sigma^{(i)}}\brbr{\dabr{\pi} C_1}$,
and hence by (\ref{220116.2106}) and (\ref{210102.1051}),
\begin{align}\label{220115.2445}
  \barq(C^{(\pi,i)})=
    \barq(C_1) 
    -2\sum_{[v,v']\in p(V_{1})}\pi([v,v'])
    +\sum_{v\in V_{1}} \sigma^{(i)}(v)
    =
    \barq(C_1) -2\bar\pi -i_0 =0,
\end{align}
which implies
$\Comp_+(G^{(\pi,i)})=\CUBplus=\cplusrest$ and hence
$\abs{\Comp_+(G^{(\pi,i)})}=d$;
if $\CBplus$ is non-empty, then 
$\Comp_+(G^{(\pi,i)})\subset\CUBplus\sqcup\cbr{C^{(\pi,i)}}$
and 
$\abs{\Comp_+(G^{(\pi,i)})}\leq\abs\CUBplus+1\leq(d-1)+1=d$.
Hence, by the assumption of the induction, we obtain 
\begin{align}\label{220324.1452}
  E\sbr{\cali^{(\pi,i)}_n} = O(n^{e(-i_0, G^{(\pi,i)})})
  \quad \text {as } n\to\infty.
\end{align}
%for $\pi\in\Pi_{V_{c_1}}(\vertWt|_{V_{c_1}})$ and $i\in I_{G,c_1}(\pi)$, and
We are going to prove 
$e(-i_0, G^{(\pi,i)})\leq e(0,G)$.

\vspsm
For $\pi\in\Pi(\vertWt|_{V_{1}})$ and $i\in I(\pi)$,
we have
\begin{align}\label{211107.1346}
 I({C^{(\pi,i)}}) = \sum_{C\in \CB}I({C}),
\end{align}
and by Proposition \ref{220110.2555},
\begin{align}
  s(C^{(\pi,i)})
  &=s(\babr{\sigma^{(i)}}\brbr{\dabr{\pi} C_1}) + \sum_{C\in\CBplus}s(C) 
    - 2(\abs{\CB}-1)
  \nn\\&
  =s(C_1) - i_0 + \sum_{C\in\CBplus}s(C) - 2\abs{\CBplus}
    = - i_0 + \sum_{C\in\CB}s(C) - 2\abs{\CBplus}.
     \label{211029.1556}
\end{align}

\vspssm
(i) First we consider $\pi,i$ such that $s(C^{(\pi, i)})\leq1$.
Since we have the following inequality
\begin{align*}
  - i_0 + e(C^{(\pi,i)}) 
  = - i_0 + (2-I({C^{(\pi,i)}}) - s(C^{(\pi,i)})) 
  &= -i_0  - s(C^{(\pi,i)}) - 2 \numCBplus 
  + \sum_{C\in \CB}(2-I(C)) 
    &&(\because(\ref{211107.1346}))
   %&&\text{($\because \numCBplus = \sum_{c\in C^\calb_+} 1$)}
  \nn\\
  &=
    \sum_{C\in \CB}(2-I(C)-s(C)) 
    &&(\because (\ref{211029.1556}))
    %&&\text{($\because s_{c_1} = i_0 + \numCBplus$ and 
    %$\sum_{c\in \CBplus} s_{c} = \numCBplus $)}
    %\\&=(2-I_{c_1}) -s_{c_1} + \sum_{c\in C_+^\calb}(1-I_{c}) + \sum_{c\in C_+^\mathcal{UB} \cup C_0} e(V_c)
  \nn\\
  &\leq
   \sum_{C\in\CB}e(C),   &&(\because(\ref{211029.1300}))
\end{align*}
the exponent $e(-i_0, G^{(\pi,i)}) $ is bounded as:
\begin{align*}
  e(-i_0, G^{(\pi,i)}) 
  &= - i_0 + \sum_{C\in \ncbr{C^{(\pi,i)}}\sqcup\CUBplus\sqcup \Cqzero} e(C)
  \leq
  \sum_{C\in\CB} e(C)
  +\sum_{C\in\CUBplus\sqcup\Cqzero} e(C)
  = e(0,G).
\end{align*}
%since $e(V_c)$ is bounded from below as $e(V_c) \geq 2-I_c -s_{c}$ in general.
%(See (\ref{211029.1300}).)

%%%%%%%%%%%%%%%%%%%%%%%%%%%%%%%%%%%%
%\begin{expoonenew}
    (ii) The case $s(C^{(\pi,i)})\geq2$. 

    (ii-i)
    First consider $\pi$ and $i$ such that $\babs\CB=1$ (i.e.  $\CB=\cbr{C_1}$)
    and $i_0=0$ hold.
    By $i_0=0$ and $\abs\CBplus=0$,
    the equalities (\ref{210102.1051}) and (\ref{211029.1556}) shows
    \begin{align}\label{220116.1711}
      \barq(C_1)=2\bar\pi \tand
      s(C_1) =
      s(C^{(\pi,i)})(\geq2)
    \end{align}
    respectively.
    Notice that $\cali_n^{(\pi)}$ is written as
    \begin{align*}
      \cali_n^{(\pi)}=
      &\sum_{j\in[n]^V} A_n(j)
          \prod_{[v,\vpr]\in p(V_{1})} 
          \beta_n\rbr{j_{v},j_{\vpr}}^{\pi_{v,v'}+\edgeWt_{v,\vpr}}
          \prod_{C\in\cplusrest\sqcup\Cqzero} B_n^C (j_{V_C}),
    \end{align*}
    and $\cali_n^{(\pi)}=\cali_n^{(\pi,i)}$ in this case.
  %Write $v_0=\argmax_{v\in V_{c_1}}\vertwtlow_v$.
  For $v\in V_{1}=V(C_1)$,
  we have 
  \begin{align}
    \pi_{v}
    &= \sum_{\substack{v'\in V_{1}\\v'\neq v}} \pi_{v',v}
    \leq \sum_{\substack{v'\in V_{1}\\ v'\neq v}}
    \sum_{\substack{v''\in V_{1}\\v''\neq v'}}\pi_{v',v''}
    = \sum_{\substack{v'\in V_{1}\\ v'\neq v}}\pi_\vpr,
    \label{220116.1632}
  \end{align}
  where we used 
  $\cbr{v}\subset\cbr{v''\in V_{1}\mid v''\neq v'}$ for $v'\in V_{1}$ with $v'\neq v$;
  %at the inequality in (\ref{220116.1632});
  hence
  \begin{align*}
    %\barq(c_1) = 
    2\bar\pi 
    = \sum_{v'\in V_{1}}\pi_\vpr
    = \sum_{\substack{v'\in V_{1}\\v'\neq v}}\pi_\vpr + \pi_{v}
    \leq 2\sum_{\substack{v'\in V_{1}\\ v'\neq v}}\pi_\vpr
    \leq 2\sum_{\substack{v'\in V_{1}\\ v'\neq v}}\vertwtlow(\vpr)
    =2(\barq({C_1})-\vertwtlow(v))
  \end{align*}
  for $v\in V_{1}$.
  We used 
  (\ref{220116.1632}) and (\ref{211225.2001})
  at the inequalities above, respectively.
  Thus, %under the assumption $\babs\CB=1$  and $i_0=0$, 
  by (\ref{220116.1711}) we obtained 
  \begin{align*}
    \vertwtlow(v) \leq\half \barq(C_1)
    \tfor v\in V_{1},
  \end{align*}
  which, with $\barq({C_1})=2\bar\pi\in2\bbZ_{\geq1}$,
  implies that the component $C_1$ satisfies the condition (\ref{210820.0124}).
  Since $\CBplus$ is empty,
  we have
  $I(C^{(\pi,i)}) = \sum_{C\in \CB}I({C}) = I({C_1})$ and
  $\barq({C^{(\pi,i)}})=0$. 
  (See (\ref{220115.2445}).)
  The latter and the assumption $s(C^{(\pi,i)})\geq2$
  suggests that  $C^{(\pi,i)}$ also satisfies 
  (\ref{210820.0124}), 
  and by (\ref{220116.1711}) we have $e(C_1)=e(C^{(\pi,i)})$; thus,
  \begin{align*}
    e(-i_0, G^{(\pi,i)})
    = e(C^{(\pi,i)}) + \sum_{C\in \CUBplus\sqcup\Cqzero} e(C) 
    = e(C_1)+ \sum_{C\in \CUBplus \sqcup \Cqzero} e(C) 
    =e(0,G).
  \end{align*}

  (ii-ii) Assume that $\babs\CB\geq2$ or $i_0\geq1$ 
  (i.e.  $\babs\CB+i_0\geq2$) hold.
  Let $\CBo = \cbr{C\in \CB\mid s(C)=1}$ 
  and $\CBt = \cbr{C\in \CB\mid s(C)\geq2}$.
  Then $\CB = \CBo \sqcup \CBt$, 
  since $\CB \subset \Comp_{+}$ and $s(C)\geq\barq(C)\geq1$ 
  for $c\in \Comp_+$.
  We have
  \begin{align}
    s(C^{(\pi,i)})-2 
    =& -i_0 + \sum_{C\in \CB}s(C)  -2\numCBplus - 2
    &&\text{($\because$ (\ref{211029.1556}))}
    \nn\\=&
    - i_0  + \sum_{C\in \CB} (s(C)-2) 
    =  -i_0 - \abs{\CBo} +\sum_{C\in \CBt}(s(C)-2),
    \label{211029.1730}
    \\
    \sum_{C\in \CB}e(C)
    %=&\sum_{c\in \CBo}e(c)+ \sum_{c\in \CBt}e(c)
    %\nn\\
    \geq&
    \sum_{C\in \CBo}(2 - I(C) - 1)
    + \sum_{C\in \CBt}(2 - I(C) - 1 + (1/2-2H) - H(s(C)-2))
    &&(\because (\ref{211121.1302}))
    \nn\\=&
    \babs\CB - \sum_{C\in \CB}I(C)
    +(1/2-2H)\abs\CBt 
    -H\sum_{C\in \CBt}(s(C) - 2),
    %- H(i_0 + \abs{\CBo} + s(c^{(\pi,i)})-2 ), &&(\because (\ref{211029.1730}))
    \label{220115.2255}
  \end{align}
  and hence
  \begin{align*}%%\label{211102.1809}
    e(0,G) - e(-i_0, G^{(\pi,i)})
    =&\sum_{C\in \CB} e(C)+ \sum_{C\in \CUBplus\sqcup\Comp_0} e(C)
    -\bbrbr{- i_0 + e(C^{(\pi,i)}) + \sum_{C\in \CUBplus\sqcup\Comp_0} e(C)}
    \nn\\=&
    \sum_{C\in \CB} e(C) + i_0 - e(C^{(\pi,i)})
    \nn\\\geq&
    \babs\CB - \sum_{C\in \CB}I(C)
    +(1/2-2H)\abs\CBt 
    -H\sum_{C\in \CBt}(s(C)-2),
    %- H(i_0 + \abs{\CBo} + s(c^{(\pi,i)})-2 )
    &&(\because (\ref{220115.2255}))
    \nn\\&
    + i_0 - \Brbr{2-I(C^{(\pi,i)}) + (1/2 - 2H)2 - H(s(C^{(\pi,i)})-2)}
    &&(\because (\ref{211121.1302}))
    \nn\\=&
    \babs\CB + i_0 -2
    + (1/2-2H)(\abs\CBt-2) - H(i_0 + \abs{\CBo})
    &&(\because (\ref{211107.1346}) \tandsm (\ref{211029.1730}))
    \nn\\\geq&
    \babs\CB + i_0 -2
    + (1/2-2H)\rbr{\abs\CB +i_0 -2}
    \nn\\=&
    (3/2-2H)\rbr{\abs\CB +i_0 -2}\geq0.
  \end{align*}
  We used 
  $-H>1/2-2H$ for $H\in\rbr{1/2,3/4}$
  at the second last inequality.
  We used the assumption $\babs\CB+i_0\geq2$ at the last inequality.
%\end{expoonenew}
%%%%%%%%%%%%%%%%%%%%%%%%%%%%%%%%%%%%

\vspssm
Therefore, in all the cases (i) and (ii), we have
$e(-i_0, G^{(\pi,i)})\leq e(0,G)$,
%and by (\ref{211101.2012}), (\ref{220118.1720})
%and the assumption of the induction, we have
%$E\big[ \cali_n^{(\pi,i)} \big] 
%= O(n^{e(-i_0, G^{(\pi,i)})}) = O(n^{e(0,G)})$ 
%for any $\pi\in\Pi(V_{c_1};\vertWt|_{V_{c_1}})$ and $i\in I(\vertWt,\pi)$.
and by (\ref{211101.1201}), (\ref{211101.2100}) and (\ref{220324.1452})
we obtain $$E\big[ \cali_n \big] = O(n^{e(0,G)}).$$
\end{proof}

\begin{proposition}\label{210407.1401}
  Let $\alpha\in\bbR$,
  $G=(V,\edgeWt,\vertWt)$ a weighted graph, 
  $A\in\cala(V)$ and $\bbf\in\calf(V)$.
  Define $\cali_n = \cali_n(\alpha,G,A,\bbf)$ by (\ref{220118.2623}) for $n\in\bbN$.

  Then, for $k\in\bbZ_{\geq1}$,
  \begin{align*}
    E\sbr{\rbr{\cali_n}^{2k}} = O(n^{2k\, e(\alpha,G)})
  \end{align*}
  as $n\to\infty$. In particular, for $p\geq1$,
  \begin{align*}
    \bnorm{\cali_n}_{p} = O(n^{e(\alpha,G)}).
  \end{align*}
\end{proposition}

\begin{proof}
  Without loss of generality, we assume that 
  $V$ is written as $V=[m]$ with some $m\in\bbN$.
  Set 
  \begin{align*}
    G' =& \vee_{k'=0,..,2k-1} \shiftg{G}{mk'}
    \\
    A'=&\rbr{A'_n(j)}_{j\in[n]^{2km}, n\in\bbN}\twith
    \\&\qquad 
    A'_n({j}) =
    A_n(j_{[m]})A_n(j_{[m+1,2m]})\cdots A_n(j_{[(2k-1)m+1,2km]})
    \tfor j\in[n]^{2km}
    \\
    \bbf' =& \rbr{\bbf'_v}_{v\in[2km]}
    \twith \bbf'_{v+mk'} = \bbf_v \tforsm v\in[m] \tandsm k'=0,..,2k-1,
  \end{align*}
  where $[k_1,k_2]=\cbr{k_1,...,k_2}$ for $k_1\leq k_2\in\bbN$.
  Then we have
  $%\begin{align*}
    \rbr{\cali_n(\alpha,G,A,\bbf)}^{2k}
    = \cali_n(2k\alpha, G', A', \bbf').
  $ %\end{align*}
  Since
  \begin{align*}
    e(2k\alpha, G')
    = 2k\alpha + \sum_{C\in \Comp(G')}e(C)
    = 2k\alpha + 2k\sum_{C\in \Comp(G)}e(C)
    = 2k\, e(\alpha, G),
  \end{align*} 
  by Proposition \ref{210108.1642} we obtain 
  \begin{align*}
    E\bsbr{\brbr{\cali_n(\alpha,G,A,\bbf)}^{2k}}
    = O(n^{e(2k\alpha, G')})
    = O(n^{2k\,e(\alpha, G)})
    \tas n\to\infty.
  \end{align*}
\end{proof}

\subsubsection{Change of the exponent by the action of $D_{u_n}$}
Given 
$A'\in\cala(V)$ with some singleton $V=\cbr{v}$,
define $\abs\calh$-valued random variable $u_n(A')$ by 
\begin{align}\label{220318.1700}
  u_n(A')=
  n^{2H-1/2} \sum_{j\in[n]} A'_n(j) I_1(1_j) 1_j,
\end{align}
where we identify $[n]=[n]^V$ and we recall 
$1_j=1_{\clop{Tn^{-1}(j-1),Tn^{-1}j}}$.
We often omit $(A')$ and simply write $u_n$ for $u_n(A')$.
We will give the estimate of how much the order of functionals in the form of
$\cali_n(\alpha,G,A,\bbf)$
changes by the action of $D_{u_n}$ in terms of the exponent.

For a weighted graph $G$,
we classify the components of $G$ as
$\Comp^s(G) = \cbr{C\in \Comp(G)\mid s(C)=s}$ for $s=0,1$ and
$\Ctp(G) = \cbr{C\in \Comp(G)\mid s(C)\geq2}$.
%and 
Moreover, let
%$\Ctpz(G) = \cbr{c\in\Ctp(G)\mid \barq(c) = 0}$,
%$\Ctpo(G) = \cbr{c\in\Ctp(G)\mid \barq(c)\text{ is odd}}$ and
%$\cpe(G) = \cbr{c\in\Ctp(G)\mid \barq(c)\text{ is positive and even}}$.
\begin{itemize}
\renewcommand\labelitemii{\labelitemi}
  \item $\Ctpz(G) = \cbr{C\in\Ctp(G)\mid \barq(C) = 0}$
  \item $\Ctpo(G) = \cbr{C\in\Ctp(G)\mid \barq(C)\text{ is odd}}$
  \item $\cpe(G) = \cbr{C\in\Ctp(G)\mid \barq(C)\text{ is positive and even}}$
  \begin{itemize}
    \item $\Cpearg{0}(G) = \cbr{C\in\cpe(G)\mid 
      \max_{v\in V(C)} \vertWt(v) \leq \half\barq(C)}$ 
    \item $\Cpearg{1}(G) = \cbr{C\in\cpe(G)\mid 
      \max_{v\in V(C)} \vertWt(v) = \half\barq(C)+1}$ 
    \item $\Cpearg{2}(G) = \cbr{C\in\cpe(G)\mid 
      \max_{v\in V(C)} \vertWt(v) \geq \half\barq(C)+2}$,
  \end{itemize}
\end{itemize}
Recall that $V(C)$ is the set of vertices of a weighted graph $C$.
We write
\begin{itemize}
  \item $\Ctpa(G) = \cbr{C\in\Ctp(G)\mid \barq(C) \text{ is odd, or } 
  \barq(C) \text{ is even and } \max_{v\in V(C)} \vertWt(v) > \half\barq(C)}$
  \item $\Ctpb(G) = \cbr{C\in\Ctp(G)\mid \barq(C) \text{ is even and } 
  \max_{v\in V(C)} \vertWt(v) \leq \half\barq(C)}$.
\end{itemize}
The above two conditions correspond to the two case 
(\ref{210820.0123}) and (\ref{210820.0124}) respectively.
%{and $C_{1-}=C_0\sqcup C^1$ since $\barq(c)=0$ for $c\in C_0$ and 
%$\barq(c)\leq1$ for $c\in C_{1-}$}.

We have
\begin{alignat*}{2}
  \Ctp=&\Ctpz\sqcup\cpe\sqcup\Ctpo=\Ctpa\sqcup\Ctpb,
  &\qquad%\\
  \cpe=&\Cpearg{0}\sqcup\Cpearg{1}\sqcup\Cpearg{2},
  \\
  \Ctpa=&\Ctpo\sqcup\Cpearg{1}\sqcup\Cpearg{2},
  &%\\
  \Ctpb=&\Ctpz\sqcup\Cpearg{0},
\end{alignat*}
where we omitted $(G)$ for brevity.

\begin{center}
 \begin{tabular}{ r|c|c|c|c }
   %\diagbox{$s$}{$\barq$} 
   & $\barq=0$ & $\barq=1$ & even, $\geq2$ & odd, $\geq2$\\ \hline\hline
   $s=0$ & $\Cszero$ & - & - & - \\ \hline
   $s=1$ & - & $\Csone$ & - & - \\ \hline
   $s\geq2$& $\Ctpz$ & $\Ctpo$ & $\Ctppe$ & $\Ctpo$ \\ \hline
 \end{tabular}
\end{center}

%
%For notational convenience, let 
%$\calc_{0}=\calc_{0,0} \vee \calc_{0,2}$ and 
%$\calc_{+}=\calc_{1} \vee \calc_{2,1}\vee\calc_{2,2} \vee \calc_{2,3}$.
%Notice that the condition $\calc_{2,2}\vee\calc_{2,3}$ 
%corresponds to (\ref{210820.0123}) and 
%$\calc_{0,2}\vee\calc_{2,1}$ to (\ref{210820.0124}).

\begin{proposition}\label{210822.1800}
  Suppose a weighted graph $G=(V, \edgeWt, \vertWt)$ and a singleton $V'$
  such that $V\cap V'=\emptyset$.
  Write $V''=V\sqcup V'$.
  Let $\alpha\in\bbR$, 
  $A\in\cala(V)$, $\bbf\in\calf(V)$ and $A'\in\cala(V')$.
  Consider
  $\cali_n = \cali_n(\alpha,G,A,\bbf)$
  given by (\ref{220118.2623}):
  \begin{align*}
    \cali_n =
    n^\alpha \sum_{j\in[n]^V} A_n(j)
    \prod_{C\in \Comp(G)} B_n^C(j_{V_C}, \bbf|_{V_C}) 
  \end{align*}
  Then,
  there exists a finite set $\Lambda$,
  $\alpha^\lambda\in\bbR$,
  a weighted graph 
  $G^\lambda=(V'',\edgeWt^\lambda,\vertWt^\lambda)$, 
  $A^\lambda\in\cala(V'')$ and $\bbf^\lambda\in\calf(V'')$
  for each $\lambda\in\Lambda$,
  such that
  $D_{u_n(A')} \cali_n$ can be written as
  \begin{align*}
    D_{u_n(A')}\cali_n = 
    \sum_{\lambda\in\Lambda} \cali_n(\alpha^\lambda, G^\lambda, 
    A^\lambda, \bbf^\lambda)
  \end{align*}
  %$(\alpha^\lambda, G^\lambda, A^\lambda, \bbf^\lambda)_{\lambda\in\Lambda}$
  and the following condition is fulfilled:
  \begin{alignat*}{3}
    &\text{(a)}&\;\;
    &\max_{\lambda\in\Lambda} e(\alpha^\lambda, G^\lambda)
    =e(\alpha, G) + 2H-\frac32
    &\hsp&\text{if $\Cpearg{1}(G)$ is empty;}
    \\
    &\text{(b)}&
    &\max_{\lambda\in\Lambda} e(\alpha^\lambda, G^\lambda) 
    =e(\alpha, G)
    &&\text{if $\Cpearg{1}(G)$ is nonempty.}
  \end{alignat*}
\end{proposition}

\begin{proof}
Without loss of generality, we assume $\alpha=0$, $\bbf=\bbone$
and $V=[m]$ and $V'=\cbr{m+1}$ with some $m\in\bbN$.
We abbreviate $u_n(A')$ as $u_n$.
Let $V_+ = \cbr{v\in V\mid \vertWt(v)>0}$.
For $v\in V$, define $C_v$ by $C\in\Comp(G)$  such that $v\in V(C)$.
Then for $v\in V_+$, we have 
$C_v\in\Csone(G)\sqcup\Ctpo(G)\sqcup\Ctppe(G)$.
The functional $D_{u_n}\cali_n$ is written as the following sum:
\begin{align*}%\label{210407.2101}
D_{u_n}\cali_n = \cali_{n}^{(0)} + \sum_{v\in V_+} \cali_{n}^{(v)},
\end{align*}
where we set 
\begin{align*}
  \cali_{n}^{(0)} =&
    n^{2H-3/2} \sum_{j\in[n]^{[m+1]}} 
    \Brbr{n D_{1_{j_\mpl}}A_n(j_{[m]})}
    A'_n(j_{m+1})\; I_1(1_{j_{m+1}})
    %a_{t_{j_\mpl-1}} I_1(1_{j_\mpl})
  \prod_{C\in \Comp(G)} B^C_n(j_{V(C)},\bbone)
  %\prod_{c\in C} \cbr{
  % \prod_{v'<v''\in V_c}\beta_{j_{v'},j_{v''}}^{\edgeWt_{v',v''}} 
  % \prod_{v\in V_c} I^{\vertwtlow_v}(1_{j_v}^{\otimes \vertwtlow_v})}
  \\
  \cali_{n}^{(v)} =&
  n^{2H-1/2} \sum_{j\in[n]^{[m+1]}} A_n(j_{[m]})A'_n(j_{m+1})\;
  %a_{t_{j_\mpl-1}}
  \bbrbr{\prod_{\substack{C\in \Comp(G)\\C\neq C_v}} B^C_n(j_{V(C)},\bbone)}\;
  %\prod_{\substack{c\in C\\c\neq c_v}} \cbr{  
  %  \prod_{v'<v''\in V_c}\beta_{j_{v'},j_{v''}}^{\edgeWt_{v',v''}} 
  %  \prod_{\vpr\in V_c} I_{\vertwtlow_\vpr}(1_{j_\vpr}^{\otimes \vertwtlow_\vpr})}
  \beta_{n}^{C_v}(j_{V\rbr{C_v}})
  %\rbr{\prod_{[v,\vpr]\in p(V_{c_v})} \beta_{j_v,j_\vpr}^{\edgeWt_{v,\vpr}} }
  \bbrbr{\prod_{\substack{v'\in V\rbr{C_v}\\v'\neq v}}
  I_{\vertwtlow_\vpr}(1_{j_\vpr}^{\otimes \vertwtlow_\vpr}) } 
  \\&%\times
    \hspace{75pt}\times \vertwtlow(v) I_{\vertwtlow(v)-1}(1_{j_v}^{\otimes (\vertwtlow(v)-1)})\,
    \beta_n\rbr{j_v,j_\mpl}\,
  %D_{1_{j_0}} I_{\vertwtlow_\vpr}(1_{j_\vpr}^{\otimes \vertwtlow_\vpr})
  I_1(1_{j_\mpl}).
\end{align*}

Let $A''$ as %and $\bbf'=(\bbf'_v)_{v\in[m+1]}$ as
\begin{align*}
  &A''=\brbr{A''_n(j) }_{j\in[n]^{m+1},  n\in\bbN}=
  \brbr{A_n(j_{[m]}) A'_n(j_{m+1})}_{j\in[n]^{m+1}, n\in\bbN}
  %\\
  %&\bbf'_v=\bbf_v\;\text{ for }\; v\in[m]\;\text{ and }\;
  %\bbf'_{m+1}=\bbone
\end{align*}
and define a weighted graph
$C_0:=(\cbr{m+1},0,2)$.
Note that 
$\barq(C_0)=2$ and
$s(C_0)=2$.
We can see that $A''\in\cala([m+1])$.% and $\bbf'$ satisfy 
%Assumption \ref{201229.1450} and \ref{211126.1140}, respectively.

(i) First we consider $(\cali_n^{(0)})_n$.
Let
$\alpha^{(0)}=2H-\frac{3}{2}$ and 
define $G^{(0)}$ and $A^{(0)}$ by
\begin{align*}
  &G^{(0)}=
  \brbr{\abr{m+1}_{-1} C_0} \vee G
  \\
  &A^{(0)}=\brbr{A^{(0)}_n(j) }_{j\in[n]^{m+1},  n\in\bbN}
  =\Brbr{n\brbr{D_{\charf{j_{m+1}}} A_n(j_{[m]})} A'_n(j_{m+1})}
  _{j\in[n]^{m+1}, n\in\bbN}.
\end{align*}
We denote 
$C^{(0)}=\abr{m+1}_{-1} C_0=(\cbr{m+1},0,1)$.
We have
$I(C^{(0)})=1$, $\barq(C^{(0)})=s(C^{(0)})=1$ and
$\Comp(G^{(0)}) = \Comp(G) \sqcup \cbr{C^{(0)}}$.
We can see that 
$A^{(0)}\in\cala([m+1])$ %satisfies Assumption \ref{201229.1450},
and
\begin{align}
  \cali_n^{(0)} = \cali_n(\alpha^{(0)}, G^{(0)}, A^{(0)},\bbone).
  \label{220118.1901}
\end{align}
The exponent $e(\alpha^{(0)}, G^{(0)})$ is 
\begin{align*}
  e(\alpha^{(0)}, G^{(0)})
  &= 2H-\frac{3}{2} + (2-I(C^{(0)})-1) + \sum_{C\in \Comp(G)} e(C)
  = e(0,G) + 2H-\frac{3}{2}.
  %\label{210407.1501}
\end{align*}

(ii) Consider $(\cali_n^{(v)})_n$ for $v\in  V_+$.
Let
$\alpha^{(v)}=2H-\frac{1}{2}$ and 
define weighted graphs $C^{(v)}$ and $G^{(v)}$ by
\begin{align*}
  C^{(v)} %{= (V_{c^*}, \edgeWt_{c^*}, \vertWt_{c^*}) }
  &= \dabr{v, m+1}_{1} (C_v\vee C_0)\\
  G^{(v)}&=
  C^{(v)} \vee \bbrbr{\mathop\vee_{C\in \Comp(G)\setminus\cbr{C_v}} C}.
\end{align*}
We write $G^{(v)}=([m+1], \edgeWt^{(v)}, \vertWt^{(v)})$.
We have
$\Comp(G^{(v)}) = (\Comp(G)\setminus\cbr{C_v})\sqcup\cbr{C^{(v)}}$.
By (\ref{220116.2106}) and (\ref{211225.2120}), 
the component $C^{(v)}$ satisfies 
\begin{align}
  I(C^{(v)}) &= I(C_v)+1
  \label{220116.2133}\\
  \barq({C^{(v)}}) &= (\barq(C_v)+\barq(C_0)) + 0 - 2 = \barq(C_v)
  \label{220116.2134}\\
  s(C^{(v)}) &= (s(C_v)+s(C_0)) + 0 - 2(2-1) = s(C_v).
  \label{220116.2135}
\end{align}
We have
\begin{align}
  \cali_n^{(v)} = \vertwtlow(v)\;\cali_n(\alpha^{(v)}, G^{(v)}, A'', \bbone).
  \label{220118.1902}
\end{align}

(ii-i) 
If $C_v\in\Csone(G)$ or $C_v\in\Ctpo(G)$,
the connected weighted graph $C^{(v)}$ satisfies 
$C^{(v)}\in\Csone(G^{(v)})$ or $C^{(v)}\in\Ctpo(G^{(v)})$, respectively, thanks to (\ref{220116.2134}) and (\ref{220116.2135}).
In either case, we have 
$e(C^{(v)}) = e(C_v) - 1$ by (\ref{220116.2133}).
%since $I_{c^*} = I_{c_v}+1$.
Hence, 
\begin{align}\label{210822.1801} 
  e(\alpha^{(v)}, G^{(v)})=
  (2H-1/2) + \sum_{C\in\Comp(G)\setminus\cbr{C_v}} e(C) + (e(C_v)-1)
  =e(0, G)+2H-3/2.
\end{align}

(ii-ii)
If $C_v\in\Ctppe(G)$ and 
$v\notin\argmax_{\vpr\in V\rbr{C_v}} \vertwtlow(\vpr)$
(i.e. $\vertwtlow(v)<\max_{\vpr\in V\rbr{C_v}} \vertwtlow(\vpr)$),
then 
\begin{align*}
  \max_{v'\in V(C^{(v)})} \vertWt^{(v)}(\vpr) 
  = 1 \vee (\vertwtlow(v)-1) \vee \max_{v'\in V\rbr{C_v}\setminus\cbr{v}} \vertwtlow(\vpr)
  = \max_{\vpr\in V\rbr{C_v}} \vertwtlow(\vpr).
\end{align*}
Hence, with the help of (\ref{220116.2134}) and (\ref{220116.2135}), 
we have  $C^{(v)}\in \Cpearg{i}(G^{(v)})$ 
with $i\in\cbr{0,1,2}$ such that $C_v\in \Cpearg{i}(G)$.
In any case, $e(C^{(v)}) = e(C_v) - 1$ holds, since 
$I(C^{(v)}) = I(C_v)+1$.
The exponents $e(0, G)$ and $e(\alpha^{(v)}, G^{(v)})$ 
satisfy %the same relation as (\ref{210822.1801}).
\begin{align*}%\label{220324.1801} 
  e(\alpha^{(v)}, G^{(v)})
  =(2H-1/2) + \sum_{C\in \Comp(G)\setminus\cbr{C_v}} e(C) + (e(C_v)-1)
  =e(0, G)+2H-3/2.
\end{align*}

\vspssm
(ii-iii) Assume that $C_v\in\Ctppe(G)$ and
$ \displaystyle \vertwtlow(v)=\max_{v'\in V\rbr{C_v}} \vertwtlow(\vpr)$.

%\noindent
(ii-iii-i) 
If $C_v\in\Cpearg{1}(G)(\subset\Ctpa(G))$,
then $\argmax_{v'\in V\rbr{C_v}} \vertwtlow(\vpr)=\cbr{v}$, and we have
\begin{align*}
  \max_{v'\in V(C^{(v)}) } \vertWt^{(v)}(\vpr) 
  = 1 \vee (\vertwtlow(v)-1) \vee \max_{v'\in V\rbr{C_v}\setminus\cbr{v}} \vertwtlow(\vpr)
  = \vertwtlow(v)-1 = \max_{\vpr\in V(C_v)} \vertwtlow(\vpr) - 1
  = \half\barq(C_v) = \half\barq(C^{(v)}),
\end{align*}
which means $C^{(v)}\in\Cpearg{0}(G^{(v)})\subset\Ctpb(G^{(v)})$.
The exponent of $C^{(v)}$ is 
\begin{align*}
  e(C^{(v)}) &= (2-I(C^{(v)})) + (1/2-2H)2 -H(s(C^{(v)})-2) 
  %= (1-I_{c_v})   + (1/2-2H)2 -H(s_{c_v}-2) 
  \\&= \Brbr{2-(I(C_v) +1)  + (1/2-2H) -H(s({C_v})-2)}  + 1/2-2H
  = e(C_v) + 1/2 - 2H
\end{align*}
by (\ref{220116.2133}) and (\ref{220116.2135}).
Hence,
\begin{align*}%\label{211104.1536} 
  e(\alpha^{(v)}, G^{(v)})=
  (2H-1/2) + \sum_{C\in \Comp(G)\setminus\cbr{C_v}} e(C)
  + (e(C_v) + 1/2 - 2H)
  =e(0,G).
\end{align*}

\vspssm%\noindent
(ii-iii-ii) 
If $C_v\in\Cpearg{2}(G)(\subset\Ctpa(G))$,
then $\argmax_{v'\in V\rbr{C_v}} \vertwtlow(\vpr)=\cbr{v}$, and we have
\begin{align*}
  \max_{v'\in V(C^{(v)})} \vertWt^{(v)}(\vpr) 
  = 1 \vee (\vertwtlow(v)-1) \vee \max_{v'\in V\rbr{C_v}\setminus\cbr{v}} \vertwtlow(\vpr)
  = \vertwtlow(v)-1
  = \max_{\vpr\in V\rbr{C_v}} \vertwtlow(\vpr) - 1
  \geq \half \barq(C_v)+1 = \half\barq(C^{(v)})+1,
\end{align*}
and hence,  
$C^{(v)}\in\Cpearg{1}(G^{(v)})\sqcup\Cpearg{2}(G^{(v)})
\subset\Ctpa(G^{(v)})$;
if $C_v\in\Cpearg{0}(G)(\subset\Ctpb(G))$,
then  we have
\begin{align*}
  \max_{v'\in  V(C^{(v)})} \vertWt^{(v)}(\vpr) 
  = 1 \vee (\vertwtlow(v)-1) \vee \max_{v'\in V\rbr{C_v}\setminus\cbr{v}} \vertwtlow(\vpr)
  \leq \max_{\vpr\in V\rbr{C_v}} \vertwtlow(\vpr)
  \leq \frac{1}{2} \barq(C_v) = \frac{1}{2} \barq(C^{(v)}),
\end{align*}
and hence,  $C^{(v)}\in\Cpearg{0}(G^{(v)})\subset\Ctpb(G^{(v)})$.
In both cases,
the relation $e(C^{(v)}) = e(C_v) - 1$ holds
from (\ref{220116.2133}) and (\ref{220116.2135}).
The exponents $e(0, G)$ and $e(\alpha^{(v)}, G^{(v)})$ 
satisfy the same relation as (\ref{210822.1801}), namely,
\begin{align*} 
  e(\alpha^{(v)}, G^{(v)})=
  (2H-1/2) + \sum_{C\in \Comp(G)\setminus\cbr{C_v}} e(C) + (e(C_v)-1)
  =e(0, G)+2H-3/2.
\end{align*}

\vspssm
By summing up the above arguments (i) and (ii),
we have
\begin{align*}
  D_{u_n} \cali_n =&\;
  \cali_{n}^{(0)}
  + \sum_{v\in V_+} \cali_{n}^{(v)} 
\end{align*}
with the representations 
(\ref{220118.1901}) and (\ref{220118.1902}).
Let
$\Lambda=\cbr{0} \sqcup V_+$.
In the case (a), where $\Cpearg{1}(G)$ is empty,
we obtain the relation 
$%\begin{align*}
  e(\alpha^{(v)}, G^{(v)})=e(0, G)+2H-3/2
$ %\end{align*}
in any case $v\in\cbr{0}\sqcup V_+=\Lambda$, and hence
\begin{align*}
  \max_{v\in\Lambda} e(\alpha^{(v)}, G^{(v)})=e(0, G)+2H-3/2
\end{align*}
In the case (b),
there exists $v\in V_+$ such that 
$C_v\in\Cpearg{1}(G)$
and $\vertwtlow(v)=\max_{v'\in V\rbr{C_v}} \vertwtlow(\vpr)$, 
for which 
$e(\alpha^{(v)}, G^{(v)})=e(0, G)$  holds.
Hence, we have
\begin{align*}
  \max_{v\in\Lambda} e(\alpha^{(v)}, G^{(v)})
  = e(0, G).
\end{align*}
\end{proof}

\begin{remark}\label{220502.2012}
  Proposition \ref{210822.1800} extends Proposition 7.1 in \cite{yoshida2020asymptotic}
  in the case $q=2$.
  When the weighted graph is written as 
  $G=\vee_{v\in V} \rbr{\cbr{v}, 0, q_v}$ 
  (i.e. each component of $G$ has one vertex),
  both $\Ctpz(G)$ and $\Cpearg{0}(G)$ are empty, 
  and we have
  \begin{align*}
    \Comp(G)=
    \Cszero(G) \sqcup \Csone(G)  \sqcup 
    \Cpearg{1}(G) \sqcup \Cpearg{2}(G) \sqcup \Ctpo(G)
  \end{align*}
  and $\Cpearg{1}(G)=\cbr{(\cbr{v},0,q_v)\mid q_v=2}$.
  Notice also that 
  $\Comp^s(G)=\cbr{(\cbr{v},0,q_v)\mid q_v=s}$ for $s=0,1$,
  $\Cpearg{2}(G)=\cbr{(\cbr{v},0,q_v)\mid q_v\geq4,\; q_v \text{ is even}}$ and 
  $\Ctpo(G) =\cbr{(\cbr{v},0,q_v)\mid q_v\geq2,\; q_v \text{ is odd}}$. 
  Hence the condition in (ii) of Proposition 7.1 of \cite{yoshida2020asymptotic} under $q=2$ 
  (i.e. there is no $m'\in\cbr{1,...,m}$ such that $q_{m'}=q$) and 
  the condition in (a) of Proposition \ref{210822.1800} (i.e.  $\Cpearg{1}(G)$ is empty) coincides.
  So does the difference occurred by the action of $D_{u_n}$, if we read $H=1/2$.
\end{remark}

\subsubsection{Change of the order of by the action of $D^i$}
We verify that the order of the functional $\cali_n$ is stable 
under the action of $D^i$ by means of the exponent.

\begin{proposition}\label{210823.2400}%\label{210407.2246}
  Let $\alpha\in\bbR$,
  $G=(V,\edgeWt,\vertWt)$ a weighted graph, 
  $A\in\cala(V)$ and $\bbf\in\calf(V)$.
  For $i\in\bbZ_{\geq1}$ and any $p>1$,
  \begin{align*}%\label{210822.2211}
    \Bnorm{ \snorm{D^i \cali_n}_{\calh^{\otimes i}} }_p
    = O(n^{e(\alpha, G)})
  \end{align*}
  as $n\to\infty$ 
  for
  $\cali_n = \cali_n(\alpha,G,A,\bbf)$
  given by (\ref{220118.2623}).
\end{proposition}

\begin{proof}
Without loss of generality, we assume that 
$\alpha=0$, $\bbf=\bbone$ and $V$ is written as $V=[m]$ with some $m\in\bbN$.
The $i$-th derivative of $\cali_n$ can be written as:
\begin{align*}
  D^i\cali_n = \sum_{\lambda\in (\cbr{0}\sqcup[m])^{[i]}} \DLamCali,
\end{align*}
where the above summation runs through 
the set $(\cbr{0}\sqcup[m])^{[i]}$ of all the mappings $\lambda$ 
from $\cbr{1,..,i}$ to $\cbr{0,..,m}$,
\begin{align*}
  \DLamCali = 
  c_\lambda \rbr{
    \sum_{j\in[n]^m}
    D^{\lambda_0}_{s_{\lambda,0}} A_n(j)
    \rbr{\prod_{v\in V}
      I_{(\vertwtlow(v)-\lambda_v)}(1_{n,j_v}^{\otimes(\vertwtlow(v)-\lambda_v)})
      1_{n,j_v}^{\otimes\lambda_v}(s_{\lambda,v})}\;
    \beta_n^{G}(j)
    %\prod_{\vpr<\vppr\in[m]} \beta_{j_\vpr,j_\vppr}^{\edgeWt_{\vpr,\vppr}}
  }_{s_1,..,s_i\in[0,T]}
\end{align*}
with the conbinatorial constant $c_\lambda=\prod_{v\in V} \frac{\vertwtlow_v!}{(\vertwtlow_v-\lambda_v)!} $, 
and for $v\in\cbr{0}\sqcup[m]$ we denote 
$\ilamv:=\abs{\lambda^{-1}\cbr{v}}$ and
$s_{\lambda,v}$ reads
$s_{k_1},..,s_{k_{\ilamv}}$
if $\lambda^{-1}(\cbr{v})=\cbr{k_1,..,k_{\ilamv}}(\subset\cbr{1,..,i})$
as in Lemma \ref{211120.1740}.
If $\lambda_v>\vertwtlow(v)$ for some $v\in[m]$, we read $\DLamCali=0$.
Thus we only consider $\lambda\in\cbr{0,..,m}^{[i]}$ 
such that $\vertwtlow(v)\geq\lambda_v$ for all $v\in [m]$ hereafter.
We have
$i = \sum_{v\in\cbr{0}\sqcup[m]} \lambda_v$.

\vspssm
Since the following inequality holds for $p\geq2$
\begin{align}\label{211120.2811}
  \Bnorm{\bnorm{D^i\cali_n}_{\calh^{\otimes i}}}_p 
  \leq
  \sum_{\substack{\lambda\in\cbr{0,..,m}^i\\\vertwtlow(v)\geq\lambda_v\tforsm v\in[m]}}
  \Bnorm{\bnorm{\DLamCali}_{\calh^{\otimes i}}}_p
  =
  \sum_{\substack{\lambda\in\cbr{0,..,m}^i\\\vertwtlow(v)\geq\lambda_v\tforsm v\in[m]}}
    \Bnorm{\bnorm{\DLamCali}_{\calh^{\otimes i}}^2 }_{p/2}^{1/2},
\end{align}
we will consider the functional
$\bnorm{\DLamCali}_{\calh^{\otimes i}}^2 
%= \langle \DLamCali,\DLamCali\rangle_{\calh^{\otimes i}}
$.
Fix $\lambda\in(\cbr{0}\sqcup[m])^i$ 
satisfying $\vertwtlow(v)\geq\lambda_v$. % for $v\in [m]$.
We can write
\begin{align*}
  \bnorm{\DLamCali}_{\calh^{\otimes i}}^2
  =&
  c_\lambda^2 \sum_{j,k\in[n]^m}
  \abr{D^{\lambda_0}A_n(j), D^{\lambda_0}A_n(k)
  }_{\calh^{\otimes \lambda_0}}
  \nn\\&\hspace{40pt}
  \times\prod_{C\in \Comp(G)}
  \bigg\{
  \beta_n^{C}(j_{V_C})\,
  %\Brbr{\prod_{\vpr<\vppr\in V_c}\beta_{j_{\vpr},j_{\vppr}}^{\edgeWt_{\vpr,\vppr}}}
  \beta_n^{C}(k_{V_C})\,
  %\Brbr{\prod_{\vpr<\vppr\in V_c}\beta_{k_{\vpr},k_{\vppr}}^{\edgeWt_{\vpr,\vppr}}}
    \Brbr{\prod_{v\in V_C}\beta_n\rbr{j_{v},k_{v}}^{\lambda_v}}
    \nn\\&\hspace{100pt}
    \Brbr{\prod_{v\in V_C} I_\rbr{\vertwtlow(v)-\lambda_v}
    (1_{j_v}^{\otimes (\vertwtlow(v)-\lambda_v)})}
    \Brbr{\prod_{v\in V_C} I_\rbr{\vertwtlow(v)-\lambda_v}
    (1_{k_v}^{\otimes (\vertwtlow(v)-\lambda_v)})}
  \bigg\},
  %\label{210822.2245}
  %\label{210407.1331}
\end{align*}
Set
$\bar\lambda_C = \sum_{v\in V(C)} \lambda_v$ for $C\in \Comp(G)$ and
we write
$\cblam{+} = \cbr{C\in \Comp(G)\mid \bar\lambda_C>0}$ and 
%$C_{\lambda,+} = \cbr{c\in C: \bar\lambda_c>0}$ and 
$\cblam{0} = \cbr{C\in \Comp(G)\mid \bar\lambda_C=0}$.
%$C_{\lambda,0} = \cbr{c\in C: \bar\lambda_c=0}$.
For $\lambda$ in our consideration, 
the following relation holds:
\begin{align}\label{211120.2800}
  \barq(C)\geq\bar\lambda_C \tfor C\in \Comp(G).
\end{align}
For $C\in \cblam{+} $, define a map $\tau^{(C)}:p(V(C)\sqcup (V(C)+m))\to\bbZ_{\geq0}$ by
\begin{alignat*}{2}
  &\tau^{(C)}([v,v+m])=\lambda_{v}
  &&\qquad\tforsm v\in V(C)\\
  &\tau^{(C)}([v,v'])=0
  &&\qquad\tforsm [v,v']\in p(V(C)\sqcup (V(C)+m))\setminus\cbr{[v,v+m]\mid v\in V(C)}
\end{alignat*}
and let
\begin{align*}
  G^{(\lambda,C)}
  %{=G^*_c = (V_{c^*}, \edgeWt_{c^*}, \vertWt_{c^*})}
  = \dabr{\tau^{(C)}}(C\vee \shiftg{C}{m}).
\end{align*}
We define a weighted graph $G^{(\lambda)}$ by 
\begin{align*}
  G^{(\lambda)}
  =(V^{(\lambda)}, \edgeWt^{(\lambda)}, \vertWt^{(\lambda)})
  :=\Brbr{\mathop\vee_{C\in \cblam{0}} (C\vee \shiftg{C}{m})}
  \vee \Brbr{\mathop\vee_{C\in \cblam{+}} G^{(\lambda,C)}},
\end{align*}
and obviously we have
$%\begin{align*}
  \Comp(G^{(\lambda)}) = \bcbr{C, \shiftg{C}{m} \mid C\in\cblam{0}}
  \sqcup \bcbr{ G^{(\lambda,C)} \mid C\in\cblam{+}}.
$ %\end{align*}
We set 
\begin{align*}
  A^{(\lambda)}
  &=\brbr{A^{(\lambda)}_n(j)}_{j\in[n]^{2m},n\in\bbN}=
  %&=\Brbr{\brbr{A^{(\lambda)}_n(j_{[2m]})}_{j_{[2m]}\in[n]^{2m}} }_n=
  \Brbr{c_\lambda^2 
  \Babr{D^{\lambda_0}A_n(j_{[m]}), D^{\lambda_0}A_n(j_{[m+1,2m]})}
  _{\calh^{\otimes \lambda_0}}}_{j\in[n]^{2m}, n\in\bbN},
  %\\
  %\bbf^*&=(\bbf^*_v)_{v\in V^*} \;\text{ with }\;
  %\bbf^*_v=\bbf^*_{v+m} =\bbf_v \tforsm v\in V,
\end{align*}
where $[m+1,2m]=\cbr{m+1,...,2m}$.
Indeed, we have
$A^{(\lambda)}\in\cala([2m])$ %and $\bbf^*$ satisfy Assumptions 
and
\begin{align}\label{220118.2051}
  \bnorm{\DLamCali}_{\calh^{\otimes i}}^2
  =\cali_n(0, G^{(\lambda)}, A^{(\lambda)}, \bbone).
\end{align}

%\contifrom{220117.1300}

For $C\in\cblam{+}$, we have
$I(G^{(\lambda, C)})=2I(C)$, 
\begin{align*}
  s(G^{(\lambda, C)}) &=
  2s(C) +0 -2(2-1) = 2(s(C)-1)
  &&(\because (\ref{211225.2120}))\\
  \barq(G^{(\lambda, C)})&=
  2\barq(C) -2\sum_{[v,v']\in p(V_C\sqcup (V_C+m))}\tau^{(C)}([v,v'])
  \Brbr{=2\barq(C) -2\sum_{v\in V(C)} \lambda_{v}}
  &&(\because (\ref{220116.2106}))
\end{align*}
and hence $\barq(G^{(\lambda, C)})$ is even.
We also have $s(C)\geq\barq(C)\geq1$ by (\ref{211120.2800}).
%for $C\in \cblam{+}$.
We can show 
$e(G^{(\lambda,C)})\leq2e(C)$ 
as follows.

\noindent 
(i) If $s(C)=1$, then
we have $s(G^{(\lambda,C)})=0$ and the exponent of $G^{(\lambda,C)}$ is 
\begin{align*}
  e(G^{(\lambda,C)}) = 2-I(G^{(\lambda,C)}) = 2(1-I(C)) = 2e(C).
\end{align*}

\noindent
(ii) If $s(C)\geq2$, then $s(G^{(\lambda,C)})\geq2$.
With $v_0\in\argmax_{v\in V(C)} \vertWt^{(\lambda)}(v)$, we have
%With $v_0\in\argmax_{v\in V(G^{(\lambda,C)})\cap[m]} \vertWt^{(\lambda)}(v)$, we have
\begin{align*}
  \max_{v\in V(G^{(\lambda,C)})} \vertWt^{(\lambda)}(v)
  = \vertWt^{(\lambda)}(v_0)
  = \vertWt^{(\lambda)}(v_0+m)
  = \half\rbr{\vertWt^{(\lambda)}(v_0)+\vertWt^{(\lambda)}(v_0+m)}
  \leq \half\barq(G^{(\lambda,C)}),
\end{align*}
and hence $G^{(\lambda,C)}\in\Ctpb(G^{(\lambda)})
%\Ctpb(G) 
= \cbr{C\in\Ctp(G^{(\lambda)})\mid \barq(C) \text{ is even and } 
  \max_{v\in V(C)} \vertWt^{(\lambda)}(v) \leq \half\barq(C)}$.
%the component $G^*_c$ satisfies the conditions
%$\barq(G^*_c)$ is even
%and $\max_{v\in V_{c^*}} \vertWt_{c^*}(v) \leq \frac{1}{2}\barq(G^*_c)$.
The exponent of $G^{(\lambda,C)}$ is bounded as
\begin{align*}
  e(G^{(\lambda,C)}) &=
  2 - I(G^{(\lambda,C)}) + (1-4H) -H (s(G^{(\lambda,C)})-2) 
  \\&=
  2\rbr{1 -I(C)  + (1/2-2H) -H (s(C)-2)}
  \leq 2e({C}).&
  & (\because (\ref{211121.1302}))
\end{align*}
%We used (\ref{211121.1302}) at the inequality above.

Thus we have 
\begin{align*}
  e(0,G^{(\lambda)}) 
  = \sum_{C\in\cblam{0}}2e(C) + \sum_{C\in\cblam{+}} e(G^{(\lambda,C)})
  \leq 2\sum_{C\in \Comp(G)}e(C) = 2 e(0,G),
\end{align*}
and hence by (\ref{220118.2051}) and  Proposition \ref{210407.1401},
\begin{align*}
  \Bnorm{\bnorm{\DLamCali}_{\calh^{\otimes i}}^2}_p
  =O(n^{2e(0,G)})
\end{align*}
for $p\geq1$.
This estimate holds for any $\lambda\in(\cbr{0}\sqcup[m])^i$ such that
$\vertwtlow(v)\geq\lambda_v$ for $v\in [m]$.
By the inequality (\ref{211120.2811}), we have
\begin{align*}
  \Bnorm{\bnorm{D^i\cali_n}_{\calh^{\otimes i}}}_p 
  \leq
  %\sum_{\lambda\in\cbr{0,..,m}^i}
  %  \Bnorm{\snorm{\DLamCali}_{\calh^{\otimes i}}}_p
  %=
  \sum_{\substack{\lambda\in\cbr{0,..,m}^i\\\vertwtlow(v)\geq\lambda_v\tforsm v\in[m]}}
    \Bnorm{\bnorm{\DLamCali}_{\calh^{\otimes i}}^2 }_{p/2}^{1/2}
  = O(n^{e(0,G)})
\end{align*}
for $p\geq2$.
\end{proof}

%\newpage
\subsection{Second exponent}
For the second exponent,
we restirct weighted graphs to consider in order to obtain a sharper estimate.
First we introduce two types of weighted graphs.
\begin{definition}\label{220322.1900}
  (i) A weighted graph $G=(V,\edgeWt,\vertWt)$ is called a cycle graph if 
  $\vertWt(v)=0$ for $v\in V$ and 
  $G$ satisfies the following conditons
  \begin{itemize}
  \renewcommand\labelitemii{\labelitemi}
    \item if $I(G)=\abs{V}=2$ and $V$ is written as $V=\cbr{v_1,v_2}$, then
    $\edgeWt([v_1,v_2])=2$;

    \item if $I(G)=\abs{V}\geq3$ and 
    $V$ is written as $V=\cbr{v_1,...,v_m}$ with $m=I(G)$, then
    \begin{itemize}
      \item $\edgeWt([v,v'])=1$ 
      if $[v,v']=[v_i,v_{i+1}]$ with some $i=1,...,m-1$ or $[v,v']=[v_m,v_1]$,
      \item $\edgeWt([v,v'])=0$ otherwise.
    \end{itemize}
  \end{itemize}
  \noindent
  (ii) A weighted graph $G=(V,\edgeWt,\vertWt)$ is called 
  {\it a path graph with weighted ends} if 
  $V$ can be written as $V=\cbr{v_1,...,v_m}$ with $m\geq2$ and 
  \begin{itemize}
    \item $\edgeWt([v,v'])=1$ 
    if $[v,v']=[v_i,v_{i+1}]$ with some $i=1,...,m-1$;
    $\edgeWt([v,v'])=0$ otherwise.
    \item $\vertWt(v_i)=0$ for $1<i<m$ and $\vertWt(v_i)=1$ for $i=1,m$.
  \end{itemize}
\end{definition}
When $G$ is a cycle graph or a path graph with weighted ends,
we will denote 
the vertices $v_1,...,v_{m}$ ($m=I(G)$) in the above definitions of 
by ${v^G_1,...,v^G_{I(G)}}$.
Notice that a cycle graph is connected, so is a path graph with weighted ends.
If $G$ is a cycle graph, then $s(G)=2$ and $\barq(G)=0$.
If $G$ is a path graph with weighted ends, then $s(G)=2$ and $\barq(G)=2$.
We consider the following condition on connected weighted graphs:
\begin{ass}\label{220317.1320}
  For a connected weighted graph $G$,
  the two following conditions are fulfuilled:
  \begin{itemize}
  \item [(i)] $s(G)\leq2$, 
  \item [(ii)] in the case $s(G)=2$ and  $I(G)\geq2$,
    \begin{itemize}
      \item [(a)] if $\barq(G)=0$, then $G$ is a cycle graph;
      \item [(b)] if $\barq(G)=2$, 
      then $G$ is a path graph with weighted ends.
    \end{itemize}
  \end{itemize}
\end{ass} 
Notice that, in general, 
$\barq(G)$ for a connected weighted graph $G$ with $s(G)=2$
can be either $0$ or $2$;
if $G$ satisfies $(s(G),\barq(G))=(2,0)$, 
then $G$ has multiple vertices (i.e. $I(G)\geq2$).
Hence, $G$ with $s(G)=2$ and $I(G)=1$ must satisfy 
$\barq(G)=2$, %since $p(V)$ is empty,
and is written as 
$G=(\cbr{v},0, 2)$ with some $v$. 
%$G=(\cbr{v},\emptyset, \vertWt)$ with some $v$ and $\vertWt(v)=2$. 

\vspssm
In this section, we will define another exponent 
for the weighted graphs satisfying the following condition.
\begin{ass}\label{211203.1620}
  For a weighted graph,
  every component of it satisfies Assumption \ref{220317.1320}.
\end{ass}

%(In fact, this follows from the initial setting. 
%This is written merely for clarity.)
%{table wo ireru}

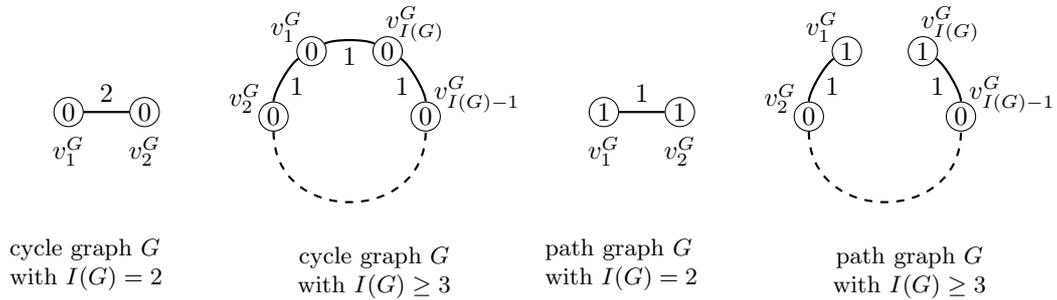
\captionsetup[subfigure]{labelformat=empty}
\begin{figure}[h]
  \centering
  \begin{subfigure}{0.15\textwidth}
    \centering
    \graphcycletwo
    \caption{cycle graph $G$\\with $I(G)=2$}
\end{subfigure}
  \begin{subfigure}{0.25\textwidth}
    \centering
    \graphcycle
    \caption{cycle graph $G$\\with $I(G)\geq3$}
\end{subfigure}
\begin{subfigure}{0.15\textwidth}
  \centering
  \graphpathtwo
  \caption{path graph $G$\\with $I(G)=2$}
\end{subfigure}
\begin{subfigure}{0.25\textwidth}
  \centering
  \graphpath
  \caption{path graph $G$\\with $I(G)\geq3$}
\end{subfigure}
%
%\begin{subfigure}{0.15\textwidth}
%  \centering
%  \graphCtto
%  \caption{{$G$ with $I(G)=1$ and $\barq(G)=2$}}
%\end{subfigure}%
\end{figure}

For a weighted graph $G$,
 we classify the components as
$\Comp^s(G) = \cbr{C\in \Comp(G)\mid s(C)=s}$ for $s=0,1,2$.
%and 
Moreover, let
$\Comp^2_{k}(G) = \cbr{C\in \Comp^2(G)\mid \barq(C) = k}$ for $k\in\cbr{0,2}$,
$\Ctto(G) =\cbr{C\in \Ctt(G)\mid I(C)=1}$ and 
$\Cttt(G) =\cbr{C\in \Ctt(G)\mid I(C)\geq2}$.
For brevity, we write 
$\Comp_{0}(G)=\Comp^0(G)\sqcup \Comp^2_0(G)$
and $\Cqom(G)=\Comp_0(G)\sqcup \Comp^1(G)$
since %$\barq(c)=0$ for $c\in C_0(G)$ and 
$\barq(C)=1$ for $C\in \Comp^{1}(G)$.

\begin{center}
  \begin{tabular}{ r|c|c|c|c }
    %\diagbox{$s$}{$\barq$} 
    & $\barq=0$ & $\barq=1$ & \multicolumn{2}{c}{$\barq=2$}\\ \hline
    & - & - & $I=1$ & $I\geq2$ \\ \hline\hline
    $s=0$ & $\Comp^0$ & - & - & - \\ \hline
    $s=1$ & - & $\Comp^1$ & - & - \\ \hline
    $s=2$& $\Comp^2_0$ & - & 
    $\Comp^2_{2,1}$ & $\Comp^2_{2,2+}$ \\ \hline
  \end{tabular}
\end{center}

\noindent
%With this notation,
In addition to $G$, suppose that 
a subset $\tensorc$ of $\Cttt(G)$ is given.
We write
$\twochaos := \Ctto(G) \sqcup \tensorc$ and 
$\nottenc := \Cttt(G) \setminus \tensorc$.
%$\tildc := \Cttt \setminus \ddotc$ and 
%$\check C^2_2 := C^2_{2,1} \cup \widetilde \Cttt$.
Notice that if $G$ satisfies Assumption \ref{211203.1620}, 
then $\Comp(G)$ can be written as the following disjoint unions:
\begin{align*}
  \Comp 
  = \Cszero \sqcup \Csone \sqcup \Ctz \sqcup \Ctto \sqcup \tensorc \sqcup \nottenc
  = \Cqzero \sqcup \Csone \sqcup \twochaos \sqcup \nottenc
  = \Cqzero \sqcup \Csone \sqcup \Ctt,
\end{align*}
where we omitted $(G)$ for brevity.

\begin{definition}
  For $n\in\bbN$, $\alpha\in\bbR$, 
  a weighted graph $G=(V,\edgeWt,\vertWt)$ satisfying Assumption \ref{211203.1620}, 
  a subset $\tensorc$ of $\Cttt(G)$,
  $A\in\cala(V)$ and $\bbf\in\calf(V)$,
  we define a functional $\cali^\T_n(\alpha,G,A,\bbf,\tensorc)$ by 
  \begin{align}\label{211126.2643}
    \funcSec(\alpha,G, \tensorc, A, \bbf) =&
    n^\alpha
    \sum_{j\in[n]^V} A_n(j)
    \prod_{C\in \Comp(G)\setminus\tensorc} 
    B_n^C(j_{V_C},\bbf|_{V_C})
    \prod_{C\in \tensorc} \check{B}_n^C(j_{V_C},\bbf|_{V_C})
    %\\=&
    %n^\alpha\sum_{j\in[n]^V} A_n(j) \prod_{c\in C(G)\setminus\tensorc}\bbcbr{ 
    %  \prod_{v\in V_c} I_{\vertwtlow_v}(\kerfvn{v}{j_v}^{\otimes \vertwtlow_v})
    %  \prod_{[v,v']\in p(V_c)}\beta_{j_{v},j_{\vpr}}^{\edgeWt_{v,\vpr}}}
    %\nn\\&\hsp\hsp\hsp\quad\times
    %\prod_{c\in \tensorc}\bbcbr{ 
    %  I_2(\kerfvn{\vc{c}{1}}{j_{\vc{c}{1}}} %\tilde
    %  \otimes \kerfvn{\vc{c}{I_c}}{j_{\vc{c}{I_c}}})
    %  \beta_{j_{\vc{c}{1}},j_{\vc{c}{2}}}\cdots\beta_{j_{\vc{c}{I_c-1}},j_{\vc{c}{I_c}}}}
  %\\=&
  %n^\alpha\sum_{j\in[n]^V} A_n(j)
  %\prod_{c\in C(G)\setminus\tensorc}\bbcbr{ 
  % \prod_{v\in V_c} I^{\vertwtlow_v}(\kerf{v}{j_v}^{\otimes \vertwtlow_v})
  % \prod_{\vpr<\vppr\in V_c}\beta_{j_{\vpr},j_{\vppr}}^{\edgeWt_{\vpr,\vppr}} 
  %}
  %\nn\\&\hsp\hsp\hsp\quad\times
  %\prod_{c\in \tensorc}\bbcbr{ 
  %  I^2(\kerf{\vc{c}{1}}{j_{\vc{c}{1}}} \tilde\otimes \kerf{\vc{c}{I_c}}{j_{\vc{c}{I_c}}})
  %  \prod_{\vpr<\vppr\in V_c}\beta_{j_{\vpr},j_{\vppr}}^{\edgeWt_{\vpr,\vppr}}
  %}
  %%
  \\=&
  n^\alpha \sum_{j\in[n]^V} A_n(j)
  \prod_{C\in \nottenc} B_n^C(j_{V_C},\bbf|_{V_C})
  \prod_{C\in \Comp(G)\setminus\nottenc} 
  \check{B}_n^C(j_{V_C},\bbf|_{V_C})
  %\prod_{c\in C_0(G)\sqcup C^1(G)\sqcup\nottenc} B_n^{\,c}(j_{V_c},\bbf|_{V_c})
  %\prod_{c\in \twochaos}\check{B}_n^{\,c}(j_{V_c},\bbf|_{V_c})
  %\\=&\delc{n^\alpha
  %\sum_{\jonem\in[n]^m} A_n(\jonem)
  %\prod_{c\in C_0\sqcup C^1\sqcup\nottenc} B^c_n(j_{V_c};\bbf)
  %\prod_{c\in \delc{\Ctto\sqcup\tensorc}} \check{B}^c_n(j_{V_c};\bbf),}
  \label{211127.1630}
  \end{align}
\end{definition}
The last expression (\ref{211127.1630}) is verified,
because the two functionals $B_n^G(j,\bbf)$ and $\check{B}_n^G(j,\bbf)$
coincides for a connected weighted graph $G=(V,\edgeWt,\vertWt)$ such that
$\babs{\cbr{v\in V\mid \vertWt(v)>0}}\leq1$.
%with $\barq(G)\leq1$ or $I(G)=1$.

\vspsm
Let 
\begin{alignat}{2}
  &e^+_2(I) = (2-I)+2\phi_H(I) = (1-2I H)\vee(-I)
  &\quad&\tfor I\geq2
  \label{220121.1137}\\
  &e^-_2(I) = (2-I)-1+\phi_H(2I)
  &&\tfor I\geq1,
  \nn
\end{alignat}
where
\begin{align}\label{211006.1341}
  \phi_H(I) = -1/2 + ((1/2-H)I)\vee (-1/2) 
  \quad\text{ for }\; I\geq2.
\end{align}
Notice that 
\begin{align}\label{211127.2424}
  %-H>\phi_H(l)\geq\phi_H(l+1)\geq-1 \quad\text{for}\; l\geq2
  -H>\phi_H(I)\geq-1 \tfor I\geq2, \tand
  \phi_H(I_1)\geq\phi_H(I_2) \tfor 2\leq I_1<I_2
\end{align}
and hence 
\begin{alignat}{4}\label{211202.2516}
  e^-_2(I) &\geq (2-I)-2=-I  &&\tfor I\geq1, %.\nn
  \tand&%\\
  e^+_2(I) &\geq e^-_2(I)   &&\tfor I\geq2.
\end{alignat}
Indeed,
$\phi_H(I)\geq-1$ and
$\phi_H(I)\geq\phi_H(2I)$.
We define the exponent $e_2$ for a connected weighted graph $C$
satisfying Assumption \ref{220317.1320} as follows.
If $s(C)=0$ or $1$, 
\begin{align}
  e_2(C)=2-I(C)-s(C);\label{220325.1321}
\end{align}
If $s(C)=2$,
\begin{subnumcases}{e_2(C)=}
  e^-_2(I(C))=1/2-2H
  &\text{if $I(C)=1$}
  \label{220325.1323}
  \\
  e^+_2(I(C))
  &\text{if $I(C)\geq2$}.
  \label{220325.1322}
\end{subnumcases}
Recall that for a connected weighted graph $C$,
the condition $(s(C),I(C))=(2,1)$ is equivalent to 
$(s(C),\barq(C),I(C))=(2,2,1)$, and 
the condition $(s(C),\barq(C))=(2,0)$ implies 
$I(C)\geq2$.
Hence we can write for $C$ with $s(C)=2$
\begin{align*}
  e_2(C)=
  \begin{cases}
  e^-_2(I(C))%=1/2-2H
  & \tif (s(C),\barq(C),I(C))=(2,2,1);
  \vspssm\\
  e^+_2(I(C))
  &\tif (s(C),\barq(C))=(2,0), \torsm (s(C),\barq(C))=(2,2) \tandsm I(C)\geq2.
\end{cases}
\end{align*}
%\end{scrap}
%
\begin{scrap}  
For a connected weighted graph $G$ satisfying Assumption \ref{220317.1320},
\begin{subnumcases}{e_2(G)=}
  2-I(G)-s(G)
  & if $s(G)=0$ or $1$;%\label{220325.1321}
  \vspssm\\
  e^+_2(I(G))
  &{if $(s(G),\barq(G))=(2,0)$, or $(s(G),\barq(G))=(2,2)$ and $I(G)\geq2$; }
  %\label{220325.1322}
  %\\& {if $s(G)=2$ and $I(G)\geq2$;}\nn
  \vspssm\\
  e^-_2(I(G))=1/2-2H
  & if $(s(G),\barq(G),I(G))=(2,2,1)$.%\label{220325.1323}
  %$(s_c,\barq(c))=(2,2)$, $I_c\geq2$ and $\labp=1$. \label{210412.1641}
\end{subnumcases}
\end{scrap}
Notice that for $C$ in the case
(\ref{220325.1321}), 
(\ref{220325.1323}) or
(\ref{220325.1322}) with $I(C)=2$,
we have 
\begin{align}
  e_2(C)=e(C),
  \label{220325.1410}
\end{align}
where $e(C)$ is the exponent defined in Section \ref{211006.0300}.
If $s(C)=2$ and $I(C)\geq3$,
then the relation $e_2(C)\leq e(C)$ holds.

For a connected weighted graph $C$ satisfying Assumption \ref{220317.1320} with
$(s(C),\barq(C))=(2,2)$ and $I(C)\geq2$,
we define another exponent $\expoT(C)$ by
\begin{align*} 
  \expoT(C) = e^-_2(I(C)).
\end{align*}
For $\alpha\in\bbR$, a weighted graph $G$ satisfying Assumption \ref{211203.1620}
and $\tensorc\subset\Cttt(G)$,
we define the second exponent $e_2(\alpha,G,\tensorc)$ by
\begin{align*}%\label{211202.1735}
  e_2(\alpha,G,\tensorc)
  &= \alpha + \sum_{C\in\Comp(G)\setminus\tensorc} e_2(C)
  + \sum_{C\in \tensorc} \expoT(C).
  %\\&= \alpha + \sum_{c\in C\setminus\check C^2_2} e_2(c)
  %+ \sum_{c\in \check C^2_2} e^{-}_2(I_c),
  %= \alpha + \sum_{c\in C\setminus \Cttt} e_2(c)
  %  + \sum_{c\in \ddot  \Cttt} e^+_2(I_c)
  %  + \sum_{c\in \widetilde \Cttt} e^{-}_2(I_c),
\end{align*}
The exponent $e_2(\alpha,G,\tensorc)$ can be written in several ways as follows:
\begin{align*}%\label{211202.1735}
  e_2(\alpha,G,\tensorc)
  &= \alpha 
  + \sum_{C\in\Comp(G)\setminus\Ctt(G)} e_2(C)
  + \sum_{C\in\Ctto(G)\sqcup\nottenc}  e_2(C)
  + \sum_{C\in \tensorc}  \expoT(C)
  \\&= \alpha 
  + \sum_{C\in\Comp(G)\setminus\Ctt(G)} e_2(C)
  + \sum_{C\in\nottenc}  e^+_2(I(C))
  + \sum_{C\in \Ctto(G)\sqcup\tensorc}  e^-_2(I(C))
  \\&= \alpha 
  + \sum_{C\in\Cqom(G)} e_2(C)
  + \sum_{C\in\nottenc}  e^+_2(I(C))
  + \sum_{C\in \twochaos}  e^-_2(I(C)).
\end{align*}
In particular, when $\tensorc=\Cttt(G)$, the exponent has the following expression:
\begin{align*}%\label{211202.1735}
  e_2(\alpha,G,\Cttt(G))
  &= \alpha 
  + \sum_{C\in\Comp(G)\setminus\Ctt(G)} e_2(C)
  + \sum_{C\in \Ctt(G)}  e^-_2(I(C)).
\end{align*}

Furthermore we introduce the notation $f(n)=\olog(n^q)$
meaning that $f(n)=O(n^q (\log n)^k)$ for some $k\geq0$.

\subsubsection{Estimate of the order of expectation of the functional 
$\funcSec(\alpha,G,\tensorc,A,\bbf)$}
First we prove the estimate of 
$E\sbr{\funcSec(\alpha,G,\tensorc, A,\bbf)}$
 under the assumption 
$\tensorc=\Cttt(G)$.

\begin{lemma} \label{211226.1426}
  Let $\alpha\in\bbR$, 
  $G=(V,\edgeWt,\vertWt)$ a weighted graph satisfying Assumption \ref{211203.1620}, 
  $A\in\cala(V)$ and $\bbf\in\calf(V)$.
  Then,
  \begin{align}\label{211226.1438}%\label{210412.1101}
    E[\cali_n] = \olog(n^{e_2(\alpha,G,\Cttt(G))})
  \end{align}
  as $n\to\infty$ for
  $\cali_n = \funcSec(\alpha,G,\Cttt(G),A,\bbf)$
  given by (\ref{211126.2643}).
\end{lemma}

\begin{proof}
%For the meantime, 
%we assume that $\kerf{v}{j}=\charfn{j}$ for $v\in V$.
We assume $\alpha=0$ and $\bbf=\bbone$ without loss of generality.
%

%In this case,
We have
$\twochaos(= \Ctto(G) \sqcup \tensorc)= \Ctt(G)$ and 
$\nottenc = \emptyset$.
We prove the estimate by induction with respect to $\abs{\Ctt(G)}$, 
namely the cardinality of $\Ctt(G)$. 

(i) First we consider the case $\abs{\Ctt(G)}=0$,
where $\tensorc=\Cttt(G)$ is empty as well.
Here we use another induction with respect to $\abs{\Csone(G)}$.
In the case of $\abs{\Csone(G)}=0$, in other words
$\Comp(G)=\Cqzero(G)(=\Cszero(G)\sqcup \Ctz(G))$,
the functional $\cali_n$ is written as
\begin{align*}
  \cali_n 
  %= n^\alpha \sum_{j\in[n]^V} A_n(j) 
  %\prod_{c\in C_0(G)} B_n^{\,c}(j_{V_c},\bbf|_{V_c})
  = %n^\alpha
    \sum_{j\in[n]^V} A_n(j)
    %\prod_{c\in C_0(G)} \beta_n^{\,c} (j_{V_c})
    \prod_{C\in \Cqzero(G)} \beta_{n}^{C}(j_{V_C}).
    %\prod_{C\in \Cqzero(G)} \beta_{n, V_C}(j_{V_C},\edgeWt|_{p(V_C)}).
\end{align*}
For $C\in\Cszero(G)$, we have $e(C)=e_2(C)$ (see (\ref{220325.1410})),
and by Lemma \ref{211028.2340}
\begin{align*}
  \sum_{j\in [n]^{V(C)}} \beta_{n}^{C}(j)
  %\sum_{j\in [n]^{V_C}} \beta_{n,V_C}(j, \edgeWt|_{p(V_C)})
  =O(n^{e(C)})=\bar O(n^{e_2(C)}).
\end{align*}
If $C\in \Ctz(G)$, then $C$ is a cycle graph by Assumption \ref{211203.1620}, 
and we have 
\begin{align*}
  \sum_{j\in [n]^{V(C)}} \beta_{n}^{C}(j)%, \edgeWt|_{p(V_C)})
  %\sum_{j\in [n]^{V_C}} \beta_{n,V_C}(j, \edgeWt|_{p(V_C)})
  &=\sum_{j\in[n]^{V(C)}} 
  \beta_n\nrbr{j_{v^C_1},j_{v^C_2}}\cdots
  \beta_n\nrbr{j_{v^C_{I(C)-1}},j_{v^C_{I(C)}}}
  \beta_n\nrbr{j_{v^C_{I(C)}},j_{v^C_1}}
  \\&=\bar O(n^{e^+_2(I(C))})
  =\bar O(n^{e_2(C)})
\end{align*}
by Lemma \ref{211006.1220}(c) with $k=I(C)$. 
(See (\ref{220325.1344}).)
Hence the following estimate holds:
\begin{align*}
  \Babs{ E[\cali_n] }
  %&=
  %\abs{ E\bbsbr{ \sum_{j\in[n]^V} A_n(j) \prod_{c\in C_0(G)}
  %  \beta_{n,V_c}(j_{V_c}, \edgeWt|_{p(V_c)})}}
  %\\
  &\leq
  \sup_{n\in\bbN, j\in[n]^V} \norm{A_n(j)}_{L^1(P)}\;
  \sum_{j\in[n]^V}\prod_{C\in\Cqzero(G)} 
  \beta_{n}^{C}(j_{V_C})
  %\beta_{n,V_c}(j_{V_c},\edgeWt|_{p(V_c)})
  \\
  &=C_{(\ref{220119.1721})}
  \prod_{C\in \Cqzero(G)}  \Brbr{\sum_{j\in [n]^{V(C)}}
  \beta_{n}^{C}(j)
  %\beta_{n,V_c}(j, \edgeWt|_{p(V_c)})
  }
  \\&=
  \bar O(n^{\sum_{C\in\Cqzero(G)}e_2(C)})
  =\bar O(n^{e_2(0,G,\emptyset)}),
\end{align*}
where $C_{(\ref{220119.1721})}$ is a constant obtained from
the bound (\ref{220119.1721}) in the assumption on $A$.

%%\newpage
\vspssm
Let $d\in\bbZ_{\geq1}$.
Assume that the estimate (\ref{211226.1438}) holds 
for any weighted graph $G'$ satisfying Assumption \ref{211203.1620} 
with $\abs{\Ctt(G)}=0$ and $\abs{\Csone(G')}\leq d-1$. 
%and any $\abs{C_0}\in\bbZ_{\geq0}$.
We are going to prove the estimate (\ref{211226.1438})
%for $G$ with
when $\abs{\Csone(G)}= d$. 
Fix 
$G=(V,\edgeWt,\vertWt)$ such that $\abs{\Csone(G)}=d$ and $\abs{\Ctt(G)}=0$, and
$A\in\cala(V)$.
%Without loss of generality, we can assume that $\alpha=0$.
%\redb{We omit $(G)$ of $C(G)$ until the end of (a-1).}
Pick $\cpr\in\Csone(G)$ arbitrarily.
The expectation of $\cali_n= \funcSec(\alpha,G,\emptyset,A,\bbone)$ can be transformed as follows:
\begin{align}
  E\sbr{\cali_n} =&
  \sum_{j\in[n]^V}  E\sbr{A_n(j)
  \Brbr{\beta_n^{\cpr}(j_{V_{\cpr}})\; I_1\brbr{1_{j_{v_{C_0}}}}}
  %\prod_{\vpr<\vppr\in V_{c'}}\beta_{j_{\vpr},j_{\vppr}}^{\edgeWt_{\vpr,\vppr}}
  \prod_{C\in \Comp(G)\setminus\cbr{\cpr}} B^C_n(j_{V_C})
  }
  \nn\\=&
  \sum_{j\in[n]^V}  
  E\sbr{ \rbr{D_{\charf{j_{v_{C_0}}}} A_n(j)}\;
  \beta_n^{\cpr}(j_{V_{\cpr}})
  %\bbcbr{\prod_{\vpr<\vppr\in V_{c'}}\beta_{j_{\vpr},j_{\vppr}}^{\edgeWt_{\vpr,\vppr}}}
  \prod_{C\in \Comp(G)\setminus\cbr{\cpr}} B^C_n(j_{V_C})
  }
  \nn\\&+
  \sum_{\cppr\in\Csone(G)\setminus\cbr{\cpr}} \sum_{j\in[n]^V}  
  E\sbr{A_n(j)\;
  \beta_n^{\cpr}(j_{V_{\cpr}})
  %\bbcbr{\prod_{\vpr<\vppr\in V_{c'}}\beta_{j_{\vpr},j_{\vppr}}^{\edgeWt_{\vpr,\vppr}}}
    %\nn\\&\hsp\hsp\hsp\hsp\hsp\times
    \rbr{D_{\charf{j_\vc{}{C_0}}} B^{\cppr}_n(j_{V_{\cppr}})}
    \prod_{C\in \Comp(G)\setminus\cbr{\cpr,\cppr}} B^C_n(j_{V_C})}
  \nn\\=&
  E\bsbr{ \cali_n^{(0)} }
  + \sum_{\cppr\in \Csone(G) \setminus \cbr{\cpr} } E\bsbr{ \cali_n^{(\cppr)} },
  \label{220530.1800}
\end{align}
where we write
\begin{align*}
  \cali_n^{(0)} = &
  n^{-1} \sum_{j\in[n]^V} \rbr{nD_{\charf{j_\vc{}{\cpr}}} A_n(j)}\;
  \beta_n^{\cpr}(j_{V_{\cpr}})
  %\bbcbr{\prod_{\vpr<\vppr\in V_{c'}}\beta_{j_{\vpr},j_{\vppr}}^{\edgeWt_{\vpr,\vppr}} }
  \prod_{C\in \Comp(G)\setminus\cbr{\cpr}} B^C_n(j_{V_C})
  \\
  \cali_n^{(\cppr)} = &
  \sum_{j\in[n]^V}
    A_n(j)\;
  \beta_n^{\cpr}(j_{V_{\cpr}})\;
  %\bbcbr{\prod_{\substack{\vpr<\vppr\\\in V_{c'}}}
  %\beta_{j_{\vpr},j_{\vppr}}^{\edgeWt_{\vpr,\vppr}}}
  \beta_n\brbr{j_\vc{}{\cpr},j_{\vc{}{\cppr}}}\;
  \beta_n^{\cppr}(j_{V_{\cppr}})\;
  %\bbcbr{\prod_{\substack{\vpr<\vppr\\\in V_{c''}}}
  %  \beta_{j_{\vpr},j_{\vppr}}^{\edgeWt_{\vpr,\vppr}}}
  \prod_{C\in \Comp(G)\setminus\cbr{\cpr,\cppr}} B^C_n(j_{V_C})
  &&\tfor \cppr\in \Csone(G)\setminus \cbr{\cpr},
\end{align*}
and 
for $C\in \Csone(G)$
we denote by $\vc{}{C}$ 
the only vertex $v\in V\rbr{C}$ such that $\vertwtlow(v)=1$.

Define a connected weighted graph $C^{(0)}$, 
a weighted graph $G^{(0)}$ and 
$A^{(0)}$ by 
\begin{align*}
  C^{(0)}&=\babr{v_{\cpr}}_{-1}(\cpr)\\
  G^{(0)}&=
  C^{(0)} \vee \Brbr{\mathop\vee_{C\in \Comp(G)\setminus\cbr{\cpr}} C}\\
  A^{(0)}&=
  \brbr{A^{(0)}_n(j) }_{j\in[n]^V,n\in\bbN}=
  \brbr{n\,D_{\charf{j_{v_{\cpr}}}} A_n(j)}_{j\in[n]^V,n\in\bbN}.
\end{align*}
We can see that 
$\Comp(G^{(0)})=\cbr{C^{(0)}}\sqcup(\Comp(G)\setminus\cbr{\cpr})$,
$A^{(0)}\in\cala(V)$ and
\begin{align*}
  \cali_n^{(0)} = \funcSec(-1, G^{(0)},\emptyset, A^{(0)},\bbone)
\end{align*}
since 
$\beta_n^{\cpr}(j_{V_{\cpr}})
=\beta_n^{C^{(0)}}(j_{V_{C^{(0)}}})
=B_n^{C^{(0)}}(j_{V_{C^{(0)}}})$.
The component $C^{(0)}$ of $G^{(0)}$ satisfies 
\begin{align*}
  I(C^{(0)})=I(\cpr) \tand
  s(C^{(0)})=s(\cpr)-1=0
  %s(C^{(0)})=2\rbr{\bartheta(\cpr)-(I(\cpr)-1)}+\barq(c^{(0)})=0
\end{align*}
by (\ref{211225.2120}).
%since $\bartheta(\cpr)-(I(\cpr)-1)=0$ and $\barq(\cpr)=1$ follows from  $s(\cpr)=1$. 
%The rest of the graph remains the same as $G$
Hence $G^{(0)}$ satisfies Assumption \ref{211203.1620} and 
 $\abs{\Csone(G^{(0)})}=\abs{\Csone(G)\setminus\cbr{\cpr}}=d-1$.
Since we have
\begin{align*}
  e_2(-1,G^{(0)},\emptyset) 
  = -1 + (2-I(C^{(0)})) + \sum_{C\in \Comp(G)\setminus\cbr{\cpr}} e_2(C)
  = (1-I\rbr{\cpr}) + \sum_{C\in \Comp(G)\setminus\cbr{\cpr}} e_2(C)
  = e_2(0,G,\emptyset), 
\end{align*}
we obtain 
\begin{align}
  E[\cali_n^{(0)}] = \olog(n^{e_2(-1,G^{(0)},\emptyset)})
  = \olog(n^{e_2(0,G,\emptyset)})
  \label{220325.1911}
\end{align}
by the assumption of 
induction with respect to $\abs{\Csone(G)}$.

\vspssm
Consider $(\cali_n^{(\cppr)})_n$ for $\cppr\in \Csone(G)\setminus\cbr{\cpr}$. %$\cali_n^{(c'')}$, 
With $C^{(\cppr)}$ and $G^{(\cppr)}$ defined by
\begin{align*}
  C^{(\cppr)}&=\dabr{v_{\cpr}, v_{\cppr}}_1 (\cpr\vee\cppr)
\\
  G^{(\cppr)}&=C^{(\cppr)}\vee\Brbr{\mathop\vee_{C\in \Comp(G)\setminus\cbr{\cpr,\cppr}} C},
\end{align*}
we have
$%\begin{align*}
  \cali_n^{(\cppr)} = \funcSec(0,G^{(\cppr)},\emptyset,A,\bbone)
$ %\end{align*}
and $\Comp(G^{(\cppr)})=\bcbr{C^{(\cppr)}}\sqcup(\Comp(G)\setminus\cbr{\cpr,\cppr})$.
By (\ref{211225.2120}), 
the component $C^{(\cppr)}$ satisfies 
$%\begin{align*}
  s(C^{(\cppr)}) = (1+1)+0-2(2-1) = 0,
$ %\end{align*}
which implies that
$G^{(\cppr)}$ satisfies Assumption \ref{211203.1620} and 
$\babs{\Csone(G^{(\cppr)})}=d-2$.
%and hence the number of the components in $G^{(c'')}$ satisfying $\calc^1$ is $d-2$.
Since we have
\begin{align*}
  e_2(C^{(\cppr)}) =2-I(C^{(\cppr)}) =(1-I(\cpr)) + (1-I(\cppr)) 
  =e_2(\cpr)+e_2(\cppr),
\end{align*}
we obtain
$e_2(0,G^{(\cppr)},\emptyset) = e_2(0,G,\emptyset)$ and
\begin{align}
  E[\cali_n^{(\cppr)}]  
  = \olog(n^{e_2(0,G^{(\cppr)},\emptyset)})
   = \olog(n^{e_2(0,G,\emptyset)}),
   \label{220325.1912}
\end{align}
by the assumption of induction.
As a result, by (\ref{220530.1800}), (\ref{220325.1911}) and (\ref{220325.1912}) we obtain
\begin{align*}
  E[\cali_n] = \olog(n^{e_2(0,G,\emptyset)}).
\end{align*}
By induction,
the estimate (\ref{211226.1438}) holds true when
$\Ctt(G)=\emptyset$.

\vspsm
%%\newpage
(ii) Let $d\in\bbZ_{\geq1}$. Assume that 
%under the additional assumption (\ref{211218.2500}) 
the estimate (\ref{211226.1438}) holds
for any $G'$ with $\abs{\Ctt(G')}\leq d-1$. 
%in the case of $\abs{C^2_2}=\abs{C^2_2(G)}\leq d-1$.
%and any $\abs{C_0}, \abs{C^1}\in\bbZ_{\geq0}$.
We are going to prove it for $G$ with $\abs{\Ctt(G)}=d$. 
Fix 
$G=(V,\edgeWt,\vertWt)$ satisfying Assumption \ref{211203.1620} and $\abs{\Ctt(G)}=d$,
and $A\in\cala(V)$.
Since we have assumed $\tensorc=\Cttt(G)$ and hence
$\nottenc=\emptyset$,
%$\twochaos = C^2_{2,1}(G) \sqcup \tensorc = C^2_2(G)$,
the functional $\cali_n$ is written as
\begin{align*}
  \funcSec(\alpha, G, \tensorc, A, \bbone) =&
  \sum_{j\in[n]^V} A_n(j)
  \prod_{C\in \Comp(G)} \check{B}^C_n(j_{V_C})
  %\prod_{c\in C_{1-}(G)} B^c_n(j_{V_c})
  %\prod_{c\in C^2_2(G)} \check{B}^c_n(j_{V_c}).
  %\label{220327.1454}
\end{align*}
by the definition (\ref{211127.1630}) of $\funcSec$.
%$C_{1-}(G)=C_0(G)\sqcup C^1(G)$
%\redb{220326: Under the assumption that all the component in $\Cttt(G)$ (and $\Ctz(G)$)
%satisfies Assumption \ref{220317.1320}}
For $C\in\Ctt(G)$,
the functional $\check{B}^C_n(j_{V_C})$ is written as
\begin{align*}
  \check{B}^C_n(j_{V_C})=
  \begin{cases}
    %I_2(1_{j_{v^c_1}}^{\otimes2})=
    \delta^2\brbr{1_{j_{v^C_1}}^{\otimes 2}}
    =\delta^2\brbr{1_{j_{v^C_1}}\otimes1_{j_{v^C_{I(C)}}}}
    \quad&\tifsm I(C)=1 \\
    \beta_n(j_{v^C_1},j_{v^C_2})\cdots
    \beta_n\nrbr{j_{v^C_{I(C)-1}},j_{v^C_{I(C)}}}
    \delta^2\brbr{1_{j_{v^C_1}}\otimes1_{j_{v^C_{I(C)}}}}
    \quad&\tifsm I(C)\geq2,
  \end{cases}
\end{align*}
where we write $v^C_1$ for the only vertex in $V(C)$ for $C\in\Ctt(G)$ with $I(C)=1$.
We can write 
$\check{B}^C_n(j_{V_C})
=\beta_n^{C}(j_{V_C})\,\delta^2\brbr{1_{j_{v^C_1}}\otimes1_{j_{v^C_{I(C)}}}}$.
%
%
%
%Pick $c'\in C^2_2$ arbitrarily.

Take $\cpr\in \Ctt(G)$ arbitrarily.
%We use the IBP formula to reduce the number of connected components $V_c$ of $c\in C_2$.
The expectation of $\cali_n$ can be transformed as follows:
\begin{align*}
  E\sbr{\cali_n} =&
  \sum_{j\in[n]^V}  E\sbr{A_n(j)
  \Brbr{\beta_n^{\cpr}(j_{V_{\cpr}})
  \delta\brbr{ I_1(\charf{j_\vc{\cpr}{I(\cpr)}}) \charf{j_\vc{\cpr}{1}}} }
  %\prod_{\vpr<\vppr\in V_{c'}}\beta_{j_{\vpr},j_{\vppr}}^{\edgeWt_{\vpr,\vppr}} 
  \prod_{C\in \Comp(G)\setminus\cbr{\cpr}} \check{B}^C_n(j_{V_C})
  %\prod_{c\in C_{1-}} B^c_n(j_{V_c})
  %\prod_{c\in C^2_2\setminus\cbr{c'}} \check{B}^c_n(j_{V_c})
  }
  \nn\\=&
  \sum_{j\in[n]^V}
  E\sbr{ \rbr{D_{\charf{j_\vc{\cpr}{1}}} A_n(j)}
  \beta_n^{\cpr}(j_{V_{\cpr}})\, I_1(\charf{j_\vc{\cpr}{I(\cpr)}})
  %\bbcbr{\prod_{\vpr<\vppr\in V_{c'}}\beta_{j_{\vpr},j_{\vppr}}^{\edgeWt_{\vpr,\vppr}}}
  \prod_{C\in \Comp(G)\setminus\cbr{\cpr}} \check{B}^C_n(j_{V_C})
  %\prod_{c\in C_{1-}} B^c_n(j_{V_c})
  %\prod_{c\in C^2_2\setminus\cbr{c'}} \check{B}^c_n(j_{V_c})
  }
  \nn\\&+
  \sum_{\cppr\in \Csone(G)\sqcup(\Ctt(G)\setminus\cbr{\cpr})} \sum_{j\in[n]^V}
  E\bbsbr{A_n(j)\,
  \beta_n^{\cpr}(j_{V_{\cpr}})\, I_1(\charf{j_\vc{\cpr}{I(\cpr)}}) 
  %\bbcbr{\prod_{\vpr<\vppr\in V_{c'}}\beta_{j_{\vpr},j_{\vppr}}^{\edgeWt_{\vpr,\vppr}}}
  %\nn\\&\hspace{110pt}\times
  D_{\charf{j_\vc{\cpr}{1}}} B^{\cppr}_n(j_{V_{\cppr}})
  \prod_{C\in \Comp(G)\setminus\cbr{\cpr,\cppr}} \check{B}^C_n(j_{V_C})}
  %\prod_{c\in C_0\sqcup (C^1\setminus\cbr{c''}) } B^c_n(j_{V_c})
  %\prod_{c\in C^2_2\setminus\cbr{c'}} \check{B}^c_n(j_{V_c})
  %%\nn\\&+
  %%\sum_{c''\in C^2_2\setminus\cbr{c'}} \sum_{j\in[n]^V}
  %%E\bigg[A_n(j)\,
  %%\beta_n^{\,c'}(j_{V_{c'}})\, I_1(\charf{j_\vc{c'}{I(c')}}) 
  %%  \nn\\&\hspace{110pt}\times
  %%  D_{\charf{j_\vc{c'}{1}}} \check{B}^{c''}_n(j_{V_{c''}})
  %%\prod_{c\in C_{1-}} B^c_n(j_{V_c})
  %%\prod_{c\in C^2_2\setminus\cbr{c',c''}} \check{B}^c_n(j_{V_c})\bigg]
  \nn\\=&
  E\bsbr{ \cali_n^{(0)} }
  + \sum_{\cppr\in \Csone(G)\sqcup (\Ctt(G) \setminus \cbr{\cpr}) } E\bsbr{ \cali_n^{(\cppr)} },
\end{align*}
where we write
\begin{align}
  \cali_n^{(0)} &=
  n^{-1} \sum_{j\in[n]^V} \rbr{nD_{\charf{j_\vc{\cpr}{1}}} A_n(j)}\,
  \beta_n^{\cpr}(j_{V_{\cpr}})\, I_1(\charf{j_\vc{\cpr}{I(\cpr)}})  
  %\prod_{c\in C\setminus\checc} B^c_n(j_{V_c})
  \prod_{C\in \Comp(G)\setminus\cbr{\cpr}} \check{B}^C_n(j_{V_C})\nn
  %\prod_{c\in C_{1-}} B^c_n(j_{V_c})
  %\prod_{c\in C^2_2\setminus\cbr{c'}} \check{B}^c_n(j_{V_c})
  %\\\text{and}\nn
  \\
  \cali_n^{(\cppr)} &=
  \sum_{j\in[n]^V} A_n(j)\,
  \beta_n^{\cpr}(j_{V_{\cpr}})\, I_1(\charf{j_\vc{\cpr}{I(\cpr)}})\,
  \Brbr{D_{\charf{j_\vc{\cpr}{1}}} B^{\cppr}_n(j_{V_{\cppr}})}
  \prod_{C\in \Comp(G)\setminus\cbr{\cpr,\cppr}}  \check{B}^C_n(j_{V_C})
  %\prod_{c\in C_0\sqcup (C^1\setminus\cbr{c''}) } B^c_n(j_{V_c})
  %\prod_{c\in C^2_2\setminus\cbr{c'}} \check{B}^c_n(j_{V_c})
  \label{211217.1858}
  %\\&\text{for $c''\in C^1$}
  %\nn\\
  %\cali_n^{(c'')} &=
  %\sum_{j\in[n]^V} A_n(j)\,
  %\beta_n^{\,c'}(j_{V_{c'}})\, I_1(\charf{j_\vc{c'}{I(c')}})\,
  %\Brbr{D_{\charf{j_\vc{c'}{1}}} \check{B}^{c''}_n(j_{V_{c''}})}
  %\prod_{c\in C_{1-}} B^c_n(j_{V_c})
  %\prod_{c\in C^2_2\setminus\cbr{c',c''}} \check{B}^c_n(j_{V_c})
  %\\&\text{for $c''\in C^2_2\setminus\cbr{c'}$.}\nn
\end{align}
for $\cppr\in \Csone(G)\sqcup(\Ctt(G)\setminus\cbr{\cpr})$.

\vspssm
(ii-i) First we consider $(\cali_n^{(0)})_n$.
Let $\alpha^{(0)}=-1$ and 
define $C^{(0)}$, $G^{(0)}$ and 
$A^{(0)}=\brbr{A^{(0)}_n(j)}_{j\in[n]^V,n\in\bbN}$ by
  \begin{alignat*}{3}
  C^{(0)} &= \babr{\vc{\cpr}{1}}_{-1} \cpr,&\quad
  G^{(0)}&= C^{(0)} \vee \Brbr{\mathop\vee_{C\in \Comp(G)\setminus\cbr{\cpr}} C}&\tand
  A^{(0)}_n(j)&
  =nD_{\charf{j_\vc{\cpr}{1}}}A_n(j).
  %A^{(0)}&=
  %\Brbr{A^{(0)}_n(j)}_{j\in[n]^V,n\in\bbN}
  %=\Brbr{ nD_{\charf{j_\vc{c'}{1}}}A_n(j)}_{j\in[n]^V,n\in\bbN}.
\end{alignat*}
We can see that $A^{(0)}\in\cala(V)$.
We have
$\Comp(G^{(0)}) = (\Comp(G) \setminus \cbr{\cpr}) \sqcup \cbr{C^{(0)}}$ 
and $s(C^{(0)})=1$,
which imply that $G^{(0)}$ satisfies Assumption \ref{211203.1620}.
Notice that
\begin{align*}
  \check{B}^{C^{(0)}}_n(j_{V_{C^{(0)}}})=
  \beta_n^{\cpr}(j_{V_{\cpr}})\, I_1(\charf{j_\vc{\cpr}{I(\cpr)}}).
\end{align*}

We have
$\Cqzero(G^{(0)})=\Cqzero(G)$ and
%$C^{0}(G^{(0)})=C^0(G)$, 
%$C^{2}_{0}(G^{(0)})=C^{2}_{0}(G)$. 
$\Csone(G^{(0)})=\Csone(G) \sqcup \cbr{C^{(0)}}$.
%$C^{2}(G^{(0)})=C^2\setminus \cbr{c'}$ and
%$s_c\leq2$ for all $c\in C^*$,
The set $\Ctt(G^{(0)})$ is decomposed as
\begin{alignat*}{3}
\Ctto(G^{(0)}) &= \Ctto(G) \setminus \cbr{\cpr} &\tand
\Cttt(G^{(0)}) &= \Cttt(G) &
&\tif
\cpr\in \Ctto(G)\\
%if $c'\in \tensorc=\Cttt$, then
\Ctto(G^{(0)}) &= \Ctto(G) &\tand
\Cttt(G^{(0)}) &= \Cttt(G) \setminus \cbr{\cpr}\quad&
&\tif
\cpr\in \Cttt(G)
\end{alignat*}
In any case,
%$C^{2}_{2,1}(G^{(0)}) = C^{2}_{2,1}(G) \setminus \cbr{c'}$,
%$C^{2}_{2,2+}(G^{(0)}) = C^{2}_{2,2+}(G) \setminus \cbr{c'}$ and
%{\Ctto(G^*)\sqcup\Cttt(G^*)=}
$\Ctt(G^{(0)}) = \Ctt(G) \setminus \cbr{\cpr}$ and
$\babs{\Ctt(G^{(0)})}=d-1$.
By (\ref{211127.1630}) we have 
\begin{align*}
  \cali_n^{(0)}=
  \funcSec(\alpha^{(0)}, G^{(0)}, \Cttt(G^{(0)}), A^{(0)}, \bbone).
  %\cali_n(\alpha^*, m, \edgeWt, \vertWt^*, A^*, \bbf, \widetilde C^{2,*}_{2,2+}).
\end{align*}
Hence we apply the assumption of the induction to this functional to obtain
\begin{align*}
  E\sbr{\cali_n^{(0)}}=
  \olog(n^{e_2(\alpha^{(0)}, G^{(0)}, \Cttt(G^{(0)}) )}).
  %\tilde O(n^{e_2(\alpha^*, G^*, \widetilde C^{2,*}_{2,2+})}).
\end{align*}

%With the notation
%$\check C^{2,*}_2:=C^{2}_{2,1}(G^*)\sqcup C^{2}_{2,2+}(G^*)=C^{2}_{2}(G^*)$,
The exponent 
$e_2(\alpha^{(0)}, G^{(0)}, \Cttt(G^{(0)}))$ 
is bounded as follows:
\begin{align*}
  e_2(\alpha^{(0)}, G^{(0)}, \Cttt(G^{(0)}))
  %=& \alpha^* + \sum_{c\in C^*\setminus\widetilde C^{2,*}_{2,2+}} e_2(c)
  %+ \sum_{c\in\widetilde C^{2,*}_{2,2+}} \widetilde{e}_2(c)
  %\nn\\
  =& \alpha^{(0)}
  + \sum_{C\in \Comp(G^{(0)})\setminus \Ctt(G^{(0)})} e_2(C)
  + \sum_{C\in \Ctt(G^{(0)}) } e^-_2(I(C))
  \nn\\=& 
  -1 
  + \sum_{C\in \Cqzero(G)\sqcup \Csone(G) \sqcup \ncbr{C^{(0)}}} e_2(C)
  + \sum_{C\in  \Ctt(G)\setminus \cbr{\cpr}} e^-_2(I(C))
  \nn%\label{211212.1252}
  \\=&
  -1 + \sum_{C\in \Cqzero(G)\sqcup \Csone(G)} e_2(C) + (1-I(C^{(0)})) 
  + \sum_{C\in \Ctt(G)} e^-_2(I(C)) - e^-_2(I(\cpr))
  \nn\\\leq& 
  \sum_{C\in \Cqzero(G)\sqcup \Csone(G)} e_2(C) 
  + \sum_{C\in  \Ctt(G)} e^-_2(I(C))
  =e_2(0, G, \Cttt(G)).
\end{align*}
%where for readability we wrote $\alpha^{*}, G^{*}$ and $C^{*}$
%for $\alpha^{(0)}$ $G^{(0)}$ and $C^{(0)}$, respectively.
We used
(\ref{211202.2516}) and $I(\cpr)=I(C^{(0)})$ 
at the last inequality.
%and used (\ref{211205.1301})
%at the first and last equality.
Therefore,
\begin{align*}
  E\bsbr{\cali_n^{(0)}}=
  \olog(n^{e_2(0, G, \Cttt(G))}).
\end{align*}

\vspsm
(ii-ii) Consider $(\cali_n^{(\cppr)})_n$ for $\cppr\in  \Ctt(G) \setminus \cbr{\cpr}$.
We have
\begin{align*}
  D_{\charf{j_\vc{\cpr}{1}}} \check{B}^{\cppr}_n(j_{V_{\cppr}})
  =&
  \rbr{\beta_n\brbr{j_\vc{\cpr}{1},j_{\vc{\cppr}{1}}} I_1(\charf{j_{\vc{\cppr}{I(\cppr)}}})
  + \beta_n\brbr{j_\vc{\cpr}{1},j_{\vc{\cppr}{I(\cppr)}}} I_1(\charf{j_{\vc{\cppr}{1}}})}
  \beta_n^{\cppr}(j_{V_{\cppr}})
  %\bbcbr{\prod_{\vpr<\vppr\in V_{c''}}\beta_{j_{\vpr},j_{\vppr}}^{\edgeWt_{\vpr,\vppr}}},
  \nn\\
  =&
  \begin{cases}
    \rbr{\beta_n\brbr{j_\vc{\cpr}{1},j_{\vc{\cppr}{1}}} I_1(\charf{j_{\vc{\cppr}{I(\cppr)}}})
  + \beta_n\brbr{j_\vc{\cpr}{1},j_{\vc{\cppr}{I(\cppr)}}} I_1(\charf{j_{\vc{\cppr}{1}}})}
  \\\hspace{120pt}\times
  \beta_n\brbr{j_{\vc{\cppr}{1}},j_{\vc{\cppr}{2}}}\cdots
  \beta_n\brbr{j_{\vc{\cppr}{I(\cppr)-1}},j_{\vc{\cppr}{I(\cppr)}}}  
  \quad&\tifsm \cppr\in \Cttt(G) \\
  2I_1(1_{j_{v^{\cppr}_1}})\beta_n\brbr{j_{v^{\cpr}_1},j_{v^{\cppr}_1}}
  \quad&\tifsm \cppr\in \Ctto(G)
  \end{cases}
\end{align*}
The functional $\cali_n^{(\cppr)}$ expressed by (\ref{211217.1858}) 
is decomposed as
%\begin{align}
$\cali_n^{(\cppr)} = \cali_n^{(\cppr,1)} + \cali_n^{(\cppr,2)},$
%\end{align}
where
\begin{align*}
  \cali_n^{(\cppr,1)} = &
  \sum_{j\in[n]^V}
    A_n(j)\;
    I_1(\charf{j_\vc{\cpr}{I(\cpr)}}) \;
    \beta_n^{\cpr}(j_{V_{\cpr}})\;
    %\bbcbr{\prod_{\substack{\vpr<\vppr\in V_{c'}}}
    % \beta_{j_{\vpr},j_{\vppr}}^{\edgeWt_{\vpr,\vppr}}}
    \beta_n\brbr{j_\vc{\cpr}{1},j_{\vc{\cppr}{1}}}\;
    I_1(\charf{j_{\vc{\cppr}{I(\cppr)}}})\;
    \beta_n^{\cppr}(j_{V_{\cppr}})\;
    %\bbcbr{\prod_{\substack{\vpr<\vppr\in V_{c''}}}
    %  \beta_{j_{\vpr},j_{\vppr}}^{\edgeWt_{\vpr,\vppr}}}
  \prod_{C\in \Comp(G)\setminus\cbr{\cpr,\cppr}} \check{B}^C_n(j_{V_C})\\
  \nn\\
  %&\hsp\hsp\times
  %\prod_{c\in C_{1-}(G)} B^c_n(j_{V_c})
  %%\prod_{c\in C_0\sqcup C^1} B^c_n(j_{V_c})
  %\prod_{c\in C^2_2(G)\setminus\cbr{c',c''}} \check{B}^c_n(j_{V_c})\\
  \cali_n^{(\cppr,2)} = &
  \sum_{j\in[n]^V} A_n(j)\;
    I_1(\charf{j_\vc{\cpr}{I(\cpr)}}) \;
    \beta_n^{\cpr}(j_{V_{\cpr}})\;
    %\bbcbr{\prod_{\substack{\vpr<\vppr\in V_{c'}}}
    %  \beta_{j_{\vpr},j_{\vppr}}^{\edgeWt_{\vpr,\vppr}}}
    \beta_n\brbr{j_\vc{\cpr}{1},j_{\vc{\cppr}{I(\cppr)}}}\;
    I_1(\charf{j_{\vc{\cppr}{1}}})\;
    \beta_n^{\cppr}(j_{V_{\cppr}})\;
    %\bbcbr{\prod_{\substack{\vpr<\vppr\in V_{c''}}}
    %  \beta_{j_{\vpr},j_{\vppr}}^{\edgeWt_{\vpr,\vppr}}}
  \prod_{C\in \Comp(G)\setminus\cbr{\cpr,\cppr}} \check{B}^C_n(j_{V_C})
  %\nn\\
  %&\hsp\hsp\times
  %\prod_{c\in C_{1-}(G)} B^c_n(j_{V_c})
  %%\prod_{c\in C_0(G)\sqcup C^1(G)} B^c_n(j_{V_c})
  %  \prod_{c\in C^2_2(G)\setminus\cbr{c',c''}} \check{B}^c_n(j_{V_c})
\end{align*}

We only consider $\cali_n^{(\cppr,1)}$ here.
By the product formula, $\cali_n^{(\cppr,1)}$ is decomposed into the sum
$\cali_n^{(\cppr,1)}=\cali_n^{(\cppr,1,0)}+\cali_n^{(\cppr,1,2)}$, 
where
\begin{align*}
  \cali_n^{(\cppr,1,0)} = &
  \sum_{j\in[n]^V}
    A_n(j)\;
    \beta_n\brbr{j_\vc{\cpr}{I({\cpr})},j_{\vc{\cppr}{I(\cppr)}}} 
    \beta_n^{\cpr}(j_{V_{\cpr}})\;
    %\bbcbr{\prod_{\substack{\vpr<\vppr\in V_{c'}}}
    %  \beta_{j_{\vpr},j_{\vppr}}^{\edgeWt_{\vpr,\vppr}}}
    \beta_n\brbr{j_\vc{\cpr}{1},j_{\vc{\cppr}{1}}} 
    \beta_n^{\cppr}(j_{V_{\cppr}})\;
    %\bbcbr{\prod_{\substack{\vpr<\vppr\in V_{c''}}}
    %\beta_{j_{\vpr},j_{\vppr}}^{\edgeWt_{\vpr,\vppr}}}
  \prod_{C\in \Comp(G)\setminus\cbr{\cpr,\cppr}} \check{B}^C_n(j_{V_C})
  \nn\\
  %&\hspace{120pt}\times
  %\prod_{c\in C_{1-}(G)} B^c_n(j_{V_c})
  %%\prod_{c\in C_0\sqcup C^1} B^c_n(j_{V_c})
  %\prod_{c\in C^2_2(G)\setminus\cbr{c',c''}} \check{B}^c_n(j_{V_c})
  %\\
  \cali_n^{(\cppr,1,2)} = &
  \sum_{j\in[n]^V}
    A_n(j)\;
    \delta^2(\charf{j_\vc{\cpr}{I(\cpr)}} \otimes \charf{j_{\vc{\cppr}{I(\cppr)}}})
    \beta_n^{\cpr}(j_{V_{\cpr}})\;
    %\bbcbr{\prod_{\substack{\vpr<\vppr\in V_{c'}}}
    %  \beta_{j_{\vpr},j_{\vppr}}^{\edgeWt_{\vpr,\vppr}}}
    \beta_n\brbr{j_\vc{\cpr}{1},j_{\vc{\cppr}{1}}}\;
    \beta_n^{\cppr}(j_{V_{\cppr}})\;
    %\bbcbr{\prod_{\substack{\vpr<\vppr\in V_{c''}}}
    %  \beta_{j_{\vpr},j_{\vppr}}^{\edgeWt_{\vpr,\vppr}}}
  \prod_{C\in \Comp(G)\setminus\cbr{\cpr,\cppr}} \check{B}^C_n(j_{V_C})
  \nn
  %\\
  %&\hspace{120pt}\times
  %\prod_{c\in C_{1-}(G)} B^c_n(j_{V_c})
  %%\prod_{c\in C_0\sqcup C^1} B^c_n(j_{V_c})
  %\prod_{c\in C^2_2(G)\setminus\cbr{c',c''}} \check{B}^c_n(j_{V_c})
\end{align*}
We define $C^{(\cppr,1,0)}$ and $G^{(\cppr,1,0)}$ by
\begin{align*}
  C^{(\cppr,1,0)}= 
  \dabr{\vc{\cpr}{1}, \vc{\cppr}{1}}_{1}
  \rbr{\dabr{\vc{\cpr}{I\rbr{\cpr}}, \vc{\cppr}{I\rbr{\cppr}}}_{1} (\cpr\vee \cppr)}
  \tand
  &G^{(\cppr,1,0)}=
  C^{(\cppr,1,0)} \vee \Brbr{\mathop\vee_{C\in \Comp(G)\setminus\cbr{\cpr,\cppr}} C}.
\end{align*}
If $\cpr,\cppr\in \Ctto(G)$,
we can read 
$C^{(\cppr,1,0)} =\dabr{\vc{\cpr}{1}, \vc{\cppr}{1}}_{2} (\cpr\vee \cppr)$.
Notice that
\begin{align*}
  \check{B}^{C^{(\cppr,1,0)}}_n(j_{V\nrbr{C^{(\cppr,1,0)}}})=
  %B^{c^{(c'',1,0)}}_n(j_{V_{c^{(c'',1,0)}}})=
  \beta_n\brbr{j_\vc{\cpr}{I({\cpr})},j_{\vc{\cppr}{I(\cppr)}}}\;
  \beta_n^{\cpr}(j_{V_{\cpr}})\;
  \beta_n\brbr{j_\vc{\cpr}{1},j_{\vc{\cppr}{1}}}\;
  \beta_n^{\cppr}(j_{V_{\cppr}}).
\end{align*}
We may write $C^{*}$ and $G^{*}$ for $C^{(\cppr,1,0)}$ and $G^{(\cppr,1,0)}$
untill the end of (ii-ii).
We have $\barq(C^{*})=0$ and $s(C^{*})=2$, and 
we can observe that the component $C^{*}$
is a cycle graph of 
$I\brbr{C^{*}}=I\rbr{\cpr}+I\rbr{\cppr}(\geq2)$ vertices 
in the sence of Definition \ref{220322.1900}.

The set of the components of $G^{*}$ is written as
$\Comp(G^{*})=(\Comp(G)\setminus\cbr{\cpr,\cppr})\sqcup\bcbr{C^{*}}$, and
the weighted graph $G^{*}$ satisfies Assumption \ref{211203.1620}.
We have
$\Cszero(G^{*})=\Cszero(G)$, 
$\Csone(G^{*}) = \Csone(G)$,
$\Ctz(G^{*})=\Ctz(G)\sqcup \bcbr{C^{*}}$, 
%\redb{Since 
%$C^{2}_{2,1}(G^{*}) = C^{2}_{2,1}(G) \setminus \cbr{c',c''}$ and
%$C^{2}_{2,2+}(G^{*}) = C^{2}_{2,2+}(G) \setminus \cbr{c',c''}$,
%we  have}
$%{\Ctto(G^*)\sqcup\Cttt(G^*)=}
\Ctt(G^{*}) = \Ctt(G) \setminus \cbr{\cpr,\cppr}$ and
$\babs{\Ctt(G^{*})}=d-2\leq d-1$.
Then, we have 
\begin{align*}
  \cali_n^{(\cppr,1,0)}=
  \funcSec(0, G^{*}, \Cttt(G^{*}), A, \bbone)  
  \tand
  E\sbr{\cali_n^{(\cppr,1,0)}}=
  \bar O(n^{e_2(0, G^{*}, \Cttt(G^{*}) )}).
  %\bar O(n^{e_2(0, G^{(c'',1,0)}, C^{2}_{2,2+}(G^{(c'',1,0)}) )}).
\end{align*}

Since 
\begin{align}
  %e_2({c^*})=
  e_2^+(I\rbr{C^*})
  &=e_2^+(I\rbr{\cpr}+I\rbr{\cppr})
  =2-(I\rbr{\cpr}+I\rbr{\cppr}) + 2 \phi_H\brbr{I\rbr{\cpr}+I\rbr{\cppr}}
  \nn\\
  &\leq(1-I\rbr{\cpr})+(1-I\rbr{\cppr}) + \phi_H(2I\rbr{\cpr})+\phi_H(2I\rbr{\cppr})
  = e^-_2(I\rbr{\cpr}) + e^-_2(I\rbr{\cppr}) 
\label{220327.1641}
\end{align}
by Lemma \ref{211006.1220} (d),
%(i.e. $2\phi_H(I_{c'}+I_{c''}) \leq \phi_H(2I_{c'})+\phi_H(2I_{c''})$),
%On the other hand, 
the exponent 
$e_2(0, G^{*}, \Cttt(G^{*}) )$ 
%$e_2(0, G^{(c'',1,0)}, C^{2}_{2,2+}(G^{(c'',1,0)}) )$ 
is bounded as follows:
\begin{align*}
  e_2(0, G^*, \Cttt(G^*) )
  =& 
  0
  + \sum_{C\in \Comp(G^*)\setminus \Ctt(G^*)} e_2(C)
  + \sum_{C\in \Ctt(G^*)} e^-_2(I(C))
  \nn\\=& 
  \sum_{C\in \Comp(G)\setminus \Ctt(G)} e_2(C) + e_2^+(I\rbr{C^*})%e_2(c^*)
  + \sum_{C\in \Ctt(G) \setminus \cbr{\cpr,\cppr}} e^-_2(I(C))
  %\label{211217.2126}
  \\\leq& 
  \sum_{C\in \Comp(G)\setminus \Ctt(G)} e_2(C) 
  + \sum_{C\in \Ctt(G) } e^-_2(I(C))
  \;=\;
  e_2(0, G, \Cttt(G))
  %\label{211217.2125}
\end{align*}
%We used (\ref{211205.1301})
%at the first and last equality.
%$e^+_2(I)\geq e^-_2(I)$ for $I\geq2$ (\ref{211202.2516}), respectively.

\vspssm
Similarly, 
we define $C^{(\cppr,1,2)}$ and $G^{(\cppr,1,2)}$ by
\begin{align*}
  C^{(\cppr,1,2)}= 
  \dabr{\vc{\cpr}{1}, \vc{\cppr}{1}}_{1}(\cpr\vee\cppr) 
  \tand
  &G^{(\cppr,1,2)}=
  C^{(\cppr,1,2)} \vee \Brbr{\mathop\vee_{C\in \Comp(G)\setminus\cbr{\cpr,\cppr}} C}.
\end{align*}
Notice that
\begin{align*}
  \check{B}^{C^{(\cppr,1,2)}}_n(j_{V\nrbr{C^{(\cppr,1,2)}}})=
  %B^{c^{(c'',1,0)}}_n(j_{V_{c^{(c'',1,0)}}})=
  \delta^2(\charf{j_\vc{\cpr}{I(\cpr)}} \otimes \charf{j_{\vc{\cppr}{I(\cppr)}}})\;
  %\beta_n\brbr{j_\vc{c'}{I({c'})},j_{\vc{c''}{I(c'')}}}\;
  \beta_n^{\cpr}(j_{V_{\cpr}})\;
  \beta_n\brbr{j_\vc{\cpr}{1},j_{\vc{\cppr}{1}}}\;
  \beta_n^{\cppr}(j_{V_{\cppr}}).
\end{align*}
We may write $C^{**}$ and $G^{**}$ for $C^{(\cppr,1,2)}$ and $G^{(\cppr,1,2)}$
untill the end of (ii-ii).
We have $\barq(C^{**})=2$ and $s(C^{**})=2$, and 
we can observe that the component $C^{**}$
is a path graph with weighted ends of 
$I(C^{**})=I\rbr{\cpr}+I\rbr{\cppr}(\geq2)$ vertices % with 
%$q^{**}_v=1$ for $v=\vc{c'}{I_{c'}},\vc{c''}{I_{c''}}$
%and $q^{**}_v=0$ for the other vertices $v$.
%We can observe that the component $c^{*}$
%is a path graph of $I_{c^*}=I_{c'}+1(\geq2)$ vertices 
and satisfies Assumption \ref{220317.1320}.

Since
$\Comp(G^{**})=(\Comp(G)\setminus\cbr{\cpr,\cppr})\sqcup\bcbr{C^{**}}$,
the weighted graph $G^{**}$ satisfies Assumption \ref{211203.1620}.
We have
$\Cszero(G^{**})=\Cszero(G)$, 
$\Csone(G^{**}) = \Csone(G)$,
%$C^{2}(G^*)=(C^{2}\setminus \cbr{c', c''}) \sqcup \cbr{c^*}$,
$\Ctz(G^{**})=\Ctz(G)$,
%$C^{2}_{2,1}(G^{**}) = C^{2}_{2,1}(G) \setminus \cbr{c',c''}$ and
%$C^{2}_{2,2+}(G^{**}) = (C^{2}_{2,2+}(G) \setminus \cbr{c',c''}) \sqcup \cbr{c^{**}}$.
%Notice that
$\Ctt(G^{**}) = (\Ctt(G) \setminus \cbr{\cpr,\cppr}) \sqcup \cbr{C^{**}}$ and
$\babs{\Ctt(G^{**})}=d-1$.
Then, we have 
\begin{align*}
  \cali_n^{(\cppr,1,2)}=
  \funcSec(0, G^{**}, \Cttt(G^{**}), A, \bbone)  
  \tand
  E\sbr{\cali_n^{(\cppr,1,2)}}=
  \bar O(n^{e_2(0, G^{**}, \Cttt(G^{**}))})
\end{align*}
by the assumption of induction.

On the other hand, the exponent 
$e_2(0, G^{**}, \Cttt(G^{**}))$ 
is bounded as follows:
\begin{align}
  e_2(0, G^{**}, \Cttt(G^{**}))
  =& 
  %\alpha^* + \sum_{c\in C^*\setminus\widetilde C^{2,*}_{2,2+}} e_2(c)
  %+ \sum_{c\in\widetilde C^{2,*}_{2,2+}} \widetilde{e}_2(c)
  %\nn\\=&
  \sum_{C\in \Comp(G^{**})\setminus \Ctt(G^{**})} e_2(C)
  + \sum_{C\in \Ctt(G^{**})} e^-_2(I(C))
  \nn\\=& 
  \sum_{C\in \Comp(G)\setminus \Ctt(G)} e_2(C)
  + \sum_{C\in \Ctt(G)\setminus \cbr{\cpr,\cppr}} e^-_2(I(C))
  + e^-_2(I(C^{**}))
  \nn\\\leq& 
  \sum_{C\in \Comp(G)\setminus \Ctt(G)} e_2(C) 
  + \sum_{C\in \Ctt(G)\setminus \cbr{\cpr,\cppr}} e^-_2(I(C))
  + e^-_2(I\rbr{\cpr}) + e^-_2(I\rbr{\cppr}) 
  \label{211226.1650}
  \\%=&\; %e_2(0, G^*, C^{2}_{2,2+}(G^*) )
  =& e_2(0, G, \Cttt(G))
  \nn
\end{align}
At the inequality 
(\ref{211226.1650}),
we used
$e^-_2(I\rbr{C^{**}})\leq e^+_2(I\rbr{C^{**}})$ by (\ref{211202.2516}) and 
the inequality (\ref{220327.1641}) with 
$I\rbr{C^{**}} = I\rbr{\cpr}+I\rbr{\cppr} = I\rbr{C^*}$.
  
\vspssm
Therefore,
\begin{align*}
  E\bsbr{\cali_n^{(\cppr,1)}}=
  E\sbr{\cali_n^{(\cppr,1,0)}} + E\sbr{\cali_n^{(\cppr,1,2)}}=
  \bar O(n^{e_2(0, G, \Cttt(G))}).
\end{align*}
This argument works for $\cali_n^{(\cppr,2)}$ with slight modification and 
we shall have
$E\bsbr{\cali_n^{(\cppr,2)}}=
\bar O(n^{e_2(0, G, \Cttt(G))})$
as well.

%%\newpage
\vspsm
(ii-iii) Consider $(\cali_n^{(\cppr)})_n$ for $\cppr\in  \Csone(G)$.
We define $C^{(\cppr)}$ and $G^{(\cppr)}$ by
\begin{align*}
  C^{(\cppr)}= 
  \dabr{\vc{\cpr}{1}, \vc{\cppr}{1}}_{1}(\cpr\vee \cppr) 
  \tand
  &G^{(\cppr)}=
  C^{(\cppr)} \vee \Brbr{\mathop\vee_{C\in \Comp(G)\setminus\cbr{\cpr,\cppr}} C}
\end{align*}
and we have
\begin{align*}
  \beta_n^{\cpr}(j_{V_{\cpr}})\, I_1(\charf{j_\vc{\cpr}{I(\cpr)}})\,
  \Brbr{D_{\charf{j_\vc{\cpr}{1}}} B^{\cppr}_n(j_{V_{\cppr}})}
  =
  %B^{c^{(c'',1,0)}}_n(j_{V_{c^{(c'',1,0)}}})=
  \beta_n^{\cpr}(j_{V_{\cpr}})\;
  I_1(\charf{j_\vc{\cpr}{I(\cpr)}})\;
  \beta_n\brbr{j_\vc{\cpr}{1},j_{v_{\cppr}}}\;
  \beta_n^{\cppr}(j_{V_{\cppr}})
  =
  \check{B}^{C^{(\cppr)}}_n(j_{V\nrbr{C^{(\cppr)}}})
\end{align*}
for $j\in[n]^V$,
where we denote the only vertex  $v\in V(\cppr)$ such that $\vertwtlow(v)=1$ by $v_{\cppr}$.

We may write $C^{*}$ and $G^{*}$ for $C^{(\cppr)}$ and $G^{(\cppr)}$
untill the end of (ii-iii).
We have %$\barq(c^{*})=1$ and 
$s(C^{*}) = (2+1)+0-2(2-1) = 1$ by (\ref{211225.2120}).
Since 
$\Comp(G^*)=(\Comp(G)\setminus\cbr{\cpr,\cppr})\sqcup\bcbr{C^{*}}$,
the weighted graph $G^*$ satisfies Assumption \ref{211203.1620}.
We have
$\Cszero(G^*)=\Cszero(G)$, 
$\Csone(G^*) = (\Csone(G)\setminus \cbr{\cppr}) \sqcup \cbr{C^*}$,
%$C^{2}(G^*) = C^{2}\setminus \cbr{c'}$,
$\Ctz(G^*)=\Ctz(G)$,
%$C^{2}_{2,1}(G^*) = C^{2}_{2,1}(G) \setminus \cbr{c'}$ and
%$C^{2}_{2,2+}(G^*) = C^{2}_{2,2+} \setminus \cbr{c'}$. Notice that
$\Ctt(G^*) =\Ctt(G)\setminus \cbr{\cpr}$
and
$\babs{\Ctt(G^*)}=d-1$.
Then, we have 
\begin{align*}
  \cali_n^{(\cppr)}=
  \funcSec(0, G^*, \Cttt(G^*), A, \bbone)
  \tand 
  E\sbr{\cali_n^{(\cppr)}}=
  \bar O(n^{e_2(0, G^*, \Cttt(G^*) )}).
\end{align*}
Since we have
\begin{align*}
  e_2(C^*)
  =(1-(I\rbr{\cpr}+I\rbr{\cppr}))
  =(2-I\rbr{\cpr}-2) + (1-I\rbr{\cppr})
  \leq
  e^-_2(I\rbr{\cpr}) + e_2(\cppr)
\end{align*}
by (\ref{211202.2516}),
the exponent 
$e_2(0, G^*, \Cttt(G^*))$ 
is bounded as follows
\begin{align}
  e_2(0, G^*, \Cttt(G^*))
  =& 
  \sum_{C\in \Comp(G^*) \setminus \Ctt(G^*)} e_2(C)
  + \sum_{C\in\Ctt(G^*)} e^-_2(I(C))
  \nn\\=& 
  \sum_{C\in \Comp(G)\setminus (\Ctt(G)\sqcup\cbr{\cppr})} e_2(C)
  + e_2(C^*)
  + \sum_{C\in \Ctt(G) \setminus\cbr{\cpr}} e^-_2(I(C))
  \nn
  \\\leq&
  \sum_{C\in \Comp(G)\setminus \Ctt(G)} e_2(C)
  + \sum_{C\in \Ctt(G)} e^-_2(I(C))
  \nn%\label{211127.2612}
  %\\=&
  %\sum_{c\in C\setminus C^{2}_2} e_2(c)
  %+ \sum_{c\in C^{2}_2} e^-_2(I_c)
  %\label{211127.2613}
  =e_2(0, G, \Cttt(G)).
\end{align}
Therefore, for $\cppr\in\Csone(G)$ we obtain
\begin{align*}
  E\bsbr{\cali_n^{(\cppr)}}=
  \bar O(n^{e_2(0, G, \Cttt(G)) }).
\end{align*}

As a result, we obtain
\begin{align*}%\label{211127.2801}
  \abs{E\sbr{\cali_n}} \leq&
  \abs{E\bsbr{ \cali_n^{(0)}}} +
  \sum_{\cppr\in \Csone(G) } \abs{E\bsbr{ \cali_n^{(\cppr)}}} + 
  \sum_{\cppr\in \Ctt(G) \setminus \cbr{\cpr} } 
   \rbr{\Babs{E\bsbr{ \cali_n^{(\cppr,1)} }} + \Babs{E\bsbr{ \cali_n^{(\cppr,2)} }}}
  = \bar O(n^{e_2(0, G, \Cttt(G))}).
\end{align*}
By induction
we have obtained the estimate (\ref{211226.1438})
for weighted graphs satisfying Assumption \ref{211203.1620}.
%under the assumption $\tensorc=\Cttt$.

For general $\bbf\in\calf(V)$, %$\kerf{v}{j}$,
we modify $A_n(j)$ by incorporating the factors 
$\abr{\kerfvn{v'}{j'}, \kerfvn{v''}{j''}}/\beta_n\rbr{j',j''}$ into it
thanks to Lemma \ref{201229.1452} (2).
\end{proof}

%\contifrom{1218.1430}
\begin{proposition}\label{210413.0100}%{210108.1642}
  Let $\alpha\in\bbR$.
  Suppose that $G=(V,\edgeWt,\vertWt)$ is a weighted graph satisfying Assumption \ref{211203.1620}, 
  $A\in\cala(V)$ and $\bbf\in\calf(V)$.
  For a subset $\tensorc$ of $\Cttt(G)$,
  \begin{align}\label{210412.1101}
    E[\cali_n] = \olog(n^{e_2(\alpha,G,\tensorc)})
  \end{align}
  as $n\to\infty$ for
  $\cali_n = \funcSec(\alpha,G,\tensorc,A,\bbf)$
  given by (\ref{211126.2643}).
\end{proposition}

\begin{proof}
Without loss of generality, we can assume that $\alpha=0$.
As seen in the proof of Lemma \ref{211226.1426},
we may only consider the case where $\bbf=\bbone.$ %$\kerf{v}{j}=\charfn{j}$ for $v\in V$.
If $\tensorc=\Cttt(G)$,  
the estimate (\ref{210412.1101}) is already proved in Lemma \ref{211226.1426}.
Consider the general case
$\tensorc\subsetneq \Cttt(G)$.

Fix a weighted graph $G=(V,\edgeWt,\vertWt)$, 
$\tensorc\subsetneq \Cttt(G)$  %$\bbf$,
and $A\in\cala(V)$.
Recall that we write $\nottenc=\Cttt(G)\setminus\tensorc$.
For $C\in\Cttt(G)$,  let
$\hatc=\dabr{\vc{C}{1}, \vc{C}{I(C)}}_{1}(C)$.
We can observe that 
$s(\hatc)=2$, $\barq(\hatc)=0$, $I(\hatc)=I(C)\geq2$ and
$\hatc$
is a cycle graph %of $I\rbr{C}$ vertices 
in the sence of Definition \ref{220322.1900}.
Notice that 
\begin{align}\label{211226.1826}
  e_2(C)=e_2(\hatc)=e_2^+(I(C))
  \quad\tfor C\in\Cttt(G).
\end{align}

By the product formula,
%{$B^C_n(j_{V_C})$ for $C\in\nottenc\subset\Cttt(G)$}
$B^C_n(j)$ for $C\in\Cttt(G)$ and $j\in[n]^{V(C)}$
can be decomposed as 
%\begin{align*}
%  B^C_n(j_{V_C}) = \check{B}^{\hatc}_n(j_{V_C}) + \check{B}^C_n(j_{V_C}),
%\end{align*}
\begin{align*}
  B^C_n(j) = \check{B}^{\hatc}_n(j) + \check{B}^C_n(j),
\end{align*}
and we have
\begin{align*}
  \cali_n =&
  \sum_{j\in[n]^V} A_n(j)
  \prod_{C\in \nottenc} B^C_n(j_{V_C})
  \prod_{C\in\Comp(G)\setminus \nottenc} \check{B}^C_n(j_{V_C})
  %\prod_{c\in C_0\sqcup C^1\sqcup\nottenc} B^c_n(j_{V_c})
  %\prod_{c\in \twochaos} \check{B}^c_n(j_{V_c})
  \nn\\=&
  \sum_{\pi\in\cbr{0,2}^{\nottenc}}
  \sum_{j\in[n]^V} A_n(j)
  \prod_{C\in \pi^{-1}(0)} \check{B}^{\hatc}_n(j_{V_C})
  \prod_{C\in (\Comp(G)\setminus \nottenc)\sqcup\pi^{-1}(2)} 
  \check{B}^C_n(j_{V_C})
  \nn%\label{220327.2153}
  %\prod_{c\in C_0\sqcup C^1} B^c_n(j_{V_c})
  %\prod_{c\in \pi^{-1}(\cbr{0})} B^{\hatc}_n(j_{V_{\hatc}})
  %\prod_{c\in \twochaos\sqcup\pi^{-1}[2]} \check{B}^c_n(j_{V_c})
  \\=:&
  \sum_{\pi\in\cbr{0,2}^{\nottenc}}
  \cali_n^{(\pi)},
  \nn
\end{align*}
where the summation runs through all the mappings
$\pi$ from $\nottenc$ to $\cbr{0,2}$ and
\begin{align}
  \cali_n^{(\pi)}=
  \sum_{j\in[n]^V} A_n(j)
  \prod_{C\in \pi^{-1}(0)} \check{B}^{\hatc}_n(j_{V_C})
  \prod_{C\in \Comp(G)\setminus\pi^{-1}(0)} 
  \check{B}^C_n(j_{V_C}).
  \label{220327.2153}
\end{align}
%Notice that
%
Define $G^{(\pi)}$ by
\begin{align*}
  G^{(\pi)}=
  \Brbr{\mathop\vee_{C\in \pi^{-1}(0)} \hatc}
  \vee \Brbr{\mathop\vee_{C\in\Comp(G)\setminus\pi^{-1}(0)} C}.
\end{align*}
Since $\pi^{-1}(0)\subset\Cttt(G)$
and $\hatc$ for $C\in\Cttt(G)$ is a cycle graph,
%Since $\hatc$ for $C\in\pi^{-1}(0)$ is a cycle graph,
$G^{(\pi)}$ satisfies Assumption \ref{211203.1620}.
By the definition of $\funcSec$ (\ref{211127.1630}), %with 
%$(C(G)\setminus \nottenc)\sqcup\pi^{-1}(2)
%=C(G)\setminus\pi^{-1}(0)$,
we have 
\begin{align*}
  \cali_n^{(\pi)}=
  \funcSec(0, G^{(\pi)}, \Cttt(G^{(\pi)}), A, \bbone),
\end{align*}
and we can apply Lemma \ref{211226.1426} to obtain
\begin{align*}
  E\sbr{\cali_n^{(\pi)}} = 
  \olog(n^{e_2\rbr{0, G^{(\pi)}, \Cttt(G^{(\pi)})}}).
\end{align*}

We have
$\Cszero(G^{(\pi)})=\Cszero(G)$, 
$\Csone(G^{(\pi)}) = \Csone(G)$,
%$C^{2}(G^*) = C^{2}\setminus \cbr{c'}$,
$\Ctz(G^{(\pi)}) =\Ctz(G)\sqcup\cbr{\hatc\mid C\in \pi^{-1}(0)}$,
%$C^{2}_{2,1}(G^{(\pi)}) = C^{2}_{2,1}(G)$ and
%$C^{2}_{2,2+}(G^{(\pi)}) = \tensorc \sqcup \pi^{-1}(2)
%=\Cttt(G)\setminus \pi^{-1}(0)$.
and
$\Ctt(G^{(\pi)})
=\Ctt(G)\setminus\pi^{-1}(0)
=\twochaos\sqcup \pi^{-1}(2)$.
Thus, the exponent $e_2\rbr{0, G^{(\pi)}, \Cttt(G^{(\pi)})}$
is bounded as
\begin{align}
  e_2\rbr{0, G^{(\pi)}, \Cttt(G^{(\pi)})}
  =& 
  \sum_{C\in \Comp(G^{(\pi)}) \setminus \Ctt(G^{(\pi)})} e_2(C)
  + \sum_{C\in \Ctt(G^{(\pi)})} e^-_2(I(C))
  \nn\\=& 
  \sum_{C\in\Cqom(G)\sqcup\cbr{\hatc\mid C\in \pi^{-1}(0)}} e_2(C) 
  + \sum_{C\in\twochaos\sqcup\pi^{-1}(2)} e^-_2(I(C))
  \nn
  \\=&
  \sum_{C\in\Cqom(G)} e_2(C) 
  + \sum_{C\in\pi^{-1}(0)} e_2^+(I(C)) 
  + \sum_{C\in\twochaos\sqcup\pi^{-1}(2)} e^-_2(I(C))
  \label{211226.1840}
  \\\leq&
  \sum_{C\in\Cqom(G)} e_2(C) 
  + \sum_{C\in\nottenc} e_2^+(I(C))
  + \sum_{C\in\twochaos} e^-_2(I(C))
  =e_2(0, G, \tensorc).
  \label{211226.1841}
\end{align}
We used %$e_2(c)=e_2(\hat c)=e_2^+(I_c)$ for $c\in\Cttt$
(\ref{211226.1826})
at (\ref{211226.1840}),
and (\ref{211202.2516})
at (\ref{211226.1841}).
Note that the above inequality is an equality,
if $\pi(C)=0$ for all $C\in \nottenc$.

As a result, we obtain
\begin{align*}
  E\sbr{\cali_n} 
  = \sum_{\pi\in\cbr{0,2}^{\nottenc}} E\sbr{\cali_n^{(\pi)}} 
  = \olog(n^{e_2(0, G, \tensorc)}).
\end{align*}
Thus we proved the statement for general
$\tensorc\subset\Cttt(G)$.
\end{proof}

\begin{proposition}\label{210407.1402}
  Let $k_0\in\bbZ_{\geq1}$ and $\alpha\in\bbR$. 
  Suppose that $G=(V,\edgeWt,\vertWt)$ is a weighted graph satisfying Assumption \ref{211203.1620}, 
  $A\in\cala(V)$ and $\bbf\in\calf(V)$.
  For a subset $\tensorc$ of $\Cttt(G)$, 
\begin{align*}
  %E[\cali_n] = \tilde O(n^{e_2(\alpha,G,\widetilde \Cttt)})
  E\sbr{{(\cali_n)}^{2k_0}} = \olog(n^{2k_0\,e_2(\alpha,G,\tensorc)})
\end{align*}
as $n\to\infty$ for
$\cali_n = \funcSec(\alpha,G,\tensorc,A,\bbf)$
given by (\ref{211126.2643}).
In particular, for $p\geq1$,
\begin{align*}
\bnorm{\cali_n}_{p} = \olog(n^{e_2(\alpha,G,\tensorc)}).
\end{align*}
\end{proposition}

\begin{proof}
  Without loss of generality, we assume that 
  $V$ is written as $V=[m]$ with some $m\in\bbN$.
  Set 
  \begin{align*}
    G' =& \vee_{k=0,...,2k_0-1} \shiftg{G}{mk}
    \\
    A'=&\rbr{A'_n(j)}_{j\in[n]^{2k_0m}, n\in\bbN}\twith
    \\&\qquad 
    A'_n({j}) =
    A_n(j_{[m]})A_n(j_{[m+1,2m]})\cdots A_n(j_{[(2k_0-1)m+1,2k_0m]})
    \tfor j\in[n]^{2k_0m}
    \\
    \bbf' =& \rbr{\bbf'_v}_{v\in[2k_0m]}
    \twith \bbf'_{v+mk} = \bbf_v \tforsm v\in[m] \tandsm k=0,..,2k_0-1,
  \end{align*}
  where $[k_1,k_2]=\cbr{k_1,...,k_2}$ for $k_1\leq k_2\in\bbN$.
  Notice that 
  $\Comp(G')=\bigsqcup_{k=0,...,2k_0-1}\cbr{\shiftg{C}{mk}\mid C\in \Comp(G)}$
  and
  $%\begin{align*}
    \Cttt(G')=
    \bigsqcup_{k=0,...,2k_0-1}\cbr{\shiftg{C}{mk}\mid C\in\Cttt(G)}
  $, %\end{align*}
  and set 
  \begin{align*}
    \Comp_{\T,G'}=
    \bigsqcup_{k=0,...,2k_0-1}\cbr{\shiftg{C}{mk}\mid C\in\tensorc}.
  \end{align*}
  \delc{
    $G'$, $A'$ and $\bbf'$ satisfies Assumption \ref{220119.2343} and
    $G'$ satisfies Assumption \ref{211203.1620}.
  }
  Obviously $\tensorg{G'}\subset\Cttt(G')$.
  Then we have
  $%\begin{align*}
    \rbr{\funcSec(\alpha,G,\tensorc,A,\bbf)}^{2k_0}
    = \funcSec(2k_0\alpha, G',\tensorg{G'}, A', \bbf')
  $ %\end{align*}
  by Proposition \ref{210413.0100}.
  Since
  \begin{align*}
    e_2(2k_0\alpha, G',\tensorg{G'})
    =& 2k_0\alpha + \sum_{C\in \Comp(G')\setminus\tensorg{G'}}e_2(C)
      + \sum_{C\in \tensorg{G'}} \expoT(C)
    \\=& 2k_0\alpha + 2k_0\sum_{C\in \Comp(G)\setminus\tensorc}e_2(C) 
      + 2k_0\sum_{C\in \tensorc} \expoT(C)
    = 2k_0\, e_2(\alpha, G,\tensorc),
  \end{align*} 
  by Proposition \ref{210413.0100} we obtain 
  \begin{align*}
    E\sbr{ \rbr{\funcSec(\alpha,G,\tensorc,A,\bbf)}^{2k_0}}
    = \olog(n^{e_2(2k_0\alpha, G', \tensorg{G'})})
    = \olog(n^{2k_0\; e_2(\alpha, G, \tensorc)})
    \tas n\to\infty.
  \end{align*}
\end{proof}

%%\newpage
\subsubsection{Change of the 2nd exponent by the action of $D_{u_n}$}
%\contifrom{211221.0521}
Recall that we defined $u_n(A')$ 
for $A'\in\cala(V)$ with some singleton $V=\cbr{v}$
at (\ref{220318.1700}):
\begin{align*}
  u_n(A')=
  n^{2H-1/2} \sum_{j\in[n]} A'_n(j) I_1(1_j) 1_j.
\end{align*}
We will give the estimate of how much the order of functionals written as
$\funcSec(\alpha,G,\tensorc,A,\bbf)$
changes by the action of $D_{u_n}$ in terms of the exponent.

\begin{proposition}\label{210413.0101}
  Consider a weighted graph $G=(V, \edgeWt, \vertWt)$ satisfying Assumption \ref{211203.1620}
  and a singleton $V'$
  such that $V\cap V'=\emptyset$. Write $V''=V\sqcup V'$.
  Let $\alpha\in\bbR$, 
  $A\in\cala(V)$, $\bbf\in\calf(V)$ and $A'\in\cala(V')$.
  For a subset $\tensorc$ of $\Cttt(G)$, % satisfying $\cond^2_{2,2+}$,
  suppose
  $\cali_n = \funcSec(\alpha,G,\tensorc,A,\bbf)$
  is given by (\ref{211126.2643}):
  \begin{align*}
    \cali_n =
    n^\alpha
    \sum_{j\in[n]^V} A_n(j)
    \prod_{C\in\Comp(G)\setminus\tensorc} B_n^C(j_{V_C},\bbf|_{V_C})
    \prod_{C\in \tensorc} \check{B}_n^C(j_{V_C},\bbf|_{V_C})
  \end{align*}
  
  Then, there exists a finite set $\Lambda$,
  $\alpha^\lambda\in\bbR$,
  a weighted graph 
  $G^\lambda=(V'',\edgeWt^\lambda,\vertWt^\lambda)$, 
  $A^\lambda\in\cala(V'')$, $\bbf^\lambda\in\calf(V'')$ and 
  a subset $\tensorc^{\lambda}$ of $\Cttt(G^\lambda)$
  for each $\lambda\in\Lambda$,
  such that
  $D_{u_n(A')} \cali_n$ can be written as
  \begin{align*}
    D_{u_n(A')}\cali_n = 
    \sum_{\lambda\in\Lambda} \funcSec(\alpha^\lambda, G^\lambda, 
    \tensorc^{\lambda},
    A^\lambda, \bbf^\lambda)
  \end{align*}
  and the following condition is fulfilled:
  \begin{alignat*}{3}
    &\text{(a)}&\;\;
    &\max_{\lambda\in\Lambda} 
    e_2(\alpha^\lambda, G^\lambda, \tensorc^{\lambda})
    = e_2(\alpha, G, \tensorc) + 2H-\frac32
    &\hsp&\text{if $\twochaos=\Ctto(G)\sqcup\tensorc$ is empty;}
    \\
    &\text{(b)}&
    &\max_{\lambda\in\Lambda} 
    e_2(\alpha^\lambda, G^\lambda, \tensorc^{\lambda})
    = e_2(\alpha, G, \tensorc) + \delta_H(I_{\min})
    &&\text{if $\twochaos$ is nonempty,}
  \end{alignat*}
    where
    $I_{\min}=\min_{C\in\twochaos} I(C)$ and
    \begin{align*}
      \delta_H(k) = 2\phi_H(k+1) - (\phi_H(2k) + \phi_H(2))
    \end{align*}
    for $k\in\bbZ_{\geq1}$. 
    In particular, if $\Ctto(G)$ is nonempty, then
    $%\begin{align*}
      \max_{\lambda\in\Lambda} e_2(\alpha^\lambda, G^\lambda, \tensorc^{\lambda})
      = e_2(\alpha, G, \tensorc).
    $%\end{align*}
\end{proposition}

The value $\delta_H(k)$ is written explicitly as
\begin{subnumcases}{\delta_H(k)=}
  0
  & for $1\leq k\leq \frac1{2(2H-1)}$\nn
  %& for $H\in\opcl{\frac12,\frac12 + \frac1{4I}}$
  \vspssm\\
  \frac12 + (1-2H) k
  & for $\frac1{2(2H-1)}\leq k\leq \frac{2-2H}{2H-1}$\nn
  %& for $H\in\sbr{\frac12 + \frac1{4I}, \frac12 + \frac1{2(I+1)}}$
  \vspssm\\
  2H-\frac32
  & for $\frac{2-2H}{2H-1}\leq k$\nn
  %& for $H\in\clop{\frac12 + \frac1{2(I+1)}, \frac34}.$
\end{subnumcases}
Notice that 
\begin{align}
  0=\delta_H(1)\geq\delta_H(k)\geq 2H-\frac32
  \tfor k\in\bbN  \tand
  \delta_H(k_1)\geq\delta_H(k_2)
  \tfor k_1<k_2\in\bbN.
  \label{211224.2242}
\end{align}
%$\delta_H(I)=0$ for $I=1$ and
%$\delta_H(I_1)\geq\delta_H(I_2)$ for $I_1<I_2\in\bbZ_{\geq1}$

\begin{remark}
  This proposition is not used in the proof of the asymptotic expansion.
  However some functionals can have sharper estimates by this proposition.
  %Some terms \redb{specific} can have a sharper estimate.
  %\redb{Adding examples are better}
\end{remark}

\begin{proof}[Proof of  Proposition \ref{210413.0101}.]
  Without loss of generality, we assume $\alpha=0$, %$\bbf=\bbone$
and $V=[m]$ and $V'=\cbr{m+1}$ with some $m\in\bbN$.
As seen in the proof of Lemma \ref{211226.1426},
we may only consider the case where $\bbf=\bbone.$ %$\kerf{v}{j}=\charfn{j}$ for $v\in V$.
We abbreviate $u_n(A')$ as $u_n$.

Then, we can write $D_{u_n} \cali_n$ as:
\begin{align*}
  D_{u_n} \cali_n =&
  n^{2H-1/2} \sum_{j\in[n]^{m+1}}
  %\brbr{D_{1_{j_{m+1}}} A_n(j_{[m]})} A'_n(j_{m+1}) %a_{t_{j_{m+1}-1}}
  \brbr{D_{1_{j_{m+1}}} A_n(j_{[m]})} A'_n(j_{m+1})\;\,
  I_1(1_{j_{m+1}})
  \prod_{C\in\Comp(G)\setminus\tensorc} B^C_n(j_{V_C})
  \prod_{C\in \tensorc} \check{B}^C_n(j_{V_C})
  \nn\\&+
  \sum_{\cppr\in \Csone(G)\sqcup(\Ctt(G)\setminus\tensorc)} 
  n^{2H-1/2} \sum_{j\in[n]^{m+1}}
  A_n(j_{[m]}) A'_n(j_{m+1})\;% a_{t_{j_{m+1}-1}}
  I_1(\charf{j_{m+1}})
  \brbr{D_{\charf{j_{m+1}}} B^{\cppr}_n(j_{V_{\cppr}})}
  \nn\\&\hspace{170pt}\times
  \prod_{C\in (\Comp(G)\setminus\tensorc)\setminus\cbr{\cppr}} 
  %\prod_{c\in C_0\sqcup (C^1\sqcup\nottenc)\setminus\cbr{c'}} 
    B^C_n(j_{V_C})
  \prod_{C\in \tensorc} \check{B}^C_n(j_{V_C})
  \nn\\&+
  \sum_{\cppr\in \tensorc} 
  n^{2H-1/2} \sum_{j\in[n]^{m+1}}
  A_n(j_{[m]}) A'_n(j_{m+1})\;%a_{t_{j_{m+1}-1}}
  I_1(\charf{j_{m+1}})
  \brbr{D_{\charf{j_{m+1}}} \check{B}^{\cppr}_n(j_{V_{\cppr}})}
  \nn\\&\hspace{110pt}\times
  \prod_{C\in \Comp(G)\setminus\tensorc} B^C_n(j_{V_C})
  \prod_{C\in \tensorc\setminus\cbr{\cppr}} \check{B}^C_n(j_{V_C})
  \\=:&\;
  \cali_{n}^{(0)}
  + \sum_{\cppr\in \Csone(G)\sqcup\Ctt(G)} \cali_{n}^{(\cppr)} 
  \nn
\end{align*}

Let $A''$ as %and $\bbf'=(\bbf'_v)_{v\in[m+1]}$ as
\begin{align*}
  &A''=\brbr{A''_n(j) }_{j\in[n]^{m+1},  n\in\bbN}=
  \brbr{A_n(j_{[m]}) A'_n(j_{m+1})}_{j\in[n]^{m+1}, n\in\bbN}
  %\\
  %&\bbf'_v=\bbf_v\;\text{ for }\; v\in[m]\;\text{ and }\;
  %\bbf'_{m+1}=\bbone
\end{align*}
and define a weighted graph
$\cpr:=(\cbr{m+1},0,2)$.
Notice that $A''\in\cala([m+1])$ and $s(\cpr)=2$.

(i) First we consider $(\cali_n^{(0)})_n$.
Let
$\alpha^{(0)}=2H-\frac{3}{2}$.
Define $C^{(0)}$, $G^{(0)}$ and 
$A^{(0)}=\brbr{A^{(0)}_n(j)}_{j\in[n]^{m+1}, n\in\bbN}$ by
\begin{align*}
  C^{(0)}&=\abr{m+1}_{-1} \cpr,\quad
  G^{(0)}=C^{(0)}\vee G \tand
  A^{(0)}_n(j)=n\brbr{D_{\charf{j_{m+1}}} A_n(j_{[m]})} A'_n(j_{m+1}).
\end{align*}
Since
$\Comp(G^{(0)}) = \Comp(G)\sqcup\cbr{C^{(0)}}$
and $s(C^{(0)})=1$,
the weighted graph $G^{(0)}$ satisfies Assumption \ref{211203.1620}.
We can see that 
$A^{(0)}\in\cala([m+1])$,
and
\begin{align}
  \cali_n^{(0)} = \funcSec(\alpha^{(0)}, G^{(0)}, \tensorc, A^{(0)}, \bbone).
  \label{211224.2201}
\end{align}
Since 
$I\rbr{C^{(0)}}=1$,
the exponent $e_2(\alpha^{(0)}, G^{(0)}, \tensorc)$ is 
\begin{align}
  e_2(\alpha^{(0)}, G^{(0)}, \tensorc)
  &= 2H-\frac{3}{2} + (2-I\brbr{C^{(0)}}-1) 
  + \sum_{C\in\Comp(G)\setminus\tensorc} e_2(C)
  + \sum_{C\in \tensorc} \expoT(C)
  = 2H-\frac{3}{2} + e_2(0,G,\tensorc).
\label{211224.2221}
\end{align}
We set $\tensorc^{(0)}=\tensorc$.

\vspssm
(ii) Consider $(\cali_n^{(\cppr)})_n$ for $\cppr\in\Csone(G)$.
Let
$\alpha^{(\cppr)}=2H-\frac{1}{2}$ and 
define $C^{(\cppr)}$ and $G^{(\cppr)}$ by
\begin{align*}
  C^{(\cppr)} = \dabr{v^{\cppr}, m+1}_{1} (\cppr\vee\cpr) \tand
  &G^{(\cppr)}= C^{(\cppr)} \vee \Brbr{\mathop\vee_{C\in\Comp(G)\setminus\cbr{\cppr}} C},
\end{align*}
where we denote the only vertex  $v\in V(\cppr)$ such that $\vertwtlow(v)=1$ by $v^{\cppr}$.
We have 
\begin{align*}
  B_n^{C^{(\cppr)}}(j_{V_{C^{(\cppr)}}})
  =
  \beta_n^{\cppr}(j_{V_{\cppr}})
  \beta_n(j_{v^{\cppr}}, j_{m+1})  I_1(\charf{j_{m+1}})
  =
  I_1(\charf{j_{m+1}})\brbr{D_{\charf{j_{m+1}}} B^{\cppr}_n(j_{V_{\cppr}})}.
\end{align*}
Since 
$\Comp(G^{(\cppr)}) = (\Comp(G)\setminus\cbr{\cppr}) \sqcup \bcbr{C^{(\cppr)}}$ and
$s(C^{(\cppr)})= (1+2) + 0 - 2(2-1) = 1$
by (\ref{211225.2120}),
the weighted graph $G^{(\cppr)}$ satisfies Assumption \ref{211203.1620}.
We have
\begin{align*}
\Cttt(G^{(\cppr)})=\Cttt(G)\supset\tensorc \tand
\Comp(G^{(\cppr)})\setminus\tensorc
=((\Comp(G)\setminus\tensorc)\setminus\cbr{\cppr}) \sqcup \bcbr{C^{(\cppr)}},
\end{align*}
and hence
\begin{align}
  \cali_n^{(\cppr)} = \funcSec(\alpha^{(\cppr)}, G^{(\cppr)}, \tensorc, A'', \bbone).
  \label{211224.2202}
\end{align}
Since
$I(C^{(\cppr)}) = I(\cppr)+1$ and
$e_2(C^{(\cppr)}) = 2-I(C^{(\cppr)})-1 %= 2-(I_{c'}+1)-1
= (2 - I(\cppr) -1) -1 = e_2(\cppr) -1$,
the exponent $e_2(\alpha^{(\cppr)}, G^{(\cppr)}, \tensorc)$ is 
\begin{align}
  e_2(\alpha^{(\cppr)}, G^{(\cppr)}, \tensorc)
  &= 
  2H-\frac{1}{2} + e_2(C^{(\cppr)})
  + \sum_{C\in \Comp(G)\setminus(\tensorc\sqcup\cbr{\cppr})} e_2(C)
  + \sum_{C\in \tensorc} \expoT(C)
  \nn\\&=
  2H-\frac{3}{2} + e_2(\cppr)
  + \sum_{C\in \Comp(G)\setminus(\tensorc\sqcup\cbr{\cppr})} e_2(C)
  + \sum_{C\in \tensorc} \expoT(C)
  = 2H-\frac{3}{2} + e_2(0,G,\tensorc).
  \label{211224.2222}
\end{align}
We set $\tensorc^{(\cppr)}=\tensorc$.

\vspssm
(iii) Consider $(\cali_n^{(\cppr)})_n$ for $\cppr\in \Ctt(G)$.
Since we have
\begin{align*}
  D_{1_{j_{m+1}}} \check{B}^{\cppr}_n(j_{V_{\cppr}})=&
  D_{1_{j_{m+1}}} B^{\cppr}_n(j_{V_{\cppr}})
  =
  \rbr{\beta_n\brbr{j_{m+1},j_{\vc{\cppr}{1}}} 
  I_1(\charf{j_{\vc{\cppr}{I(\cppr)}}})
  + \beta_n\brbr{j_{m+1},j_{\vc{\cppr}{I(\cppr)}}} 
  I_1(\charf{j_{\vc{\cppr}{1}}})}
  \beta_n^{\cppr}(j_{V_{\cppr}})
  \nn\\
  =&
  \begin{cases}
    \rbr{\beta_n\brbr{j_{m+1},j_{\vc{\cppr}{1}}}
    I_1(\charf{j_{\vc{\cppr}{I(\cppr)}}})
    + \beta_n\brbr{j_{m+1},j_{\vc{\cppr}{I(\cppr)}}}
    I_1(\charf{j_{\vc{\cppr}{1}}})}
  \\\hspace{110pt}\times
  \beta_n\brbr{j_{\vc{\cppr}{1}},j_{\vc{\cppr}{2}}}\cdots
  \beta_n\brbr{j_{\vc{\cppr}{I(\cppr)-1}},j_{\vc{\cppr}{I(\cppr)}}}  
  \quad&\tifsm \cppr\in \Cttt(G) \\
  2I_1(1_{j_{v^{\cppr}_1}})\beta_n\brbr{j_{m+1},j_{v^{\cppr}_1}}
  \quad&\tifsm \cppr\in \Ctto(G),
  \end{cases}
\end{align*}
the functional $\cali_n^{(\cppr)}$ 
is decomposed as
%\begin{align}
$\cali_n^{(\cppr)} = \cali_n^{(\cppr,1)} + \cali_n^{(\cppr,2)},$
%\end{align}
where
\begin{align*}
  \cali_n^{(\cppr,1)} = &
  n^{2H-1/2} \sum_{j\in[n]^{m+1}} A''_n(j)\;
  I_1(\charf{j_{m+1}})
  \beta_n\brbr{j_{m+1},j_{\vc{\cppr}{1}}} 
  I_1(\charf{j_{\vc{\cppr}{I(\cppr)}}})\;
  \beta_n^{\cppr}(j_{V_{\cppr}})
  %\bbcbr{\prod_{\vpr<\vppr\in V_{c'}}\beta_{j_{\vpr},j_{\vppr}}^{\edgeWt_{\vpr,\vppr}}}
  \nn\\&\hspace{80pt}\times
  \prod_{C\in (\Comp(G)\setminus\tensorc)\setminus\cbr{\cppr}} B^C_n(j_{V_C})
  \prod_{C\in \tensorc\setminus\cbr{\cppr}} \check{B}^C_n(j_{V_C})
  %\label{211222.1804}
  \\
  \cali_n^{(\cppr,2)} = &
  n^{2H-1/2} \sum_{j\in[n]^{m+1}}
  A''_n(j)
  I_1(\charf{j_{m+1}})
  \beta_n\brbr{j_{m+1},j_{\vc{\cppr}{I(\cppr)}}} 
  I_1(\charf{j_{\vc{\cppr}{1}}})
  \beta_n^{\cppr}(j_{V_{\cppr}})
  %\bbcbr{\prod_{\vpr<\vppr\in V_{c'}}\beta_{j_{\vpr},j_{\vppr}}^{\edgeWt_{\vpr,\vppr}}}
  \nn\\&\hspace{80pt}\times
  \prod_{C\in (\Comp(G)\setminus\tensorc)\setminus\cbr{\cppr}} B^C_n(j_{V_C})
  \prod_{C\in \tensorc\setminus\cbr{\cppr}} \check{B}^C_n(j_{V_C}).
  \nn
\end{align*}
We only consider $\cali_n^{(\cppr,1)}$ here.
(Similar aruguments work for $\cali_n^{(\cppr,2)}$.)
Let
$\alpha^{(\cppr,1)}=2H-\frac{1}{2}$
and 
define $C^{(\cppr,1)}$ and $G^{(\cppr,1)}$ by
\begin{align*}
  C^{(\cppr,1)} =\dabr{\vc{\cppr}{1}, m+1}_{1} (\cppr\vee \cpr) \tand
  &G^{(\cppr,1)}=
  C^{(\cppr,1)} \vee 
  \Brbr{\mathop\vee_{C\in \Comp(G)\setminus\cbr{\cppr}} C}.
\end{align*}
We can observe that $C^{(\cppr,1)}$
is a path graph with weighted ends of $I\brbr{C^{(\cppr,1)}}=I\brbr{\cppr}+1(\geq2)$ vertices.
%with 
%$q^{*}_v=1$ for $v=\vc{c'}{I_{c'}},m+1$
%and $q^{*}_v=0$ for the other vertices $v$.
Since 
$\Comp(G^{(\cppr,1)}) = (\Comp(G)\setminus\cbr{\cppr}) \sqcup 
\bcbr{C^{(\cppr,1)}}$,
the weighted graph $G^{(C',1)}$ satisfies Assumption \ref{211203.1620}.
We may abbreviate
$\alpha^{(\cppr,1)}$, $C^{(\cppr,1)}$ and $G^{(\cppr,1)}$ 
as $\alpha^*$, $C^*$ and $G^*$ respectively.
Notice that
\begin{align*}
  B^{C^*}_n(j_{V_{C^*}})=
  I_1(\charf{j_{m+1}})
  \beta_n\brbr{j_{m+1},j_{\vc{\cppr}{1}}} 
  I_1(\charf{j_{\vc{\cppr}{I(\cppr)}}})\;
  \beta_n^{\cppr}(j_{V_{\cppr}}).
\end{align*}

%{We have$\Cqom(G^{(\cppr,1)})=\Cqom(G)$.}

\vspssm
(iii-i) The case $\cppr\in\nottenc$. 
%The last two products in (\ref{211222.1804}) read 
%$\prod_{c\in \nottenc\setminus\cbr{c'}}$ and 
%$\prod_{c\in C_{1-}(G)\sqcup\twochaos}$, respectively.
Since $\tensorc\setminus\cbr{\cppr}=\tensorc$,
the functional $\cali_n^{(\cppr,1)}$ is written as
\begin{align}
  \cali_n^{(\cppr,1)} &=
  n^{2H-1/2} \sum_{j\in[n]^{m+1}} A''_n(j)\;
  \prod_{C\in\nrbr{\Comp(G)\setminus(\tensorc\sqcup\cbr{\cppr})}\sqcup\ncbr{C^{*}}} B^C_n(j_{V_C})
  \prod_{C\in \tensorc} \check{B}^C_n(j_{V_C})
  \nn\\&= \funcSec(\alpha^*, G^{*}, \tensorc, A'', \bbone).
  \label{211224.2203}
\end{align}
Notice that 
$\Cttt(G^{*})=(\Cttt(G)\setminus\cbr{\cppr})\sqcup\bcbr{C^{*}}
\supset\tensorc$ and 
$\Comp(G^{*})\setminus\tensorc
=\brbr{\Comp(G)\setminus(\tensorc\sqcup\cbr{\cppr})}\sqcup\bcbr{C^{*}}$.

The exponent $e_2(C^{(\cppr,1)})$ is bounded as 
\begin{align*}
  e_2(C^{*}) =e_2^+(I(C^{*}))&=
  2-I(C^{*}) + 2\phi_H(I(C^{*}))=
  2- I(\cppr) -1 + 2 \phi_H(I(\cppr)+1)
  \\&\leq
  -1 + (2-I(\cppr) + 2 \phi_H(I(\cppr)))
  =-1 + e_2^+(I(\cppr))
  =-1 + e_2(\cppr)
\end{align*}
using $\phi_H(I+1)\leq\phi_H(I)$  (\ref{211127.2424}).
The exponent $e_2(\alpha^*, G^{*}, \tensorc)$ is 
\begin{align}
  e_2(\alpha^*, G^{*}, \tensorc)
  &= 2H-\frac{1}{2} + e_2(C^{*})
  + \sum_{C\in\Comp(G)\setminus(\tensorc\sqcup\cbr{\cppr})} e_2(C)
  + \sum_{C\in \tensorc} \expoT(C)
  %\nn\\&=
  %2H-\frac{1}{2} + 2-I_{c^{*}} + 2\phi_H(I_{c^{*}})
  %+ \sum_{c\in C\setminus(\cbr{c'}\sqcup\tensorc)} e_2(c)
  %+ \sum_{c\in \tensorc} \expoT(c)
  \nn\\&\leq
  2H-\frac{3}{2} + e_2(\cppr)
  + \sum_{C\in \Comp(G)\setminus(\tensorc\sqcup\cbr{\cppr})} e_2(C)
  + \sum_{C\in \tensorc} \expoT(C)
  = 2H-\frac{3}{2} + e_2(0,G,\tensorc).
  \label{211224.2223}
\end{align}
We set $\tensorc^{(\cppr,1)}=\tensorc$.

\vspssm
(iii-ii) The case $\cppr\in\tensorc$.
Since
$(\Comp(G)\setminus\tensorc)\setminus\cbr{\cppr}=\Comp(G)\setminus\tensorc$,
the functional $\cali_n^{(\cppr,1)}$ is written as
\begin{align}
  \cali_n^{(\cppr,1)} &=
  n^{2H-1/2} \sum_{j\in[n]^{m+1}} A''_n(j)\;
  \prod_{C\in (\Comp(G)\setminus\tensorc)\sqcup\ncbr{C^{*}}} B^C_n(j_{V_C})
  \prod_{C\in \tensorc\setminus\cbr{\cppr}} \check{B}^C_n(j_{V_C})
  \nn\\&= \funcSec(\alpha^*, G^{*}, \tensorc\setminus\cbr{\cppr}, A'', \bbone).
  \label{211224.2204}
\end{align}
Notice that
%$\Ctto(G^{*})=\Ctto(G)$,
$\Cttt(G^{*})=(\Cttt(G)\setminus\cbr{\cppr})\sqcup\bcbr{C^{*}}
\supset\tensorc\setminus\cbr{\cppr}$ and 
$\Comp(G^{*})\setminus(\tensorc\setminus\cbr{\cppr})
=(\Comp(G)\setminus\tensorc)\sqcup\bcbr{C^{*}}$.

Since 
\begin{align*}
  e_2(C^{*}) 
  = e_2^+(I(C^{*})) 
  &= 2-(I(\cppr)+1) + 2\phi_H(I(\cppr)+1)
  \\&=
  \Brbr{2-I(\cppr) -1  + \phi_H(2I(\cppr))} + \phi_H(2) + \delta_H(I(\cppr))
  %\\\leq& 
  %(2-I_{c'} -1  + \phi_H(2I_{c'})) + \phi_H(2)
  %(2-I_{c'} -1  + \phi_H(2I_{c'})) + (2- 1 -1  + \phi_H(2))
  %= e_2^-(I_{c'}) + e_2^-(1) 
  \\&= e_2^-(I(\cppr)) + \frac{1}{2} -2H + \delta_H(I(\cppr))
  = \expoT(\cppr) + \frac{1}{2} -2H + \delta_H(I(\cppr)),
\end{align*}
the exponent $e_2(\alpha^*, G^{*}, \tensorc\setminus\cbr{\cppr})$ is 
\begin{align*}
  e_2(\alpha^*, G^{*}, \tensorc\setminus\cbr{\cppr})
  &= 2H-\frac{1}{2} + e_2(C^{*}) 
  + \sum_{C\in\Comp(G)\setminus\tensorc} e_2(C)
  + \sum_{C\in \tensorc\setminus\cbr{\cppr}} \expoT(C)
  =  e_2(0,G,\tensorc) + \delta_H(I(\cppr)) .
\end{align*}
We set $\tensorc^{(\cppr,1)}=\tensorc\setminus\cbr{\cppr}$.

\vspsm
(iii-iii) The case $\cppr\in \Ctto(G)$.
Since $\tensorc\setminus\cbr{\cppr}=\tensorc$,
the functional $\cali_n^{(\cppr,1)}$ is written as
\begin{align}
  \cali_n^{(\cppr,1)} &=
  n^{2H-1/2} \sum_{j\in[n]^{m+1}} A''_n(j)\;
  \prod_{C\in \nrbr{\Comp(G)\setminus(\tensorc\sqcup\cbr{\cppr})}\sqcup\ncbr{C^{*}}} B^C_n(j_{V_C})
  \prod_{C\in \tensorc} \check{B}^C_n(j_{V_C})
  \nn\\&= \funcSec(\alpha^*, G^{*},\tensorc, A'', \bbone).
  \label{211224.2205}
\end{align}
Notice that
%$\Ctto(G^{*})=\Ctto(G)\setminus\cbr{c'}$,
$\Cttt(G^{*})=\Cttt(G)\sqcup\bcbr{C^{*}}
\supset\tensorc$ and 
$\Comp(G^{*})\setminus\tensorc=\brbr{\Comp(G)\setminus(\tensorc\sqcup\cbr{\cppr})}\sqcup\bcbr{C^{*}}$.

Since we have
\begin{align*}
  e_2(C^{*}) 
  = e_2^+(I(C^*)) 
  = e_2^+(2) 
  = 2\phi_H(2)
  %= 2e_2^-(1)
  = e_2^-(I(\cppr)) + \frac12 -2H
  = e_2(\cppr) + \frac12 -2H,
\end{align*}
the exponent $e_2(\alpha^*, G^{*}, \tensorc)$ is
\begin{align*}
  e_2(\alpha^*, G^{*}, \tensorc)
  &= 2H-\frac{1}{2} + e_2(C^{*}) 
  + \sum_{C\in (\Comp(G)\setminus\tensorc)\setminus\cbr{\cppr}} e_2(C)
  + \sum_{C\in \tensorc} \expoT(C)
  %\nn\\&{\text{the other way to decompose the sum would be easier to understand}}
  \nn\\&=
  2H-\frac{1}{2} 
  %+ e_2^-(I_{c'}) 
  + e_2(\cppr)
  + \frac{1}{2} -2H
  + \sum_{C\in (\Comp(G)\setminus\tensorc)\setminus\cbr{\cppr}} e_2(C)
  + \sum_{C\in \tensorc} \expoT(C)
  \\&=  e_2(0,G,\tensorc)
  =  e_2(0,G,\tensorc) + \delta_H(1).
\end{align*}
We set $\tensorc^{(\cppr,1)}=\tensorc$.

\vspssm
By summing up the above arguments (i), (ii), (iii-i), (iii-ii) and (iii-iii),
we have
\begin{align*}
  D_{u_n} \cali_n =&\;
  \cali_{n}^{(0)}
  + \sum_{C\in\Csone(G)} \cali_{n}^{(C)} 
  + \sum_{C\in \Ctt(G)} \rbr{\cali_{n}^{(C,1)} + \cali_{n}^{(C,2)} },
\end{align*}
where 
$\cali_{n}^{(0)}$, 
$\cali_{n}^{(C)}$ for $C\in\Csone(G)$ and
$\cali_{n}^{(C,1)}$ for $C\in\Ctt(G)$ 
have the representations 
(\ref{211224.2201}), (\ref{211224.2202}), (\ref{211224.2203}), (\ref{211224.2204})
and (\ref{211224.2205}).
We can see that $\cali_{n}^{(C,2)}$ $(C\in\Ctt(G))$ 
has a similar representation corresponding 
to one of (\ref{211224.2203}), (\ref{211224.2204})
and (\ref{211224.2205}).
Let
$\Lambda=\cbr{0} \sqcup \Csone(G) \sqcup \cbr{(C,1),(C,2)\mid C\in\Ctt(G)}$.

Consider the case (a) $\twochaos$ is empty,
where 
$\Lambda=\cbr{0} \sqcup \Csone(G) \sqcup \cbr{(C,1),(C,2)\mid C\in\nottenc}$.
We have
\begin{align*}
  \max_{\lambda\in\Lambda} 
  e_2(\alpha^\lambda, G^\lambda, \tensorc^\lambda)
  = e_2(0, G, \tensorc) + 2H-\frac32
\end{align*}
from (\ref{211224.2221}), (\ref{211224.2222}) and (\ref{211224.2223})
and $0\in\Lambda$ in any case.
%{When the functionals belonging to $A$ are deteministic,
%we can set $\Lambda=\Csone \sqcup \cbr{(c,1),(c,2): c\in\nottenc}$}
%
In the case (b) $\twochaos$ is nonempty, 
we have
\begin{align*}
  \max_{\lambda\in\Lambda} 
  e_2(\alpha^\lambda, G^\lambda, \tensorc^\lambda)
  = e_2(0, G, \tensorc) + 
  \rbr{2H-\frac32} \vee
  \max_{c\in\twochaos} \delta_H(I(C))
  = e_2(0, G, \tensorc) + \delta_H(I_{0})
\end{align*}
with $I_0=\min_{C\in\twochaos} I(C)$
by (\ref{211224.2242}).
\end{proof}

\subsubsection{Change of the order of by the action of $D^i$ }
\begin{proposition}\label{210407.2246}
  Let $\alpha\in\bbR$.
  Suppose that $G=(V,\edgeWt,\vertWt)$ is a weighted graph satisfying Assumption \ref{211203.1620}, 
  $A\in\cala(V)$ and $\bbf\in\calf(V)$.
  For a subset $\tensorc$ of $\Cttt(G)$ % satisfying $\cond^2_{2,2+}$,
  and $i\in\bbZ_{\geq1}$,
  \begin{align*}
    \Bnorm{ \snorm{D^i \cali_n}_{\calh^{\otimes i}} }_p
    = \olog(n^{e_2(\alpha, G, \tensorc)})
  \end{align*}
  for any $p>1$ %for any $p\geq2$
  as $n\to\infty$ for
  $\cali_n = \funcSec(\alpha,G,\tensorc,A,\bbf)$
  given by (\ref{211126.2643}).
\end{proposition}

\begin{proof}
  As seen in the proof of Lemma \ref{211226.1426},
  we may only consider the case where $\bbf=\bbone$.
  Without loss of generality, we assume that 
$\alpha=0$ and $V$ is written as $V=[m]$ with some $m\in\bbN$.
Fix %$\bbf$,
$G=(V,\edgeWt,\vertWt)$, $\tensorc\subset \Cttt(G)$ and $A\in\cala(V)$.
%Without loss of generality, we can assume $\alpha=0$.
%We fix 
%$G,A,\bbf,\tensorc$
%and 
Consider 
$\cali_n = \funcSec(\alpha,G,\tensorc,A,\bbf)$ defined at (\ref{211126.2643}).
The $i$-th derivative of $\cali_n$ can be written as
\begin{align*}
  D^i\cali_n = 
  \sum_{\bar\lambda\in (\cbr{0}\sqcup \Comp(G))^{[i]}} 
  D^{\bar\lambda} \cali_n,
\end{align*}
where the above summation runs through 
the set $(\cbr{0}\sqcup \Comp(G))^{[i]}$ of all the mappings 
$\bar\lambda$ 
from $\cbr{1,..,i}$ to $\cbr{0}\sqcup \Comp(G)$,
\begin{align*}
  D^{\bar\lambda}\cali_n =
  \rbr{
  %n^\alpha
  \sum_{j\in[n]^m} D^{\bar\lambda_0}_{s_{\bar\lambda^{-1}[0]}} A_n(j)
  \prod_{C\in\Comp(G)\setminus\tensorc} 
    D^{\bar\lambda_C}_{s_{\bar\lambda^{-1}[C]}} B^C_n(j_{V_C})
  \prod_{C\in \tensorc} 
    D^{\bar\lambda_C}_{s_{\bar\lambda^{-1}[C]}} \check{B}^C_n(j_{V_C})
  }_{s_1,..,s_i\in[0,T]}
\end{align*}
with the notations
$\bar\lambda^{-1}[C]:= \bar\lambda^{-1}(\cbr{C})$ and
$\blamc:=\abs{\bar\lambda^{-1}[C]}$ for $C\in\cbr{0}\sqcup \Comp(G)$, and
$s_{K}$ reads
$s_{k_{1}},..,s_{k_{\abs{K}}}$
for $K=\cbr{k_1,..,k_{\abs{K}}}\subset\cbr{1,..,i}$. % with $k_1<..<k_{\abs{K}}$.
For  $\bar\lambda\in (\cbr{0}\sqcup \Comp(G))^{[i]}$ and
$k\in\bbZ_{\geq0}$,
we write
$\cblam{k}=\cbr{C\in \Comp(G)\mid \bar\lambda_C=k}$.
We only consider $\bar\lambda$ satisfying
$\bar\lambda_C\leq\barq(C)$ for $C\in \Comp(G)$,
in other words
\begin{align*}%\label{211211.2300}
  \cblam{1}\subset \Csone(G)\sqcup \Ctt(G),\quad
  \cblam{2}\subset \Ctt(G) \tand 
  \bigsqcup_{k=0,1,2}\cblam{k}=\Comp(G);
\end{align*}
%$C_0\subset \cblam{0}$,
%$C^1\subset \cup_{k=0,1}\cblam{k}$ and
%$C^2_2\subset\cup_{k=0,1,2}\cblam{k}$
otherwise 
$D^{\bar\lambda}\cali_n$ is zero.

Since 
$B^C_n(j_{V_C})$ and $\check{B}^C_n(j_{V_C})$ for $C\in\Cttt(G)$
differ only by some deterministic value thanks to the product formula,
we have
\begin{align*}
  D^{k} B^C_n(j_{V_C})
  =
  D^{k} \check{B}^C_n(j_{V_C})
\end{align*}
for $k\in\bbZ_{\geq1}$.
We can write, hence, for a fixed $\bar\lambda$ 
\begin{align*}
  D^{\bar\lambda}\cali_n =
  \sum_{\lambda}D^\lambda \cali_n
\end{align*}
 with 
\begin{align*}
  D^\lambda \cali_n = 
  c_\lambda \Bigg(%\rbr{
  %n^\alpha
  \sum_{j\in[n]^m} 
    &D^{\lambda_0}_{s_{\lambda^{-1}[0]}} A_n(j)
  \prod_{C\in \cblam{0}\setminus\tensorc} 
    B^C_n(j_{V_C})
  \prod_{C\in \cblam{0}\cap\tensorc} 
    \check{B}^C_n(j_{V_C})
  \nn\\
  &\times
  \prod_{C\in \cblam{1}\sqcup\cblam{2}} 
  \rbr{
    \beta_n^C(j_{V_C})\;
    \prod_{v\in V(C)} 
    I_{(\vertwtlow_v-\lambda_v)}(1_{j_v}^{\otimes(\vertwtlow_v-\lambda_v)})\;
    1_{j_v}^{\otimes\lambda_v}(s_{\lambda^{-1}[v]})
  }\Bigg)_{s_1,..,s_i\in[0,T]},
\end{align*}
where the above summation runs through 
all the mappings $\lambda$ 
from $\cbr{1,..,i}$ to $\cbr{0}\sqcup V$
satisfying that
$\bar\lambda^{-1}[0]=\lambda^{-1}[0]$,
%$\lambda(\bar\lambda^{-1}[0])\subset\cbr{0}$,
%$\lambda(k)=0$ for $k\in\bar\lambda^{-1}(\cbr{0})$ and 
$\bar\lambda^{-1}[C]= \lambda^{-1}(V(C))$ for $C\in \Comp(G)$ and 
%$\lambda(\bar\lambda^{-1}[c])\subset V_c$ for $c\in C$ and 
$\lambda_v\leq \vertwtlow(v)$ for $v\in V$,
%$\lambda(k)\in V_c$ for $k\in\bar\lambda^{-1}(\cbr{c})$ ($c\in C$),
with the notations
$\lambda^{-1}[v]:=\lambda^{-1}(\cbr{v})$ and
$\lambda_v:=\abs{\lambda^{-1}[v]}$ for $v\in \cbr{0}\sqcup V$, and
$c_\lambda=\prod_{v\in V} \frac{\vertwtlow_v!}{(\vertwtlow_v-\lambda_v)!} 
=2^{\abs{\Ctto(G)\cap(\cblam{1}\sqcup\cblam{2})}}$ 
is a conbinatorial constant.
We write $\lambda\ll\bar\lambda$
when $\lambda$ satisfies these relations with $\bar\lambda$.
We have
$\bar\lambda_C = \sum_{v\in V(C)} \lambda_v$ for $C\in\Comp(G)$ and 
$i = \sum_{v\in\cbr{0}\sqcup V} \lambda_v$.

Since 
\begin{align}\label{211211.1711}
  \Bnorm{\bnorm{D^i\cali_n}_{\calh^{\otimes i}}}_p 
  \leq
  \sum_{\bar\lambda:\bar\lambda_C\leq\barq(C)}% \text{ for } c\in C}
  \sum_{\lambda\ll\bar\lambda}
    \Bnorm{\snorm{D^\lambda \cali_n}_{\calh^{\otimes i}}}_p
  =
  \sum_{\bar\lambda:\bar\lambda_C\leq\barq(C)}
  \sum_{\lambda\ll\bar\lambda}
    \Bnorm{\snorm{D^\lambda \cali_n}_{\calh^{\otimes i}}^2 }_{p/2}^{1/2}
\end{align}
for $p\geq2$,
we will consider the functional
$
\snorm{D^\lambda \cali_n}_{\calh^{\otimes i}}^2 
$
for fixed $\bar\lambda\in (\cbr{0}\sqcup \Comp(G))^{[i]}$ 
and $\lambda\in (\cbr{0}\sqcup V)^{[i]}$ such that
$\bar\lambda_C\leq\barq(C)$ for $C\in\Comp(G)$ and
$\lambda\ll\bar\lambda$.
We can write
\begin{align*}
  \snorm{D^\lambda\cali_n}_{\calh^{\otimes i}}^2
  =
  c_\lambda^2 \sum_{j,k\in [n]^m}&
  \abr{D^{\lambda_0}A_n(j), D^{\lambda_0}A_n(k)}_{\calh^{\otimes \lambda_0}}
  %\nn\\&\times
  \Brbr{\prod_{C\in \cblam{0}\setminus\tensorc} 
    B^C_n(j_{V_C}) 
    B^C_n(k_{V_C})}
  \Brbr{\prod_{C\in \cblam{0}\cap\tensorc} 
    \check{B}^C_n(j_{V_C}) 
    \check{B}^C_n(k_{V_C})}
  \nn\\&
  \times\prod_{C\in\cblam{1}\sqcup\cblam{2}} 
  \bigg\{
    \beta_n^{C}(j_{V_C})\; \beta_n^{C}(k_{V_C})\;
    \prod_{v\in V(C)}\beta_n\rbr{j_{v},k_{v}}^{\lambda_v}
    \nn\\&\hspace{100pt}\times
    \Brbr{\prod_{v\in V(C)} I_\rbr{\vertwtlow_v-\lambda_v}(1_{j_v}^{\otimes (\vertwtlow_v-\lambda_v)})}
    \Brbr{\prod_{v\in V(C)} I_\rbr{\vertwtlow_v-\lambda_v}(1_{k_v}^{\otimes (\vertwtlow_v-\lambda_v)})}
  \bigg\}
\end{align*}
and 
\begin{align*}
  &\beta_n^{C}(j_{V_C})\; \beta_n^{C}(k_{V_C})\;
    \prod_{v\in V(C)}\beta_n\rbr{j_{v},k_{v}}^{\lambda_v}
    \Brbr{\prod_{v\in V(C)} I_\rbr{\vertwtlow_v-\lambda_v}(1_{j_v}^{\otimes (\vertwtlow_v-\lambda_v)})}
    \Brbr{\prod_{v\in V(C)} I_\rbr{\vertwtlow_v-\lambda_v}(1_{k_v}^{\otimes (\vertwtlow_v-\lambda_v)})}
    \\=&
    \begin{cases}
      \beta_n^{C}(j_{V_C})\; \beta_n^{C}(k_{V_C})\;
      \beta_n\rbr{j_{v^C},k_{v^C}},
      & \tif C\in \cblam{1}\cap\Csone(G)\;
      \text{($v^C$ is $v\in V(C)$ with $\lambda_v=1$.)}
      \vspssm\\
      \beta_n\rbr{j_{v^C},k_{v^C}} I_1(1_{j_{v^C}}) I_1(1_{k_{v^C}})
      & \tif C\in \cblam{1}\cap\Ctto(G)\;
      \text{($v^C$ is the only vertex in $V(C)$.)}
      \vspssm\\
      \beta_n^{C}(j_{V_C})\; \beta_n^{C}(k_{V_C})\;
      \beta_n\rbr{j_{v^C},k_{v^C}}\; I_1(1_{j_{v^C_0}})I_1(1_{k_{v^C_0}})
      &\tif C\in \cblam{1}\cap\Cttt(G)
      \vspssm\\
      \beta_n\rbr{j_{v^C},k_{v^C}}^2
      &\tif C\in \cblam{2}\cap\Ctto(G)\;
      \text{($v^C$ is the only vertex in $V(C)$.)}
      \vspssm\\
      \beta_n^{C}(j_{V_C})\; \beta_n^{C}(k_{V_C})\;
      \beta_n\brbr{j_{v^C_1},k_{v^C_1}}\beta_n\brbr{j_{v^C_{I(C)}},k_{v^C_{I(C)}}}
      &\tif C\in \cblam{2}\cap\Cttt(G),
    \end{cases}
\end{align*}
where 
for $C\in\cblam{1}\cap\Cttt(G)$, 
$v^C$ is $v\in V(C)$ with $\lambda_v=1$ and
$v^C_0$ is $v(\neq v^C)\in V(C)$ such that $\vertwtlow(v)=1$;
for $C\in\cblam{2}\cap\Cttt(G)$, 
we recall that $v^C_1$ and $v^C_{I(C)}$ are the two vertices $v\in V(C)$ such that $\vertwtlow(v)=1$.
  
Let
\begin{align*}
  G^{(\lambda,C)} :=&
  \dabr{v^C,v^C+m}_{1} (C\vee \shiftg{C}{m}),
  &&\text{ if $C\in\cblam{1}$.  ($v^C$ is $v\in V(C)$ with $\lambda_v=1$.)}
  \\
  G^{(\lambda,C)} :=&
  \dabr{v^C_1,v^C_1+m}_{1}\dabr{v^C_{I(C)},v^C_{I(C)}+m}_{1}
   (C\vee \shiftg{C}{m})
  &&\text{ if $C\in \cblam{2}(\subset\Ctt(G))$,}
\end{align*}
where 
%we recall that $v^c_1$ and $v^c_{I_c}$ are the two terminal vertices of $c\in\Cttt(G)$, and 
for $C\in \cblam{2}\cap \Ctto(G)$ we read
  $G^{(\lambda,C)} = \dabr{v^C,v^C+m}_{2} (C\vee \shiftg{C}{m})$
  with $\cbr{v^C}=V(C)$.
We define a weighted graph $G^{(\lambda)}$ by 
\begin{align*}
  G^{(\lambda)} = 
  \Brbr{\mathop\vee_{C\in \cblam{0}} (C\vee \shiftg{C}{m})}
  %\vee \Brbr{\mathop\vee_{c\in \cblam{1}} G_c^{1}}
  %\vee \Brbr{\mathop\vee_{c\in \cblam{2}} G_c^{2}}.
  \vee \Brbr{\mathop\vee_{C\in \cblam{1}\sqcup\cblam{2}} G^{(\lambda,C)}}.
\end{align*}

Recall $\cblam{1}\subset\Csone(G)\sqcup\Ctt(G)$.
For $C\in\cblam{1}\cap \Csone(G)$,
we have
$ s(G^{(\lambda,C)})=(1+1) + 0 -2(2-1)=0$ 
by (\ref{211225.2120})
and $G^{(\lambda,C)}$ belongs to $\Cszero(G^{(\lambda)})$.
For $C\in\cblam{1}\cap\Ctt(G)$,
we can observe that $G^{(\lambda,C)}$ is a path graph with weighted ends of $2I(C)$ vertices,
%and hence the component $G^*_c$ satisfies the condition (ii)(b)
%in Assumption \ref{211203.1620}
and belongs to $\Cttt(G^{(\lambda)})$.
For $C\in \cblam{2}(\subset\Ctt(G))$,
we can observe that
$G^{(\lambda,C)}$ is a cycle graph of $2I(C)$ vertices in the sense of Definition \ref{220322.1900},
%and thus the component $G^*_c$ satisfies the condition (ii)(a)
%in Assumption \ref{211203.1620}
and belongs to $\Ctz(G^*)$.
In total, we have
\begin{align*}
  \Cszero(G^{(\lambda)}) =& 
  \bcbr{C, \shiftg{C}{m}\mid C\in  \cblam{0}\cap \Cszero(G)(=\Cszero(G))} 
  %\bcbr{c, \shiftg{c}{m}:c\in \cblam{0}\cap C^0(G)} 
  \sqcup \bcbr{G^{(\lambda,C)}\mid C\in\cblam{1}\cap \Csone(G)}\\
  \Csone(G^{(\lambda)}) =& 
  \bcbr{C, \shiftg{C}{m}\mid C\in\cblam{0}\cap \Csone(G)}\\
  \Ctz(G^{(\lambda)}) =& 
  \bcbr{C, \shiftg{C}{m}\mid C\in\cblam{0}\cap \Ctz(G)(=\Ctz(G))} 
  %\bcbr{c, \shiftg{c}{m}:c\in \cblam{0}\cap C^2_0(G)} 
  \sqcup \bcbr{G^{(\lambda,C)}\mid C\in\cblam{2}\cap \Ctt(G)(=\cblam{2})}\\
  \Ctto(G^{(\lambda)}) =& 
  \bcbr{C, \shiftg{C}{m}\mid C\in\cblam{0}\cap \Ctto(G)}\\
  \Cttt(G^{(\lambda)}) =& 
  \bcbr{C, \shiftg{C}{m}\mid C\in\cblam{0}\cap \Cttt(G)} 
  \sqcup \bcbr{G^{(\lambda,C)}\mid C\in\cblam{1}\cap \Ctt(G)}
\end{align*}
and we see that the graph $G^{(\lambda)}$ satisfies Assumption \ref{211203.1620}.

We set 
\begin{align*}
  A^{(\lambda)}=\brbr{A^{(\lambda)}_n(j)}_{j\in[n]^{2m}, n\in\bbN}=
  \Brbr{c_\lambda^2 
  \abr{D^{\lambda_0}A_n(j_{[m]}), D^{\lambda_0}A_n(j_{[m+1,2m]})}
  _{\calh^{\otimes \lambda_0}}}_{j\in[n]^{2m},n\in\bbN},
\end{align*}
and 
%\begin{align*}
$  \tensorc^{(\lambda)} = \bcbr{C, \shiftg{C}{m}\mid C\in\cblam{0}\cap \tensorc}.
$ %\end{align*}
Indeed 
$A^{(\lambda)}\in\cala([2m])$ and
$\tensorc^{(\lambda)}$ is a subset of $\Cttt(G^{(\lambda)})$.
We have
\begin{align*}
  \snorm{D^\lambda\cali_n}_{\calh^{\otimes i}}^2
  =
  \funcSec(0, G^{(\lambda)}, \tensorc^{(\lambda)}, A^{(\lambda)}, \bbone),
\end{align*}
since 
\begin{align}
  \Comp(G^{(\lambda)})\setminus\tensorc^{(\lambda)}
  &=\Cszero(G^{(\lambda)})\sqcup \Csone(G^{(\lambda)})
  \sqcup \Ctz(G^{(\lambda)})
  \sqcup \Ctto(G^{(\lambda)})\sqcup \rbr{\Cttt(G^{(\lambda)})
  \setminus \tensorc^{(\lambda)}}
  \nn\\&=\bcbr{C, \shiftg{C}{m}\mid C\in \cblam{0}\setminus\tensorc} 
  \sqcup \bcbr{G^{(\lambda,C)}\mid C\in \cblam{1}\sqcup\cblam{2}}.
  \label{220531.1915}
\end{align}

We proceed to estimate the exponent 
$e_2(0,G^{(\lambda)},\tensorc^{(\lambda)})$.
For $C\in\cblam{1}\cap\Csone(G)$, we have
\begin{align*}
  e_2(G^{(\lambda,C)})=2-I(G^{(\lambda,C)})
  =2(1-I(C))=2e_2(C).
\end{align*}
In the both cases 
$C\in \cblam{1}\cap\Ctt(G)$ and 
$C\in \cblam{2}\cap\Ctt(G)$,
we have
\begin{align*}
  e_2(G^{(\lambda,C)})&=e^+_2(I(G^{(\lambda,C)}))=e^+_2(2I(C))
  \\&=(2-2I(C))+2\phi_H(2I(C))
  =2\brbr{(2-I(C))-1+\phi_H(2I(C))}
  =2e^-_2(I(C)).
  %\leq 2e^+_2(I_{c}).
\end{align*}

%$\cblam{1}\sqcup\cblam{2}=(\cblam{1}\cap\Csone(G))\sqcup
%(\brbr{\cblam{1}\sqcup\cblam{2}}\cap\Ctt(G))$,
With the notation 
$\cblam{+}=\cblam{1}\sqcup\cblam{2}$,
we have
\begin{align*}
  \sum_{C\in \cblam{+}} e_2(G^{(\lambda,C)}) 
  &=
  \sum_{C\in \cblam{+}\cap \Csone(G)} 2 e_2(C) +
  \sum_{C\in \cblam{+}\cap \Ctto(G)} 2 e^-_2(I(C)) +
  \sum_{C\in \cblam{+}\cap \tensorc} 2 e^-_2(I(C)) +
  \sum_{C\in \cblam{+}\cap \nottenc} 2 e^-_2(I(C)) 
  \\&\leq
  2\bbcbr{
    \sum_{C\in \cblam{+}\cap \Csone(G)} e_2(C) +
  \sum_{C\in \cblam{+}\cap \Ctto(G)} e^-_2(I(C)) +
  \sum_{C\in \cblam{+}\cap \tensorc} e^-_2(I(C)) +
  \sum_{C\in \cblam{+}\cap \nottenc} e^+_2(I(C)) 
  }
  %\\=&
  %2\bbcbr{
  %  \sum_{c\in \cblam{+}\cap C^1(G)} e_2(c) +
  %\sum_{c\in \cblam{+}\cap \Ctto(G)} e_2(c) +
  %\sum_{c\in \cblam{+}\cap \tensorc} \expoT(c) +
  %\sum_{c\in \cblam{+}\cap \nottenc} e_2(c) 
  %}
  \\&=
  2\bbcbr{
    \sum_{C\in \cblam{+}\setminus \tensorc} e_2(C) +
    \sum_{C\in \cblam{+}\cap \tensorc} \expoT(C) 
  }.
\end{align*}
We used the following relation at the last equality.
\begin{align*}
  \brbr{\cblam{+}\cap \Csone(G)}\sqcup
\brbr{\cblam{+}\cap \Ctto(G)}\sqcup
\brbr{\cblam{+}\cap \nottenc}
=
\cblam{+}\cap \brbr{\Cqzero(G)\sqcup \Csone(G)\sqcup\Ctto(G)\sqcup \nottenc}
=
\cblam{+}\setminus\tensorc
\end{align*}

Thus, we obtain
\begin{align*}
  e_2(0,G^{(\lambda)},\tensorc^{(\lambda)})=&\;
  \sum_{C\in \Comp(G^{(\lambda)})\setminus\tensorc^{(\lambda)}} e_2(C) + 
  \sum_{C\in \tensorc^{(\lambda)}} \expoT(C)
  \\=&\;
  2\sum_{C\in\cblam{0}\setminus\tensorc} e_2(C) + 
  \sum_{C\in \cblam{1}\sqcup \cblam{2}} e_2(G^{(\lambda,C)})+
  2\sum_{C\in \cblam{0}\cap \tensorc} \expoT(C)
  \\\leq&\;
  2\sum_{C\in \Comp(G)\setminus\tensorc} e_2(C) +
  2\sum_{C\in \tensorc} \expoT(C)
  =%\\=&
  2e_2(0,G,\tensorc),
\end{align*}
where we used (\ref{220531.1915}) at the first equality. 
Hence, we have the estimate
\begin{align*}
  \Bnorm{\snorm{D^\lambda\cali_n}_{\calh^{\otimes i}}^2}_p
  =\bar O(n^{2e_2(0,G,\tensorc)})
\end{align*}
for $p>1$ by Proposition \ref{210407.1402}.
This estimate holds
for any $\bar\lambda\in (\cbr{0}\sqcup \Comp(G))^{[i]}$ 
and $\lambda\in (\cbr{0}\sqcup V)^{[i]}$ such that
$\bar\lambda_C\leq\barq(C)$ for $C\in\Comp(G)$ and
$\lambda\ll\bar\lambda$. 
By the inequality (\ref{211211.1711}), we have
\begin{align*}
  \Bnorm{\bnorm{D^i\cali_n}_{\calh^{\otimes i}}}_p 
  \leq
  \sum_{\bar\lambda:\bar\lambda_c\leq\barq(c)}% \text{ for } c\in C}
  \sum_{\lambda\ll\bar\lambda}
    \Bnorm{\snorm{D^\lambda \cali_n}_{\calh^{\otimes i}}}_p
  =
  \sum_{\bar\lambda:\bar\lambda_c\leq\barq(c)}
  \sum_{\lambda\ll\bar\lambda}
    \Bnorm{\snorm{D^\lambda \cali_n}_{\calh^{\otimes i}}^2 }_{p/2}^{1/2}
  =\bar O(n^{e_2(0,G,\tensorc)})
\end{align*}
for $p>2$.
\end{proof}

\subsection{Estimates related to the definition of exponents}
Recall that
$\beta_{j_1,j_2} = \langle 1_{j_1}, 1_{j_2}\rangle$
for $j_1,j_2\in[n]$ and $n\in\bbZ_{\geq1}$ with 
$1_{j} = 1_{(T(j-1)/n,Tj/n]}$.
Notice that $\beta_{j_1,j_2}(>0)$ and $1_{j}$ implicitly depend on $n$, 
$H\in\rbr{\frac12,\frac34}$ and $T>0$.
We write
  $\rho_H(\ell)=\abr{1_{\sbr{0,1}}, 1_{\sbr{\ell,\ell+1}}}
  =E\sbr{(B^H_1-B^H_0) (B^H_{\ell+1}-B^H_\ell)}
  =\frac12\rbr{\abs{\ell+1}^{2H}+\abs{\ell-1}^{2H}-2\abs{\ell}^{2H}}$ for $\ell\in\bbZ$.

\begin{lemma}\label{211006.1220}
Assume that $H\in(\frac{1}{2},\frac{3}{4})$. 
\begin{itemize}
  \item [(a)] $\sup_{j_1,j_2\in[n]} \beta_{j_1,j_2} = O(n^{-2H})$ as $n\to\infty$.
  % const depends on $H,T$.%\label{201229.1505}
  \item [(b)] $\sup_{j_1\in[n]} \rbr{\sum_{j_2\in[n]} \beta_{j_1,j_2}}= O(n^{-1})$
  as $n\to\infty$.
  %\label{211005.2115}
  \item [(c)] Let $k\geq2$ and 
  \begin{align*}
    \betacycle{n}{k} = \sum_{j\in[n]^k} 
    %$B_{n,k} = \sum_{j_1,..,j_k\in[n]}
    \beta_n\rbr{j_1,j_2}\cdots
    \beta_n\rbr{j_{k-1},j_k}
    \beta_n\rbr{j_{k},j_1}. 
  \end{align*}
  Then, for $k=2$,
  \begin{align*}
    \betacycle{n}{k} &= O(n^{1-4H}),
  \end{align*}
  and for $k\geq3$,
  \begin{alignat*}{2}
    \betacycle{n}{k}&=
    \begin{cases}
      O(n^{1-2kH})  \quad&\tifsm H< H(k) \\
      O(n^{1-2kH} \;\log n)   \quad&\tifsm H= H(k) \\
      O(n^{-k})   \quad&\tifsm H> H(k) 
    \end{cases}
  \end{alignat*}
  as $n\to\infty$, where $H(k)=\frac12(1+\frac{1}{k})$.
  With $e_2^+(k)=(1-2kH)\vee (-k)$ ($k\geq2$) defined at (\ref{220121.1137}),
  we can write 
  \begin{align}
    \betacycle{n}{k}=\olog(n^{e_2^+(k)})
    \qquad\text{as $n\to\infty$.}
    \label{220325.1344}
  \end{align}

  In particular, we can roughly estimate
  \begin{align*}
    \betacycle{n}{k}=O(n^{1-4H-(k-2)}).
    %\label{220325.1345}
  \end{align*}

\item [(d)] For an integer $k\geq2$,
let $\phi_H(k) =-\frac12 +((\frac12-H)k)\vee (-\frac12)$ 
as in (\ref{211006.1341}).
For $k_1,k_2\geq1$,
\begin{align*}%\label{220120.1949}
 2\phi_H(k_{1}+k_{2}) \leq \phi_H(2k_{1})+\phi_H(2k_{2}).
\end{align*}
\end{itemize}

\end{lemma}

\begin{proof}
  (a) By the self-similarity of the law of fBm and $\rho_H(\ell)\leq1$,
  \begin{align*}
    \beta_{j_1,j_2} = \rbr{T/n}^{2H} \rho_H(j_1-j_2) \tand
    \sup_{j_1,j_2\in[n]}\beta_{j_1,j_2} = \rbr{T/n}^{2H}.
  \end{align*}

  (b) For $n\in\bbN$ and $j_1\in[n]$,
  \begin{align*}
    \sum_{j_2\in[n]} \beta_{j_1,j_2}
    &=\sum_{j_2\in[n]} \abr{1_{j_1},1_{j_2}}
    =\abr{1_{j_1},1_{\sbr{0,T}}}
    =\alpha_H\int_{\sbr{0,T}^2} 1_{j_1}(t) 1_{\sbr{0,T}}(t') \abs{t-t'}^{2H-2}dtdt'
    \\&\leq
    \alpha_H\int_{0}^{T} 1_{j_1}(t) dt
    \int_{-T}^{T}  \abs{t'}^{2H-2}dt'
    =\alpha_H  C_{H,T} T \cdot n^{-1},
  \end{align*}
  where we put 
  $\alpha_H=H(2H-1)$ and
  $C_{H,T}=\int_{-T}^{T}  \abs{t'}^{2H-2}dt'$.

  %{(b) can be proved as in the proof of Lemma A.1(e) of \cite{nourdin2016quantitative}.}
  
  \vspssm
  (c) The estimate is based on the Lemmas 6, 8 and 10 of \cite{tudor2020asymptotic}.
  At (52) of \cite{tudor2020asymptotic}, $a_{n,k}$ ($n\in\bbN,k\geq2$) is defined by
  \begin{align*}
    a_{n,k}=
    \frac1n\sum_{j\in[n]^k}
    \rho_H(j_1-j_2)\rho_H(j_2-j_3)\cdots\rho_H(j_{k-1}-j_k)\rho_H(j_k-j_1).
  \end{align*}
  We can observe
  $T^{2kH}\cdot n^{1-2kH}a_{n,k}=\betacycle{n}{k}$.
  For $k=2$, $a_{n,k}=O(1)$  by Lemma 6(a) in \cite{tudor2020asymptotic};
  For $k\geq3$, by Lemmas 8 and 10 in \cite{tudor2020asymptotic},
  $a_{n,k}$ is estimated as 
  \begin{align*}
    a_{n,k}&=
    \begin{cases}
      O(1)  \quad&\text{if } H< H(k) \\
      O(\log n)   \quad&\text{if } H= H(k) \\
      O(n^{(2H-1)k-1})   \quad&\text{if } H> H(k) 
    \end{cases}
  \end{align*}
  as $n\to\infty$.
  Hence we obtained the estimate of $\betacycle{n}{k}$.

  The estimate $\betacycle{n}{k}=O(n^{1-4H-(k-2)})$ follows from
  $1-2kH=1-4H-2H(k-2)\leq1-4H-(k-2)$ and 
  $-k<-k+3-4H=1-4H-(k-2)$ for $H\in(1/2,3/4)$ and $k\geq2$.
  Notice that we can remove the factor $\log n$ from the estimate.

  (d) Assume $1\leq k_1\leq k_2$.
  Since
  \begin{align*}
    \phi_H(k)=
    \begin{cases}
      - \frac{1}{2} + k(\frac{1}{2}-H)
      & \text{if}\; k\leq\frac1{2H-1}
      \quad\brbr{i.e.\ \frac{1}{2}<H\leq H(k)=\frac12(1+\frac{1}{k})}
      \vspssm\\
      - 1
      & \text{if}\; k\geq \frac1{2H-1}
      \quad\brbr{i.e.\ H(k)\leq H<\frac34}
    \end{cases} 
  \end{align*}
  for $k\geq2$, we have
  \begin{align*}
    \phi_H(2k_1) + \phi_H(2k_2)=
    \begin{cases}
      -1 + (k_1+k_2)(1-2H)
      & \text{if}\; \frac12<H\leq H(2k_2)=\frac12(1+\frac1{2k_2});
      \vspssm\\
      -\frac32 + k_1(1-2H)
      & \text{if}\; 
      H(2k_2)\leq H \leq H(2k_1)=\frac12(1+\frac1{2k_1});
      \vspssm\\
      -2
      & \text{if}\; H(2k_1)\leq H<\frac34.
    \end{cases} 
  \end{align*}
  On the other hand,
  \begin{align*}
    2\phi_H(k_1+k_2)=
    \begin{cases}
      - 1 + (k_1+k_2)(1-2H)
      & \text{if}\; \frac12<H\leq H(k_1+k_2) =\frac12(1+\frac1{(k_1+k_2)})
      \vspssm\\
      - 2
      & \text{if}\; H(k_1+k_2)\leq H<\frac34
    \end{cases} 
  \end{align*}
  Since $\frac1{2k_2}\leq\frac1{k_1+k_2}\leq\frac1{2k_1}$,
  we have $ H(2k_2)\leq H(k_1+k_2)\leq  H(2k_1)$, and
  we can observe that
  $\phi_H(2k_1) + \phi_H(2k_2) \geq 2\phi_H(k_1+k_2)$.
\end{proof}

For $n\in\bbN$ and a connected weighted graph $G=(V,\edgeWt,0)$
  with zero vertex-weight, where $0$ stands for a constant zero map on $V$,
  consider 
  \begin{align}\label{211006.1210}
    \bar\beta_n(G) %:=\bar\beta_{n,V}(\edgeWt) 
    :=\sum_{j\in[n]^V} \beta_{n}^{G}(j)
    =\sum_{j\in[n]^V} \beta_{n,V}(j,\edgeWt)
    = \sum_{j\in[n]^V} \prod _{[v,v']\in p(V)} 
    \beta_{n}(j_{v}, j_{v'})^{\edgeWt([v,v'])}.
    %\beta_{j_{v}, j_{v'}}^{\edgeWt_{v,v'}}.
  \end{align}
\begin{lemma}\label{211028.2340}
  %Assume $H\in(\frac{1}{2},\frac{3}{4})$. 
  %For $n\in\bbN$, a finite set $V$, a map $\edgeWt:p(V)\to\bbZ_{\geq0}$,
  For a connected weighted graph $G=(V,\edgeWt,\vertWt)$ with $\vertWt=0$,
  the following estimate holds, as $n$ tends to $\infty$:
  \begin{align*}
    \bar\beta_n(G)=O(n^{e(G)}),
  \end{align*}
  %$\bar\beta_{n,V}(\edgeWt) =O(n^{e(G)})$ 
  where $e(G)$ is the exponent defined in Section \ref{211006.0300}.
\end{lemma}
\noindent
Notice that for any connected weighted graph $G=(V,\edgeWt,\vertWt)$ with $\vertWt=0$,
we have 
\begin{align*}%\label{220322.2042}
  s(G)=2\brbr{\bartheta(G)-I(G)}+2, 
\end{align*}
and the conditions $\bartheta(G)=I(G)-1$ and $\bartheta(G)\geq I(G)$
are equivalent to $s(G)=0$ and $s(G)\geq2$, respectively.
The above estimate is written as 
\begin{subnumcases}{\bar\beta_n(G)=}%{\bar\beta_{n,V}(\edgeWt)=}
    O(n^{2-I(G)}) 
    & if $\bartheta(G)=I(G)-1$%\redb{$s(G)=0$}
    %\quad (i.e. $\edgeWt(G)=I(G)-1$).	
    \label{211006.1151}
    \\
    O(n^{2-I(G) +(1-4H)-2H(\bartheta(G)-I(G))}) 
    & if $\bartheta(G)\geq I(G)$%\redb{$s(G)\geq2$}
    %O(n^{2-I(G) +(1-4H)-2H(\edgeWt(G)-I(G))}) & if $\edgeWt(G)\geq I(G)$
    %\quad (i.e. $\edgeWt(G)\geq I(G)$). 
    \label{211006.1152}
  \end{subnumcases}
  For $v\in V$, we write
  $\edgeWt_{v}=\sum_{v'\in V, v'\neq v}\edgeWt([v,v'])$.

\begin{proof}[Proof of Lemma \ref{211028.2340}]
  Firstly, we consider the case $\bartheta(G)=I(G)-1$.
  Since we assumed that $G$ is connected, 
  the non-weighted graph $\frakg(G)$ is a tree 
  (i.e. a connected graph without cycle),
  and we have $\edgeWt_{v,v'}=1$  
  for $[v,v']\in p(V)$ with $\edgeWt_{v,v'}>0$.
  We prove the estimate by induction in $I(G)(\geq1)$.
  If $I(G)=1$, then 
    $\bar\beta_n(G)= \sum_{j\in[n]^V} 1 =n$ 
  and the estimate (\ref{211006.1151}) holds true.
  
  Let $d\geq2$ and suppose that the estimate (\ref{211006.1151}) holds 
  for any connected weighted graph $G'$ with zero vertex-weight
  such that $\bartheta(G')=I(G')-1$ and $I(G')\leq d-1$.
  Consider a connected weighted graph $G=(V,\edgeWt,0)$ %with no weight on vertices 
  such that 
  $\bartheta(G)=I(G)-1$ and $I(G)=d$.
  There is at least one vertex $v_0\in V$ such that
  $\edgeWt_{v_0}=1$
  and we denote by $v_1$ the vertex $v\in V$ such that $\edgeWt_{v_0,v}=1$.
  Indeed, if there is no such $v_0$, then by the connectedness of $G$,
  we have $\edgeWt_{v}\geq2$ for $v\in V$ and hence
  $2\,\bartheta(G)=\sum_{v\in V}\edgeWt_v\geq2\,I(G)$, 
  which contradicts to the assumption 
  $\bartheta(G)=I(G)-1$.
  Hence, $\bar\beta_n(G)$ is bounded as follows:
  \begin{align}\label{211006.1241}
    \bar\beta_n(G) &= 
    \sum_{j\in[n]^{V}}
      \Brbr{\prod _{[v,v']\in p(V\setminus \cbr{v_0})} 
      \beta_{n}\rbr{j_{v}, j_{v'}}^{\edgeWt_{v,v'}}}
      \beta_{n}(j_{v_0}, j_{v_1})
    \nn\\&=  
    \sum_{\substack{j\in[n]^{V\setminus \cbr{v_0}}}}
    \beta_{n,V\setminus \cbr{v_0}}\Brbr{j,\edgeWt|_{p(V\setminus \cbr{v_0})}}
    %\Brbr{\prod _{[v,v']\in p(V\setminus \cbr{v_0})} 
    %\beta_{j_{v}, j_{v'}}^{\edgeWt_{v,v'}}}
    \Brbr{\sum_{j'\in[n]} \beta_{n}(j', j_{v_1})}
    \nn\\&\leq 
    \sup_{j''\in[n]} \Brbr{\sum_{j'\in[n]}\beta_{n}\rbr{j', j''}}\;
    \bar\beta_n(G|_{V\setminus\cbr{v_0}}).
  \end{align}
  We can observe that $G|_{V\setminus\cbr{v_0}}$ is again a connected weighted graph with 
  $I(G|_{V\setminus\cbr{v_0}})=I(G)-1=d-1$, 
  $\bartheta(G|_{V\setminus\cbr{v_0}})
  =\bartheta(G)-\edgeWt_{v_0}=\bartheta(G)-1$ and hence
  $\bartheta(G|_{V\setminus\cbr{v_0}})=I(G|_{V\setminus\cbr{v_0}})-1$.
  We apply the assumption of the induction to 
  %$G|_{V\setminus\cbr{v_0}}
  %=(V\setminus\cbr{v_0},\edgeWt|_{p(V\setminus\cbr{v_0})},0)$
  $G|_{V\setminus\cbr{v_0}}$
  and use Lemma \ref{211006.1220}(b) to obtain
  \begin{align*}
    \bar\beta_n(G) =O(n^{-1}n^{2-I(G|_{V\setminus\cbr{v_0}})})=O(n^{2-I(G)}).
  \end{align*}
  The estimate (\ref{211006.1151}) is proved by induction.

  \vspsm
  Secondly, we consider the case $\bartheta(G)=I(G)$, 
  which implies $I(G)\geq2$.
  Again by induction in $I(G)$, we prove (\ref{211006.1152}), 
  which is the following estimate in this case:
  \begin{alignat}{2}\label{220112.2430}
    \bar\beta_n(G)&=O(n^{2-I(G)+(1-4H)})
    &&\quad\text{ as }n\to\infty.
  \end{alignat}
  If $I(G)=2$, then we can write
  $%\begin{align*}
    \bar\beta_n(G) = \sum_{j_1,j_2\in[n]} \beta_n\rbr{j_{1}, j_{2}}^{2},
  $ %\end{align*}
  and we have the estimate (\ref{220112.2430}), that is 
  $\bar\beta_n(G)=O(n^{1-4H})$, 
  by Lemma \ref{211006.1220}(c).
  Let $d\geq3$ and suppose that the estimate (\ref{220112.2430}) holds 
  for any connected weighted graph $G'$ with zero vertex-weight
  such that $\bartheta(G')=I(G')$ and $I(G')\leq d-1$.
  Take a connected weighted graph $G=(V,\edgeWt,0)$ 
  such that $\bartheta(G)=I(G)$ and $I(G)=d$.
  
  First we consider the case where there is no  $v\in V$ such that
  $\edgeWt_{v}=1$.
  By the connectedness of $G$,
  we have $\edgeWt_{v}\geq2$ for $v\in V$ and hence
  $2\,\bartheta(G)=\sum_{v\in V}\edgeWt_v\geq2I(G)$,
  which implies $\edgeWt_{v}=2$ for all $v\in V$.
  If there exists $[v_1,v_2]\in p(V)$ such that $\edgeWt([v_1,v_2])=2$, 
  then we have $\edgeWt([v_i,v])=0$
  for $i=1,2$ and $v\in V\setminus\cbr{v_1,v_2}\neq\emptyset$,
  which contradicts to the connectedness of $G$.
  Therefore, for each $v\in V$, there exist two vertices $v',v''\in V$ 
  ($v',v''\neq v$) such that
  $\edgeWt([v,v'])=\edgeWt([v,v''])=1$,
  and by the connectedness of $G$, 
  $G$ is a cycle graph in the sense of Definition \ref{220322.1900}.
  Hence we can write 
  \begin{align*}
    \bar\beta_{n}(G) = \sum_{j_1,..,j_{I(G)}\in[n]}
    \beta_n\rbr{j_1,j_2} \beta_n\rbr{j_2,j_3} \cdots 
    \beta_n\rbr{j_{I(G)-1},j_{I(G)}} \beta_n\rbr{j_{I(G)},j_1},
  \end{align*}
  and we have 
  $\bar\beta_n(G)=O(n^{1-4H-(I(G)-2)})$
  by Lemma \ref{211006.1220}(c),
  which is the estimate (\ref{220112.2430}) in this case.

  Consider the case where there exists a vertex $v_0\in V$ such that
  $\edgeWt_{v_0}=1$.
  %Writing $v_1$ for the vertex $v\in V$ such that $\edgeWt_{v_0,v}=1$,
  By the same argument as (\ref{211006.1241}), 
  we obtain
  %{but for the replacement of $G$, $v_0$ and $v_1$:}
  \begin{align*}
    \bar\beta_n(G) &\leq 
    \sup_{j''\in[n]} \Brbr{\sum_{j'\in[n]}\beta_n\rbr{j', j''}}
    \bar\beta_n(G|_{V\setminus\cbr{v_0}})
  \end{align*}
  We can observe that $G|_{V\setminus\cbr{v_0}}$ is a connected weighted graph with 
  $I(G|_{V\setminus\cbr{v_0}})=I(G)-1=d-1$, 
  $\bartheta(G|_{V\setminus\cbr{v_0}})=\bartheta(G)-1$ and hence
  $\bartheta(G|_{V\setminus\cbr{v_0}})=I(G|_{V\setminus\cbr{v_0}})$.
  We apply the assumption of the induction to 
  %$G|_{V\setminus\cbr{v_0}}
  %=(V\setminus\cbr{v_0},\edgeWt|_{p(V\setminus\cbr{v_0})},0)$
  $G|_{V\setminus\cbr{v_0}}$
  and use Lemma \ref{211006.1220}(b) to obtain
  \begin{align*}
    \bar\beta_n(G) =O(n^{-1}n^{2-I(G|_{V\setminus\cbr{v_0}})+(1-4H)})
    =O(n^{2-I(G)+(1-4H)}).%\label{220112.1522}
  \end{align*}
  %and the estimate (\ref{211006.1151}) is true by induction.
  Thus we have the estimate (\ref{220112.2430}) in the both cases
  of $I(G)=d$, and
  by induction it holds for any $G$ satisfying $\bartheta(G)=I(G)$.
  
  \vspsm
  Thirdly, consider the case $\bartheta(G) >I(G)$.
  %\contifrom{220112.1000}
  We can find a function
  $\edgeWt':p(V)\to\bbZ_{\geq0}$ such that
  $\edgeWt'([v,v'])\leq\edgeWt([v,v'])$ for each $[v,v']\in p(V)$ and
  the weighted graph $G'=(V,\edgeWt',0)$ is connected with 
  $\bartheta(G')=I(G')\brbr{=I(G)}$ as follows.
  There is a subset $E'\subset p(V)$ such that
  the graph $(V,E')$ is a spanning tree 
  of the graph $\frakg(G)$ since $G$ is connected.
  ($(V,E')$ is a spanning tree of a connected graph $(V,E)$,
  if $(V,E')$ is a subgraph of $(V,E)$ and a tree.) 
  We define another function $\edgeWt'':p(V)\to\bbZ_{\geq0}$ by
  \begin{align*}
    \edgeWt''([v,v'])=
    \begin{cases}
      1 & \text{if}\; [v,v']\in E'\\
      0 & \text{if}\; [v,v']\not\in E'.
    \end{cases}
  \end{align*}
  Since 
  $\sum_{[v,v']\in p(V)} \edgeWt([v,v'])=\bartheta(G)
  >I(G)-1=\sum_{[v,v']\in p(V)} \edgeWt''([v,v'])$
  and 
  $\edgeWt([v,v'])\geq\edgeWt''([v,v'])$ for $[v,v']\in p(V)$,
  there is $[v_1,v_2]\in p(V)$ such that
  $\edgeWt([v_1,v_2])\geq\edgeWt''([v_1,v_2])+1$.
  Then we define $\edgeWt':p(V)\to\bbZ_{\geq0}$ by 
  \begin{align*}
    \edgeWt'([v,v'])=
    \begin{cases}
      \edgeWt''([v,v']) +1  & \text{if}\; [v,v']= [v_1,v_2]\\
      \edgeWt''([v,v'])       & \text{if}\; [v,v']\not= [v_1,v_2],
    \end{cases}
  \end{align*}
  which satisfies the above condition.
  We write
  $\edgeWt'_{v,v'}=\edgeWt'([v,v'])$ for $[v,v']\in p(V)$.
  We have
  \begin{alignat*}{2}
    \bar\beta_n(G) &%=\sum_{j\in[n]^V} B_n(G;j)
    = \sum_{j\in[n]^V} 
    \prod_{[v,v']\in p(V)} \beta_n\rbr{j_{v}, j_{v'}}^{\edgeWt_{v,v'}}
    \\&= \sum_{j\in[n]^V} 
    \bbrbr{\prod_{[v,v']\in p(V)} 
    \beta_n\rbr{j_{v}, j_{v'}}^{\edgeWt_{v,v'}-\edgeWt'_{v,v'}}}
    \bbrbr{\prod_{[v,v']\in p(V)} \beta_n\rbr{j_{v}, j_{v'}}^{\edgeWt'_{v,v'}}}
    \\&\leq 
    \prod _{[v,v']\in p(V)} 
    \Brbr{\sup_{j_{1},j_{2}\in[n]}\beta_n\rbr{j_{1}, j_{2}}}^{\edgeWt_{v,v'}-\edgeWt'_{v,v'}}
    \sum_{j\in[n]^V} 
    \prod_{[v,v']\in p(V)} \beta_n\rbr{j_{v}, j_{v'}}^{\edgeWt'_{v,v'}}
    \\&=
      O(n^{-2H(\bartheta(G)-I(G))})\; \bar\beta_n(G')
      &&(\because \text{Lemma \ref{211006.1220}(a)})
    \\&=
      O(n^{ 2-I(G)+(1-4H)-2H(\bartheta(G)-I(G))}).
  \end{alignat*}
  We applied the estimate (\ref{220112.2430}) 
  to $G'$ since $\bartheta(G')=I(G')\brbr{=I(G)}$.
  Therefore, we obtain the estimate (\ref{211006.1152})
  for $G$ with $\bartheta(G) >I(G)$.
\end{proof}

\section{Derivation of the asymptotic expansion of realized volatility}
\label{220317.2030}

In this section, we first obtain a stochastic expansion of $Z_n$
and specify $u_n$ and $N_n$, 
to which we apply the general theory explained in Section \ref{220421.2330}.
In Section \ref{220220.2416}, we examine the functionals appearing repeatedly in
Section \ref{210430.1627} and \ref{220404.1900}.
We determine the random symbol appearing the asymptotic expansion in Section \ref{210430.1627}
and in Section \ref{220404.1900} we verify the condition [D] to prove Theorem \ref{220404.1630}.

By \cite{nualart2009malliavin}, 
the existence of the malliavin derivative of SDE (\ref{210430.1615}) of any order 
is proved under Assumption \ref{220404.1535},
and the derivatives can be written as solutions to linear SDE's.
The first-order derivative is written as follows:
\begin{align*}
  D_sX_t = 
  V^{[1]}(X_s) +
  \int^t_s V^{[2;1]}(X_\tpr) D_sX_\tpr\, dt' +
  \int^t_s V^{[1;1]}(X_\tpr) D_sX_\tpr\, dB_{t'} 
  \quad\tforsm t\geq s,
\end{align*}
and $D_sX_t=0$ for $t<s$,
where we denote the derivative of $V^{[k]}:\bbR\to\bbR$ by $V^{[k;1]}$ ($k=1,2$).
For the derivative of the higher order, see Proposition 7 in \cite{nualart2009malliavin}.
The estimates on the SDE and its derivatives used in the following discussion 
are summarized in Section \ref{220405.1158}.

We denote the $i$-th derivative of $V^{[k]}$ ($k=1,2$)
by $V^{[k;i]}$ and
abbreviate $V^{[k]}(X_t)$ and $V^{[k;i]}(X_t)$
as $V^{[k]}_t$ and $V^{[k;i]}_t$, respectively.
We write $a(x)=V^{[1]}(x)^2$ and $a_t=a(X_t)$ for brevity.
Recall that 
we write 
$\tj=Tj/n$
%$I_j=\clop{\tjm,\tj}\subset[0,T]$,
$1_j=1_{\clop{\tjm,\tj}}$
for $n\in\bbN$ and $j\in[n]=\cbr{1,...,n}$.

\subsection{Stochastic expansion}
%memo: $a_{t_{j-1}}=(V^{[1]}_{t_{j-1}})^2$.
Recall that we set 
\begin{align*}
  %%% for n\in\bbN
\bbV_n = n^{2H-1}\sum_{j=1}^n (\Delta_jX)^2,\quad
\bbV_\infty = T^{2H-1} \int_0^T (V^{[1]}_t)^2dt \tand
Z_n = %n^{2H-\half}
n^{1/2}\big(\bbV_n-\bbV_\infty).
\end{align*}
%$r_n \yeq n^{2H-\frac{3}{2}}$

\begin{lemma}\label{211007.1600}
The functional $Z_n$ has the following stochastic expansion
with $r_n = n^{2H-3/2}$:
\begin{align}
  %%% for n\in\bbN
  Z_n =  \delta(u_n) + r_n N_n,
  \label{220207.1041}
\end{align}
where
%%% the $\calh$-valued random variable $u_n$ for $n\in\bbN$
\begin{align}
  u_n 
  %{= n^{2H-1/2} \sum_{j\in[n]} (V^{[1]}_{\tjm})^2 I_1(1_j) 1_j}
  = n^{2H-1/2} \sum_{j\in[n]} a_{t_{j-1}}  I_1(1_j) 1_j
  \label{220404.1632}
\end{align}
%with $a_{t_{j-1}}=(V^{[1]}_{\tjm})^2$ 
and
\begin{align}
  N_n &=
    \sum_{j\in[n]} n\rbr{D_{1_j}a_{t_{j-1}}}  I_1(1_j)
    + 2T \sum_{j\in[n]} V^{[2]}_{\tjm} V^{[1]}_{\tjm} I_1(1_j) 
    + T^2 n^{-1} \sum_{j\in[n]} (V^{[2]}_{\tjm})^2  
    +N'_n
    \label{220404.1633}
\end{align}
with some functional $N'_n$ of 
$O_M(n^\alpha)$ with some $\alpha<0$.
  \end{lemma}

\begin{proof}%[Proof of Lemma \ref{211007.1600}]
We start with decomposing the increment of SDE (\ref{210430.1615})
using the change of variables formula for Young integral.
We write $V^{[(1;1),1]}_{t}=V^{[1;1]}_{t} V^{[1]}_{t}$ for $t\in[0,T]$.
For $0\leq\rze<\ron\leq T$, 
we have %\comm{almost surely}
\begin{align*}
  %X_s-X_t =&
  X_\ron-X_\rze =&
    \int_\rze^\ron V^{[2]}_t dt
    +\int_\rze^\ron V^{[1]}_t dB_t
  \\=&
    V^{[2]}_\rze (\ron-\rze)
    + \int_\rze^\ron \rbr{V^{[2]}_t - V^{[2]}_\rze}dt
  %\\&
    + V^{[1]}_\rze I_1(1_{[\rze,\ron]})
    + V^{[(1;1),1]}_{\rze}%V^{[1;1]}_{\rze} V^{[1]}_{\rze}
    \times \frac12\rbr{I_2(1_{[\rze,\ron]}^{\otimes 2}) + (\ron-\rze)^{2H}}
  \\&
    + \int_\rze^\ron \int_\rze^t 
    \rbr{V^{[(1;1),1]}_{\tpr}-V^{[(1;1),1]}_{\rze}}dB_{\tpr} dB_t
    %(V^{[1;1]}_{\tppr} V^{[1]}_{\tppr}-V^{[1;1]}_{\rze} V^{[1]}_{\rze}))dB_{\tppr} dB_\tpr
    + \int_\rze^\ron \int_\rze^t V^{[(1;1),2]}_{\tpr} d\tpr dB_t.
    %+ \int_\rze^\ron \int_\rze^\tpr V^{[1;1]}_{\tppr} V^{[2]}_{\tppr} d\tppr dB_\tpr.
\end{align*}
Thus, for $n\in\bbN$ and $j=1,...,n$, 
\begin{align}\label{211020.1120}
  \Delta_jX = X_\tj-X_\tjm = \sum_{k=1}^{5}T^{k}_{n,j},
\end{align}
where we set 
\begin{alignat}{3}
  T^{1}_{n,j} &= V^{[1]}_\tjm I_1(1_j),&\quad
  T^{2}_{n,j} &= V^{[2]}_\tjm (\tj-\tjm)
  = V^{[2]}_\tjm T n^{-1},
  \nn\\
  T^{3}_{n,j} &= \frac12 V^{[(1;1),1]}_{\tjm} I_2(1_j^{\otimes 2}),
  %T^{3}_{n,j} &= \frac12 V^{[1,1]}_{\tjm} V^{[1]}_{\tjm} \delta^2(1_{[\tjm,\tj]}^{\otimes 2}),
  & \quad
  T^{4}_{n,j} &= \frac12 V^{[(1;1),1]}_{\tjm} (\tj-\tjm)^{2H}
  = \frac12  V^{[(1;1),1]}_{\tjm} T^{2H} n^{-2H}
\label{220206.2332}
\end{alignat}
and $T^{5}_{n,j} = \sum_{k=1}^{3}R^{k}_{n,j}$
with 
\begin{align}
  R^{1}_{n,j} &= \int_\tjm^\tj \int_\tjm^t 
  (V^{[(1;1),1]}_{\tpr} -V^{[(1;1),1]}_{\tjm}) dB_{\tpr} dB_t,&%\\
  %(V^{[1,1]}_{\tppr} V^{[1]}_{\tppr}-V^{[1,1]}_{\tjm} V^{[1]}_{\tjm}) dB_{\tppr} dB_\tpr,&%\\
  R^{2}_{n,j} &= \int_\tjm^\tj \rbr{V^{[2]}_t-V^{[2]}_\tjm}dt,
  \nn\\
  R^{3}_{n,j} &= \int_\tjm^\tj \int_\tjm^t V^{[(1;1),2]}_{\tpr} d\tpr dB_t.
  %R^{3}_{n,j} &= \int_\tjm^\tj \int_\tjm^\tpr V^{[1,1]}_{\tppr}V^{[2]}_{\tppr} d\tppr dB_\tpr.
  \label{220206.2351}
\end{align}
By Proposition \ref{220413.1450} (ii),  for any $i\in\bbZ_{\geq0}$ and $p>1$,
\begin{align*}
  \sup_{n\in\bbN, j\in[n]}\bnorm{V^{[1]}_\tjm}_{i,p}, 
  \sup_{n\in\bbN, j\in[n]}\bnorm{V^{[2]}_\tjm}_{i,p}
  \tand
  \sup_{n\in\bbN, j\in[n]}\bnorm{V^{[(1;1),1]}_{\tjm}}_{i,p}
\end{align*}
are all finite. Hence we have 
\begin{align}
  \sup_{j\in[n]} \snorm{T^{1}_{n,j}}_{i,p}=O(n^{-H}),
  \quad\sup_{j\in[n]} \snorm{T^{2}_{n,j}}_{i,p}=O(n^{-1})
  \tand
  \sup_{j\in[n]} \snorm{T^{k}_{n,j}}_{i,p}=O(n^{-2H}) \quad (k=3,4)
  \label{220415.1630}
\end{align}
as $n\to\infty$.
 
\vspsm
The functional $Z_n$ is expanded as
\begin{align*}
  Z_n =& n^{1/2}\big(\bbV_n-\bbV_\infty) 
    = n^{1/2} \big( n^{2H-1} \sum_{j\in[n]} (\Delta_j X)^2 -  \bbV_\infty \big)
  \\=& 
    \bar M_n + \mR_{n}  %M^{\bf(gap)}_n 
  +\sum_{\substack{k_1,k_2=1,..,5 \\ (k_1,k_2)\neq(1,1)}}
  n^{2H-1/2} \sum_{j\in[n]} T^{k_1}_{n,j} T^{k_2}_{n,j},
  %\\&+& n^{2H-1/2} (\sum _{\# \in A} \bar M^{\#}_n 
  %+ \sum _{\# \notin A } \bar M^{\#}_n)
  %+ n^{2H-1/2}  \sum_j \{ 2R^{(\ref{0211120524})}_j \cdot T_j + (R^{(\ref{0211120524})}_j)^2 \}
\end{align*}
where
\begin{align}
  \bar M_n=&
  n^{2H-1/2}\sum_{j\in[n]}(V^{[1]}_{\tjm})^2 I_2(1_j ^{\otimes 2}) 
  =
  n^{2H-1/2}\sum_{j\in[n]} a_{\tjm} I_2(1_j ^{\otimes 2}) 
  %{=\delta(u_n) + r_n \cdot n \sum_{j\in[n]} I_1(1_j) D_{1_j}a_{\tjm}}
  \nn\\
  \mR_{n}=&
  %n^{2H-1/2}\sum_{j\in[n]}(V^{[1]}_{\tjm})^2 (T/n)^{2H}
  %-n^{1/2}T^{2H-1}\int^T_0 (V^{[1]}_{t})^2 dt
  %\\=&
  n^{1/2}\bbrbr{n^{2H-1}\sum_{j\in[n]}(V^{[1]}_{\tjm})^2 (T/n)^{2H} 
  -\bbV_\infty}
  =
  n^{1/2}\:T^{2H-1}\int^T_0\rbr{(V^{[1]}_{t_{j(t)-1}})^2-(V^{[1]}_{t})^2}dt.
  \label{210421.2254}
\end{align}
%where we denote $V^{[1]}(\cdot)^2$ by $a(\cdot)$.
%

First we consider 
\begin{align*}
  r_n^{-1}n^{2H-1/2} \sum_{j\in[n]} T^{k_1}_{n,j} T^{k_2}_{n,j}
  =n \sum_{j\in[n]} T^{k_1}_{n,j} T^{k_2}_{n,j}
\end{align*}
for $k_1,k_2=1,...,5$ such that $(k_1,k_2)\neq(1,1)$.
If at least one of $k_1$ and $k_2$ is $5$,
then by the estimate (\ref{220415.1630}) and Lemma \ref{210504.2200} (ii),
we have
\begin{align*}
  n\sum_{j\in[n]} T^{k_1}_{n,j}\, T^{k_2}_{n,j} 
  =O_M(n^{2+(-1-\beta)-H})
  =O_M(n^{1-\beta-H})
\end{align*}
for any $\beta\in(1/2,H)$.
Consider the cases where $(k_1,k_2)=(1,3),(3,1)$.
By the product formula,
\begin{align*}
  n\sum_{j\in[n]} T^{1}_{n,j}T^{3}_{n,j} 
  =& 
  n\sum_{j\in[n]} \frac12  V^{[(1;1),1*2]}_{\tjm}
    I_1(1_j)\ I_2(1_j^{\otimes 2})
  \\=&
  \frac{1}{2} n\sum_{j\in[n]} V^{[(1;1),1*2]}_{\tjm}
    I_3(1_j^{\otimes 3})
  +
  T^{2H}n^{1-2H}\sum_{j\in[n]} V^{[(1;1),1*2]}_{\tjm} I_1(1_j) 
  \\
  =& 
  \frac12\:
  \cali_n(1,(\cbr{1},0,3),\brbr{(V^{[(1;1),1*2]}_{\tjm})_{j\in[n]}}_n,\bbone)
  +T^{2H}\:
  \cali_n(1-2H,(\cbr{1},0,1),\brbr{(V^{[(1;1),1*2]}_{\tjm})_{j\in[n]}}_n,\bbone),
\end{align*}
where we write
$V^{[(1;1),1*2]}_{\tjm}=V^{[1;1]}_{\tjm}\brbr{V^{[1]}_\tjm}^2$.
The corresponding exponents are
\begin{align*}
  e(1,(\cbr{1},0,3))=&1+((2-1)-1+(1/2-2H)-H)=3/2-3H
  \\
  e(1-2H,(\cbr{1},0,1))=&1-2H+((2-1)-1)=1-2H,
\end{align*}
and by Proposition \ref{210823.2400} we obtain
\begin{align*}
  n\sum_{j\in[n]}T^{1}_{n,j}T^{3}_{n,j} = O_M(n^{1-2H}).
\end{align*}
Similary, we have 
\begin{align*}
  n\sum_{j\in[n]}T^{1}_{n,j}T^{2}_{n,j} 
  =& O_M(1),
  &
  n\sum_{j\in[n]}T^{1}_{n,j}T^{4}_{n,j} 
  =& O_M(n^{1-2H}),
  &
  n\sum_{j\in[n]}T^{2}_{n,j}T^{2}_{n,j} 
  =& O_M(1),
  \\
  n\sum_{j\in[n]}T^{2}_{n,j}T^{3}_{n,j} 
  =& O_M(n^{1/2-2H}),
  &
  n\sum_{j\in[n]}T^{2}_{n,j}T^{4}_{n,j} 
  =& O_M(n^{1-2H}),
  &
  n\sum_{j\in[n]}T^{3}_{n,j}T^{3}_{n,j} 
  =& O_M(n^{2-4H}),
  \\
  n\sum_{j\in[n]}T^{3}_{n,j}T^{4}_{n,j} 
  =& O_M(n^{3/2-4H}),
  &
  n\sum_{j\in[n]}T^{4}_{n,j}T^{4}_{n,j} 
  =& O_M(n^{2-4H}).
\end{align*}

\vspssm
%\noindent
The functional $\bar M_n$ is written as 
$\bar M_n=\delta(u_n) + M_{0,n}$
with 
\begin{align*}
  M_{0,n} 
  %{r_n N_{0,n}}
  %{\dot N_n,r_n N^{(0)}_n}
  =n^{2H-1/2} \sum_{j\in[n]} \rbr{D_{1_j} a_{t_{j-1}}} I_1(1_j),
\end{align*}
which is of $O_M(r_n)$ by an argument using the exponent.
We also have
$\mR_{n} =O_M(n^{-1/2})$
by Lemma \ref{210504.2151}, and hence
$r_n^{-1}\mR_{n} =O_M(n^{1-2H})$.
Thus, by setting
\begin{align}
  N'_n=&r_n^{-1}\mR_{n}+
  \sum_{\substack{(k_1,k_2)\neq\\(1,1),(1,2),(2,1),(2,2)}}
      n\sum_{j\in[n]} T^{k_1}_{n,j} T^{k_2}_{n,j}
  \label{220220.2100}
  \\
  N_n=&
  r_n^{-1}M_{0,n}
  %n \sum_{j\in[n]}\rbr{D_{1_j}a_{\tjm}}\:I_1(1_j)
  +2n\sum_{j\in[n]} T^{1}_{n,j} T^{2}_{n,j}
  +n\sum_{j\in[n]} \rbr{T^{2}_{n,j}}^2
  +N'_n,
  \label{220220.2101}
\end{align}
we have $N'_{n} =O_M(n^{1-\beta-H})$ with $1-\beta-H<0$
and obtain the stochastic expansion (\ref{220207.1041}).
\end{proof}
\begin{remark}\label{220228.1345}
  Notice that 
  $N'_n=O_M(n^{1-2H+\epsilon})$ for any $\epsilon>0$.
  By decomposing $\Delta_jX$ further, we can show
  $N'_n=O_M(n^{1-2H})$.
\end{remark}

\subsection{Functionals and their exponents}
\label{220220.2416}
We will investigate the functionals %and their exponents 
appearing when decomposing 
$D_{u_n}M_n$, $D_{u_n}^2\bar M_n$, $N_n$ and $D_{u_n}N_n$.
Recall we defined
\begin{align*}
  u_n =\;&
   n^{2H-1/2} \sum_{j\in[n]} a_{t_{j-1}}  I_1(1_j) 1_j
  \\
  M_n=\;&\delta(u_n)
  =n^{2H-1/2}\sum_{j\in[n]}a_\tjm I_2(1_j^{\otimes2}) 
  -n^{2H-1/2}\sum_{j\in[n]} \rbr{D_{1_j} a_{t_{j-1}}} I_1(1_j)
  \\
  \bar M_n=\;& n^{2H-1/2}\sum_{j\in[n]}a_\tjm I_2(1_j ^{\otimes 2})
  \\
  \cmpns=\;&n^{2H-1/2} \sum_{j\in[n]} \rbr{D_{1_j}a_{\tjm}}I_1(1_j).
\end{align*}
\begin{itemize}
  \item [\bf (i)]  $D_{u_n}M_n$
\end{itemize}
$D_{u_n}M_n$ decomposes as follows
  \begin{align}
    D_{u_n}M_n=& D_{u_n}\Brbr{\bar M_n-\cmpns}
    =
    D_{u_n}\Brbr{
    n^{2H-1/2} \sum_{j\in[n]} a_{t_{j-1}}  I_2(1_j^{\otimes2})}
    -D_{u_n}\Brbr{n^{2H-1/2}\sum_{j\in[n]} \rbr{D_{1_j} a_{t_{j-1}}} I_1(1_j)}
    \nn\\=&
    \cali^{M(2,0)}_{1,n}+\cali^{M(2,0)}_{2,n}
    -(\cali^{M(2,0)}_{3,n}+\cali^{M(2,0)}_{4,n}),
    \label{220221.1811}
  \end{align}
  where
  \begin{align*}
    \cali^{M(2,0)}_{1,n}
    =& 2 n^{4H-1} \sum_{j\in[n]^2} 
    a_{t_{j_1-1}} a_{t_{j_2-1}} 
    I(1_{j_1})I(1_{j_2})\beta_{j_1,j_2}
    \\
    \cali^{M(2,0)}_{2,n}
    =& n^{4H-1} \sum_{j\in[n]^2} 
    \rbr{D_{1_{j_2}} a_{t_{j_1-1}}} a_{t_{j_2-1}}
    I_2(1_{j_1}^{\otimes2})I(1_{j_2})
    \\
    \cali^{M(2,0)}_{3,n}
    =& n^{4H-1}\sum_{j\in[n]^2} 
    \rbr{D_{1_{j_2}} D_{1_{j_1}} a_{t_{j_1-1}}}a_{t_{j_2-1}} 
    I(1_{j_1})I(1_{j_2})
    \\
    \cali^{M(2,0)}_{4,n}
    =& n^{4H-1}\sum_{j\in[n]^2} 
    \rbr{D_{1_{j_1}}a_{t_{j_1-1}}} a_{t_{j_2-1}}
    I(1_{j_2})\beta_{j_1,j_2}.
  \end{align*} 
  The functionals above can be written as
  \begin{align*}
    \cali^{M(2,0)}_{1,n}
    =& 
    2\:\cali_n(4H-1,G^{M(2,0)}_{1},A^{M(2,0)}_{1},\bbone),
    &A^{M(2,0)}_{1,n}(j)=&a_{t_{j_1-1}} a_{t_{j_2-1}}
    \\
    \cali^{M(2,0)}_{2,n}
    =& 
    \cali_n(4H-2,G^{M(2,0)}_{2},A^{M(2,0)}_{2},\bbone),
    &A^{M(2,0)}_{2,n}(j)=&n\rbr{D_{1_{j_2}} a_{t_{j_1-1}}} a_{t_{j_2-1}}
    \\
    \cali^{M(2,0)}_{3,n}
    =&
    \cali_n(4H-3,G^{M(2,0)}_{3},A^{M(2,0)}_{3},\bbone),
    &A^{M(2,0)}_{3,n}(j)=&n^2\rbr{D_{1_{j_2}} D_{1_{j_1}} a_{t_{j_1-1}}} a_{t_{j_2-1}}
    \\
    \cali^{M(2,0)}_{4,n}
    =&
    \cali_n(4H-2,G^{M(2,0)}_{4},A^{M(2,0)}_{4},\bbone),
    &A^{M(2,0)}_{4,n}(j)=&n\rbr{D_{1_{j_1}}a_{t_{j_1-1}}} a_{t_{j_2-1}},
  \end{align*} 
  where the weighted graphs $G^{M(2,0)}_{k}$ ($k=1,...,4$) are expressed as below.
%
%\noindent
Hence the corresponding exponents are calculated as 
\begin{align}
  e(4H-1,G^{M(2,0)}_{1})&= 4H-1 + (2-2+2(1/2-2H))=0
  \nn\\
  e(4H-2,G^{M(2,0)}_{2})&= 4H-2 + (2-1-1+(1/2-2H)) +(2-1-1)=2H-3/2
  \label{220220.2422}
  \\
  e(4H-3,G^{M(2,0)}_{3})&= 4H-3 + 2(2-1-1)=4H-3
  \label{220220.2423}
  \\
  e(4H-2,G^{M(2,0)}_{4})&= 4H-2 + (2-2-1)=4H-3.
  \label{220220.2424}
\end{align} 

Furthermore, the functional $\cali^{M(2,0)}_{1,n}$ can be decomposed as 
$%\begin{align*}
  \cali^{M(2,0)}_{1,n}
  =\cali^{M(2,0)}_{1,1,n}+\cali^{M(2,0)}_{1,2,n}
$ %\end{align*}
with 
\begin{align}
  \cali^{M(2,0)}_{1,1,n}
  =& 2 n^{4H-1} \sum_{j\in[n]^2} 
  a_{t_{j_1-1}} a_{t_{j_2-1}} 
  \beta_{j_1,j_2}^2
  &=&2\,\cali_n(4H-1,G^{M(2,0)}_{1,1},A^{M(2,0)}_{1},\bbone)
  \label{220301.1029}\\
  \cali^{M(2,0)}_{1,2,n}
  =& 2 n^{4H-1} \sum_{j\in[n]^2} 
  a_{t_{j_1-1}} a_{t_{j_2-1}} 
  \delta^2(1_{j_1}\otimes1_{j_2})\beta_{j_1,j_2}
  %I_2(1_{j_1}\tilde\otimes1_{j_2})\beta_{j_1,j_2}
  &=&
  2\,\cali_n(4H-1,G^{M(2,0)}_{1},
  \bcbr{G^{M(2,0)}_{1}},
  A^{M(2,0)}_{1},\bbone),
  \nn
\end{align} 
and the corresponding exponents are calculated as
\begin{align}
  e(4H-1,G^{M(2,0)}_{1,1})&= 4H-1 + (2-2+2(1/2-2H))=0
  \nn\\
  e_2(4H-1,G^{M(2,0)}_{1},\bcbr{G^{M(2,0)}_{1}})&= 
  4H-1 +(1/2-4H)\vee(-2)
  %+ \bcbr{2-2-1+\brbr{-1/2+(2-4H)\vee(-1/2)}}
  =(-1/2)\vee(4H-3).
  \label{220220.2425}
\end{align} 
\captionsetup[subfigure]{labelformat=empty}
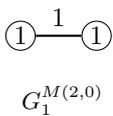
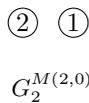
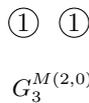
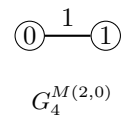
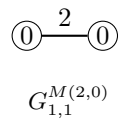
\begin{figure}[H]
    \centering
  \begin{subfigure}[t]{0.18\textwidth}
      \centering
      \graphMtwon{}{}
      \caption{$G^{M(2,0)}_{1}$}
  \end{subfigure}
  \begin{subfigure}[t]{0.18\textwidth}
    \centering
    \graphMtwtw{}{}
    \caption{$G^{M(2,0)}_{2}$}
  \end{subfigure}
  \begin{subfigure}[t]{0.18\textwidth}
    \centering
    \graphMtwth{}{}
    \caption{$G^{M(2,0)}_{3}$}
  \end{subfigure}
  \begin{subfigure}[t]{0.18\textwidth}
    \centering
    \graphMtwfo{}{}
    \caption{$G^{M(2,0)}_{4}$}
  \end{subfigure}
  \begin{subfigure}[t]{0.18\textwidth}
    \centering
    \graphMtwonon{}{}
    \caption{$G^{M(2,0)}_{1,1}$}
  \end{subfigure}
\end{figure}

\begin{itemize}
  \item [\bf (ii)]  $D_{u_n}^2\bar M_n$
\end{itemize}
$D_{u_n}^2\bar M_n$ can be decomposed as follows
\begin{align}
  D_{u_n}^2\bar M_n=&
  D_{u_n}\rbr{\cali^{M(2,0)}_{1,n}+\cali^{M(2,0)}_{2,n}}
  =%\\=&
  \rbr{2\:\cali^{M(3,0)}_{1,n}+2\:\cali^{M(3,0)}_{2,n}}
  +\rbr{\cali^{M(3,0)}_{3,n}+\cali^{M(3,0)}_{4,n}
  +\cali^{M(3,0)}_{1,n}+\cali^{M(3,0)}_{5,n}}
  \nn\\=&
  3\:\cali^{M(3,0)}_{1,n}+2\:\cali^{M(3,0)}_{2,n}
  +\cali^{M(3,0)}_{3,n}+\cali^{M(3,0)}_{4,n}+\cali^{M(3,0)}_{5,n},
  \label{220221.1812}
\end{align}
where
\begin{align*}
  \cali^{M(3,0)}_{1,n}=&
  2n^{6H-3/2} \sum_{j\in[n]^3}
  \rbr{D_{1_\jth}a_{t_{\jon-1}}} a_{t_{\jtw-1}}a_{t_{\jth-1}}
  I(1_\jon)I(1_\jtw)I(1_\jth)\beta_{\jon,\jtw}
  \\
  \cali^{M(3,0)}_{2,n}=&
  2n^{6H-3/2} \sum_{j\in[n]^3}
  a_{t_{\jon-1}}a_{t_{\jtw-1}}a_{t_{\jth-1}}
  I(1_\jtw)I(1_\jth)\beta_{\jon,\jtw}\beta_{\jon,\jth}
  \\
  \cali^{M(3,0)}_{3,n}=&
  n^{6H-3/2} \sum_{j\in[n]^3}
  \rbr{D_{1_\jth}D_{1_\jtw} a_{t_{\jon-1}}} a_{t_{\jtw-1}} a_{t_{\jth-1}}
  I^2(1_\jon^{\otimes 2}) I(1_\jtw)I(1_\jth)
  \\
  \cali^{M(3,0)}_{4,n}=&
  n^{6H-3/2} \sum_{j\in[n]^3}
  \rbr{D_{1_\jtw}a_{t_{\jon-1}}} \rbr{D_{1_\jth}a_{t_{\jtw-1}}} a_{t_{\jth-1}}
  I^2(1_\jon^{\otimes 2})I(1_\jtw)I(1_\jth)
  \\
  \cali^{M(3,0)}_{5,n}=&
  n^{6H-3/2} \sum_{j\in[n]^3}
  \rbr{D_{1_\jtw}a_{t_{\jon-1}}} a_{t_{\jtw-1}} a_{t_{\jth-1}}
  I^2(1_\jon^{\otimes 2})I(1_\jth)\beta_{\jtw,\jth}.
\end{align*} 
%{$3$ and $4$ can be made one functional}
The functionals above can be written as
\begin{align}
  \cali^{M(3,0)}_{1,n}=&
  2\,\cali_n(6H-5/2,G^{M(3,0)}_{1},A^{M(3,0)}_{1},\bbone),
  &A^{M(3,0)}_{1,n}(j)=&\rbr{nD_{1_\jth}a_{t_{\jon-1}}} a_{t_{\jtw-1}}a_{t_{\jth-1}}
  \label{220222.2004}\\
  \cali^{M(3,0)}_{2,n}=&
  2\,\cali_n(6H-3/2,G^{M(3,0)}_{2},A^{M(3,0)}_{2},\bbone),
  &A^{M(3,0)}_{2,n}(j)=&a_{t_{\jon-1}}a_{t_{\jtw-1}}a_{t_{\jth-1}}
  \nn\\
  \cali^{M(3,0)}_{3,n}=&
  \cali_n(6H-7/2,G^{M(3,0)}_{3},A^{M(3,0)}_{3},\bbone),
  &A^{M(3,0)}_{3,n}(j)=&
  n^2\rbr{D_{1_\jth}D_{1_\jtw} a_{t_{\jon-1}}} a_{t_{\jtw-1}} a_{t_{\jth-1}}
  \label{220222.2003}\\
  \cali^{M(3,0)}_{4,n}=&
  \cali_n(6H-7/2,G^{M(3,0)}_{3},A^{M(3,0)}_{4},\bbone),
  &A^{M(3,0)}_{4,n}(j)=&
  n^2\rbr{D_{1_\jtw}a_{t_{\jon-1}}} \rbr{D_{1_\jth}a_{t_{\jtw-1}}} a_{t_{\jth-1}}
  \nn\\
  \cali^{M(3,0)}_{5,n}=&
  \cali_n(6H-5/2,G^{M(3,0)}_{5},A^{M(3,0)}_{1},\bbone),
  &A^{M(3,0)}_{5,n}(j)=&
  n\rbr{D_{1_\jtw}a_{t_{\jon-1}}} a_{t_{\jtw-1}} a_{t_{\jth-1}}
  \nn
\end{align} 
where the weighted graphs $G^{M(3,0)}_{k}$ ($k=1,2,3,5$) are expressed as below.
Hence the corresponding exponents are calculated as 
\begin{align}
  e(6H-5/2,G^{M(3,0)}_{1})&= 6H-5/2 + (2-2+2(1/2-2H))+(2-1-1)=2H-3/2
  \nn\\
  e(6H-3/2,G^{M(3,0)}_{2})&= 6H-3/2 + (2-3+2(1/2-2H))=2H-3/2
  \label{220330.1242}\\
  %e_2(6H-3/2,G^{M(3,0)}_{2},\emptyset)&= 6H-3/2 + (1-6H)\vee(-3)
  %=(-1/2)\vee(6H-9/2)
  %\label{220227.1601}\\
  e(6H-7/2,G^{M(3,0)}_{3})&= 6H-7/2 + (2-1-1+(1/2-2H)) + 2(2-1-1)=4H-3
  \label{220222.1954}\\
  e(6H-5/2,G^{M(3,0)}_{5})&= 6H-5/2 + (2-1-1+(1/2-2H)) + (2-2-1)=4H-3
  \label{220227.1600}
\end{align} 

For $\cali^{M(3,0)}_{2,n}$, we can write
\begin{align}
  \cali^{M(3,0)}_{2,n}=&
  2\,\funcSec(6H-3/2,G^{M(3,0)}_{2},
  \emptyset, A^{M(3,0)}_{2}, \bbone),
  \nn
\end{align} 
and we have
\begin{align}
  e_2(6H-3/2,G^{M(3,0)}_{2},\emptyset)&= 6H-3/2 + (1-6H)\vee(-3)
  =(-1/2)\vee(6H-9/2).
  \label{220227.1601}
\end{align} 
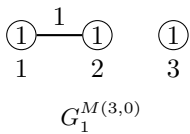
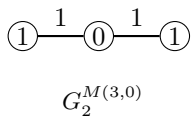
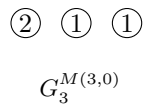
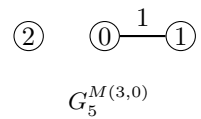
\begin{figure}[H]
  \centering
\begin{subfigure}[t]{0.23\textwidth}
    \centering
    \graphMthon{1}{2}{3}
    \caption{$G^{M(3,0)}_{1}$}
\end{subfigure}
\begin{subfigure}[t]{0.23\textwidth}
  \centering
  \graphMthtw{}{}{}
  \caption{$G^{M(3,0)}_{2}$}
\end{subfigure}
\begin{subfigure}[t]{0.23\textwidth}
  \centering
  \graphMthth{}{}{}
  \caption{$G^{M(3,0)}_{3}$}
\end{subfigure}
\begin{subfigure}[t]{0.23\textwidth}
  \centering
  \graphMthfi{}{}{}
  \caption{$G^{M(3,0)}_{5}$}
\end{subfigure}
\end{figure}

\begin{itemize}
  \item [\bf (iii)] $N_n$, $D_{u_n}N_n$
\end{itemize}
The functional $N_n$ defined by (\ref{220220.2101}) is decomposed as 
$N_n=\pertur{0}+\pertur{1}+\pertur{2}+N'_n$
with
\begin{align*}
  \pertur{0}=&\sum_{j\in[n]} \rbr{nD_{1_j} a_{t_{j-1}}} I_1(1_j),
  &%\\
  \pertur{1}=&\sum_{j\in[n]}\rbr{2\,T\: V^{[2]}_{\tjm} V^{[1]}_{\tjm}} I_1(1_j),
  &%\\
  \pertur{2}=&T^2 n^{-1} \sum_{j\in[n]} (V^{[2]}_{\tjm})^2
\end{align*}
and $N'_n$ defined by (\ref{220220.2100}).
Note that $r_nN_{0,n}=M_{0,n}$.

The functionals above can be written as
\begin{align*}
  \cali^{N(1,0)}_{1,n}:=\pertur{0}+\pertur{1}=&
  \cali_n(0,G^{N(1,0)}_{1},A^{N(1,0)}_{1},\bbone),
  &A^{N(1,0)}_{1,n}(j)=&
  nD_{1_j} a_{t_{j-1}}+2\,T\: V^{[2]}_{\tjm} V^{[1]}_{\tjm}
  %\label{220222.2051}
  \\
  %\cali^{N(1,0)}_{2,n}:=
  \pertur{2}=&
  \cali_n(-1,G^{N(1,0)}_{2},A^{N(1,0)}_{2},\bbone),
  &A^{N(1,0)}_{2,n}(j)=&T^2\,(V^{[2]}_{\tjm})^2
  %\label{220222.2052}
\end{align*} 
with $G^{N(1,0)}_{1}=\singlegraph{1}{1}$ 
and $G^{N(1,0)}_{2}=\singlegraph{1}{0}$.
Recall that we denote by $\singlegraph{v}{k}$
a weighted graph $(V,\edgeWt,\vertWt)$ 
such that $V=\cbr{v}$ with some $v$ and $\vertWt(v)=k$.
The exponents are 
\begin{align}
  e(0,G^{N(1,0)}_{1})&= 0 + (2-1-1)=0,&
  e(-1,G^{N(1,0)}_{2})&= -1 + (2-1-0)=0.
  \label{220222.2100}
\end{align} 

Since the only component in $G^{N(1,0)}_{1}$ belongs to 
$\Csone(G^{N(1,0)}_{1})$, and hence
$\Cpearg{1}(G^{N(1,0)}_{1})$ is empty,
by Proposition \ref{210822.1800},
there exist finite sets $\Lambda^{N(2,0)}_{1}$ and $V''$,
$\alpha^\lambda\in\bbR$,
a weighted graph 
$G^\lambda=(V'',\edgeWt^\lambda,\vertWt^\lambda)$, 
$A^\lambda\in\cala(V'')$ and $\bbf^\lambda\in\calf(V'')$
for each $\lambda\in\Lambda^{N(2,0)}_{1}$,
such that
$D_{u_n}\cali^{N(1,0)}_{1,n}$ can be written as
\begin{align} 
  D_{u_n}\cali^{N(1,0)}_{1,n} =&
  \sum_{\lambda\in\Lambda^{N(2,0)}_{1}} 
  \cali_n(\alpha^{\lambda}, G^{\lambda}, A^{\lambda}, \bbf^\lambda),%\bbone),
  \label{220218.1201}
\end{align}
%$(\alpha^\lambda, G^\lambda, A^\lambda, \bbf^\lambda)_{\lambda\in\Lambda}$
and
\begin{align}
  \max_{\lambda\in\Lambda^{N(2,0)}_{1}} e(\alpha^\lambda, G^\lambda)
  = e(0,G^{N(1,0)}_{1}) +2H-3/2 =2H-3/2.
  \label{220218.1202}
\end{align}

  The functional $D_{u_n}N_{2,n}$ can be written as
  \begin{align}
    D_{u_n}N_{2,n}
    =&
    %D_{u_n}\Brbr{ T^2 n^{-1} \sum_{j\in[n]} (V^{[2]}_{\tjm})^2}
    %\\=&
    T^2 n^{2H-3/2}\sum_{j\in[n]^2} 
    D_{1_\jtw}\brbr{(V^{[2]}_{t_{\jon-1}})^2} a_{t_{\jtw-1}}I_1(1_\jtw)
    =%\\=&
    \cali_n(2H-5/2,G^{N(2,0)}_{2},A^{N(2,0)}_{2},\bbone),
    \label{220222.2111}
  \end{align}
where
$A^{N(2,0)}_{2,n}(j)=
nD_{1_\jtw}\brbr{(V^{[2]}_{t_{\jon-1}})^2} a_{t_{\jtw-1}}$ and 
$G^{N(2,0)}_{2}=\singlegraph{1}{0}\vee\singlegraph{2}{1}$. % is expressed by
The exponent $e(2H-5/2,G^{N(2,0)}_{2})$ is 
\begin{align}
  e(2H-5/2,G^{N(2,0)}_{2})
  =2H-5/2+(2-1-0)+(2-1-1)=2H-3/2.
  \label{220222.2112}
\end{align}

%\newpage
\subsection{Random symbols}\label{210430.1627}
Here we consider the random symbols
$(\mT_n,\mT)= (\mS_n^{(3,0)},\mS^{(3,0)})$,
$(\mS_{0,n}^{(2,0)},\mS_0^{(2,0)})$,
$(\mS_n^{(1,0)},\mS^{(1,0)})$ and
$(\mS_{1,n}^{(2,0)},\mS_1^{(2,0)})$
in Condition {\bf [D]} (iii).
We calculate the limit random symbol, which appear in the asymptotic expansion formula, and prove 
the convergence $\mT_n\to\mT$ in $L^p$.
%$\mS^{(3)}$, $\mS_0^{(2)}$, $\mS^{(1)}$ and $\mS_1^{(2)}$
%
%\subsubsection{Notations}
%indep. of $T$
We introduce the following definitions.
\begin{alignat*}{2}
  %\rho_H(k) &:= \frac12 \rbr{\abs{k+1}^{2H} +\abs{k-1}^{2H} -2\abs{k}^{2H}}
  %=\abr{1_{[0,1]}, 1_{[k,k+1]}}_\calh
  %\\
  %c_H^2 &=\sum_{k\in\bbZ} \rho_H(k)^2 
  %\\
  %\beta_n\rbr{j_1,j_2}
  %&:= \abr{1_{\jon}, 1_\jtw}_\calh= T^{2H} n^{-2H} \rho_H(\jon-\jtw)
  %&&\tfornsp \jon,\jtw\in[n]\nn
%\\
\rho_{\tau}(s)&=\alpha_H\, T\, \abs{s-\tau}^{2H-2}
&&\tfornsp s,\tau\in[0,T]
\nn\\
\dotd(\tau;t)&=\int_{[0,T]}D_sX_t\; \rho_{\tau}(s) ds
&&\tfornsp \tau,t\in[0,T]
\nn\\
\ddotd(\tau;t)&=
\int_{[0,T]^2}D^2_{\son,\stw} X_{t}\; \rho_{\tau}(\son)\rho_{\tau}(\stw)d\son d\stw
&\quad&\tfornsp \tau,t\in[0,T]\nn
\end{alignat*}
with $\alpha_H=H(2H-1)$.
Recall we have defined
\begin{alignat*}{2}
  \rho_H(k) &= \frac12 \rbr{\abs{k+1}^{2H} +\abs{k-1}^{2H} -2\abs{k}^{2H}}
  =\abr{1_{[0,1]}, 1_{[k,k+1]}}_\calh
  \nn\\
  c_H^2 &
    =\sum_{k\in\bbZ} \rho_H(k)^2 
    %\text{ where } \rho_H(k) = 1/2(|k+1|^{2H}-2|k|^{2H}+|k-1|^{2H}).
  \\
  \beta_n\rbr{j_1,j_2}=\beta_{j_1,j_2}%=\beta_{j_1,j_2;T}
  &:= \abr{1_{\jon}, 1_\jtw}_\calh
  = T^{2H} n^{-2H} \rho_H(\jon-\jtw)
  &&\tfornsp \jon,\jtw\in[n]\nn
\end{alignat*}
and 
$a(x)=(V^{[1]}(x))^2$.
Let us write
$g(x)=2c_H^2\; T^{4H-1}\; (V^{[1]}(x))^4$
and we can write
$G_\infty=\int^T_0g(X_t)dt$.
%$g'(x)=2c_H^2\cdot T^{4H-1}\cdot4(V^{[1]}(x))^3\cdot V^{[1,1]}(x)$
%Set $j(t)=[n\cdot t/T]+1$, namely
%{For $t\in\clop{0,T}$, we denote $j\in[n]$ such that $t\in\clop{\tjm,\tj}$ by $j_n(t)$ or $j(t)$.
%By convention, we set $j(T)=n$.}

%%\newpage
%\subsubsection{Quasi-torsion}
\begin{itemize}
  \item [\bf (i)] \bf The random symbol $\mS_n^{(3,0)}$ (Quasi-torsion)
\end{itemize}
First we treat the random symbol %$\mS_n^{(3\redcomm{,0})}[(\sfi\sfz)^3]$
$\mS_n^{(3,0)}$ and its weak limit $\mS^{(3,0)}$.
Recall that the random symbol $\mS_n^{(3,0)}$ is written as 
\begin{align*}
  \mS_n^{(3,0)}(\sfi\sfz) 
  = \frac13 \qtor_n[\sfi\sfz]^3
  = \frac13 r_n^{-1} D^2_{u_n[\sfi\sfz]}M_n[\sfi\sfz]
  = \frac13 r_n^{-1} D^2_{u_n}M_n\:[\sfi\sfz]^3.
\end{align*}
%\begin{align*}
%  \qtor_n[\sfi\sfz]^3
%  = 3 \mS_n^{(3,0)}(\sfi\sfz) 
%  = r_n^{-1} D^2_{u_n[\sfi\sfz]}M_n[\sfi\sfz]
%  = r_n^{-1} D^2_{u_n}M_n\:[\sfi\sfz]^3.
%\end{align*}
By the decomposition (\ref{220221.1811}) and (\ref{220221.1812}),
we have
\begin{align*}
  D_{u_n}^2M_n=&
  D_{u_n}\rbr{\cali^{M(2,0)}_{1,n}+\cali^{M(2,0)}_{2,n}
    -(\cali^{M(2,0)}_{3,n}+\cali^{M(2,0)}_{4,n})}
  \\=&
  3\:\cali^{M(3,0)}_{1,n}+2\:\cali^{M(3,0)}_{2,n}
  +\cali^{M(3,0)}_{3,n}+\cali^{M(3,0)}_{4,n}+\cali^{M(3,0)}_{5,n}
  -D_{u_n}\rbr{\cali^{M(2,0)}_{3,n}+\cali^{M(2,0)}_{4,n}}
  %\\D_{u_n}\cmpns=&\cali^{M(2,0)}_{3,n}+\cali^{M(2,0)}_{4,n},
\end{align*}
By the exponents (\ref{220222.1954}), (\ref{220227.1600}) 
and Proposition \ref{210407.1401},
we have 
 $\cali^{M(3,0)}_{3,n}$, $\cali^{M(3,0)}_{4,n}$, $\cali^{M(3,0)}_{5,n}$
$=O_{L^p}(n^{4H-3})=o_{L^p}(r_n)$ for $p>1$.
Similarly, by (\ref{220227.1601}) and Proposition \ref{210407.1402},
we have
$\cali^{M(3,0)}_{2,n}
=O_{L^p}(n^{(-1/2)\vee(6H-9/2)})=o_{L^p}(r_n)$.
By the exponents (\ref{220220.2423}) and (\ref{220220.2424}),
using Proposition \ref{210823.2400} and 
$\bnorm{u_{n}}_{p}=O(1)$ $(p>1)$ from (\ref{220222.1937}),
we have 
$D_{u_n}\cali^{M(2,0)}_{3,n}$,
$D_{u_n}\cali^{M(2,0)}_{4,n}$
$=O_{L^p}(n^{4H-3})=o_{L^p}(r_n)$ for $p>1$.

The functional $\cali^{M(3,0)}_{1,n}$ can be decomposed further as
\begin{align*}
  \cali^{M(3,0)}_{1,n}=
  \cali^{M(3,0)}_{1,1,n}+\cali^{M(3,0)}_{1,2,n},
\end{align*}
where
\begin{align*}
  %%% for n\in\bbN
  \cali^{M(3,0)}_{1,1,n}=&
  2n^{6H-3/2} \sum_{j\in[n]^3}
  \rbr{D_{1_\jth}a_{t_{\jon-1}}} a_{t_{\jtw-1}}a_{t_{\jth-1}}
  I(1_\jth)\beta_{\jon,\jtw}^2
  \\
  \cali^{M(3,0)}_{1,2,n}=&
  2n^{6H-3/2} \sum_{j\in[n]^3}
  \rbr{D_{1_\jth}a_{t_{\jon-1}}} a_{t_{\jtw-1}}a_{t_{\jth-1}}
  I(1_\jth)
  \delta^2(1_\jon\otimes1_\jtw)
  %I_2(1_\jon\tilde\otimes1_\jtw)
  \beta_{\jon,\jtw}.
\end{align*}
The functional $\cali^{M(3,0)}_{1,2,n}$ can be written as
\begin{align*}
  \cali^{M(3,0)}_{1,2,n}=&
  2\,\cali_n(6H-5/2,G^{M(3,0)}_{1},
  \bcbr{G^{M(3,0)}_{1}|_{\cbr{1,2}}},
  A^{M(3,0)}_{1}, \bbone),
\end{align*}
and the corresponding exponent is 
\begin{align*}
  e_2(6H-5/2,G^{M(3,0)}_{1},\bcbr{G^{M(3,0)}_{1}|_{\cbr{1,2}}})
  =&6H-5/2 + (2-2-1+\phi_H(4)) + (2-1-1)
  %\\=&6H-5/2 + (-1+(3/2-4H)\vee (-1))
  \\=&(2H-2)\vee (6H-9/2).
\end{align*}
Hence, by (\ref{220227.1601}) and Proposition \ref{210407.1402},
we have
$\cali^{M(3,0)}_{1,2,n}
=O_{L^p}(n^{(2H-2)\vee (6H-9/2)})=o_{L^p}(r_n)$. 

\vspssm
Consider 
$\cali^{M(3,0)}_{1,1,n}$.
The action of the random symbol $\cali^{M(3,0)}_{1,1,n}[\tti\sfz]^3$ 
to $\Psi(\sfz)$ under the expectation is decomposed as
\begin{align*}
  E\sbr{\Psi(\sfz) \;\cali^{M(3,0)}_{1,1,n} [\tti\sfz]^{3}}
  =E\sbr{\Psi(\sfz) \;\bar\cali^{M(3,0)}_{1,1,n} (\tti\sfz)},
  %\quad{\tfornsp \sfz\in\bbR\tandsm n\in\bbN}
\end{align*}
where 
$\bar\cali^{M(3,0)}_{1,1,n} (\tti\sfz)
=r_n\rbr{\bar{\mS}_{1,n}^{(3,0)} [\tti\sfz]^5
+\bar{\mS}_{2,n}^{(3,0)} [\tti\sfz]^3}$
and
\begin{align}
  \bar{\mS}_{1,n}^{(3,0)}&= 
  2n^{4H-1} \sum_{j\in[n]^3}
  2^{-1} \rbr{D_{1_{j_3}} G_\infty} A^{M(3,0)}_{1,n}(j) \:\beta_{j_1,j_2}^2
  \label{220403.2145}
  \\
  \bar{\mS}_{2,n}^{(3,0)}&=
  2n^{4H-1} \sum_{j\in[n]^3} 
  \brbr{D_{1_{j_3}} A^{M(3,0)}_{1,n}(j)} \:\beta_{j_1,j_2}^2,
  \label{220403.2146}
\end{align}
where $A^{M(3,0)}_{1,n}$ was defined at (\ref{220222.2004}).
Hence, we define 
\begin{align*}
  \bar{\mS}_n^{(3,0)}(\sfi\sfz) 
  =& r_n^{-1}\bar\cali^{M(3,0)}_{1,1,n} (\tti\sfz)
  +\mR_n^{(3,0)}[\sfi\sfz]^3
  = \bar{\mS}_{1,n}^{(3,0)} [\tti\sfz]^5
  +\bar{\mS}_{2,n}^{(3,0)} [\tti\sfz]^3
  +\mR_n^{(3,0)}[\sfi\sfz]^3
\end{align*}
with
\begin{align*}
  \mR_n^{(3,0)}
  =3^{-1} r_n^{-1}
  \Brbr{3\:\cali^{M(3,0)}_{1,2,n}
  +2\:\cali^{M(3,0)}_{2,n}
  +\sum_{k=3,4,5}\cali^{M(3,0)}_{k,n}
  -D_{u_n}\Brbr{\sum_{k=3,4}\cali^{M(2,0)}_{k,n}}}.
\end{align*}

Define the functional ${\mS}^{(3,0)}_{k,\infty}$ ($k=1,2$) by
%{Constant may be wrong}
%\grnb{Check the constant\tto seems okay}
\begin{alignat*}{3}
  {\mS}^{(3,0)}_{1,\infty}&=
  \frac1{4T} \int_0^T \bbrbr{\int_0^T g'(X_t) \dotd(\tau,t) dt}^2
  a_{\tau}d\tau
  \\
  {\mS}^{(3,0)}_{2,\infty}&=
  2 T^{4H-2}c_H^2 \bigg\{
  \int_0^T\bbrbr{ \int_0^T 
  \rbr{a''(X_{t})\;\dotd(\tau,{t})^2 +a'(X_{t})\;\ddotd(\tau,{t})} 
   a_{t} dt} a_{\tau} d\tau
  \nn\\&\hspace*{75pt}+
  %2 T^{4H-2}c_H^2 
  \int_0^T \rbr{\int_0^T \rbr{a'(X_{t}) \dotd(\tau,t)} a_{t}dt }
  \rbr{a'(X_{\tau})\dotd(\tau,\tau)} d\tau
  +%\nn\\&+
  %2 T^{4H-2}c_H^2 
  \int_0^T \rbr{\int_0^T \rbr{a'(X_{t})\dotd(\tau,{t})}^2 dt}a_{\tau} d\tau
  \bigg\}
\end{alignat*}
and let 
\begin{align}
  \mS^{(3,0)}(\tti\sfz) =
  {\mS}^{(3,0)}_{1,\infty}[\tti\sfz]^5 +{\mS}^{(3,0)}_{2,\infty}[\sfi\sfz]^3.
  \label{220308.2520}
\end{align}

\begin{proposition}\label{211010.2330}
  For any $p>1$,
  \begin{align*}
    \bar{\mS}_n^{(3,0)}(\sfi\sfz)\to 
    \mS^{(3,0)}(\tti\sfz)
  \end{align*}
  in $L^p$ as $n\to\infty$ in the sense that
  the coefficient random variable of each degree of $\bar{\mS}_n^{(3,0)}$
  converges to the counterpart of $\mS^{(3,0)}$ in $L^p$.
\end{proposition}

\begin{proof}
For $k=1,2$,
$\bar{\mS}^{(1,0)}_{k,n}$ converges almost surely to
${\mS}^{(1,0)}_{k,\infty}$ by Lemma \ref{210505.1440}.
By arguments using the exponent, we have
%r_n^{-1}\bar\cali^{M(3,0)}_{1,1,k,n}
$\bar{\mS}_{k,n}^{(3,0)}
=O_{L^p}(1)$ 
for any $p>1$.
Thus we have the convergence 
%r_n^{-1}\bar\cali^{M(3,0)}_{1,1,k,n} 
$\bar{\mS}_{k,n}^{(3,0)}
\to\mS^{(3,0)}_{k,\infty}$
in $L^p$ for any $p>1$.
Since $\mR_n^{(3,0)}=o_{L^p}(1)$, 
we obtain the result.
\end{proof}

%%\newpage
%\subsubsection{Quasi-tangent}
\begin{itemize}
  \item [\bf (ii)] \bf The random symbol $\mS_n^{(2,0)}$ (Quasi-tangent)
\end{itemize}
We treat the random symbol 
$\mS_{0,n}^{(2,0)}$ and its weak limit
$\mS_0^{(2,0)}$.
Recall that the random symbol $\mS_{0,n}^{(2,0)}$ and 
the quasi tangent are written as 
\begin{align*}
  \qtan_n[\tti\sfz]^2 
  = 2 \mS_{0,n}^{(2,0)}(\tti\sfz) 
  = r_n^{-1} \rbr{D_{u_n[\sfi\sfz]}M_n[\tti\sfz]-G_\infty[\tti\sfz]^2}
  = r_n^{-1} \rbr{D_{u_n}M_n-G_\infty}\:[\tti\sfz]^2.
\end{align*}
We decompose $D_{u_n}M_n-G_\infty$ into six terms as follows:
\begin{align*}
  G_n^{(2,0)}
  &=D_{u_n}M_n-G_\infty
  \\&=
  \Brbr{\cali^{M(2,0)}_{1,1,n}-G_\infty}
  +\cali^{M(2,0)}_{1,2,n}
  +\cali^{M(2,0)}_{2,n}
  -(\cali^{M(2,0)}_{3,n}+\cali^{M(2,0)}_{4,n}),
  %\\&=
  %{
  %n^{4H-1}  \abr{\sum_{j_1\in[n]}
  %D\rbr{a_{t_{j_1-1}} I^2(1_{j_1}^{\otimes2}) -D_{1_{j_1}} a_{t_{j_1-1}} I(1_{j_1})},
  %\sum_{j_2\in[n]} a_{t_{j_2-1}} I(1_{j_2}) 1_{j_2}}
  %- 2 c_H^2 T^{4H-1} \int^T_0 a_t^2 dt}
\end{align*}
By the exponents calculated at 
(\ref{220220.2423}), (\ref{220220.2424}) and (\ref{220220.2425}),
we have
$\cali^{M(2,0)}_{k,n} =O_{L^p}(n^{4H-3})=o_{L^p}(r_n)$ for $k=3,4$ and
$\cali^{M(2,0)}_{1,2,n} =\bar O_{L^p}(n^{(-1/2)\vee(4H-3)})=o_{L^p}(r_n)$  for $p>1$.
We also have
$\cali^{M(2,0)}_{1,1,n}-G_\infty= o_{L^{\infty-}}(n^{2H - \frac{3}{2}})$
by Lemma \ref{220309.2200}.

Consider 
$\cali^{M(2,0)}_{2,n}$.
The action of the random symbol $\cali^{M(2,0)}_{2,n}[\tti\sfz]^2$ 
to $\Psi(\sfz)$ under the expectation is decomposed as

\begin{align*}
  %%% for \sfz\in\bbR, n\in\bbN
  E \sbr{\Psi(\sfz)\;\cali^{M(2,0)}_{2,n}[\tti\sfz]^2}
  &=
    n^{4H-2} \sum_{j\in[n]^2} E\sbr{\abr{ 
      D\brbr{\Psi(\sfz)\: A^{M(2,0)}_{2,n}(j)
      I(1_{j_2})},
      I(1_{j_1})1_{j_1}
     }} [\tti\sfz]^{2}
  \\&=
    E\sbr{\Psi(\sfz)\; \bar\cali^{M(2,0)}_{2,n}  (\tti\sfz)},
    %\bcbr{ \bar\cali^{M(2,0)}_{2,1,n}  [\tti\sfz]^{4}
    %  + \bar\cali^{M(2,0)}_{2,2,n} [\tti\sfz]^{2}
    %  + \bar\cali^{M(2,0)}_{2,3,n} [\tti\sfz]^{2}}
\end{align*}
where
$\bar\cali^{M(2,0)}_{2,n}  (\tti\sfz)
=\bar\cali^{M(2,0)}_{2,1,n}  [\tti\sfz]^{4}
+ \bar\cali^{M(2,0)}_{2,2,n} [\tti\sfz]^{2}
+ \bar\cali^{M(2,0)}_{2,3,n} [\tti\sfz]^{2}$ with
%$\Psi(\sfz) = \exp(2^{-1}G_\infty [(\tti\sfz)^{\otimes2}])$ 
\begin{alignat*}{2}
  \bar\cali^{M(2,0)}_{2,1,n} &=
    n^{4H-2} \sum_{j\in[n]^2}
      %\big(D_{1_{j_1}} \Psi(\sfz) \big)
    \rbr{2^{-1} D_{1_{j_1}} G_\infty}\:
    A^{M(2,0)}_{2,n}(j)
    I(1_{j_2})I(1_{j_1})
    &&=\cali_n(4H-3,G^{M(2,0)}_{2,1},A^{M(2,0)}_{2,1},\bbone)
  \\
  \bar\cali^{M(2,0)}_{2,2,n} &=
    n^{4H-2} \sum_{j\in[n]^2}
    \rbr{D_{1_{j_1}} A^{M(2,0)}_{2,n}(j)}\:
    I(1_{j_2})I(1_{j_1})
    &&=\cali_n(4H-3,G^{M(2,0)}_{2,1},A^{M(2,0)}_{2,2},\bbone)
  \\
  \bar\cali^{M(2,0)}_{2,3,n} &=
    n^{4H-2} \sum_{j\in[n]^2}A^{M(2,0)}_{2,n}(j)
    I(1_{j_1})\beta_{j_1,j_2}
    &&=\cali_n(4H-2,G^{M(2,0)}_{2,3},A^{M(2,0)}_{2},\bbone),
  \end{alignat*}
  where we define 
  $A^{M(2,0)}_{2,1,n}(j)
  =\rbr{2^{-1} nD_{1_{j_1}} G_\infty}\:A^{M(2,0)}_{2,n}(j)$ and
  $A^{M(2,0)}_{2,2,n}(j)=nD_{1_{j_1}} A^{M(2,0)}_{2,n}(j)$,
  and the weighted graphs are as follows.
\begin{figure}[H]
    \centering
  \begin{subfigure}[t]{0.18\textwidth}
    \centering
    \graphMtwth{}{}
    \caption{$G^{M(2,0)}_{2,1}$}%$G^{M(2,0)}_{3}$
  \end{subfigure}
  \begin{subfigure}[t]{0.18\textwidth}
    \centering
    \graphMtwtwth{}{}
    \caption{$G^{M(2,0)}_{2,3}$}%$G^{M(2,0)}_{4}$
  \end{subfigure}
\end{figure}
The corresponding exponent of the above functionals are
\begin{align*}
  e(4H-3,G^{M(2,0)}_{2,1})
  &=4H-3+2(2-1-1)=4H-3\\
  e(4H-2,G^{M(2,0)}_{2,3})
  &=4H-2+(2-2-1)=4H-3
\end{align*}
and, we have
$\bnorm{\bar\cali^{M(2,0)}_{2,k,n}}_{L^p} = O(n^{4H-3}) = o(r_n)$
for $k=1,2,3$ and $p>1$ by Proposition \ref{210407.1401}.

Hence, we define 
\begin{align*}
  \bar{\mS}_{0,n}^{(2,0)}(\sfi\sfz) =& 
  2^{-1} r_n^{-1} \rbr{
  \bar\cali^{M(2,0)}_{2,n} (\tti\sfz)
  +\bcbr{\brbr{\cali^{M(2,0)}_{1,1,n}-G_\infty}
  +\cali^{M(2,0)}_{1,2,n}
  -(\cali^{M(2,0)}_{3,n}+\cali^{M(2,0)}_{4,n})}[\tti\sfz]^2},
\end{align*}
and $\bar{\mS}_{0,n}^{(2,0)}$ converges to 
$\mS_{0}^{(2,0)}=0$ in $L^p$ as $n\to\infty$.

%\subsubsection{The random symbols $\mS^{(1,0)}$ and $\mS^{(2,0)}_1$
%derived from the perturbation term $N_n$}
\begin{itemize}
  \item [\bf (iii)] \bf The random symbols $\mS^{(1,0)}$ 
\end{itemize}

Recall
\begin{align*}
  \mS^{(1,0)}_n(\tti\sfz) = N_n[\tti\sfz]
\end{align*}
and 
$N_n =\cali^{N(1,0)}_{1,n} +N_{2,n} +N'_n$
with $N'_n=o_M(1)$.
The action of the random symbol $\cali^{N(1,0)}_{1,n}[\tti\sfz]$
to $\Psi(\sfz)$ under the expectation is written as
\begin{align*}
  E\sbr{\Psi(\sfz) \cali^{N(1,0)}_{1,n} [\tti\sfz]}
  &= E \sbr{\Psi(\sfz) \bar\cali^{N(1,0)}_{1,n} (\tti\sfz)},
\end{align*}
with
\begin{align*}
  \bar\cali^{N(1,0)}_{1,n}(\tti\sfz)=&\;
  %\bar\cali^{N(1,0)}_{1,1,n}&=
  n^{-1}\sum_{j\in[n]} \rbr{2^{-1}nD_{1_j} G_\infty}
  \rbr{nD_{1_j} a_{t_{j-1}} + 2\,T\; V^{[2,1]}_{\tjm}}
  [\tti\sfz]^3
  \\
  %\bar\cali^{N(1,0)}_{1,2,n}&=
  &+n^{-1}\sum_{j\in[n]}
  \rbr{n^2D_{1_j}D_{1_j}a_{t_{j-1}} +2T\;nD_{1_j}V^{[2,1]}_{\tjm}}
  [\tti\sfz],
\end{align*}
where we write $V^{[2,1]}(x)=V^{[2]}(x) V^{[1]}(x)$.
We define 
\begin{align}
  \bar{\mS}^{(1,0)}_{1,n}=&\;
  n^{-1}\sum_{j\in[n]} \rbr{2^{-1}nD_{1_j} G_\infty}
  \rbr{nD_{1_j} a_{t_{j-1}} + 2\,T\; V^{[2,1]}_{\tjm}}
  \label{220403.2151}
  \\
  \bar{\mS}^{(1,0)}_{2,n}=&\;
  n^{-1}\sum_{j\in[n]}
  \rbr{n^2D_{1_j}D_{1_j}a_{t_{j-1}} +2T\;nD_{1_j}V^{[2,1]}_{\tjm}}
  +N_{2,n}
  \label{220403.2152}
\end{align}
and
\begin{align*}
  \bar{\mS}_n^{(1,0)}(\sfi\sfz) 
  =\bar{\mS}^{(1,0)}_{1,n}[\tti\sfz]^3
  +\bar{\mS}^{(1,0)}_{2,n}[\sfi\sfz]
  +N'_n[\sfi\sfz].
\end{align*}
Define the functional ${\mS}^{(1,0)}_{k,\infty}$ ($k=1,2$) by
\begin{align}
  {\mS}^{(1,0)}_{1,\infty}&=
  \int_0^T \rbr{\int^T_0 g'(X_t)\; \dotd(\tau,t)dt}
  \rbr{2^{-1}T^{-1} a'(X_{\tau})\; \dotd(\tau,\tau) +V^{[2,1]}_{\tau}} d\tau
  \label{220403.2131}
  \\
  {\mS}^{(1,0)}_{2,\infty}&=
  \int_0^T \cbr{
  T^{-1}\rbr{a''(X_{\tau})\;\brbr{\dotd(\tau,\tau)}^2 
  +a'(X_{\tau})\;\ddotd(\tau,\tau) }
  +2V^{[(2,1);1]}_{\tau}\;\dotd(\tau,\tau) 
  +T\;V^{[2,2]}_{\tau}}d\tau,
  \label{220403.2132}
\end{align}
where $V^{[(2,1);1]}(x)
=V^{[2;1]}(x) V^{[1]}(x)+V^{[2]}(x) V^{[1;1]}(x)$,
and let 
\begin{align}
  \mS^{(1,0)}(\tti\sfz) =
  {\mS}^{(1,0)}_{1,\infty}[\tti\sfz]^3
  +{\mS}^{(1,0)}_{2,\infty}[\sfi\sfz].
  \label{220308.2521}
\end{align}

\begin{proposition} 
  For any $p>1$,
  \begin{align*}
    \bar{\mS}_n^{(1,0)}(\sfi\sfz)\to 
    \mS^{(1,0)}(\tti\sfz)
  \end{align*}
  in $L^p$ as $n\to\infty$ in the sense that
  the coefficient random variable of each degree of 
  $\bar{\mS}_n^{(1,0)}$
  converges to the counterpart of $\mS^{(1,0)}$ in $L^p$.
\end{proposition}%%MODIF%%

\begin{proof}
For $k=1,2$,
$\bar{\mS}^{(1,0)}_{k,n}$ converges almost surely to
${\mS}^{(1,0)}_{k,\infty}$ by Lemma \ref{210505.1440}.
By arguments using the exponent,
we have 
$\sup_n \snorm{\bar{\mS}^{(1,0)}_{k,n}}_{p}<\infty$ for any $p>1$,
and the convergences 
$\bar{\mS}^{(1,0)}_{k,n} \to{\mS}^{(1,0)}_{k,\infty}$
in $L^p$ follows.
The term  $N'_n$ is of $o_M(1)$ and negligible.
Thus, we obtain the result.
\end{proof}

%\subsubsection*{The limit of $D_{u_n} N_n$}
%For $D_{u_n}N_n$,
%\vspsm
\begin{itemize}
  \item [\bf (iv)] \bf The random symbols $\mS^{(2,0)}_1$
\end{itemize}
Recall $\mS^{(2,0)}_{1,n} (\tti\sfz) = D_{u_n}N_n [(\tti\sfz)^{2}]$.
By (\ref{220218.1201}) and (\ref{220218.1202}),
we have
$\bnorm{ D_{u_n}\cali^{N(1,0)}_{1,n}}_{p} =O(n^{2H-3/2}) =o(1)$.
By (\ref{220222.2111}) and (\ref{220222.2112}),
$\bnorm{ D_{u_n}N_{2,n}}_{p} =O(n^{2H-3/2}) =o(1)$.
Since
$N'_n$ is of $o_M(1)$ and
$\bnorm{\snorm{u_n}_{\calh}}_{p}=O(1)$ for any $p>1$ by (\ref{211019.1935}),
we also have 
$\bnorm{ D_{u_n}N'_n}_{p} =o(1)$.
Hence we obtain
$\snorm{ D_{u_n}N_n }_{p} = o(1)$
and hence $\mS_1^{(2,0)}=0$.

%%\newpage
\subsection{Proof of Theorem \ref{220404.1630}}
We will verify the conditions in {\bf [D]}.
\label{220404.1900}
%Note that when we don't consider reference variable, 
%we don't need the conditions in Section \ref{210430.2121}.

%\subsubsection{Condition (\ref{220215.1241}) of [D](ii)}
\begin{itemize}
  \item [{\bf [D]}] \bf (ii) (\ref{220215.1241})
\end{itemize}
Recall the representation of $u_n$:
\begin{align*}
  u_n=\rbr{u_n(s)}_{s\in[0,T]}
  =\bbrbr{n^{2H-1/2}\sum_{j\in[n]} a_\tjm I_1(1_j)1_j(s)}_{s\in[0,T]}.
\end{align*}
For $i\in\bbZ_{\geq0}$, by the Leibniz rule,  we have
\begin{align*}
  D^{i}u_n&=
  \rbr{D^{i}_{s_1,..,s_i}u_n(s_0)}_{s_0,s_1,...,s_i\in[0,T]}
  %\\&=
  %\rbr{n^{2H-1/2}\sum_{j\in[n]} \cbr{
  %  \sum_{k=1}^i \rbr{D^{i-1}_{s_1,...,\hat s_k,...,s_i} a_\tjm} 1_j(s_k) 1_j(s_0)
  %  +\rbr{D^{i}_{s_1,...,s_i} a_\tjm} I_1(1_j) 1_j(s_0)
  %}}_{s_0,s_1,...,s_i\in[0,T]}
  %\\&
  =\sum_{k=0}^i \cali^{i,k}_{n},
\end{align*}
with $\calh^{\otimes i+1}$-valued random variables 
$\cali^{i,k}_{n}$ ($k=0,...,i$) defined by 
\begin{align*}
  \cali^{i,0}_{n}&=\bbrbr{
  n^{2H-1/2}\sum_{j\in[n]} \rbr{D^{i}_{s_1,..,s_i} a_\tjm} I^1(1_j) 1_j(s_0)
  }_{s_0,s_1,...,s_i\in[0,T]}
  \\
  \cali^{i,k}_{n}&=\bbrbr{
  n^{2H-1/2}\sum_{j\in[n]} \rbr{D^{i-1}_{s_1,..,\hat s_k,..,s_i} a_\tjm} 1_j(s_k) 1_j(s_0)
  }_{s_0,s_1,...,s_i\in[0,T]}
  \tforsm k=1,...,i
\end{align*}

For $k=1,..,i$,
\begin{align}
  \snorm{\cali^{i,k}_{n}}_{\calh^{\otimes i+1}}^2
  &=
  n^{4H-1}\sum_{j\in[n]^2}
  \abr{D^{i-1} a_{t_{\jon-1}},D^{i-1} a_{t_{\jtw-1}}}_{\calh^{\otimes i-1}}
  \beta_{\jon,\jtw}^2
  =\cali_n(4H-1,G^{(\ref{220220.2351})},A^{i},\bbone)
  \label{220220.2351}
\end{align}
with
$A^{i}_n(j)=
\abr{D^{i-1} a_{t_{\jon-1}},D^{i-1} a_{t_{\jtw-1}}}_{\calh^{\otimes i-1}}$
and $G^{(\ref{220220.2351})}$ as below.
Since the corresponding exponent of this term is 
$e(4H-1,G^{(\ref{220220.2351})})=(4H-1)+(1-4H)=0$,
we have
\begin{align*}
  \snorm{\snorm{ \cali^{i,k}_{n} }_{\calh^{\otimes i+1}} }_{p}
  =\snorm{\snorm{ \cali^{i,k}_{n} }_{\calh^{\otimes i+1}}^2 }_{{p/2}}^{1/2}
  =O(1)
\end{align*}
for $p\geq2$ by Proposition \ref{210407.1401}.
Similary,
\begin{align}
  \snorm{\cali^{i,0}_{n}}_{\calh^{\otimes i+1}}^2
  &=
  n^{4H-1}\sum_{j\in[n]^2}
  \abr{D^{i} a_{t_{\jon-1}},D^{i} a_{t_{\jtw-1}}}_{\calh^{\otimes i}}
  I^1(1_\jon) I^1(1_\jtw)\beta_{\jon,\jtw}
  =\cali_n(4H-1,G^{(\ref{220220.2352})},A^{i,0},\bbone)
  \label{220220.2352}
\end{align}
with
$A^{i,0}_n(j)=
\abr{D^{i} a_{t_{\jon-1}},D^{i} a_{t_{\jtw-1}}}_{\calh^{\otimes i}}$
and $G^{(\ref{220220.2352})}$ as below.
The corresponding exponent is 
$e(4H-1,G^{(\ref{220220.2352})})=(4H-1)+(1-4H)=0$,
and we have 
$\snorm{\snorm{\cali^{i,0}_{n}}_{\calh^{\otimes i+1}} }_{p} 
%\bbnorm{\bnorm{ \cali_{n;0} }_{\calh^{\otimes i+1}}^2 }_{L^{p/2}}^{1/2}
=O(1)$ for $p\geq2$ by Proposition \ref{210407.1401}.
%By the triangle inequality of the norms $\snorm{ \cdot }_{\calh^{\otimes i+1}}$ and $\snorm{ \cdot }_{L^p}$,
Hence, we have
\begin{align}
  \Bnorm{ \bnorm{ D^{i}u_n }_{\calh^{\otimes i+1}} }_{p}
  =O(1).
  \label{211019.1935}
\end{align}
Therefore, we obtain 
\begin{align}
  \snorm{ u_n }_{i,p} = O(1)
  \label{220222.1937}
\end{align}
for any $i\in\bbZ_+$ and $p>1$.

\begin{figure}[H]
  \centering
  \begin{subfigure}[t]{0.18\textwidth}
    \centering
    \graphMtwonon{}{}
    \caption{$G^{(\ref{220220.2351})}$}%{$G^{M(2,0)}_{1,1}$}
  \end{subfigure}
  \begin{subfigure}[t]{0.18\textwidth}
    \centering
    \graphMtwon{}{}
    \caption{$G^{(\ref{220220.2352})}$}%{$G^{M(2,0)}_{1}$}
  \end{subfigure}
\end{figure}

%\subsubsection{Condition (3.19) of [C](ii)}
\begin{itemize}
  \item [{\bf [D]}] \bf (ii) (\ref{220215.1242})
\end{itemize}
%%WRITE CLEARLY%%
Using the definitions in Subsection \ref{220220.2416},
$G^{(2)}_n$ is decomposed as
\begin{align*}
  G_n^{(2)}
  &=D_{u_n}M_n-G_\infty
  \\&=
  \Brbr{\cali^{M(2,0)}_{1,1,n}-G_\infty}
  +\cali^{M(2,0)}_{1,2,n}
  +\cali^{M(2,0)}_{2,n}
  -(\cali^{M(2,0)}_{3,n}+\cali^{M(2,0)}_{4,n}),
  %\\&=
  %{
  %n^{4H-1}  \abr{\sum_{j_1\in[n]}
  %D\rbr{a_{t_{j_1-1}} I^2(1_{j_1}^{\otimes2}) -D_{1_{j_1}} a_{t_{j_1-1}} I(1_{j_1})},
  %\sum_{j_2\in[n]} a_{t_{j_2-1}} I(1_{j_2}) 1_{j_2}}
  %- 2 c_H^2 T^{4H-1} \int^T_0 a_t^2 dt}
\end{align*}
By the exponents calculated at 
(\ref{220220.2422}), (\ref{220220.2423}), (\ref{220220.2424}) and (\ref{220220.2425}),
we have
$\bnorm{\cali^{M(2,0)}_{k,n}}_{7,p} = O(r_n)$ for $k=2,3,4$ 
by Proposition \ref{210823.2400} and 
$\bnorm{\cali^{M(2,0)}_{1,2,n}}_{7,p} = O(r_n)$
by Proposition \ref{210407.2246}
for every $p>1$.
%when without reference variable
%i=8, when with reference variable
%
%
Since
$\cali^{M(2,0)}_{1,1,n}-G_\infty= o_M(n^{2H - \frac{3}{2}})$
by Lemma \ref{220309.2200},
we obtain 
\begin{align*}
  \bnorm{ G_n^{(2)}}_{7,p} = O(n^{2H-3/2}).
\end{align*}
%when without reference variable
%i=8, when with reference variable

%\subsubsection{Condition (3.20), (3.21) of [C](ii)}
\begin{itemize}
  \item [{\bf [D]}] \bf (ii) (\ref{220215.1243}) and (\ref{220215.1244})
\end{itemize}
\begin{align}
  G_n^{(3)}=D_{u_n}G_\infty
  =n^{2H-3/2} \sum_{j\in[n]} \rbr{nD_{1_j} G_\infty} \; a_\tjm  I(1_j)
  =\cali_n(2H-3/2,G^{(\ref{220221.1020})},A^{(\ref{220221.1020})},\bbone),
  \label{220221.1020}
\end{align}
where 
$G^{(\ref{220221.1020})}=\singlegraph{1}{1}$ and 
$A^{(\ref{220221.1020})}_n(j)=\rbr{nD_{1_j} G_\infty} \; a_\tjm$.
The corresponding exponent is 
$e(2H-3/2,G^{(\ref{220221.1020})})=2H-3/2+(2-1-1)=2H-3/2$, 
and hence by Proposition \ref{210823.2400} we have
\begin{align*}
  \bnorm{G_n^{(3)}}_{7,p}=O(r_n).
\end{align*}

Since the only component in $G^{(\ref{220221.1020})}$ belongs to 
$\Csone(G^{(\ref{220221.1020})})$, and hence
$\Cpearg{1}(G^{(\ref{220221.1020})})$ is empty,
by Proposition \ref{210822.1800} and following the argument around 
(\ref{220218.1201}) and (\ref{220218.1202}), we obtain
\begin{align}
  \bnorm{D_{u_n} G_n^{(3)}}_{\comm{8},p}=O(n^{4H-3})=O(r_n^{2})
  \label{220221.1800}
\end{align}
%when without reference variable
%i=8,9, resp. when with reference variable

%\subsubsection{Condition (3.22) of [C](ii)}
\begin{itemize}
  \item [{\bf [D]}] \bf (ii) (\ref{220215.1245})
\end{itemize}
Since 
$D^2_{u_n}G_n^{(2)}=D^3_{u_n}M_n - D_{u_n}^2 G_\infty$ and
$\snorm{D_{u_n}^2 G_\infty}_{6,p} = \bnorm{D_{u_n} G_n^{(3)}}_{6,p}=O(r_n^2)$
by (\ref{220221.1800}),
we only need to treat $D^3_{u_n}M_n$.
From (\ref{220221.1811}) and (\ref{220221.1812}) in the previous subsection, 
we have
\begin{align*}
  D_{u_n}^3M_n
  =&D_{u_n}^3\bar M_n - D_{u_n}^3M_{0,n}
  \\=&D_{u_n}\rbr{3\:\cali^{M(3,0)}_{1,n}+2\:\cali^{M(3,0)}_{2,n}
  +\cali^{M(3,0)}_{3,n}+\cali^{M(3,0)}_{4,n}+\cali^{M(3,0)}_{5,n}}
  -D_{u_n}^2(\cali^{M(2,0)}_{3,n}+\cali^{M(2,0)}_{4,n})
\end{align*}

Consider $\:\cali^{M(3,0)}_{1,n}=
2\,\cali_n(6H-5/2,G^{M(3,0)}_{1},A^{M(3,0)}_{1},\bbone)$.
The two components in $G^{M(3,0)}_{1}$ belong to 
$\Csone(G^{N(1,0)}_{1})$ or $\Cpearg{0}(G^{N(1,0)}_{1})$,
and hence
$\Cpearg{1}(G^{N(1,0)}_{1})$ is empty.
Hence by Proposition \ref{210822.1800},
there exist finite sets $\Lambda^{M(4,0)}_{1}$ and $V''$,
$\alpha^\lambda\in\bbR$,
a weighted graph 
$G^\lambda=(V'',\edgeWt^\lambda,\vertWt^\lambda)$, 
$A^\lambda\in\cala(V'')$ and $\bbf^\lambda\in\calf(V'')$
for each $\lambda\in\Lambda^{M(4,0)}_{1}$,
such that
$D_{u_n}\cali^{M(3,0)}_{1,n}$ can be written as
\begin{align*} 
  D_{u_n}\cali^{M(3,0)}_{1,n} =&
  \sum_{\lambda\in\Lambda^{M(4,0)}_{1}} 
  \cali_n(\alpha^{\lambda}, G^{\lambda}, A^{\lambda}, \bbf^\lambda),%\bbone),
\end{align*}
%$(\alpha^\lambda, G^\lambda, A^\lambda, \bbf^\lambda)_{\lambda\in\Lambda}$
and
\begin{align*}
  \max_{\lambda\in\Lambda^{M(4,0)}_{1}} e(\alpha^\lambda, G^\lambda)
  = e(6H-5/2,G^{M(3,0)}_{1}) +2H-3/2
  =4H-3.
\end{align*}
By Proposition \ref{210823.2400}, we have
\begin{align*}
  \snorm{D_{u_n}\cali^{M(3,0)}_{1,n}}_{6,p}=O(n^{4H-3}).
\end{align*}
For $D_{u_n}\cali^{M(3,0)}_{2,n}$, 
since the only component in $G^{M(3,0)}_{2}$ belongs to $\Cpearg{0}(G^{N(1,0)}_{2})$,
%and $\Cpearg{1}(G^{M(3,0)}_{2})$ is empty, 
following the argument above, we have 
\begin{align*}
  \snorm{D_{u_n}\cali^{M(3,0)}_{2,n}}_{6,p}
  =O(n^{e(6H-3/2,G^{M(3,0)}_{2})+2H-3/2})
  =O(n^{4H-3}),
\end{align*}
where we used 
$e(6H-3/2,G^{M(3,0)}_{2}) =2H-3/2$
from (\ref{220330.1242}).

Consider $D_{u_n}\cali^{M(3,0)}_{3,n}$.
Since $\cali^{M(3,0)}_{3,n}=
\cali_n(6H-7/2,G^{M(3,0)}_{3},A^{M(3,0)}_{3},\bbone)$
and 
$e(6H-7/2,G^{M(3,0)}_{3})=4H-3$
(see (\ref{220222.2003}) and (\ref{220222.1954})),
we have 
$\snorm{\bnorm{D^i\cali^{M(3,0)}_{3,n}}_{\calh^{\otimes i}}}_p=O(n^{4H-3})$
by Proposition \ref{210823.2400} and 
$\bnorm{u_{n}}_{i,p}=O(1)$ from (\ref{220222.1937})
for any $i\in\bbZ_{\geq1}$ and $p>1$.
Hence we obtain
\begin{align*}
  \snorm{D_{u_n}\cali^{M(3,0)}_{3,n}}_{6,p}
  =O(n^{4H-3})
  \tfor p>1.
\end{align*}
By similar arguments, we have
\begin{alignat*}{2}
  \snorm{D_{u_n}\cali^{M(3,0)}_{4,n}}_{6,p},\quad
  %&=O(n^{4H-3}),&
  \snorm{D_{u_n}\cali^{M(3,0)}_{5,n}}_{6,p},\quad
  %&=O(n^{4H-3})
  %\\
  \snorm{D_{u_n}^2\cali^{M(2,0)}_{3,n}}_{6,p},\quad
  %&=O(n^{4H-3}),&\qquad
  \snorm{D_{u_n}^2\cali^{M(2,0)}_{4,n}}_{6,p}
  &=O(n^{4H-3}).
\end{alignat*}
(We can obtain sharper estimates about 
$D_{u_n}^2\cali^{M(2,0)}_{3,n}$ and $D_{u_n}^2\cali^{M(2,0)}_{4,n}$
by using Proposition \ref{210822.1800}.)
%They are estimated as $= O(n^{8H-6})$.
Therefore, we obtain
\begin{align*}
  \snorm{ D^2_{u_n}G_n^{(2)} }_{6,p} = O(r_n^{2}).
\end{align*}

%\subsubsection{Condition (3.25), (3.26) of [C](ii)}
\begin{itemize}
  \item [{\bf [D]}] \bf (ii) (\ref{220215.1246}) and (\ref{220215.1247})
\end{itemize}
The functional $N_n$ is decomposed as 
\begin{align*}
  N_n=\cali^{N(1,0)}_{1,n}+\pertur{2}+N'_n
\end{align*}
%(See Subsection \ref{220220.2416}.)
with $N'_n$ defined by (\ref{220220.2100}).
By the exponents calculated at (\ref{220222.2100}) and Proposition \ref{210823.2400},
we have
$\bnorm{\cali^{N(1,0)}_{1,n}}_{8,p} = O(1)$ and 
$\bnorm{\pertur{2}}_{8,p} = O(1)$
for every $p>1$.
Since $N'_n=o_M(1)$, the condition 
$\snorm{N_n}_{8,p}=O(1)$ for every $p>1$ is satisfied.

Since $D_{u_n}\cali^{N(1,0)}_{1,n}$ is written as 
the sum (\ref{220218.1201}) of functionals satisfying (\ref{220218.1202}),
by Proposition \ref{210823.2400},
we have
%$D_{u_n}\cali^{N(1,0)}_{1,n}$ can be written as
\begin{align*}
  \Bnorm{\bnorm{D^i
  \brbr{D_{u_n}\cali^{N(1,0)}_{1,n}}}_{\calh^{\otimes i}}}_p 
  =O(n^{2H-3/2}) 
\end{align*}
for any $i\in\bbZ_{\geq1}$ and $p>1$.
With $\bnorm{u_{n}}=O_M(1)$ from (\ref{220222.1937}), 
we have 
\begin{align*}
  \snorm{D_{u_n}^2\cali^{N(1,0)}_{1,n}}_{7,p}=O(n^{2H-3/2}).
\end{align*}
Similar arguments work for 
$D_{u_n}^2N_{2,n}$,
and by the exponent (\ref{220222.2112}) we obtain 
\begin{align*}
  \snorm{D_{u_n}^2N_{2,n}}_{7,p}=O(n^{2H-3/2}).
\end{align*}
(We can obtain sharper estimates of 
$D_{u_n}^2\cali^{N(1,0)}_{1,n}$ and $D_{u_n}^2N_{2,n}$
by using Proposition \ref{210822.1800}.)
%They are estimated as $= O(n^{4H-3})$.

By Remark \ref{220228.1345}, we have 
$N'_n=O_M(n^{1-2H+\epsilon})$ for $\epsilon>0$.
By (\ref{220222.1937}), we can write
\begin{align*}
  \snorm{D_{u_n}^2N'_n}_{7,p}=O(n^{1-2H+\epsilon})=O(r_n^\kappa)
\end{align*}
with some $\kappa\in(0,1\wedge\frac{1-2H}{2H-\frac32})$,
%with some $\kappa\in(0,1\wedge(2H-1)/(\frac32-2H))$,
with which we obtain 
$\snorm{D_{u_n}^2N_n}_{7,p}=O(r_n^\kappa)$ for every $p>1$.

%\subsubsection{[C](iv)b}
\begin{itemize}
  \item [{\bf [D]}] \bf(iv)(b)
\end{itemize}
\begin{align}
  DM_n=&
  n^{2H-1/2} \sum_{j_1\in[n]} \cbr{
    (Da_{t_{j_1-1}}) I^2(1_{j_1}^{\otimes2})
    +a_{t_{j_1-1}} 2I(1_{j_1}) 1_{j_1}
    -\big(D(D_{1_{j_1}}a_{t_{j_1-1}})\big)I(1_{j_1})
    -(D_{1_{j_1}}a_{t_{j_1-1}})1_{j_1}
  } 
  \nn\\=&
  2n^{2H-1/2} \sum_{j_1\in[n]} a_{t_{j_1-1}} I(1_{j_1}) 1_{j_1}
  \nn\\&+
  n^{2H-1/2} \sum_{j_1\in[n]} \cbr{
    (Da_{t_{j_1-1}}) I^2(1_{j_1}^{\otimes2})
    -\big(D(D_{1_{j_1}}a_{t_{j_1-1}})\big)I(1_{j_1})
    -(D_{1_{j_1}}a_{t_{j_1-1}})1_{j_1}
  } 
\label{210430.1543}
\end{align}
Noticing the first term in (\ref{210430.1543}) is equal to $2u_n$,
we decompose $\Delta_{M_n}$ as
\begin{align}
  \Delta_{M_n}&=\abr{DM_n,DM_n}
  =\abr{2u_n,2u_n} + 2\abr{2u_n,DM_n-2u_n} +\abr{DM_n-2u_n,DM_n-2u_n}
  \nn\\&=
  2\cali^{M(2,0)}_{1,1,n}+ 2\cali^{M(2,0)}_{1,2,n}
  +4\Brbr{\sum_{k=2,3,4}\cali^{M(2,0)}_{k,n}}
  +\snorm{DM_n-2u_n}^2
  \nn%\label{210427.1521}
\end{align}

Let 
\begin{align}
  \mR^{(\ref{220309.2210})}_n=
  2\rbr{\cali^{M(2,0)}_{1,1,n}-G_\infty}
  + 2\cali^{M(2,0)}_{1,2,n}
  +4\Brbr{\sum_{k=2,3,4}\cali^{M(2,0)}_{k,n}}.
  \label{220309.2210}
\end{align}
By the exponents calculated at 
(\ref{220220.2422}), (\ref{220220.2423}), (\ref{220220.2424}) and (\ref{220220.2425}),
we have
$\cali^{M(2,0)}_{k,n} = O_{L^{\infty-}}(r_n)$ for $k=2,3,4$
by Proposition \ref{210407.1401} and 
$\cali^{M(2,0)}_{1,2,n}= O_{L^{\infty-}}(r_n)$
by Proposition \ref{210407.1402}.
Since
$\cali^{M(2,0)}_{1,1,n}-G_\infty= o_M(n^{2H - \frac{3}{2}})$
by Lemma \ref{220309.2200},
we obtain 
\begin{align*}
  \bnorm{\mR^{(\ref{220309.2210})}_n}_{p} = O(n^{2H-3/2})
\end{align*}
for any $p>1$.

Set 
$s_\infty = s_n = G_\infty.$
We can see $G_\infty\in \bbD^{\infty}$.
We have $G_\infty^{-1}\in L^{\infty-}$ by Assumption \ref{220404.1535} (ii).
Since $\Delta_{M_n}$ is decomposed as 
$\Delta_{M_n}=2G_\infty+\mR_n^{(\ref{220309.2210})}+\snorm{DM_n-2u_n}^2$ 
and $\snorm{DM_n-2u_n}^2\geq0$,
we have
\begin{align*}
  P\sbr{\Delta_{M_n}<s_\infty}
  &\leq P\Bsbr{\mR_n^{(\ref{220309.2210})} <-G_\infty}
  \leq P\Bsbr{\abs{\mR_n^{(\ref{220309.2210})}} >G_\infty}
  = P\Bsbr{\abs{\mR_n^{(\ref{220309.2210})}}\:G_\infty^{-1} >1}
  \\&\leq
  E\Bsbr{\abs{\mR_n^{(\ref{220309.2210})}}^L\; \rbr{G_\infty}^{-L}}
  \leq
  \bnorm{\mR_n^{(\ref{220309.2210})}}_{2L}^L
  \snorm{G_\infty^{-1}}_{2L}^L
  =O(r_n^L)
\end{align*}
for any $L>1/2$.
By taking $L$ large enough, say $L=2$,
we obtain 
\begin{align*}
  P\sbr{\Delta_{M_n}<s_\infty} = O(r_n^{2}).
  %= O(r_n^{1+\kappa}).
\end{align*}

Summing up the above arguments, we have proved Theorem \ref{220404.1630}.
\hfill\qedsymbol
%\comm{Then $s_\infty \in \bbD^\infty$ and $s_\infty^{-1} \in L^{\infty-}$
%since we assumed $\inf_{x\in\bbR} |V^{[1]}(x)|>0$.}

\subsection{In the case of fractional Ornstein-Uhlenbeck process}
\label{220518.2114}
When we apply Theorem \ref{220404.1630} to the quadratic variation of 
a fractional Ornstein-Uhlenbeck (fOU) process, 
Assumption \ref{220404.1535} (i) fails to be satisfied.
However, we can bypass the argument relying upon Assumption \ref{220404.1535} (i)
to  obtain the following theorem.
\begin{theorem}
  Consider the one-dimensional fractional OU process 
  \begin{align*}
    X_t = X_0 - b\int_0^t X_sds + \sigma B_t 
    \qquad\tfor t\in[0,T]
  \end{align*}
  with some $T>0$, $\sigma,b>0$ and $X_0\in\bbR$.
  Define $Z_n$ by $Z_n=n^{1/2}\rbr{\bbV_n - \bbV_\infty}$ with 
  $\bbV_n$ defined at (\ref{220420.1130}) and $\bbV_\infty=\sigma^2\,T^{2H}$.
  Let $r_n=n^{2H-3/2}$.
  Then, for any $M,\gamma>0$ the following estimate holds:
\begin{align*}%\label{220421.1700}
  \sup_{f\in\hat\cale(M,\gamma)}
  \abs{E\sbr{f(Z_n)}- \int_{\bbR}f(z)\hat p_n(z)dz}
  =o(r_n)
  \quad\tassm n\to\infty,
\end{align*}
where 
\begin{align*}
  %\hat p_n(z)&=
  %E\sbr{\exp\rbr{-\frac12 G_\infty^{-1} z^2}}
  %+r_n E\sbr{\mS(\partial_z)^* \exp\rbr{-\frac12 G_\infty^{-1} z^2}}
  %\\
  \hat p_n(z)&=
  E\sbr{\phi(z;0,G_\infty)}
  +r_n E\sbr{\mS(\partial_z)^* \phi(z;0,G_\infty)}
  %\\
  %\hat\Delta_n(f)&=
  %\abs{E\sbr{f(Z_n)}- \int_{z\in\bbR}f(z)\hat p_n(z)dz}
\end{align*}
with 
$G_\infty =2 c_H^2\, \sigma^4\, T^{4H}$ and
\begin{align*}
  \mS(\sfi\sfz)=\mS^{(1,0)}(\sfi\sfz)
  = \int^T_0\rbr{-2b\sigma\, \dotd(\tau,\tau) + Tb^2 X_\tau^2}d\tau\; [\sfi\sfz].
\end{align*}
\end{theorem}

In this case, the decomposition (\ref{211020.1120}) of the increment of $X_t$ 
is obtained in the same manner with 
$T^{1}_{n,j}=\sigma I_1(1_j)$,
$T^{2}_{n,j}=-bX_{\tjm}T n^{-1}$,
$T^{3}_{n,j}=T^{4}_{n,j}=0$,
$R^{1}_{n,j}=R^{3}_{n,j}=0$ and 
$R^{2}_{n,j}=-b\int^\tj_\tjm(X_t-X_\tjm)dt$,
and the argument of the stochastic expansion of $Z_n$ becomes much shorter.
While determining the limit random symbol and checking Condition {\bf [D]},
we need the estimate of the Malliavin derivative of $X_t$.
Although we cannot use Proposition \ref{220413.1450} due to the unboundedness of $V^{[2]}$,
we can obtain similar estimates from 
an expression of fOU process and its derivatives, namely
$X_t=X_0e^{-bt} + \sigma\int^t_0e^{-b(t-s)}dB_s$,
$D_sX_t=\sigma\, 1_{[0,t]}(s)\, e^{-b(t-s)}$ and 
$D^iX_t=0$ for $i\geq2$.

%\newpage
\section{Technical Lemmas}\label{220405.1158}

\subsection{Estimates related to SDE driven by fBm ($H>1/2$)}
\label{220423.1450}
We review the results from \cite{hu2016rate} in order to obtain bounds related to 
the solution of SDE driven by fBm.
To simplify the notation, we only state the result in the case 
where the dimensions of the space of the process and fBm are both one.
Let $\beta$ be any number satisfying $1/2<\beta<H$.

We introduce the following seminorm and norm.
Let $t,t'\in[0,T]$ with $t<t'$. % and $\beta\in(1/2,H)$.
For a function $x:[0,T]\to\bbR$,
we define the $\beta$-H\"{o}lder seminorm of $x$ on $[t,t']$ by
\begin{align*}
  \norm{x}_{t,t',\beta}=
  \sup\cbr{\frac{\abs{x(s)-x(s')}}{\abs{s-s'}^\beta} \mid s<s'\in[t,t']}
\end{align*}
and denote the uniform norm of $x$ on $t,t'$ by
$\norm{x}_{t,t',\infty}=
\sup_{s\in[t,t']} \abs{x(s)}$.
When $(t,t')=(0,T)$,
we abbreviate $\norm{x}_{t,t',\beta}$ and $\norm{x}_{t,t',\infty}$ as 
$\norm{x}_{\beta}$ and $\norm{x}_{\infty}$, respectively.

%\vspsm
Given a process 
$Y=\cbr{Y_t, t\in[0,T]}$ such that $Y_t\in\bbD^{N,2}$,
for each $t$ and some $N\geq1$ 
we define a random variable $\cald_N^* Y$ 
\begin{align*}
  \cald_N^* Y = 
  \max\Bcbr{\sup_{r_0\in[0,T]}\abs{Y_{r_0}},
  \sup_{r_0, r_1\in[0,T]}\abs{D_{r_1}Y_{r_0}},...,
  \sup_{r_0,...,r_N\in[0,T]}\abs{D^N_{r_1,...,r_N}Y_{r_0}}}
\end{align*}
and 
\begin{align*}
  \cald_N Y = 
  \max\Bcbr{
    \cald_N^* Y,\; \norm{Y}_\beta,\;
    \sup_{r_1\in[0,T]}\norm{ D_{r_1}Y_{\cdot}}_{r_1,T,\beta},...,
  \sup_{r_1,...,r_N\in[0,T]}\norm{D^N_{r_1,...,r_N}Y_{\cdot}}_{r_1\vee,...,\vee r_N,T,\beta}}.
\end{align*}
Here we are taking a good version of the density representing the Malliavin derivative of a functional.
For $k\in\bbZ_{\geq0}$, we denote by $C^k_b(\bbR)$ the space of $k$ times continuously
differentiable functions $f:\bbR\to\bbR$
which are bounded together with its derivatives of order up to $k$.
The following two propositions are basic for estimating functionals of the SDE.

\begin{proposition}[Proposition 3.1 from \cite{hu2016rate}]
\label{210507.1755}
Let $X$ be a solution of SDE (\ref{210430.1615}).
Fix $N\geq0$ and suppose that $V^{[1]}\in C^{N+1}_b(\bbR)$ and $V^{[2]}\in C^{N}_b(\bbR)$.
%Fix $N\geq0$ and suppose that $b\in C^N_b(\bbR)$, $\sigma\in C^{N+1}_b(\bbR)$.
Then there exists a positive constant $K$ such that 
the random variable $\cald_N X$ is bounded by 
$Ke^{K \norm{B}_\beta^{1/\beta}}$.
\end{proposition}

\begin{proposition}\label{220413.1450}
(i) The functional $\norm{B}_\beta$ has finite moments of any order.

  \item (ii) 
  Suppose that $(X_t)_{t\in[0,T]}$ is the solution of SDE (\ref{210430.1615}) under Assumption \ref{220404.1535}(i).
  Then, for any $N\in\bbN$, $\cald_N X$ belongs to $L^{\infty-}$.
\end{proposition}
\begin{proof}%[Proof of Proposition \ref{220413.1450}]
  %\koko
  By the Fernique theorem, 
  there exists some $\alpha>0$ such that 
  \begin{align*}
    E[\exp(\alpha\norm{B}_\beta^2)]<\infty.
  \end{align*}
  Hence (i) follows immediately.
  Since $1/\beta<2$, 
  the functional $Ke^{K \norm{B}_\beta^{1/\beta}}$ obtained in Propsition \ref{210507.1755} 
  also has finite moments of any order,
  so does $\cald_N X$.
\end{proof}
\noindent
%\comm{We use Proposition \ref{220413.1450} (ii) in the following form:
%\begin{align} \label{201230.1200}
%  \sup_{s_1,...,s_i,t\in[0,T]}\norm{D^i_{s_1,...,s_i}X_t}_{p}<
%  \bbnorm{\sup_{s_1,...,s_i,t\in[0,T]}\abs{D^i_{s_1,...,s_i}X_t}}_{p}<\infty.
%\end{align}}

\vspsm
The next lemma is a version of Lemma A.1 of \cite{hu2016rate} in the one dimensional case.
We use this lemma with $\beta=\beta'\in(1/2,H)$ and $f(x)=x$.
\begin{lemma}[Lemma A.1 (ii) in HLN]\label{220411.1720}
  Let $\beta,\beta'\in(0,1)$ satisfying $\beta+\beta'>1$.
  Assume that $z=\cbr{z_t}_{t\in[0,T]}$ and $x=\cbr{x_t}_{t\in[0,T]}$ are
  H\"older continuous functions with index $\beta,\beta'\in(0,1)$ respectively.
  Suppose that $f:\bbR\to\bbR$ is continuously differentiable. 
  Then the following estimate holds:
  \begin{align*}
    \abs{\int_\rze^\ron f(x_t) dz_t}\leq
    K_1\; \sup_{t\in[\rze,\ron]}\abs{f(x_t)}\;\norm{z}_{\rze,\ron,\beta} \abs{\ron-\rze}^\beta
    +K_2\; \sup_{t\in[\rze,\ron]}\abs{f'(x_t)}\; \norm{x}_{\rze,\ron,\beta'}
    \norm{z}_{\rze,\ron,\beta} \abs{\ron-\rze}^{\beta+\beta'}
  \end{align*}
  where the constants $K_1,K_2$ depend only on $\beta$ and $\beta'$.
  %\begin{align*}
  %  \abs{\int_\rze^\ron f_t dg_t}\leq
  %  K_1\; \norm{f}_{\rze,\ron,\infty}\norm{g}_{\rze,\ron,\beta} \abs{\ron-\rze}^\beta
  %  +K_2\; \norm{f}_{\rze,\ron,\beta}\norm{g}_{\rze,\ron,\beta} \abs{\ron-\rze}^{2\beta}
  %\end{align*}
\end{lemma}

\vspssm
Recall the norm of $\abs\calh^{\otimes l}$ was given by
\begin{align*}
  \norm\phi_{\abs\calh^{\otimes l}}^2
  =\alpha_H^l \int_{[0,T]^{l}} \int_{[0,T]^{l}} \abs{\phi(u)} \abs{\phi(v)}
  \abs{u_1-v_1}^{2H-2}...\abs{u_{l}-v_{l}}^{2H-2}dudv,
\end{align*}
and $\norm{\cdot}_{\calh^{\otimes l}} \leq \norm\cdot_{\abs\calh^{\otimes l}}.$
When we estimate the Malliavin norm in the following subsections, 
we often rely on the next inequality.

\begin{lemma}\label{220412.1015}
  (i) (Inequality (2.6) in \cite{hu2016rate})
  For $\phi\in L^{1/H}([0,T]^l)$,
  the following inequality holds:
  \begin{align*}
    \norm\phi_{\abs\calh^{\otimes l}}\leq 
    b_{H,l}\norm\phi_{L^{1/H}([0,T]^l)}.
  \end{align*}

  \item
  (ii) Let $p\geq1/H$.
  For a measurable function $f:\Omega\times[0,T]^l\to\bbR$ such that 
  $\bnorm{ \norm{f}_{L^p(P,\Omega)}}_{L^{1/H}([0,T]^l)}<\infty$,
  the following estimate holds:
  \begin{align*}
    \bnorm{ \norm{f}_{\calh^{\otimes l}} }_{L^p(P,\Omega)} \simleq 
    \bnorm{ \norm{f}_{L^p(P,\Omega)}}_{L^{1/H}([0,T]^l)}
  \end{align*}
\end{lemma}
\begin{proof}
  (i) We omit the proof. See (2.6) in \cite{hu2016rate} and references there.
  \item (ii) Using the triangle inequality for $L^{pH}(P)$-norm, we can prove the estimate.
  \begin{align*}
    \bnorm{ \norm{f}_{\calh^{\otimes l}} }_{L^p(P,\Omega)} &\simleq 
    \bnorm{ \norm{f}_{L^{1/H}([0,T]^l)} }_{L^p(P,\Omega)} 
    %\\&=
    %\rbr{\int_\Omega
    %\rbr{\int_{s\in[0,T]^l} \abs{f(\omega,s)}^{1/H} ds}^{pH}\;d\omega}^{1/p}
    =%\\&=
    \norm{\int_{s\in[0,T]^l} \abs{f(\omega,s)}^{1/H} ds}_{L^{pH}(P)}^H
    \\&\leq
    \rbr{\int_{s\in[0,T]^l} \norm{\abs{f(\cdot,s)}^{1/H}}_{L^{pH}(P)} ds}^H
    =%\\&=
    \bnorm{ \norm{f}_{L^p(P,\Omega)}}_{L^{1/H}([0,T]^l)}
  \end{align*}
\end{proof}

%\subsubsection{\redb{Proof of Lemma \ref{210315.1401}}}
We proceed to give estimates on terms which appeared in Section \ref{220317.2030}.
Recall the residual terms 
$R^{i}_{n,j}$ ($i=1,2,3$)
in the decomposition (\ref{211020.1120}):
\begin{align*}
    R^{1}_{n,j} &= 
    \int_\tl^\tr \int_\tl^t (V^{[(1;1),1]}(X_{\tpr}) - V^{[(1;1),1]}(X_{\tl})) dB_{\tpr} dB_t, 
    %R^{1}_{n,j} &= \int_\tl^\tr \int_\tl^t (f(X_{\tpr})-f(X_{\tl})) dB_{\tpr} dB_t, 
    \\
    R^{2}_{n,j} &= \int_\tl^\tr \rbr{V^{[2]}_t-V^{[2]}_\tl}dt
    \\
    R^{3}_{n,j} &= 
    \int_\tl^\tr \int_\tl^t V^{[(1;1),2]}_{\tpr} d\tpr dB_t.
\end{align*}

\begin{lemma}\label{210315.1401}
  Let $\beta\in(1/2,H)$.
  Suppose that $(X_t)_{t\in[0,T]}$ is the solution of SDE (\ref{210430.1615}) under Assumption \ref{220404.1535}(i).
  
  \item (i)
  Let $f:\bbR\to\bbR$ a smooth function on $\bbR$ with bounded derivatives.
  %and $j+i,j\in{1,..,n}$.
  %
  Then for any $i\in\bbZ_{\geq0}$ and $p>1$,
  there exists a constant $C_{i,p}$ such that
  for any $\rze,\ron \in[0,T]$,
  \begin{align*}
  %f(X_{t_{j+i}})-f(X_{\tj})
  \snorm{f(X_{\rze})-f(X_{\ron}) }_{i,p}
  %=O_M(()^H),
  \leq C_{i,p} |\rze-\ron|^\beta.
  \end{align*}
  %as $n$ goes to $\infty$,
  %where the constant of $O_M$ is independent of $i$ and $j$.
  \label{210504.2200}
  
  \item (ii)
  Consider the functionals $R^{k}_{n,j}$ ($k=1,2,3$) defined in (\ref{220206.2351}).
  For any $i\in\bbZ_{\geq0}$ and $p > 1$,
  the Malliavin norm of  $R^{k}_{n,j}$ is bounded as
\begin{align*}
  \sup_{j\in[n]} \snorm{R^{1}_{n,j}}_{i,p} =O(n^{-3\beta})
  \quad\tand\quad
  \sup_{j\in[n]} \snorm{R^{k}_{n,j}}_{i,p} =O(n^{-1-\beta})
  \quad(k = 2,3)
\end{align*}
as $n\to\infty$.
In particular, 
$\sup_{j\in[n]} \snorm{T^{5}_{n,j}}_{i,p}=O(n^{-1-\beta})$.%R^{(\ref{0211120524})}_{n,j}
  
\end{lemma}
\begin{proof}
  By the Leibniz rule, we can show that
  there exists a random variable $Z_N$
  which is a polynomial in 
  $\cald_N X$ and 
  $\norm{{f^{(i)}}}_\infty=\sup_{x\in\bbR}\abs{f^{(i)}(x)}$ ($i=1,..,N$)
  such that
  for $0\leq s_1,..,s_N\leq t\leq T$,
  \begin{align*}
    \abs{D^N_{\son,..,s_N} f(X_t)}
    \leq Z_N.
  \end{align*}
  Similarly, using the definition of the $\beta$-H\"older seminorm,
  we have
  \begin{align}
    \abs{D^N_{\son,..,s_N} f(X_{t'})- D^N_{\son,..,s_N} f(X_t)}
    \leq \zndelta\,\abs{t-t'}^\beta
    \label{220411.2030}
  \end{align}
  uniformly for $0\leq s_1,..,s_N\leq t,t'\leq T$,
 where $\zndelta$ is a random variable written as a polynomial in 
  $\cald_N X$ and $\norm{{f^{(i)}}}_\infty$ ($i=1,..,N+1$).
  By Proposition \ref{220413.1450} (ii), we have
  $Z_{N}, \zndelta\in L^{\infty-}(P)$.
  Notice that $Z_{N}, \zndelta$ are independent of $t$ and $t'$.
  We write 
  $\displaystyle \bar s=\max_{k=1,..,N} s_k$,\;
  $\displaystyle \maxs{i}=\max_{k\neq i} s_k$
  and
  $\displaystyle \maxs{i,j}=\max_{k\neq i,j} s_k$.

  \vspssm
  (i)
  Suppose $0\leq\rze<\ron\leq T$.
  Consider $N\geq1$.
  %For notational convenience, we denote the random variable 
  %$f(X_{\ron})-f(X_{\rze})$ by $z_{\rze,\ron}$.
  If $\bar s\in (\rze,\ron]$, 
  then the $N$-th Malliavin derivative of $f(X_\ron)-f(X_{\rze})$ is bounded as 
  \begin{align*}
    \abs{D^N_{\son,..,s_N} \brbr{f(X_\ron)-f(X_{\rze})}}=
    \abs{D^N_{\son,..,s_N} f(X_\ron)} &\leq Z_{N}.
  \end{align*}
  When $\bar s \in [0,\rze]$, we have
  \begin{align*}
    \abs{D^N_{\son,..,s_N} \brbr{f(X_\ron)-f(X_{\rze})}}
    &\leq Z^\Delta_{N}\; \abs{\ron-\rze}^{\beta}.
  \end{align*}
  Hence, by Lemma \ref{220412.1015}(ii), we have the estimate
  \begin{align*}
    \Bnorm{ \norm{ 
      D^N \brbr{f(X_{\ron})-f(X_{\rze})}
      %D^N z_{\rze,\ron}
    }_{\calh^{\otimes N}} }_{L^p(P)}
    &\simleq
    \Bnorm{ \norm{
      D^N_{\son,..,s_N} \brbr{f(X_{\ron})-f(X_{\rze})}
      %D^N_{\son,..,s_N} z_{\rze,\ron}
    }_{L^p(P)} }_{L^{1/H}([0,T]^N)}
    \\&\leq
    \snorm{ 
      1_{\cbr{\rze<\bar s\leq\ron}}\,
      \snorm{Z_{N}}_{L^p(P)}
      +
      1_{\cbr{\bar s\leq\rze}}\, \abs{\ron-\rze}^{\beta}\;
      \norm{Z^\Delta_{N}}_{L^p(P)}
      }_{L^{1/H}([0,T]^N)}
    \\&\leq \bar C_{N,p} \abs{\ron-\rze}^{\beta},
  \end{align*}
  for $p>1/H$, where $\bar C_{N,p}$ is independent of $\rze,\ron$.
  By similar arguments, we have
  \begin{align*}
    \bnorm{f(X_{\ron})-f(X_{\rze}) }_{L^p(P)}
    &\leq \bar C_{0,p} \abs{\ron-\rze}^{\beta}.
  \end{align*}
  By setting $C_{i,p} = \sum_{N=0,..,i}\bar C_{N,p}$,
  we obtain (ii).

  \vspssm
  (ii) Let $p>1/H$.
  Fix $0\leq\rze<\ron\leq T$ for a while.
  For notational convenience, 
  we write
  \begin{align*}
    \kerone{r_0}{t} =&\; f(X_{t})-f(X_{\rze})
    \hspace{80pt}\tfor t\in[r_0,r_1],\\
    %z_t =&\; f(X_{t})-f(X_{\rze})\\
    R^{1}_{\rze,\ron}=&\;
    \int_\rze^\ron \rbr{\int_\rze^t (f(X_{\tpr}) - f(X_{\rze})) dB_{\tpr}} dB_t,
  \end{align*}
  where $f(x)=V^{[1;1]}(x)V^{[1]}(x) = V^{[(1;1),1]}(x)$ for brevity.

  For $N\geq1$, the $N$-th Malliavin derivative of 
  the functional $R^{1}_{\rze,\ron}$ has the following expression:
  \begin{align}
    D^N_{s} R^{1}_{\rze,\ron} =
    D^N_{\son,..,s_N} R^{1}_{\rze,\ron} =&
    \int_{\rze\vee \bar s}^\ron 
    \rbr{\int_{\rze\vee \bar s}^t D^N_{\son,..,s_N} \kerone{r_0}{\tpr} dB_{\tpr}} dB_t
    \label{210929.2131}
    \\&+ 
    \sum_{i=1}^{N} 
    1_{[\rze,\ron]}(s_i)
    \rbr{\int_{s_i}^\ron dB_t}
    \rbr{D^{N-1}_{\son,..,\hat{s_i},..,s_N} \kerone{r_0}{s_i}} 
    \label{210929.2132}
    \\&+ 
    \sum_{i=1}^{N} 
    1_{[\rze,\ron]}(s_i)
    \int_\rze^{s_i} D^{N-1}_{\son,..,\hat{s_i},..,s_N} \kerone{r_0}{\tpr} dB_\tpr 
    \label{210929.2133}
    \\&+ 
    \sum_{\substack{i,j=1,..,N\\i\neq j}}
    1_{[\rze,\ron]}(s_i) 1_{[\rze,s_i]}(s_j)\;
    D^{N-2}_{\son,..,\hat{s_i},..,\hat{s_j},..,s_N} \kerone{r_0}{s_j}
    \label{210929.2134}       
  \end{align}
  for $s_1,..,s_N\in[0,T]$ satisfying $\bar s\leq r_1$.
  When $N=1$, the above sum reads the sum of the first three summands.
  %This expression can be obtained following the argument in \cite{nualart2009malliavin} using 
  %the Fr\'{e}chet differentiability of the corresponding map from $W^{1-\alpha}_2$ to $\bbR$.
  %
  We estimate these terms using
  Lemma \ref{220411.1720}
  for each $s_1,..,s_N\in[0,T]$.
  
  %We write $\bar s_N=\max_k s_k$, $\bar s_{N,\hat i}=\max_{k\neq i} s_k$
  %and $\bar s_{N,\hat i,\hat j}=\max_{k\neq i,j} s_k$.
  %When $N=0$, we read $\bar s_N=0$.

  \vspssm
  First we consider the term (\ref{210929.2131}) in the case $\bar s \in (\rze,\ron]$, 
  where 
  \begin{align*}
    D^N_{\son,..,s_N} \kerone{r_0}{t} 
    = D^N_{\son,..,s_N} f(X_t)
    = 1_\cbr{\bar s\leq t}\; D^N_{\son,..,s_N} f(X_t)
    %= 1_{[\bar s,\ron]}(t)\; D^N_{\son,..,s_N} f(X_t)
    \tfor t\in[r_0,r_1].
  \end{align*}
  %$D^N_{\son,..,s_N} z_t = 1_{[\bar s_N,\ron]}(t) D^N_{\son,..,s_N} a(X_t)$.
%
  We have
  \begin{align*}
    \norm{D^N_{\son,..,s_N} f(X_\cdot) }_{\bar s, \ron, \beta}
    \leq \zndelta
    %\leq Z^-_{N,\beta}
    \tand %\\
    \norm{D^N_{\son,..,s_N} f(X_\cdot) }_{\bar s, \ron, \infty}
    \leq Z_{N}.
    %\leq Z^+_{N,\infty}.
  \end{align*}
  %By Proposition \ref{210507.1755}, we have
  %$Z^+_{N,\beta}, Z^+_{N,\infty}\in L^{\infty-}(P)$.
  %
  Let 
  \begin{align*}
    z^{1,+}_{\son,..s_N}(t)
    = \int_{\bar s}^t D^N_{\son,..,s_N} \kerone{r_0}{t'}\;  dB_{\tpr}
    = \int_{\bar s}^t D^N_{\son,..,s_N} f(X_\tpr)\; dB_{\tpr}
    \tfor t\in[\bar s, r_1].
  \end{align*}
  %Let $z^{1,N,+}_{\son,..s_N}(t)= \int_{\bar s_N}^t D^N_{\son,..,s_N} a(X_\tppr)dB_{\tppr}$.
  By Lemma \ref{220411.1720},
  $\norm{ z^{1,+}_{\son,..s_N}}_{\bar s, \ron, \infty}$ and
  $\norm{ z^{1,+}_{\son,..s_N}}_{\bar s, \ron, \beta}$ can be bounded as 
  \begin{align*}
    \norm{ z^{1,+}_{\son,..s_N}}_{\bar s, \ron, \infty}
    &\leq
    \sup_{t\in[\bar s, \ron]} \cbr{
    K_1\, \norm{ D^N_{\son,..,s_N} f(X_\cdot) }_{\bar s, t, \infty}
    \norm{B}_\beta\, \abs{t-\bar s}^\beta
    +K_2\, \norm{ D^N_{\son,..,s_N} f(X_\cdot) }_{\bar s, t, \beta}
    \norm{B}_\beta\, \abs{t-\bar s}^{2\beta}}
    \\&\leq
    (K_1\, Z_{N} + K_2\, \zndelta\, \abs{\ron-\rze}^{\beta})\,
    \norm{B}_\beta \abs{\ron-\rze}^\beta,
  \end{align*}
  and 
  \begin{align*}
    \norm{ z^{1,+}_{\son,..s_N} }_{\bar s, \ron, \beta}
    &\leq
    \sup_{t,\tpr\in[\bar s, \ron]} \bigg\{
    K_1 \norm{ D^N_{\son,..,s_N} f(X_\cdot) }_{t, \tpr, \infty}
    \norm{ B }_\beta \abs{t-\tpr}^\beta
    %\\&\hspace{130pt}
    +K_2 \norm{ D^N_{\son,..,s_N} f(X_\cdot) }_{t, \tpr, \beta}
    \norm{B}_\beta \abs{t-\tpr}^{2\beta}
    \bigg\}/\abs{t-\tpr}^{\beta}
    \\&\leq
    \rbr{K_1\, Z_{N} + K_2\, \zndelta\,  \abs{\ron-\rze}^{\beta}}
    \norm{B}_\beta.
  \end{align*}
  Hence, we obtain 
  \begin{align*}
    \abs{\int_{\bar s}^\ron z^{1,+}_{\son,..s_N} (t) dB_t}
    &\leq
    K_1 \norm{z^{1,+}_{\son,..s_N}}_{\bar s, \ron, \infty}
    \norm{B}_\beta \abs{\ron-\bar s}^\beta
    + 
    K_2 \norm{z^{1,+}_{\son,..s_N}}_{\bar s, \ron, \beta}
    \norm{B}_\beta \abs{\ron-\bar s}^{2\beta}
    \\&\leq
    Z^{(\ref{210929.2131})}_{N,+}\; \abs{\ron-\rze}^{2\beta},
  \end{align*}
  where $Z^{(\ref{210929.2131})}_{N,+} = 
  (K_1+K_2)(K_1\, Z_{N} +K_2\, \zndelta\, T^\beta)  \norm{B}_\beta^2$,
  which has moments of any order.

  \vspssm
  Secondly we consider the case $\bar s \in [0,\rze]$, 
  where 
  \begin{align*}
    D^N_{\son,..,s_N} \kerone{r_0}{t}
    = D^N_{\son,..,s_N} f(X_t) -  D^N_{\son,..,s_N} f(X_\rze).
  \end{align*}
  By (\ref{220411.2030}), we have
  \begin{align*}
      \norm{ D^N_{\son,..,s_N} z_{r_0}(\cdot)}_{\rze, \ron, \beta}
      =\norm{D^N_{\son,..,s_N} f(X_\cdot) }_{\rze, \ron, \beta}
      \leq \zndelta
      \tand %\\
      \norm{ D^N_{\son,..,s_N} z_\rze(\cdot)}_{\rze, \ron, \infty}
      \leq \zndelta \abs{\ron-\rze}^{\beta}.
  \end{align*}
  Let 
  \begin{align*}
    z^{1,-}_{\son,...,s_N}(t)= \int_{\rze}^t D^N_{\son,..,s_N} \kerone{\rze}{\tpr} dB_{\tpr}
    \qquad\tfor t\in[\rze,\ron].
  \end{align*}
  Then, we have
  \begin{align*}
    \norm{z^{1,-}_{\son,...,s_N}}_{\rze, \ron, \infty}
    &\leq
    \sup_{t\in[\rze, \ron]} \rbr{
    K_1 \snorm{D^N_{\son,..,s_N} z_{r_0}}_{\rze, t, \infty}
    \norm{B}_\beta \abs{t-\rze}^\beta
    + K_2 \snorm{D^N_{\son,..,s_N} z_{r_0} }_{\rze, t, \beta}
    \norm{B}_\beta \abs{t-\rze}^{2\beta}}
    \\&\leq
      (K_1+ K_2)\, \zndelta\, \norm{B}_\beta\, \abs{\ron-\rze}^{2\beta}.
    \\
      \norm{z^{1,-}_{\son,...,s_N}}_{\rze, \ron, \beta}
    &\leq
      (K_1 + K_2)\, \zndelta\, \norm{B}_\beta\, \abs{\ron-\rze}^{\beta}
  \end{align*}
  and 
  \begin{align*}
    \abs{\int_\rze^\ron z^{1,-}_{\son,..s_N}(\tpr) dB_\tpr }
    &\leq
      Z^{(\ref{210929.2131})}_{N,-}\; \abs{\ron-\rze}^{3\beta},
  \end{align*}
  where $Z^{(\ref{210929.2131})}_{N,-} = 
  (K_1+K_2)^2\; \zndelta\; \norm{B}_\beta^2$,
  which has moments of any order.

  Hence, we have the estimate
  \begin{align*}
    &\norm{\norm{
      \int_{\rze\vee\bar s}^\ron \rbr{\int_{\rze\vee\bar s}^t D^N_{\son,..,s_N} 
      \kerone{\rze}{\tpr}\, dB_{\tpr}} dB_t
    }_{L^p(P)}}_{L^{1/H}([0,T]^N)}
    \\\leq&\;
    \norm{
      1_{\{\bar s\leq\rze\}}\abs{\ron-\rze}^{3\beta}\;
      \bnorm{Z^{(\ref{210929.2131})}_{N,-}}_{L^p(P)}
      + 1_{\{\rze<\bar s\leq\ron\}}\abs{\ron-\rze}^{2\beta}\;
      \bnorm{Z^{(\ref{210929.2131})}_{N,+}}_{L^p(P)} 
    }_{L^{1/H}([0,T]^N)}
    \\\leq&\; C_{N}^{(\ref{210929.2131})}|\ron-\rze|^{3\beta}
  \end{align*}
  with $C_{N}^{(\ref{210929.2131})}$ independent of $\rze,\ron$.
  We used 
  $\norm{1_{\{\rze<\bar s\leq\ron\}}}_{L^{1/H}([0,T]^N)}\leq 
  \rbr{NT^{N-1}}^H \abs{\ron-\rze}^H$
  and 
  $\abs{\ron-\rze}^H \leq \abs{\ron-\rze}^\beta T^{H-\beta}$
  for $\beta\in(1/2,H)$.

  \vspsm
  Similar arguments work for the terms 
  (\ref{210929.2132}), (\ref{210929.2133}) and (\ref{210929.2134}), 
  and we have for $i=1,...,N$,
  \begin{align*}
    &\norm{1_{[\rze,\ron]}(s_i)\;\Bnorm{
      \Brbr{\int_{s_i}^\ron dB_\tpr} 
      D^{N-1}_{\son,..,\hat{s_i},..,s_N} \kerone{\rze}{s_i} 
    }_{L^p(P)} }_{L^{1/H}([0,T]^N)}
    \\\leq&\;
    \bbnorm{ 
      1_{\cbr{\rze<\maxs{i}\leq s_i\leq\ron}}\;
      %1_{[\maxs{i},\ron]}(s_i)\; 1_{\opcl{\rze,\ron}}(\maxs{i}) \;
      \abs{\ron-\rze}^{\beta}
      \Bnorm{ Z^{(\ref{210929.2132})}_{N-1,+} }_{L^p(P)}
      + 
      1_{\cbr{\maxs{i}\leq\rze\leq s_i\leq\ron}}\;
      %1_{\{\bar s_{n,\hat i}\leq\rze]\}}1_{[\rze,\ron]}(s_i) 
      \abs{\ron-\rze}^{2\beta} 
      \bnorm{ Z^{(\ref{210929.2132})}_{N-1,-}}_{L^p(P)} 
      }_{L^{1/H}([0,T]^N)}
    \\
      \leq&\; C_{N}^{(\ref{210929.2132})}
      \abs{\ron-\rze}^{3\beta},
    \\[10pt]
    &
    \snorm{ 1_{[\rze,\ron]}(s_i)\; \snorm{
      \int_\rze^{s_i}
      D^{N-1}_{\son,..,\hat{s_i},..,s_N} \kerone{\rze}{\tpr} dB_\tpr 
    }_{L^p(P)} }_{L^{1/H}([0,T]^N)}
    \\\leq&\;
    \snorm{ 
      1_{\cbr{\rze<\maxs{i}\leq s_i\leq\ron}}\;
      %1_{\{\rze<\bar s_{n,\hat i}\leq\ron]\}}1_{[\bar s_{n,\hat i},\ron]}(s_i)
      \abs{\ron-\rze}^{\beta} 
      \snorm{ Z^{(\ref{210929.2133})}_{N-1,+} }_{L^p(P)}
    +   
    1_{\cbr{\maxs{i}\leq\rze\leq s_i\leq\ron}}\;
    %1_{\{\bar s_{n,\hat i}\leq\rze]\}}1_{[\rze,\ron]}(s_i)
    \abs{\ron-\rze}^{2\beta} 
    \snorm{Z^{(\ref{210929.2133})}_{N-1,-} }_{L^p(P)} 
    }_{L^{1/H}([0,T]^N)}
    \\\leq&\;
      C_{N}^{(\ref{210929.2133})}
      \abs{\ron-\rze}^{3\beta}
    \\[10pt]
    \text{and for }&\; i\neq j=1,...,N,\\
    &
    \snorm{ 
      1_{[\rze,\ron]}(s_i) 1_{[\rze,s_i]}(s_j)\;
      \norm{D^{N-2}_{\son,..,\hat{s_i},..,\hat{s_j},..,s_N} \kerone{\rze}{s_j}}_{L^p(P)} 
      }_{L^{1/H}([0,T]^N)}
    \\\leq&\;
    \snorm{ 
      1_{\cbr{\rze<\maxs{i,j}\leq s_j\leq s_i \leq\ron}}
      \bnorm{Z^{(\ref{210929.2134})}_{N-2,+} }_{L^p(P)}
      +   
      1_{\cbr{\maxs{i, j}\leq\rze\leq s_j\leq s_i \leq\ron]}}
      \abs{\ron-\rze}^{\beta} 
      \bnorm{Z^{(\ref{210929.2134})}_{N-1,-}}_{L^p(P)} 
      }_{L^{1/H}([0,T]^N)}
    \\\leq&\;
    C_{N}^{(\ref{210929.2134})}
    \abs{\ron-\rze}^{3\beta}
  \end{align*}
  where  
  $Z^{(\ref{210929.2132})}_{N-1,+}$, $Z^{(\ref{210929.2132})}_{N-1,-}$,
  $Z^{(\ref{210929.2133})}_{N-1,+}$, $Z^{(\ref{210929.2133})}_{N-1,-}$,
  $Z^{(\ref{210929.2134})}_{N-2,+}$ and $Z^{(\ref{210929.2134})}_{N-2,-}$
  are random variables with finite moments of any order and
  $C_{N}^{(\ref{210929.2132})}$,
  $C_{N}^{(\ref{210929.2133})}$ and
  $C_{N}^{(\ref{210929.2134})}$ are constants.
  They are all independent of $\rze,\ron$.
  Thus we obtain
  \begin{align*}
    \norm{\norm{D^N R^{1}_{\rze,\ron}}_{\calh^{\otimes N}}}_{L^p(P)}
    &\simleq
    \norm{\norm{D^N_{s_1,...,s_N} R^{1}_{\rze,\ron}}_{L^p(P)}}_{L^{1/H}([0,T]^N)}
    \\&\leq 
    \rbr{
      C_{N}^{(\ref{210929.2131})}
      +N\, C_{N}^{(\ref{210929.2132})}
      +N\, C_{N}^{(\ref{210929.2133})}
      +N(N-1)\, C_{N}^{(\ref{210929.2134})}
    }
    \abs{\ron-\rze}^{3\beta}
  \end{align*}
  Here we used Lemma \ref{220412.1015}(ii).
  By similar arguments, we have 
    $\norm{R^{1}_{\rze,\ron}}_{L^p(P)}
    \simleq
    \abs{\ron-\rze}^{3\beta}$.
  Applying this estimate to $(r_0,r_1)=(\tjm,\tj)$,
  we obtain 
  \begin{align*}
    \sup_{j\in[n]} \snorm{R^{1}_{n,j}}_{i,p} =O(n^{-3\beta}).
  \end{align*}
  %$\sup_{j\in[n]} \snorm{R^{1}_{n,j}}_{i,p} =O(n^{-3\beta})$.

  The same line of arguments shows that 
  the estimates of $R^{2}_{\rze,\ron}$ and $R^{3}_{\rze,\ron}$ stand.
\end{proof}

\vsp
\begin{lemma}\label{210504.2151}
  %stochastic expansion \tilde N_n
  %quasi tangent, G^{(2)}
  %
  Let $f\in C^\infty(\bbR)$ a smooth function with bounded derivatives
  and $(X_t)_{t\in[0,T]}$ the solution of SDE (\ref{210430.1615}) under Assumption \ref{220404.1535}(i).
  %\grnb{Check whether $f^{[1]}(X_{\tjm})$ satisfies the condition for $A$
  %\redb{looks okay}}
  The following estimate holds as $n\to\infty$:
  \begin{align*}
  \int^T_0 \rbr{f(X_{t_{j_n(t)-1}}) -f(X_{t})}dt
  =O_M(n^{-1}),
  \end{align*}
  where $j_n(t)-1=[nt/T]$.
  \end{lemma}
  
  \begin{proof}
  \begin{align}
  \int^T_0\rbr{f(X_{t}) -f(X_{t_{j_n(t)-1}}) }dt
  =&\;
   \int^T_0\rbr{\int^t_{t_{j_n(t)-1}}f'(X_\tpr)V^{[1]}(X_\tpr) dB_\tpr}dt
  +\int^T_0\rbr{\int^t_{t_{j_n(t)-1}}f'(X_\tpr)V^{[2]}(X_\tpr) d\tpr}dt
  \nn\\=&\;
  S^{(\ref{220308.1500})}_{1,n}+S^{(\ref{220308.1500})}_{2,n}+S^{(\ref{220308.1500})}_{3,n},
  \label{220308.1500}
  \end{align}
  where we define the terms 
  $S^{(\ref{220308.1500})}_{1,n},S^{(\ref{220308.1500})}_{2,n}$ and 
  $S^{(\ref{220308.1500})}_{3,n}$ with the notation 
  $f^{[1]}(x)=f'(x)V^{[1]}(x)$ by
  \begin{align*}
  S^{(\ref{220308.1500})}_{1,n}=&
  \int^T_0 \rbr{\int^t_{t_{j(t)-1}} f^{[1]}(X_{t_{j(t)-1}}) dB_\tpr}dt
  =\sum_{j\in[n]} f^{[1]}(X_{\tjm})\int^\tj_\tjm(\tj-t) dB_t
  \\
  S^{(\ref{220308.1500})}_{2,n}=&
  \int^T_0\rbr{\int^t_{t_{j(t)-1}} \rbr{f^{[1]}(X_\tpr) -f^{[1]}(X_{t_{j(t)-1}})} dB_\tpr} dt
  =\sum_{j\in[n]} \int^\tj_\tjm 
  \rbr{\int^t_{\tjm} \rbr{f^{[1]}(X_\tpr) -f^{[1]}(X_{\tjm})} dB_\tpr} dt
  \\
  S^{(\ref{220308.1500})}_{3,n}=&
  \int^T_0 \rbr{\int^t_{t_{j(t)-1}}f'(X_\tpr)V^{[2]}(X_\tpr)\, d\tpr} dt.
  \end{align*}
  The functional $S^{(\ref{220308.1500})}_{1,n}$ can be written as 
  \begin{align*}
    S^{(\ref{220308.1500})}_{1,n} = n^{-1} \sum_{j\in[n]} f^{[1]}(X_{\tjm}) I_1(f_j)
  \end{align*}
  with 
  $f_j =\rbr{1_{j}(t)\; n(\tj-t)}_{t\in[0,T]}$ for $j\in[n]$.
  By Proposition \ref{210823.2400}, % and {the remark(!!!)},%%MODIF%%
  we have
  $S^{(\ref{220308.1500})}_{1,n} =O_M(n^{-1})$.
  %\comm{this is sharp}
  Using the same arguments as in the proof of Lemma \ref{210315.1401} (ii),
  we have 
  $S^{(\ref{220308.1500})}_{2,n} =O_M(n^{-2\beta})$ for any $\beta\in(1/2,H)$ and 
  $S^{(\ref{220308.1500})}_{3,n} = O_M(n^{-1})$.
  %\comm{this is not sharp but enough for here}
  \end{proof}

%%\newpage
\vspsm
%\subsubsection{Estimates of terms related to $D_{u_n} M_n$}
We proceed to estimate terms related to $D_{u_n} M_n$.
Recall 
\begin{align*}  
  \rho_H(k) = \abr{1_{[0,1]}, 1_{[k,k+1]}}_\calh
  = \frac12 \rbr{\abs{k+1}^{2H} + \abs{k-1}^{2H} - 2\abs{k}^{2H}}
\end{align*}
for $k\in\bbZ$ and
$c_H^2 =\sum_{k\in\bbZ} \rho_H(k)^2$.

\begin{lemma}\label{210427.2310}
  There exists some $\kappa>0$ such that
\begin{align*}
n^{-1}\sum_{j\in[n]}
\bbabs{\sum_{k\in\bbZ: j+k\in[n]}\rho_H(k)^2-c_H^2}
=O(n^{(2H-3/2)(1+\kappa)})
\quad\tassm n\to\infty.
\end{align*}
\end{lemma}

\begin{proof}
For $\epsilon\in(0,1)$, the sum in question can be bounded as
  \begin{align*}
&\quad n^{-1}\sum_{j\in[n]}\bbabs{\sum_{j+k\in[n]}\rho_H(k)^2-c_H^2}
\\&=
n^{-1}\sum_{1+n^\epsilon \leq j\leq n-n^\epsilon}
\bbabs{\sum_{j+k\in[n]}\rho_H(k)^2 -c_H^2}
+n^{-1}\sum_{j<1+n^\epsilon \torsm n-n^\epsilon<j}
\abs{\sum_{j+k\in[n]}\rho_H(k)^2 -c_H^2}
\\&=
O(n^{\epsilon(4H-3)}+n^{-1}\; n^\epsilon).
\end{align*}
By taking $\epsilon$ satisfying $1/2<\epsilon<2H-1/2$, 
we have $(\epsilon(4H-3))\vee(\epsilon-1)<2H-3/2$.
\end{proof}

We often need the estimate of the difference of 
the functionals %$\cali^{M(2,0)}_{1,1,n}$ and $G_\infty$ defined at 
%(\ref{220301.1029}) and (\ref{220301.1030}), respectively.
\begin{align*}
  \cali^{M(2,0)}_{1,1,n}
  = 2 n^{4H-1} \sum_{j\in[n]^2} 
  a(X_{t_{j_1-1}}) a(X_{t_{j_2-1}})\,
  %a_{t_{j_1-1}} a_{t_{j_2-1}} 
  \beta_{j_1,j_2}^2
  \wtand
  G_\infty =2 c_H^2 T ^{4H-1}\int^T_0 (V^{[1]}(X_{t}))^4 dt.
\end{align*}
\begin{lemma} \label{220309.2200}
  For the solution $(X_t)_{t\in[0,T]}$ of SDE (\ref{210430.1615}) satisfying Assumption \ref{220404.1535}(i),
  consider the following functional
  \begin{align}
    S^{(\ref{211010.1630})}_n:=
    \cali^{M(2,0)}_{1,1,n}-G_\infty&=
    2 n^{4H-1} \sum_{j\in[n]^2} a_{t_{j_1-1}} a_{t_{j_2-1}} \beta_{j_1,j_2}^2
    - 2\, c_H^2\, T^{4H-1} \int^T_0 a_t^2 dt.
    \label{211010.1630}
  \end{align}
  Then, 
  $S^{(\ref{211010.1630})}_n =o_M(n^{2H - \frac{3}{2}})$.
\end{lemma}
\begin{proof}
  The functional $S^{(\ref{211010.1630})}_n$ decomposes as
  \begin{align*}
    S^{(\ref{211010.1630})}_n
    &=
    %2 n^{4H-1} \sum_{j_1,j_2} a_{t_{j_1-1}} a_{t_{j_2-1}} \beta_{j_1,j_2}^2
    %-2 c_H^2 T^{4H-1} \int^T_0 a_t^2 dt
    %\nn\\&=
    2\, T^{4H} \rbr{S^{(\ref{211010.1630})}_{1,n}
    +S^{(\ref{211010.1630})}_{2,n} 
    +S^{(\ref{211010.1630})}_{3,n}},
  \end{align*}
where
\begin{align*}
  S^{(\ref{211010.1630})}_{1,n}&=
  n^{-1} \sum_{\substack{j\in[n]\\j+k\in[n]}}
  %n^{-1} \sum_{\substack{1\leq j_1\leq n\\1\leq j_1+j\leq n}}
  a_{t_{j-1}} a_{t_{j+k-1}} \rho_H(k)^2
  - n^{-1} \sum_{\substack{j\in[n]\\j+k\in[n]}}
  %- n^{-1} \sum_{\substack{1\leq j_1\leq n\\1\leq j_1+j\leq n}}
  a_{t_{j-1}}^2 \rho_H(k)^2
  \\
  S^{(\ref{211010.1630})}_{2,n}&=
    n^{-1} \sum_{j\in[n]}
    %n^{-1} \sum_{\substack{1\leq j_1\leq n\\1\leq j_1+j\leq n}}
      a_{t_{j-1}}^2 \sum_{j+k\in[n]}\rho_H(k)^2
    - c_H^2\; n^{-1} \sum_{j\in[n]} a_{t_{j-1}}^2
    %- c_H^2 n^{-1} \sum_{1\leq j\leq n} a_{t_{j-1}}^2
  \\
  S^{(\ref{211010.1630})}_{3,n}&=
    c_H^2\; n^{-1} \sum_{j\in[n]} a_{t_{j-1}}^2
    %c_H^2 n^{-1} \sum_{1\leq j\leq n} a_{t_{j-1}}^2
    - c_H^2T^{-1}\int^T_0 a_t^2 dt.
\end{align*}
For $i\in\bbN$, $p>1$ and $\beta\in(1/2,H)$,
there exists a constant $C_{i,p,\beta}$ 
independent of $n$, $j$ and $k$ such that
%$
%a(X_{t_{j+i-1}})-a(X_{\tjm})
%=O_M((i/n)^H)$,
\begin{align*}
\snorm{ a(X_{t_{j+k-1}})-a(X_{t_{j-1}}) }_{i,p}
\leq C_{i,p,\beta} (\abs{k}/n)^\beta
\end{align*}
by Lemma \ref{210504.2200} (i).
By the fact 
$\rho_H(k)=O(\abs{k}^{2H-2})$ as $k\to\infty$,
we have the following estimate
\begin{align*}
  \snorm{S^{(\ref{211010.1630})}_{1,n}}_{i,p}
  \simleq
  n^{-1}\sum_{\substack{j\in[n]\\j+k\in[n]}}
  \abs{k}^{4H-4} \; (\abs{k}/n)^\beta
  =O(n^{(-\beta)\vee(4H-3)})
  = o(n^{2H-\frac{3}{2}}).
\end{align*}
By Lemma \ref{210427.2310}, we have
$S^{(\ref{211010.1630})}_{2,n}= o_M(n^{2H-3/2})$
since 
$\sup_{t\in[0,T]}\norm{a_t}_{k,p}<\infty$ for any $k\geq1$ and $p>1$.
By Lemma \ref{210504.2151},
$S^{(\ref{211010.1630})}_{3,n}=O_M(n^{-1})$. 
Hence, we obtain 
$S^{(\ref{211010.1630})}_{n}= o_M(n^{2H - \frac{3}{2}})$.
\end{proof}

%%\newpage
\vspsm
\subsection{Lemmas related to the random symbols}%%MODIF%%
Recall that
%$X_t$ is the solution to SDE (\ref{210430.1615}) and
we define
$a(x)=(V^{[1]}(x))^2$,
$g(x)=2c_H^2\; T^{4H-1}\; (V^{[1]}(x))^4$ and
we write $a_t=a(X_t)$.
Similarly we write
$f_t := f(X_t)$ and $f'_t := f'(X_t)$ for a generic function $f$.
For $t\in\clop{0,T}$, we write $j_n(t)$ or $j(t)$ for
$j\in[n]$ such that $t\in\clop{\tjm,\tj}$.
By convention, we set $j(T)=n$.

Define the functions 
$\rho_\tau$ ($\tau\in[0,T]$) and $\rho_{n,j}$ ($j\in[n]$) on $[0,T]$ 
\begin{align*}
  \rho_{\tau}(s)= 
  \alpha_H\,T\; \abs{s-\tau}^{2H-2}
  \quad\tand\quad
  \rho_{n,j}(s)=%\rho_{j}(s)=
  \alpha_H\,n \int_{s'\in I_{j}} \abs{s-s'}^{2H-2} ds'
  =\rbr{\frac{T}{n}}^{-1}\int_{s'\in I_{j}} \rho_{s'}(s)ds'.
\end{align*}
for $s\in[0,T]$.
%\redb{For $\omega$-almost surely,(ToDelete)}
Define the ($\omega$-dependent) functions
$\dot d_{n,j}$ and $\ddot d_{n,j}$ on $[0,T]$
for $j\in[n]$  by
%for $n\in\bbN$ and $j\in[n]$ 
\begin{align*}
  \dotd_{n,j}(t)
  &:= n\abr{1_{j},DX_t}_\calh
  =\int_{[0,T]}D_sX_t\; \rho_{n,j}(s) ds,
  \\
  \ddotd_{n,j}(t)
  &:= n^2\abr{1_{j}^{\otimes 2},D^2X_t}_{\calh^{\otimes 2}}
  =\int_{[0,T]^2}D^2_{\son,\stw} X_{t}\; \rho_{n,j}(\son)\rho_{n,j}(\stw)d\son d\stw
\end{align*}
and the functions
$\dotd_{n}$ and $\ddotd_{n}$ on $[0,T]^2$ by
\begin{align*}
  \dotd_{n}(\tau,t)=\dotd_{n,j(\tau)}(t),\quad
  \ddotd_{n}(\tau,t)=\ddotd_{n,j(\tau)}(t).
\end{align*}
Recall that the functions
$\dotd$ and $\ddotd$ on $[0,T]^2$ 
are defined by 
\begin{align*}
  \dotd(\tau,t)&:=\int_0^T D_sX_t\; \rho_{\tau}(s) ds,
  \qquad\qquad
  \ddotd(\tau,t):=
  \int_{[0,T]^2}D^2_{\son,\stw} X_{t}\; \rho_{\tau}(\son)\rho_{\tau}(\stw)d\son d\stw.
\end{align*}
The asymptotic variance $G_\infty$ was defined at (\ref{220301.1030}):
\begin{align*}
  G_\infty =2 c_H^2 T ^{4H-1}\int^T_0 (V^{[1]}(X_{t}))^4 dt.
\end{align*}

\subsubsection{About the coefficients of quasi-torsion}
Set the finite measures $\mu_n$ and $\mu_\infty$ on $[0,T]^2$ by 
\begin{align*}
  \mu_{n}(ds_1,ds_2)&=
  2\, n^{4H-1}\sum_{j\in[n]^2}
  \beta_{j_1,j_2}^2\; \delta_{(t_{j_1-1},t_{j_2-1})}(ds_1,ds_2)
  =
  2\, T^{4H}n^{-1}
  \sum_{j\in[n]^2}
  \rho_H(\jtw-\jon)^2\; \delta_{(t_{j_1-1},t_{j_2-1})}(ds_1,ds_2)
  \\
  \mu_{\infty}(ds_1,ds_2)&
  = 2\, T^{4H-1}c_H^2\; ds_1\,\delta_{s_1}(ds_2)
  = 2\, T^{4H-1}c_H^2\; \phi_*\mu_{[0,T]},
\end{align*}
where $\phi_*\mu_{[0,T]}$ is the pushforward measure of 
the Lebesgue measure $\mu_{[0,T]}$ on $[0,T]$ by
$\phi:[0,T]\to[0,T]^2$ defined by $\phi(s)=(s,s)$.
Define the ($\omega$-dependent) functions $\qtorker{1}{n}$ and $\qtorker{2,k}{n}$ ($k=1,2,3$) 
on $[0,T]^2$ by
%and the functionals $\qtorker{1}{n}$ and $\qtorker{2,k}{n}$ ($k=1,2,3$) 
%dependent on $s_1,s_2\in[0,T]$ by
%%%%this is a wrong way to see $\qtorker{1}{n}$
\begin{align*}
  \qtorker{1}{n}(\son,\stw) &:=
  %K^{(1)}_n(\son,\stw) &:=
  n\sum_{j\in[n]} 2^{-1}\rbr{D_{1_{j}}G_\infty}
  \rbr{D_{1_{j}}a_{\son}}a_{t_{j-1}}a_{\stw},
  &
  \qtorker{2,1}{n}(\son,\stw) &:=
  n \sum_{j\in[n]} (D^2_{1_{j}^{\otimes 2}}a_{\son}) a_{t_{j-1}} a_{\stw},
  \\
  \qtorker{2,2}{n}(\son,\stw) &:=
  n \sum_{j\in[n]}(D_{1_{j}}a_{\son})(D_{1_{j}}a_{t_{j-1}})a_{\stw},
  &
  \qtorker{2,3}{n}(\son,\stw) &:=
  n \sum_{j\in[n]} (D_{1_{j}}a_{\son}) a_{t_{j-1}} (D_{1_{j}}a_{\stw}).
\end{align*}
Let $\qtorker{2}{n} = \sum_{k=1,2,3} \qtorker{2,k}{n}$.
The functional $\bar{\mS}_{k,n}^{(3,0)}$ defined at
(\ref{220403.2145}) and (\ref{220403.2146})
are written as
\begin{align*}
  %\comm{2^{-1}}
  %\comm{r_n^{-1}\; \bar\cali_{1,1,k,n}^{M(3,0)}}
  \bar{\mS}_{k,n}^{(3,0)}&=
  \int_{[0,T]^2} \mu_{n}(ds_1,ds_2)\qtorker{k}{n}(s_1,s_2)
  \hsp\tfor k=1,2.
\end{align*}

We can write 
%\grnb{Check the constant\tto seems okay}
\begin{align*}
  \qtorker{1}{n}(\son,\stw) &=
  %n\sum_{j\in[n]} 2^{-1}\rbr{D_{1_{j}}G_\infty}\rbr{D_{1_{j}}a_{\son}}a_{t_{j-1}}a_{\stw}  
  %=
  \frac{1}{2T} \int_0^T %\int_{\tau\in[0,T]}
  \rbr{\int_0^T g'_t \dotd_{n}(\tau,t) dt}
  \brbr{a'_{\son}\dotd_{n}(\tau,\son)}
  a_{t_{j(\tau)-1}}a_{\stw}d\tau
  \\
  \qtorker{2,1}{n}(\son,\stw) &=
  %n \sum_{j\in[n]} (D_{1_{j}}D_{1_{j}}a_{\son}) a_{t_{j-1}} a_{\stw}
  %=
  \frac{1}{T} \int_0^T %\int_{\tau\in[0,T]}
  \brbr{a''_{\son}\;\dotd_{n}(\tau,\son)^2  +a'_{\son}\ddotd_{n}(\tau,\son)}
  a_{t_{j(\tau)-1}} a_{\stw} d\tau
  \\
  \qtorker{2,2}{n}(\son,\stw) &=
  %n \sum_{j\in[n]}(D_{1_{j}}a_{\son})(D_{1_{j}}a_{t_{j-1}})a_{\stw}
  %=
  \frac{1}{T} \int_0^T %\int_{\tau\in[0,T]}
  \brbr{a'_{\son} \dotd_{n}(\tau,\son)}
  \brbr{a'_{t_{j(\tau)-1}} \dotd_{n}(\tau,t_{j(\tau)-1})}
  a_{\stw} d\tau
  \\
  \qtorker{2,3}{n}(\son,\stw) &=
  %n \sum_{j\in[n]} (D_{1_{j}}a_{\son}) a_{t_{j-1}} (D_{1_{j}}a_{\stw})
  %=
  \frac{1}{T} \int_0^T %\int_{\tau\in[0,T]}
  \brbr{a'_{\son} \dotd_{n}(\tau,\son)}
  \brbr{a'_{\stw} \dotd_{n}(\tau,\stw)}
  a_{t_{j(\tau)-1}} d\tau
\end{align*}
and we define functions 
$\qtorker{1}{\infty}$ and $\qtorker{2,k}{\infty}$ ($k=1,2,3$) 
on $[0,T]^2$ by
\begin{align*}
  \qtorker{1}{\infty}(\son,\stw) &:=
  \frac{1}{2T} \int_0^T
  \bbrbr{\int_0^T g'(X_t) \dotd(\tau,t) dt}
  \brbr{a'(X_{\son})\dotd(\tau,\son)}
  a_{\tau}a_{\stw}d\tau
  \\
  \qtorker{2,1}{\infty}(\son,\stw) &:=
  \frac{1}{T} \int_0^T
  \brbr{a''(X_{\son})\;\dotd(\tau,\son)^2 +a'(X_{\son})\;\ddotd(\tau,\son) }
  a_{\tau} a_{\stw} d\tau
  \\
  \qtorker{2,2}{\infty}(\son,\stw) &:=
  \frac{1}{T} \int_0^T
  \brbr{a'(X_{\son})\dotd(\tau,\son)}
  \brbr{a'(X_{\tau})\dotd(\tau,\tau)}
  a_{\stw}d\tau
  \\
  \qtorker{2,3}{\infty}(\son,\stw) &:=
  \frac{1}{T} \int_0^T
  \brbr{a'(X_{\son})\dotd(\tau,{\son})}
  \brbr{a'(X_{\stw})\dotd(\tau,{\stw})}
  a_{\tau} d\tau.
\end{align*}
We write 
$\qtorker{2}{\infty} = \sum_{k=1,2,3} \qtorker{2,k}{\infty}$.
Notice that %$\mS^{(3,0)}(\tti\sfz)$ \comm{koko}
\begin{align*}
  %\comm{2^{-1}}
  %\comm{r_n^{-1}\; \bar\cali_{1,1,k,n}^{M(3,0)}}
  \mS_{k,\infty}^{(3,0)}&=
  \int_{[0,T]^2} \mu_{\infty}(ds_1,ds_2)\, \qtorker{k}{\infty}(s_1,s_2)
  \hsp\tfor k=1,2,
\end{align*}
where $\mS_{k,\infty}^{(3,0)}$ are defined
in page \pageref{220308.2520}.

\begin{lemma}\label{210423.2110}
(i) For $k=1,2$, the following convergence holds almost surely as $n\to\infty$:
\begin{align*}
  \sup_{(\son,\stw)\in[0,T]^2}
  \abs{\qtorker{k}{n}(\son,\stw) -\qtorker{k}{\infty}(\son,\stw)}
  \to 0
\end{align*}

\item (ii) For $k=1,2$, 
$\qtorker{k}{\infty}$ 
%$\qtorker{k}{\infty}(\son,\stw)$ 
is a continuous function on $[0,T]^2$ almost surely.
\end{lemma}

\begin{proof}(i)
Consider $\qtorker{1}{n}$ and $\qtorker{1}{\infty}$.
The following inequality holds for
$\son,\stw\in[0,T]$ and $n\in\bbN$ almost surely:
\begin{align}
&2T\abs{\qtorker{1}{n}(\son,\stw)-\qtorker{1}{\infty}(\son,\stw)}
\nn\\=&\;
\abs{\int_{0}^T
\rbr{\int_0^T g'_t\;\dotd(\tau,t) dt}
\rbr{(a'_{\son}\;\dotd(\tau,\son)} a_{\tau}a_{\stw}d\tau
-\int_0^T
\rbr{\int_0^T g'_t\;\dotd_n(\tau,t) dt}
\rbr{(a'_{\son}\;\dotd_n(\tau,\son)}
a_{t_{j(\tau)-1}}a_{\stw}d\tau}
\nn\\\leq&\;
\abs{a'_{\son}}\abs{a_{\stw}}
\int_{0}^T\abs{
\rbr{\int_0^T g'_t\; \dotd(\tau,t) dt} \dotd(\tau,\son)
\rbr{a_{\tau}-a_{t_{j(\tau)-1}}}} d\tau
\nn\\&+
\abs{a'_{\son}}\abs{a_{\stw}}
\int_{0}^T\abs{
\rbr{\int_0^T g'_t\;\dotd(\tau,t) dt}
\rbr{\dotd(\tau,\son)-\dotd_n(\tau,\son)}
a_{t_{j(\tau)-1}}}d\tau
\nn\\&+
\abs{a'_{\son}}\abs{a_{\stw}}
\int_{0}^T \abs{
\rbr{\int_0^T g'_t\;\rbr{\dotd(\tau,t) - \dotd_n(\tau,t)} dt}
\dotd_n(\tau,\son) a_{t_{j(\tau)-1}}} d\tau
\nn\\\leq&\;
T\;
\sup_{t\in[0,T]}\abs{a'_{t}}\; \sup_{t\in[0,T]}\abs{a_{t}}\; 
\sup_{t\in[0,T]}\abs{g'_t}\;
\sup_{\tau,t\in[0,T]}\abs{\dotd(\tau,t)}^2\;
\int_{0}^T \abs{a_{\tau}-a_{t_{j(\tau)-1}}}d\tau
\nn\\&+
T^2\;
\sup_{t\in[0,T]}\abs{a'_{t}}\; \sup_{t\in[0,T]}\abs{a_{t}}^2\;
\sup_{t\in[0,T]}\abs{g'_t}\;
\sup_{\tau,t\in[0,T]}\abs{\dotd(\tau,t)}\;
\sup_{\tau,t\in[0,T]}\abs{\dotd(\tau,t) -\dotd_n(\tau,t)}
\nn\\&+
T^2\;
\sup_{t\in[0,T]}\abs{a'_{t}} \sup_{t\in[0,T]}\abs{a_{t}}^2
\sup_{t\in[0,T]}\abs{g'_{t}}
\sup_{n\in\bbN,\tau,t\in[0,T]} \abs{\dot d_n(\tau,t)}\;
\sup_{\tau,t\in[0,T]} \abs{\dotd(\tau,t) - \dot d_n(\tau,t)}
\label{210506.1804}
\end{align}
The last sum of the inequality,
which are independent of $s_1$ and $s_2$,
converges to $0$
by (\ref{210506.1501}), (\ref{210506.1511}) and (\ref{210506.1517})
%\comm{Check 220303.1300-2.}
Hence, we obtain
\begin{align*}
  \sup_{(\son,\stw)\in[0,T]^2}\abs{\qtorker{1}{n}(\son,\stw)-\qtorker{1}{\infty}(\son,\stw)}
  \to 0
  \qquad\text{as $n\to\infty$ almost surely.}
\end{align*}
%as $n$ goes to $\infty$ almost surely.

For $\qtorker{2,1}{n}$ and $\qtorker{2,1}{\infty}$, 
the following inequality holds for
$\son,\stw\in[0,T]$ and $n\in\bbN$ almost surely:
\begin{align*}
&T\abs{\qtorker{2,1}\infty(\son,\stw)-\qtorker{2,1}{n}(\son,\stw)}
\\=&\;
\abs{\int_{0}^T
\rbr{a''_{\son}\;\dotd(\tau,{\son})^2 +a'_{\son}\;\ddotd(\tau,{\son})}
a_{\tau}a_{\stw}d\tau
-\int_0^T
\rbr{a''_{\son}\;\dotd_n(\tau,{\son})^2+a'_{\son}\;\ddotd_n(\tau,{\son})}
a_{t_{j(\tau)-1}}a_{\stw}d\tau}
%\\\leq&
%\abs{a''_{\son}} \abs{a_{\stw}}
%\abs{\int_{0}^T
%\rbr{\dotd(\tau,{\son})^2 a_{\tau}-\dotd_n(\tau,{\son})^2 a_{t_{j(\tau)-1}}} d\tau}
%\\&+
%\abs{a'_{\son}} \abs{a_{\stw}}
%\abs{\int_{0}^T \ddotd(\tau,{\son}) a_{\tau}d\tau
%-\int_0^T \ddotd_n(\tau,{\son}) a_{t_{j(\tau)-1}} d\tau}
\\\leq&
\abs{a''_{\son}} \abs{a_{\stw}}\rbr{
\abs{\int_{0}^T \dotd(\tau,{\son})^2
\rbr{a_{\tau} -a_{t_{j(\tau)-1}}} d\tau}
+\abs{\int_0^T \rbr{\dotd(\tau,{\son})^2 -\dot d_n(\tau,{\son})^2}
a_{t_{j(\tau)-1}}d\tau}}
\\&+
\abs{a'_{\son}} \abs{a_{\stw}}\rbr{
\abs{\int_{0}^T\ddotd(\tau,{\son})\rbr{a_{\tau}-a_{t_{j(\tau)-1}}}d\tau}
+\abs{\int_0^T \rbr{\ddotd(\tau,{\son}) -\ddot d_n(\tau,{\son})}
a_{t_{j(\tau)-1}}d\tau}}
\\\leq&
\sup_{t\in[0,T]}\abs{a''_{t}} \sup_{t\in[0,T]}\abs{a_{t}}
\sup_{\tau,t\in[0,T]}\abs{\dotd(\tau,{t})}^2
\int_{0}^T \abs{a_{\tau} -a_{t_{j(\tau)-1}}} d\tau
\\&+
T\;
\sup_{t\in[0,T]}\abs{a''_{t}} \sup_{t\in[0,T]}\abs{a_{t}}^2\;
\rbr{\sup_{\tau,\son\in[0,T]}\abs{\dotd(\tau,\son)}
+\sup_{n\in\bbN,\tau,\son\in[0,T]}\abs{\dot d_n(\tau,\son)}}
\sup_{\tau,\son\in[0,T]}\abs{\dotd(\tau,\son) -\dot d_n(\tau,\son)}
\\&+
\sup_{t\in[0,T]}\abs{a'_{t}} \sup_{t\in[0,T]}\abs{a_{t}}
\sup_{\tau,t\in[0,T]}\abs{\ddotd(\tau,{t})}
\int_{0}^T\abs{a_{\tau}-a_{t_{j(\tau)-1}}}d\tau
\\&+
T\;
\sup_{t\in[0,T]}\abs{a'_{t}} \sup_{t\in[0,T]}\abs{a_{t}}^2
\sup_{\tau,\son\in[0,T]} \abs{\ddotd(\tau,\son) -\ddot d_n(\tau,\son)}
\end{align*}
%The last sum of the inequality are independent of $s_1$ and $s_2$.
By (\ref{210506.1501}) (\ref{210506.1511}), (\ref{210506.1503}), 
(\ref{210506.1512}), (\ref{210506.1517}) and the continuity of $X_t$ in $t$, 
the last sum converges to $0$ almost surely.
%\comm{Check 220304.1000-1.}
Hence, we obtain
\begin{align*}
  \sup_{(\son,\stw)\in[0,T]^2}
  \abs{\qtorker{2,1}{n}(\son,\stw)-\qtorker{2,1}{\infty}(\son,\stw)}
  \to 0
  \qquad\text{as $n\to\infty$ almost surely.}
\end{align*}
%as $n$ goes to $\infty$ almost surely.

\vspssm
For $\qtorker{2,2}{n}$ and $\qtorker{2,2}{\infty}$, 
we have the following inequality for 
$\son,\stw\in[0,T]$ and $n\in\bbN$ almost surely:
\begin{align*}
&T\abs{\qtorker{2,2}{\infty}(\son,\stw)-\qtorker{2,2}{n}(\son,\stw)}
\\=&\;
\abs{\int_{0}^T \rbr{a'_{\son}\; \dotd(\tau,\son)}
\rbr{a'_{\tau}\; \dotd(\tau,\tau)}a_{\stw}d\tau
-\int_{0}^T \rbr{a'_{\son}\; \dotd_n(\tau,\son)}
\rbr{a'_{t_{j(\tau)-1}}\; \dotd_n(\tau,t_{j(\tau)-1})}a_{\stw}d\tau}
%\\\leq&\;
%\abs{a'_{\son}} \abs{a_{\stw}} 
%\abs{\int_{0}^T \dotd(\tau,\son)
%\rbr{a'_{\tau}\;\dotd(\tau,\tau)}d\tau
%- \int_{0}^T \dotd_n(\tau,\son)
%\rbr{a'_{t_{j(\tau)-1}}\; \dotd_n(\tau,t_{j(\tau)-1})}d\tau}
\\\leq&\;
\abs{a'_{\son}}\abs{a_{\stw}}\bigg\{
\abs{\int_{0}^T \dotd(\tau,\son)
\rbr{a'_{\tau} -a'_{t_{j(\tau)-1}}} \dotd(\tau,\tau) d\tau}
+
\abs{\int_{0}^T
\rbr{\dotd(\tau,\son)-\dotd_n(\tau,\son)}
a'_{t_{j(\tau)-1}}\; \dotd(\tau,\tau) d\tau}
\\&\hsp\hsp+
\abs{\int_{0}^T \dotd_n(\tau,\son) a'_{t_{j(\tau)-1}}
\rbr{\dotd(\tau,\tau)-\dotd_n(\tau,t_{j(\tau)-1})}d\tau}\bigg\}
%\\&+
%\abs{\int_{0}^T \dotd_n(\tau,\son) a'_{t_{j(\tau)-1}}
%\rbr{\dotd(\tau,\tau)-\dotd(\tau,t_{j(\tau)-1})}d\tau}
%+
%\abs{\int_{0}^T \dot d_n(\tau,\son) a'_{t_{j(\tau)-1}}
%\rbr{\dotd(\tau,t_{j(\tau)-1})-\dot d_n(\tau,t_{j(\tau)-1})}d\tau}\bigg\}
\\\leq&
\sup_{t\in[0,T]} \abs{a'_{t}} \sup_{t\in[0,T]}\abs{a_{t}}
\sup_{\tau,\son\in[0,T]}\abs{\dotd(\tau,\son)}^2
\int_{0}^T \abs{a'_{\tau}-a'_{t_{j(\tau)-1}}}d\tau
\\&+
T\;\sup_{t\in[0,T]}\abs{a'_{t}}^2 
\sup_{t\in[0,T]}\abs{a_{t}}
\sup_{\tau,t\in[0,T]} \abs{\dotd(\tau,t)}
\sup_{\tau,\son\in[0,T]} \abs{\dotd(\tau,\son) -\dot d_n(\tau,\son)}
\\&+
\sup_{t\in[0,T]}\abs{a'_{t}}^2 
\sup_{t\in[0,T]}\abs{a_{t}}
\sup_{n\in\bbN,\tau,\son\in[0,T]}\abs{\dot d_n(\tau,\son)}
\int_{0}^T \abs{\dotd(\tau,\tau)-\dotd_n(\tau,t_{j(\tau)-1})} d\tau
%\\&+
%\sup_{t\in[0,T]}\abs{a'_{t}}^2\, \sup_{t\in[0,T]}\abs{a_{t}}
%\sup_{n\in\bbN,\tau,\son\in[0,T]}\abs{\dot d_n(\tau,\son)}
%\int_{0}^T \abs{\dotd(\tau,\tau)-\dotd(\tau,t_{j(\tau)-1})} d\tau
%\\&+
%T\;\sup_{t\in[0,T]}\abs{a'_{t}}^2\, \sup_{t\in[0,T]}\abs{a_{t}}
%\sup_{n\in\bbN,\tau,\son\in[0,T]}\abs{\dot d_n(\tau,\son)}
%\sup_{\tau,t\in[0,T]} \abs{\dotd(\tau,t) -\dotd_n(\tau,t)}
\end{align*}
%The last sum of the inequality are independent of $s_1$ and $s_2$.
By (\ref{210506.1501}), (\ref{210506.1517}), (\ref{210506.1511}), (\ref{220308.1200})
and the continuity of $X_t$ in $t$, 
the last sum converges to $0$ almost surely.
%\comm{Check 220305.1300-2,3.}
Hence, we obtain
\begin{align*}
  \sup_{(\son,\stw)\in[0,T]^2}
  \abs{\qtorker{2,2}{n}(\son,\stw)-\qtorker{2,2}{\infty}(\son,\stw)}
  \to 0
  \qquad\text{as $n\to\infty$ almost surely.}
\end{align*}

For $\qtorker{2,3}{n}$ and $\qtorker{2,3}{\infty}$, 
we have the following inequality for
$\son,\stw\in[0,T]$ and $n\in\bbN$ almost surely:
\begin{align*}
&T\abs{\qtorker{2,3}\infty(\son,\stw)-\qtorker{2,3}{n}(\son,\stw)}
\\=&
\abs{\int_{0}^T \rbr{a'_{\son}\;\dotd(\tau,{\son})}
a_{\tau}\rbr{a'_{\stw} \dotd(\tau,{\stw})} d\tau
- \int_{0}^T \rbr{a'_{\son}\;\dotd_n(\tau,\son)} 
a_{t_{j(\tau)-1}} \rbr{a'_{\stw}\;\dotd_n(\tau,\stw)} d\tau}
%\\\leq&
%\abs{a'(X_{\son})a'(X_{\stw})}
%\abs{\int_{0}^T \dotd(\tau,{\son}) a_{\tau} \dotd(\tau,{\stw}) d\tau
%- \int_{0}^T \dot d_n(\tau,\son) a_{t_{j-1}} \dot d_n(\tau,\stw) d\tau}
\\\leq&
\abs{a'_{\son}}\abs{a'_{\stw}}
\abs{\int_{0}^T \dotd(\tau,{\son}) \rbr{a_{\tau} -a_{t_{j(\tau)-1}}} \dotd(\tau,{\stw}) d\tau}
\\&+
\abs{a'_{\son}} \abs{a'_{\stw}}
\abs{\int_{0}^T \rbr{\dotd(\tau,{\son})-\dotd_n(\tau,\son)} a_{t_{j(\tau)-1}} \dotd(\tau,\stw) d\tau}
\\&+
\abs{a'_{\son}}\abs{a'_{\stw}}
\abs{\int_{0}^T \dotd_n(\tau,\son) a_{t_{j(\tau)-1}}
\rbr{\dotd(\tau,{\stw})-\dotd_n(\tau,\stw)} d\tau}
\\\leq&
\sup_{t\in[0,T]}\abs{a'_{t}}^2
\sup_{\tau,\son\in[0,T]}\abs{\dotd(\tau,\son)}^2
\int_{0}^T \abs{a_{\tau} -a_{t_{j(\tau)-1}}} d\tau 
\\&+
T\;\sup_{t\in[0,T]}\abs{a'_{t}}^2
\sup_{\tau,\son\in[0,T]}\abs{\dotd(\tau,\son)}
\sup_{t\in[0,T]} \abs{a_{t}} 
\sup_{\tau,t\in[0,T]} \abs{\dotd(\tau,t) -\dotd_n(\tau,t)}
\\&+
T\;\sup_{t\in[0,T]}\abs{a'_{t}}^2
\sup_{n\in\bbN,\tau,\son\in[0,T]}\abs{\dotd_n(\tau,\son)}
\sup_{t\in[0,T]}\abs{a_{t}}
\sup_{\tau,t\in[0,T]} \abs{\dotd(\tau,t) -\dotd_n(\tau,t)}
\end{align*}
%Again, the last sum of the inequality are independent of $s_1$ and $s_2$.
By (\ref{210506.1501}) (\ref{210506.1511}), (\ref{210506.1517}) 
and the continuity of $X_t$ in $t$, 
the last sum converges to $0$ almost surely.
Hence, we obtain
\begin{align*}
  \sup_{(\son,\stw)\in[0,T]^2}
  \abs{\qtorker{2,3}{n}(\son,\stw)-\qtorker{2,3}{\infty}(\son,\stw)}
  \to 0
  \qquad\text{as $n\to\infty$ almost surely.}
\end{align*}

\item(ii)
We can see that $\qtorker{1}{n}$ and $\qtorker{2}{n}$
are continuous functions on $[0,T]^2$ for fixed $\omega$, $n\in\bbN$.
This follows from the continuity of $\dotd_{n,j}$ and $\ddotd_{n,j}$ ($j\in[n]$).
(See Lemma \ref{220415.1540} (3).)
Since the convergence of $\qtorker{k}{n}$ to $\qtorker{k}{\infty}$ ($k=1,2$) 
is uniform on $[0,T]^2$,
the limit $\qtorker{k}{\infty}$ 
is also continuous on $[0,T]^2$ almost surely.
\end{proof}

\begin{lemma}\label{220307.1330}
  For $k=1,2$,
  %\comm{$r_n^{-1}\bar\cali^{M(3,0)}_{1,1,k,n}$}
  $\bar{\mS}_{k,n}^{(3,0)}$
  converges to 
  $\mS^{(3,0)}_{k,\infty}$ almost surely.
\end{lemma}

%comm{check the constant\tto seems okay 220309.2500}
\begin{proof}
For $k=1,2$, we have
\begin{align*}
  &\abs{ \bar{\mS}_{k,n}^{(3,0)}
  -\mS^{(3,0)}_{k,\infty} }
  \\\leq&
    \abs{ \int_{[0,T]^2} \mu_{n}(d\son,d\stw) 
    \rbr{\qtorker{k}{n}(\son,\stw) - \qtorker{k}{\infty}(\son,\stw)}}
  \\&+
    \abs{\int_{[0,T]^2} \mu_{n}(d\son,d\stw) \qtorker{k}{\infty}(\son,\stw)
    - \int_{[0,T]^2} \mu_{\infty}(d\son,d\stw) \qtorker{k}{\infty}(\son,\stw)}.
\end{align*}
The first term converges to zero almost surely
since the total mass of the (deterministic) measures $\mu_{n}$  on $[0,T]^2$
are bounded by $C_{(\ref{210416.2230})}$ (Lemma \ref{211020.2435} (i))
and $\qtorker{k}{n}$ converges to $\qtorker{k}\infty$ in the sup-norm on $[0,T]^2$ 
almost surely (Lemma \ref{210423.2110}).
We can prove that the second term also converges to zero almost surely
using the weak convergence $\mu_{n}\to\mu_{\infty}$  (Lemma \ref{211020.2435} (ii))
and the continuity of $\qtorker{k}\infty(\son,\stw)$ in $\son,\stw$.
Hence we have proved
%$r_n^{-1}\bar\cali^{M(3,0)}_{1,1,k,n}$
$\bar{\mS}_{k,n}^{(3,0)}$
converges to $\mS^{(3,0)}_{k,\infty}$ almost surely.
\end{proof}

\subsubsection{About the coefficients of random symbols from perturbation terms}
Define the functional $\bar{\mS}^{(1,0)}_{k,n}$ ($k=1,...,5$) by
\begin{align*}
  \bar{\mS}^{(1,0)}_{1,1,n}&=
  2^{-1}T^{-1} \int_0^T
  \rbr{\int^T_0 g'_t\; \dotd_n(\tau,t)dt}
  a'_{t_{j(\tau)-1}}\; \dotd_n(\tau,t_{j(\tau)-1}) d\tau
  \\
  \bar{\mS}^{(1,0)}_{1,2,n}&=
  \int_0^T
  \rbr{\int_0^T g'_t\;\dotd_n(\tau, t)dt} V^{[2,1]}_{t_{j(\tau)-1}} d\tau
  \\
  \bar{\mS}^{(1,0)}_{2,1,n}&=
  T^{-1}\int_0^T
  \rbr{a''_{t_{j(\tau)-1}}\;\brbr{\dotd_n(\tau,t_{j(\tau)-1})}^2
  +a'_{t_{j(\tau)-1}}\;\ddotd_n(\tau,t_{j(\tau)-1}) }d\tau
  \\
  \bar{\mS}^{(1,0)}_{2,2,n}&=
  2\int_0^T
  V^{[(2,1);1]}_{t_{j(\tau)-1}}\;\dotd_n(\tau,t_{j(\tau)-1}) d\tau
  \\
  \bar{\mS}^{(1,0)}_{2,3,n}&=N_{2,n}=
  T\int_0^T
  V^{[2,2]}_{t_{j(\tau)-1}} d\tau,
\end{align*}
and the functional $\bar{\mS}^{(1,0)}_{k,\infty}$ ($k=1,...,5$) by
\begin{alignat*}{3}
  {\mS}^{(1,0)}_{1,1,\infty}&=
  2^{-1}T^{-1} \int_0^T
  \rbr{\int^T_0 g'_t\; \dotd(\tau,t)dt}
  a'_{\tau}\; \dotd(\tau,\tau) d\tau
  ,&\quad%\\
  {\mS}^{(1,0)}_{1,2,\infty}&=
  \int_0^T
  \rbr{\int_0^T g'_t\;\dotd(\tau, t)dt} V^{[2,1]}_{\tau} d\tau
  \\
  {\mS}^{(1,0)}_{2,1,\infty}&=
  T^{-1}\int_0^T 
  \rbr{a''_{\tau}\;\brbr{\dotd(\tau,\tau)}^2 
  +a'_{\tau}\;\ddotd(\tau,\tau) }d\tau
  ,&%\\
  {\mS}^{(1,0)}_{2,2,\infty}&=
  2\int_0^T
  V^{[(2,1);1]}_{\tau}\;\dotd(\tau,\tau) d\tau
  \\
  {\mS}^{(1,0)}_{2,3,\infty}&=
  T\int_0^T V^{[2,2]}_{\tau} d\tau.
\end{alignat*}
Then we have
\begin{align*}
  \bar{\mS}^{(1,0)}_{1,n}&=
  \bar{\mS}^{(1,0)}_{1,1,n}+\bar{\mS}^{(1,0)}_{1,2,n},&
  \bar{\mS}^{(1,0)}_{2,n}&=
  \bar{\mS}^{(1,0)}_{2,1,n}+\bar{\mS}^{(1,0)}_{2,2,n}+\bar{\mS}^{(1,0)}_{2,3,n}
  \\
  {\mS}^{(1,0)}_{1,\infty}&=
  {\mS}^{(1,0)}_{1,1,\infty}+{\mS}^{(1,0)}_{1,2,\infty},&
  {\mS}^{(1,0)}_{2,\infty}&=
  {\mS}^{(1,0)}_{2,1,\infty}+{\mS}^{(1,0)}_{2,2,\infty}+{\mS}^{(1,0)}_{2,3,\infty},
\end{align*}
where 
$\bar{\mS}^{(1,0)}_{1,n}$, $\bar{\mS}^{(1,0)}_{2,n}$,
${\mS}^{(1,0)}_{1,\infty}$ and ${\mS}^{(1,0)}_{2,\infty}$
are defined 
%in page \pageref{220308.2521}.
at 
(\ref{220403.2151}), (\ref{220403.2152}),
(\ref{220403.2131}) and (\ref{220403.2132}), respectively.

\begin{lemma}\label{210505.1440}
  For $k=1,2$, 
  $\bar{\mS}^{(1,0)}_{k,n}$ converges to
  $\mS^{(1,0)}_{k,\infty}$ almost surely.
\end{lemma}

\begin{proof}
  %\contifrom{220307.1400}
  Consider $\bar{\mS}^{(1,0)}_{1,n}$.
By (\ref{210506.1501}) and the continuity of $X_t$ in $t$,
we have the following estimate
\begin{align*}
  \abs{g'_t\; \dotd_n(\tau,t)\; a'_{t_{j(\tau)-1}}\; \dotd_n(\tau,t_{j(\tau)-1})}
  \leq
  \sup_{t\in[0,T]}\abs{g'_t}
  \sup_{t\in[0,T]}\abs{a'_t}
  \bbrbr{\sup_{n\in\bbN, j\in[n], t\in[0,T]}\abs{\dotd_{n,j}(t)}}^2
  <\infty
\end{align*}
for $n\in\bbN,\;\tau\in[0,T],\;t\in[0,T]$ almost surely.
By (\ref{210506.1511}), (\ref{210506.1514}) and the continuity of $X_t$ in $t$,
\begin{align*}
  g'_t\; \dotd_n(\tau,t)\; a'_{t_{j(\tau)-1}}\; \dotd_n(\tau,t_{j(\tau)-1})
  \to g'_t\; \dotd(\tau,t)\; a'_{\tau}\; \dotd(\tau,\tau)
\end{align*}
for $(t,\tau)\in[0,T]^2$ as $n\to\infty$.
By the Lebesgue convergence theorem, 
we obtain the convergence
\begin{align*}
  \bar{\mS}^{(1,0)}_{1,n} \to {\mS}^{(1,0)}_{1,\infty}
\end{align*}
as $n\to\infty$ almost surely.

Similarly, by (\ref{210506.1501}), (\ref{210506.1511}) and the pathwise continuity of $X_t$,
we have
$\bar{\mS}^{(1,0)}_{2,n} \to {\mS}^{(1,0)}_{2,\infty}$.
By (\ref{210506.1501}), (\ref{210506.1503}), (\ref{210506.1514}), (\ref{210506.1516})
and the continuity of $X_t$, we have 
$\bar{\mS}^{(1,0)}_{3,n} \to {\mS}^{(1,0)}_{3,\infty}$.
By (\ref{210506.1501}), (\ref{210506.1514}) and the continuity of $X_t$,
we have
$\bar{\mS}^{(1,0)}_{4,n} \to {\mS}^{(1,0)}_{4,\infty}$ almost surely.
By the continuity of $X_t$, we have 
$\bar{\mS}^{(1,0)}_{5,n} \to {\mS}^{(1,0)}_{5,\infty}$ almost surely.
\end{proof}

%%\newpage
\subsubsection{Basic estimates}
Recall the definition of functions 
$\rho_{\tau}$ and $\rho_{n,j}$ 
defined on $[0,T]$:
\begin{align*}
  \rho_{\tau}(s)=&\alpha_H\,T\; \abs{s-\tau}^{2H-2}
  &&\tfornsp \tau\in[0,T],
  \\
  \rho_{n,j}(s)=&%\rho_{j}(s)=
  \alpha_H\,n \int_{s'\in I_{j}} \abs{s-s'}^{2H-2} ds'
  =\rbr{\frac{T}{n}}^{-1}\int_{s'\in I_{j}} \rho_{s'}(s)ds'
  %\rho_{j,n}(s)=&%\rho_{j}(s)=
  %\alpha_H\; n\int_{s'\in I_{j}} \abs{s-s'}^{2H-2}ds'
  &&\tfornsp j\in[n].
\end{align*}
\begin{lemma}\label{210507.1804}
(1) %There exists a constant $C_{(\ref{210507.1725})}>0$ independent of $n$ and $\tau$ such that 
The functions $\rho_{\tau}$ and $\rho_{n,j}$ are bounded in $L^1([0,T])$, that is
\begin{align}
  C_{(\ref{210507.1725})}:=
  \sup_{\tau\in[0,T]} \snorm{\rho_{\tau}}_{L^1}\vee
  \sup_{n\in\bbN,j\in[n]} \snorm{\rho_{n,j}}_{L^1}
  %\sup_{\tau\in[0,T],n\in\bbN} \snorm{\rho_{j_n(\tau),n}}_{L^1}
  <\infty.
  %\leq C_{(\ref{210507.1725})}.
  \label{210507.1725}
\end{align}
(2)
The function $\rho_{n,j_n(\tau)}$ on $[0,T]$ converges to $\rho_{\tau}$
in $L^1([0,T])$ uniformly in  $\tau\in[0,T]$.
In other words,
\begin{align*}
\lim_{n\to\infty}\sup_{\tau\in[0,T]}
\snorm{\rho_{n, j_n(\tau)}-\rho_{\tau}}_{L^1}=0.
\end{align*}
\end{lemma}
\begin{proof}
  (1)
Let $\rho(s)=\alpha_H\:T \;\abs{s}^{2H-2}$ for $s\in[-T,T]$.
We have
\begin{align*}%\label{210507.1725}
  \snorm{\rho_{\tau}}_{L^1} =\int^T_0 \rho_\tau(s)ds 
  =\int^{T-\tau}_{-\tau}\rho(s)ds
  <\int^{T}_{-T}\rho(s)ds<\infty
  \tfor \tau\in[0,T].
\end{align*}
Similary, by the Fubini's theorem,
\begin{align*}
  \snorm{\rho_{n,j}}_{L^1} =\int^T_0 \rho_{n,j}(s) ds
  =\rbr{\frac{T}{n}}^{-1}
  \int^{t_{j}}_{t_{j-1}}ds'\int^T_0 \rho_{s'}(s) ds
  %=T^{-1} n \int^T_0 \int^{t_{j}}_{t_{j-1}} \rho(s-s') ds'ds
  <\int^{T}_{-T}\rho(s)ds<\infty
  \tfor j\in[n] \tandsm n\in\bbN.
\end{align*}

\vspssm
%By the integrability of $\rho$ (in $s=0$),
\noindent(2) 
For $\tau\in[0,T]$, $M>0$ and $n\in\bbN$,
define functions on $[0,T]$ by
\begin{align*}
  \rho^M_{\tau}(s)=\rho_\tau(s)\wedge M \tand
  \rho_{n,j}^M(s)=T^{-1}n\int_{s'\in I_{j}}\rho^M_{s'}(s)ds'.
  %=T^{-1}n\int_{s'\in I_{j}}(\rho(s-s')\wedge M)ds'
\end{align*}
%Note that $I_j$ implicitly depends on $n$.
For $\epsilon>0$, there exists $M_\epsilon>0$ such that 
\begin{align*}
  \int_{-T}^T\rho(s)ds - \int_{-T}^T\rho(s)\wedge M_\epsilon\, ds <\epsilon.
\end{align*}
Since $\rho(s)\wedge {M_\epsilon}$ is uniformly continuous in $s\in[-T,T]$,
there exists $n_\epsilon\in\bbN$ %$N_\epsilon\in\bbN$ 
such that 
\begin{align*}
  \abs{\rho(s_1)\wedge {M_\epsilon} - \rho(s_2)\wedge {M_\epsilon}}<\epsilon/T
\end{align*}
for $s_1,s_2\in[-T,T]$ satisfying $\abs{s_1-s_2}\leq T/n_\epsilon$. %$|s_1-s_2|\leq T/N_\epsilon$.
Then for $n\geq n_\epsilon$ and $\tau\in[0,T]$, we have
\begin{align*}
  \snorm{\rho_{n,j_n(\tau)}^{M_\epsilon}-\rho_\tau^{M_\epsilon}}_{L^1}
  &=
  \bbnorm{T^{-1}n\int_{s'\in I_{j_n(\tau)}}
\rbr{\rho_{s'}(\cdot)\wedge M_\epsilon - \rho_{\tau}(\cdot)\wedge M_\epsilon}ds'}_{L^1}
  \\
  &\leq T^{-1}n\int_0^T\int_{s'\in I_{j_n(\tau)}}
  \babs{\rho_{s'}(s)\wedge M_\epsilon - \rho_{\tau}(s)\wedge M_\epsilon}ds'ds
  \leq\epsilon,
\end{align*}
since $|s'-\tau|\leq T/n_\epsilon$ for $s'\in I_{j_n(\tau)}$.
Note that this estimate holds uniformly in $\tau$.
We also have
\begin{align*}
  \snorm{\rho_{n,j}-\rho_{n,j}^{M_\epsilon}}_{L^1}
  =&
  %\snorm{T^{-1}n\int_{s'\in I_{j}}\rho_{s'}(\cdot)ds'
  %-T^{-1}n\int_{s'\in I_{j}}\rho_{s'}^{M_\epsilon}(\cdot)ds'}_{L^1}\\
  %=&
  T^{-1}n\int_0^T
  \rbr{\int_{s'\in I_{j}}\rbr{\rho_{s'}(s)-\rho_{s'}(s)\wedge {M_\epsilon}}ds'}ds
  <\epsilon
  &&\text{for any }j\in[n]\tandsm n\in\bbN
  \\
  \snorm{ \rho_\tau^{M_\epsilon} -\rho_\tau}_{L^1}
  <&\int_{-T}^T \rbr{\rho(s) -\rho(s)\wedge M_\epsilon}ds<\epsilon
  &&\text{for any }\tau\in[0,T].
\end{align*}
By the inequality
\begin{align*}
  \snorm{ \rho_{n,j_n(\tau)}-\rho_\tau }_{L^1}\leq
  \snorm{\rho_{n,j_n(\tau)}-\rho_{n,j_n(\tau)}^{M_\epsilon}}_{L^1}
  +\snorm{\rho_{n,j_n(\tau)}^{M_\epsilon}-\rho_\tau^{M_\epsilon}}_{L^1}
  +\snorm{\rho_\tau^{M_\epsilon}-\rho_\tau}_{L^1}
\end{align*}
for $\tau\in[0,T]$ and $n\in\bbN$, we obtain the convergence.
\end{proof}

\begin{lemma}\label{220415.1540}
(1) The following bounds hold almost surely:
\begin{alignat}{2}
\sup_{n\in\bbN}\sup_{j\in[n],t\in[0,T]} \abs{\dot d_n(j;t)}
%\sup_{n\in\bbN}\sup_{\tau,t\in[0,T]} \abs{\dot d_n(\tau;t)}
\vee \sup_{\tau,t\in[0,T]}\abs{\dotd(\tau;t)}
<&\infty &\quad&a.s.		\label{210506.1501}\\
\sup_{n\in\bbN}\sup_{j\in[n],t\in[0,T]} \abs{\ddot d_n(j;{t})}
%\sup_{n\in\bbN}\sup_{\tau,t\in[0,T]} \abs{\ddot d_n(\tau;{t})}
\vee \sup_{\tau,t\in[0,T]}\abs{\ddotd(\tau;{t})}
<&\infty &&a.s.			\label{210506.1503}
%\\
%\sup_{n\in\bbN}\sup_{\tau\in[0,T]} \abs{\int_{[0,T]} g'(X_t)\: \dot d_n(\tau;t)  dt}
%<&\infty &&a.s.		\nn%\label{210506.1504}
\end{alignat}

\item (2)
The following convergences hold almost surely as $n$ goes to $\infty$:
\begin{alignat}{2}
\sup_{\tau,t\in[0,T]} \abs{\dot d_n(\tau;t) -\dotd(\tau;t)}&\to0 &\quad&a.s.
\label{210506.1511}\\
\sup_{\tau,t\in[0,T]} \abs{\ddot d_n(\tau;t) -\ddotd(\tau;t)}&\to0 &&a.s.
\label{210506.1512}\\
\abs{\dotd_n(\tau,t_{j(\tau)-1}) -\dotd(\tau,\tau)} &\to0 
&&\text{ for any $\tau\in[0,T]$  a.s.}			\label{210506.1514}\\
\abs{\ddotd_n(\tau,t_{j(\tau)-1}) -\ddotd(\tau,\tau)} &\to0
&&\text{ for any $\tau\in[0,T]$  a.s.}			\label{210506.1516}
%
%
%\abs{\dotd(\tau;t_{j(\tau)-1}) -\dotd(\tau;\tau)} &\to0 
%&&\text{ for any $\tau\in[0,T]$  a.s.}			\\%\label{210506.1514}\\
%\abs{\ddotd(\tau;t_{j(\tau)-1}) -\ddotd(\tau;\tau)} &\to0
%&&\text{ for any $\tau\in[0,T]$  a.s.}			\\%\label{210506.1516}
\end{alignat}
In partcular, 
\begin{align}
  \int_{0}^T \abs{\dotd_n(\tau,t_{j(\tau)-1}) -\dotd(\tau,\tau)}d\tau &\to0 \quad a.s.
  \label{220308.1200}
\end{align}
\item (3)
For fixed a.s. $\omega\in\Omega$, $n\in\bbN$ and $j\in[n]$,
$\dotd_{n,j}$ and $\ddotd_{n,j}$ are continuous functions on $[0,T]$.

\item (4)
Let $f$ a continuous function on $\bbR$.
Then the following integral converges to $0$ as $n\to\infty$.
\begin{align}
\int_{0}^T \abs{f(X_{\tau})-f(X_{t_{j(\tau)-1}})} d\tau &\to0 
\quad a.s. \label{210506.1517}
%\\
%\int_{0}^T \abs{\dotd(\tau;\tau)-\dotd(\tau;t_{j(\tau)-1})}d\tau &\to0 
%\quad a.s.\text{--------- to be deleted - NAZE} 
%\nn\\
%\label{210506.1513}\\
%\sup_\tau
%\abs{\int_{t\in[0,T]} g'(X_t)\;\dotd_n(\tau;t)dt -\int_{[0,T]}g'(X_t)\; \dotd(\tau;t) dt} &\to0 
%\quad a.s. \nn%\label{210506.1515}
\end{align}
\end{lemma}

\begin{proof} 
  (1) The following estimate holds almost surely
  \begin{align*}
    \abs{\dot d_n(\tau;t)} &\leq
    \int_{[0,T]}\abs{D_sX_t}\; \abs{\rho_{j(\tau),n}(s)} ds
    %\leq \sup_{s,t\in[0,T]}\abs{D_sX_t} \snorm{\rho_{j(\tau),n}}_{L^1([0,T])}
    \leq C_{(\ref{210507.1725})} \sup_{s,t\in[0,T]}\abs{D_sX_t}
    \\
    \abs{\ddot d_n(\tau;t)} &=
    \int_{[0,T]^2} \abs{D^2_{\son,\stw} X_{t}}\;
    \abs{\rho_{j(\tau),n}(\son)}\: \abs{\rho_{j(\tau),n}(\stw)} d\son d\stw
    \leq 
    (C_{(\ref{210507.1725})})^2
    \sup_{\son,\stw,t\in[0,T]} \abs{D^2_{\son,\stw} X_{t}}\:
  \end{align*}
  for any $t,\tau\in [0,T]$ and  $n\in\bbN$.
  Similarly,  for $t,\tau\in [0,T]$, we have
  \begin{align*}
  \abs{\dotd(\tau;t)}\leq
  C_{(\ref{210507.1725})} \sup_{s,t\in[0,T]}\abs{D_sX_t},\qquad%\tand
  \abs{\ddotd(\tau;t)}\leq 
  (C_{(\ref{210507.1725})})^2
  \sup_{\son,\stw,t\in[0,T]}\abs{D^2_{\son,\stw} X_{t}}.
  \end{align*}
  Since
  $ \sup_{s,t}\abs{D_sX_t}$ and $ \sup_{\son,\stw,t}\abs{D^2_{\son,\stw} X_{t}}$ 
  are almost surely finite by Proposition \ref{210507.1755},
  we obtain (\ref{210506.1501}) and (\ref{210506.1503}).
  %\comm{The estimate (\ref{210506.1504}) follows from (\ref{210506.1501}).}
  %% MEMO: using the inequality:
  %\beas
  %\bigg|\int_{t\in[0,T]}  g'(X_t)\dot d_n(\tau;t)  dt\bigg|\leq
  %T\cdot \sup_t |g'(X_t)| \sup_{s,t}|D_sX_t| \cdot C^{(\ref{210507.1725})}
  %\hspsm\text{a.s.}
  %\eeas
  %for $n\in\bbN$ and $\tau\in[0,T]$.

\vspssm\noindent (2)
Since the following inequalities stand almost surely for $\tau,t\in[0,T]$ and $n\in\bbN$:
\begin{align*}
  \abs{\dot d_n(\tau;t)-\dotd(\tau;t)}&=
  \abs{\int_{[0,T]}D_sX_t\; \rbr{\rho_{j(\tau),n}(s) -\rho_{\tau}(s)}ds}\leq
  \sup_{s,t\in[0,T]}\abs{D_sX_t}\;
  \sup_{\tau\in[0,T]}\snorm{\rho_{j_n(\tau),n}-\rho_{\tau}}_{L^1([0,T])}
  \\
  \abs{\ddot d_n(\tau;t)-\ddotd(\tau;t)}
  &\leq
  2\sup_{\son,\stw,t\in[0,T]}\abs{D^2_{\son,\stw} X_{t}}\;
  %\sup_\tau \snorm{ \rho_{\tau}(s) }_{L^1(s\in[0,T])}
  \sup_{\tau\in[0,T]}\snorm{\rho_{j_n(\tau),n}-\rho_{\tau}}_{L^1([0,T])}\;
  \bbrbr{\sup_{\tau\in[0,T]} \snorm{\rho_{\tau}}_{L^1}\vee
  \sup_{n\in\bbN,j\in[n]} \snorm{\rho_{j,n}}_{L^1}},
  %\:C_{(\ref{210507.1725})},
\end{align*}
we obtain (\ref{210506.1511}) and (\ref{210506.1512})
by Lemma \ref{210507.1804}.

%%\newpage
For $\tau\in[0,T]$ and $n\in\bbN$, we have
\begin{align*}
\abs{\dotd_n(\tau,t_{j_n(\tau)-1}) -\dotd(\tau,\tau)}
&\leq
\abs{\dotd_n(\tau,t_{j_n(\tau)-1}) -\dotd(\tau,t_{j_n(\tau)-1})}
+\abs{\dotd(\tau,t_{j_n(\tau)-1}) -\dotd(\tau,\tau)}
\\&\leq
\sup_{\tau,t\in[0,T]} \abs{\dotd_n(\tau,t) -\dotd(\tau,t)}+
\int_0^T \abs{D_sX_{t_{j_n(\tau)-1}} -D_sX_\tau}\; \rho_{\tau}(s) ds
\end{align*}
The first term converges to zero almost surely by (\ref{210506.1511}).
Consider the second term.
For $s<\tau$, 
$D_sX_{t_{j(\tau)-1}}$ converges to $D_sX_\tau$ as $n\to\infty$,
since $D_sX_t$ is continuous in $t$ except for $t=s$;
For $s>\tau$, we have $D_sX_{t_{j(\tau)-1}}=D_sX_\tau=0$.
The integrand is bounded by $2\sup_{s,t\in[0,T]}\abs{D_sX_t}\;\rho_\tau(s)$,
which is an integrable function in $s$.
Hence, by Lebesgue's dominated convergence theorem,
the second term converges to zero almost surely,
and we obtain (\ref{210506.1514}).
%\comm{uniform continuity of $D_sX_t$ in $s,t$ can be used to obtain uniform convergence in $\tau$, 
%which is stronger and not necessary here.}
By similar arguments, we can prove (\ref{210506.1516}) and
the continuity of $\dotd_{n,j}$ and $\ddotd_{n,j}$ .
%$\abs{\ddotd_n(\tau,t_{j(\tau)-1}) - \ddotd(\tau,\tau)}\to0$  for any $\tau\in[0,T]$  a.s.	

Thanks to Lebesgue's dominated convergence theorem,
we have (\ref{220308.1200}) from (\ref{210506.1501}) and (\ref{210506.1514}); 
(\ref{210506.1517}) is proved by the continuity of $f$ and $X_t$ in $t$.
\end{proof}

%\subsubsection{lemmas}

\begin{lemma}\label{211020.2435}
(i) Let $f:[0,T]^2\to\bbR$  a bounded continuous function.
Then there exists a constant $C_{(\ref{210416.2230})}>0$ such that
\begin{align}
  \sup_{n\in\bbN} \abs{\int_{[0,T]^2} \mu_n(ds_1,ds_2) f(s_1,s_2)}
  \leq C_{(\ref{210416.2230})} \snorm{f}_\infty
  \label{210416.2230}
\end{align}

\item
(ii) The measure $\mu_n$ on $[0,T]^2$ weakly converges to $\mu_\infty$.
%where $\mu_\infty$
%\begin{align*}
%\int_{[0,T]^2}\mu_\infty(ds_1,ds_2) f(s_1,s_2) 
%=2T^{4H-1}c_H^2\int_{[0,T]} f(t,t)dt
%\end{align*}
\end{lemma}
We omit the proof.

\subsection{Lemmas related to the exponents}
%\comm{Add proofs}
Recall that we introduce a class $\cala$ of families of functionals
and a class $\calf$ of families of elements in $\abs\calh$.
See (\ref{220119.1721}) and (\ref{210518.0001}).
Since the following lemmas are elementary,
we omit the proof.
\begin{lemma}\label{201229.1452}
  (1) (a) 
  Let $V$ a nonempty finite set, $v\in V$ and $i\in\bbN$.
  Consider $A=(A_n(j))_{j\in[n]^V, n\in\bbN}\in\cala(V)$ and 
  $\bbf=(\bbf_v)\in\calf(\cbr{v})$.
  Define $A'=(A'_n(j))_{j\in[n]^{V},n\in\bbN}$ by 
  \begin{align*}
    A'_n(j)=
    n^i\Babr{D^i A_n(j),\kerfvn{v}{j_v}^{\otimes i}}_{\calh^{\otimes i}}.
  \end{align*}
  Then $A'$ belongs to $\cala(V)$.

\begin{scrap}
  (1) (a) 
  Let $V$ a nonempty finite set, $\cbr{v_0}\notin V$ a singleton and $i\in\bbN$.
  Consider $A=(A_n(j))_{j\in[n]^V, n\in\bbN}\in\cala(V)$ and 
  $\bbf=(\bbf_v)_{v\in\cbr{v_0}}\in\calf(\cbr{v_0})$.
  Define $A'=(A'_n(k))_{k\in[n]^{V\sqcup\cbr{v_0}},n\in\bbN}$ by 
  \begin{align*}
    A'_n(k)=
    n^i\Babr{D^i A_n(k_{V}), f_{v_0,n;k_{v_0}}^{\otimes i}}_{\calh^{\otimes i}}.
  \end{align*}
  Then $A'$ belongs to $\cala(V\sqcup\cbr{v_0})$.

  (b-1) Let $V$ and $V'$ two disjoint nonempty finite sets.
  Consider $A\in\cala(V)$ and $A'\in\cala(V')$.
  The family of functionals $A''=(A''_n(k))_{k\in[n]^{V\sqcup V'},n\in\bbN}$ defined by 
  \begin{align*}
    A''_n(k) = A_n(k_V) A'_n(k_{V'})
  \end{align*}
  belongs to $\cala(V\sqcup V')$.

\end{scrap}

\item 
(b-1) Let $V\subset V'$ two nonempty finite sets.
Consider $A\in\cala(V)$.
The family of functionals $A'=(A'_n(j))_{j\in[n]^{V'},n\in\bbN}$ defined by 
\begin{align*}
  A'_n(j) = A_n(j_V)
\end{align*}
belongs to $\cala(V')$.

\item 
(b-2) Let $V$ a nonempty finite set.
Consider $A,A'\in\cala(V)$.
The family of functionals $A''=(A''_n(j))_{j\in[n]^{V},n\in\bbN}$ defined by 
\begin{align*}
  A''_n(j) = A_n(j) A'_n(j)
\end{align*}
belongs to $\cala(V)$.

\item
(c) Let $V$ and $V'$ two disjoint nonempty finite sets and $i\in\bbN$.
For $A\in\cala(V)$ and $A'\in\cala(V')$, 
the family of functionals $A''=(A''_n(j))_{j\in[n]^{V\sqcup V'},n\in\bbN}$ defined by 
\begin{align*}
  A''_n(j) = \abr{D^i A_n(j_V), D^i A'_n(j_{V'})}_{\calh^{\otimes i}}
\end{align*}
belongs to $\cala(V\sqcup V')$.

\vspssm\item
(2)
For $i=1,2$, $n\in\bbN$ and $j_1,j_2\in[n]$, consider 
$\kerfvn{i}{j_i}\in\abs\calh$
%$f_{i,n;j_i}\in\abs\calh$
such that $\supp(\kerfvn{i}{j_i})\subset [(j_i-1)T/n, j_iT/n]$ 
%such that $\supp(f_{i,n;j_i})\subset [(j_i-1)T/n, j_iT/n]$ 
and that there exists a constant $C>0$ such that
\begin{align*}
  \max_{i=1,2} \sup_{n\in\bbN}
    \sup_{j_i\in[n]} \sup_{t\in[0,T]} \abs{\kerfvn{i}{j_i}(t)}
  \leq C.
\end{align*}
Then, the following bound holds:
\begin{align}\label{211019.2220}
  \sup_{n\in\bbN}\sup_{j_1,j_2\in[n]}
  \abs{\frac{\abr{\kerfvn{1}{j_1}, \kerfvn{2}{j_2}}}{\beta_n\rbr{j_1,j_2}}}
  %\abs{\frac{\langle f_{1,n;j_1}, f_{2,n;j_2}\rangle}{\beta_n\rbr{j_1,j_2}}}
  \leq C^2.
\end{align}
\end{lemma}

\begin{lemma}\label{211031.2620}
  Consider a nonempty finite set $V$, 
  a function $\vertWt:V\to \bbZ_{\geq0}$ and 
  $f_v\in \calh$ for each $v\in V$.
  We set 
  \begin{align*}
  \Pi(\vertWt)=
  \cbr{\pi:p(V)\to\bbZ_{\geq0}\mid
  \vertWt(v)\geq\pi_v \tforsm v\in V}
  \end{align*}
  with
  $\pi_v = \sum_{v'\in V,v'\neq v}\pi([v,v'])$. % for $v\in V$.
  We write
  $\bar\pi = \sum_{[v,v']\in p(V_{c_1})}\pi([v,v'])$ and 
  $\bar\vertwtlow = \sum_{v\in V}\vertWt(v)$.

  Then the following equality holds:
  \begin{align*}
    &\prod_{v\in V} I_{\vertWt(v)}(f_v^{\otimes \vertWt(v)})
    = \prod_{v\in V} \delta^{\vertWt(v)}(f_v^{\otimes \vertWt(v)})
    \\=&
    \sum_{\pi\in\Pi(\vertWt)} c(\pi)\,
      \delta^{(\bar\vertwtlow - 2\bar\pi)}
      %\delta^{(\bar\vertwtlow - 2\sum_{[v,v']\in p(V)}\pi([v,v']))}
      \brbr{\subotimes{v\in V} (f_v^{\otimes(\vertwtlow_v - \pi_{v})})}\;
      \prod_{[v,v']\in p(V)} \abr{f_{v}, f_{v'}}^{\pi([v,v'])},
  \end{align*}
  where $c(\pi)$ is a integer-valued constant.
  We read $I_{q}(f^{\otimes q})= 1$ if $q=0$ and 
  $\delta^{(\bar\vertwtlow - 2\bar\pi)}
  \brbr{\subotimes{v\in V} (f_v^{\otimes(\vertwtlow_v - \pi_{v})})}=1$
  if $\bar\vertwtlow - 2\bar\pi=0$ by convention.
  When $\abs{V}=1$, hence $p(V)=\emptyset$, 
  we consider $\Pi(\vertWt)$ is a singleton 
  with the function $\pi:\emptyset\to\bbZ_{\geq0}$, 
  and read $\pi_v=0$ for $\cbr{v}=V$ and $\bar\pi=0$.
\end{lemma}

\begin{lemma} \label{211120.1740}
  Let $V$ a finite set.
  Suppose $(F_v)_{v\in V}$ such that 
  $F_v\in\bbD^\infty$ and
  can be written as 
  $D^iF_v = (D^i_{s_i,..,s_1}F_v)_{s_1,..,s_i\in[0,T]}$
  for $i\in\bbZ_{\geq1}$.
  Then for $i\in\bbZ_{\geq1}$,
  the $i$-th derivative of $\prod_{v\in V} F_v$ can be written as:
  \begin{align}
    D^i \prod_{v\in V} F_v
    =
    \rbr{
      \sum_{\lambda\in V^{i}} \prod_{v\in V} 
      D^{\ilamv}_{s_{\lambda,v}} F_v
    }_{\son,..,s_i\in[0,T]},
  \end{align}
  where the above summation runs through 
  the set $V^i$ of all the mappings $\lambda$ from $\cbr{1,..,i}$ to $V$,
  and %for $v\in V$ 
  we denote 
  $\ilamv:=\abs{\lambda^{-1}(\cbr{v})}$.
  $D^{\ilamv}_{s_{\lambda,v}} F_v$ reads $F_v$ if $\ilamv=0$, and 
  $D^{\ilamv}_{s_{\lambda,v}} F_v$ reads
  $D^{\ilamv}_{s_{k_1},..,s_{k_{\ilamv}}} F_v$
  if $\lambda^{-1}(\cbr{v})=\cbr{k_1,..,k_{\ilamv}}\subset\cbr{1,..,i}$.
\end{lemma}
%$\lambda[v]:=\lambda^{-1}(\cbr{v})$ and
%$D^{|\lambda[v]|}_{s_{\lambda[v]}}$ means
%$D^{i_v}_{s_{k_1},..,s_{k_{i_v}}}$ 

%\bibliographystyle{plain}
%\bibliography{bibtex20080401}
%\bibliography{bibtex-20200418-20200531-20201001}   % name your BibTeX data base
%\bibliography{../../_general/_bib/bibtex-20200418-20200531-20201001} 
%\bibliography{../../../_bib/bibtex220423}
%\bibliography{../../../_bib/bibtex220423-220512}
\bibliography{bibtex220423-220512}
\end{document}